\documentclass[reqno,11pt]{amsart}

\usepackage[margin=1.1in]{geometry}
\usepackage[linktocpage]{hyperref}
\hypersetup{hidelinks}
\usepackage[us,12hr]{datetime}
\usepackage{enumitem}
\usepackage{cases}
\usepackage{mathrsfs}
\usepackage[T1]{fontenc}
\usepackage{mathrsfs}
\usepackage{amsmath,latexsym,amssymb,amsfonts,amsbsy, amsthm, cite}
\usepackage{bm}
\usepackage[usenames]{color}
\usepackage{xspace,colortbl}
\usepackage{epsfig}
\usepackage{graphicx}
\usepackage{subfigure}
\usepackage{amsmath,amsfonts,amsthm,amssymb,amscd}
\usepackage{color}
\usepackage[title,titletoc,toc]{appendix}
\usepackage{bigints}
\usepackage{cases}\usepackage{extarrows}
\usepackage{setspace}
\usepackage{titletoc}
\usepackage{color}
\usepackage{threeparttable}
\usepackage{array}
\usepackage{booktabs}
\usepackage{float}
\usepackage{tikz}
\usepackage{epstopdf}
\usetikzlibrary{patterns,decorations.pathreplacing}
\allowdisplaybreaks

\numberwithin{equation}{section}

\theoremstyle{plain}
\newtheorem{theorem}{Theorem}[section]
\newtheorem{corollary}{Corollary}[section]

\newtheorem{lemma}{Lemma}[section]

\theoremstyle{definition}

\theoremstyle{remark}

\newtheorem{remark}{Remark}[section]
\numberwithin{equation}{section}

\begin{document}

\title[Semilinear Schr\"{o}dinger equation with harmonic potential and nonlinearities]
{Semilinear Schr\"{o}dinger equation with harmonic potential and combined power-type nonlinearities}

\author[J. Han]{Jiangbo Han}

\address{Jiangbo Han\newline
College of Mathematical Sciences, Harbin Engineering University, 150001 Harbin, China.}

\email{han1825141364@163.com,  2014112111@hrbeu.edu.cn}

\author[R. Xu]{Runzhang Xu}

\address{Runzhang Xu\newline
College of Mathematical Sciences, Harbin Engineering University, 150001 Harbin, China;\newline
The Institute of Mathematical Sciences, The Chinese University of Hong Kong, Shatin, NT, Hong Kong, China;\newline
Sanya Nanhai Innovation and Development Base of Harbin Engineering University, Sanya, China.}

\email{xurunzh@hrbeu.edu.cn; rzxu@ims.cuhk.edu.hk; xurunzh@163.com}

\author[C. Yang]{Chao Yang}

\address{Chao Yang\newline
The Institute of Mathematical Sciences, The Chinese University of Hong Kong, Shatin, NT, Hong Kong, China;\newline
Faculty of Applied Mathematics, AGH University of Krakow, Krakow {\rm 30-059}, Poland.}

\email{cyang@ims.cuhk.edu.hk; chaoyang@agh.edu.pl; yangchao\_yc98@163.com}

\begin{abstract}
We study the Cauchy problem for the semilinear Schr\"{o}dinger equation with a harmonic potential and combined power-type nonlinearities. The goal is to classify initial data that lead to either global existence or blowup solutions by using an enhanced cross-constrained variational theory. First, we review and analyze this variational approach, pinpointing the main issue for our specific problem. We then introduce new variational structures by carefully selecting terms in the energy functional and adjusting their weights, enabling us to create eight cross-constrained variational combinations and related manifolds to distinguish initial data that result in global solutions or finite-time blowup. Using these structures, we examine three initial energy levels: negative, non-negative low, and high energy. Notably, we prove that for any initial data with negative energy, solutions blow up in finite time regardless of their size. Throughout the study, we utilize these new variational frameworks and functionals to systematically classify initial data for global existence or finite-time blowup. We highlight the comprehensiveness of our approach as we analyze numerous combinations of the energy and mass terms, despite the complexity involved. Interestingly, we find that when initial energy is minimized, classifications based on different functionals coincide. For high initial energy, an unresolved issue even for the standard Schr\"{o}dinger equation, we develop a strategy controlling bounds of energy to classify the solution dynamics, including global existence or blowup, for our model. We also provide a complete proof for the standard case addressing this open problem.
\end{abstract}



\maketitle

\tableofcontents

\section{Introduction}\label{sec1}
     In this paper, we consider the Cauchy problem of the semilinear Schr\"{o}dinger equation with the harmonic potential and combined power-type nonlinearities:
\begin{equation}\label{1.1}
\begin{cases}
\displaystyle i \varphi_{t}+\Delta \varphi-|x|^{2}\varphi+\lambda_{1}|\varphi|^{p-2}\varphi+\lambda_{2}|\varphi|^{q-2}\varphi=0, \ \ x\in \mathbb{R}^n, \ t>0,\\[2pt]
\displaystyle \varphi(x,0)=\varphi_0(x), \ \ \ \ \ \ \ \ \ \ \ \ \ \ \ \ \ \ \ \ \ \ \ \ \ \ \ \ \ \ \ \ \ \ \ \ \ \ \ \ \ \ \ x\in \mathbb{R}^n,
\end{cases}
\end{equation}
where ${i}$ is the imaginary unit, $\varphi(x,t)$ is the unknown complex function, ${p}$ and ${q}$ satisfy
\begin{equation}\label{1.2}
\begin{cases}
2+\frac{4}{n}<q<p<\frac{2n}{n-2}, &n\geq3,\\[2pt]
2+\frac{4}{n}<q<p<+\infty, &n=1,2,\\
\end{cases}
\end{equation}
and $\lambda_{1}>0$, $\lambda_{2}>0$ are constants.

\subsection{Physical background.} \label{sec1.1}
The equation in problem (\ref{1.1}) considered in this paper originates from various well-known mathematical and physical models.
For example,  to describe the Bose-Einstein condensation, Gross \cite{r13-2-1,r13-2-2} and Pitaevskii \cite{r13-2-3} independently introduced the so-called Gross-Pitaevskii equation:
\begin{equation}\label{1.3}
i\hbar\varphi_{t}+\frac{\hbar^{2}}{2m}\Delta \varphi-V_{ext}\varphi-g|\varphi|^{2}\varphi=0,
\end{equation}
where $\varphi$ is the condensate wave function, $\hbar$ is the Plank constant, $m$ is the mass of particle,
\begin{align*}
g:=\frac{4\pi \hbar^{2}a}{m}
\end{align*}
is related to the scattering length $a$, and  the harmonic potential term
\begin{equation*}
V_{ext}(r):=\frac{m}{2}\omega^{2}r^{2}
\end{equation*}
is used for the approximation of the confining potential of portraying the magnetic traps for alkali atoms (cf. \cite{r13-2}), a condensate of
$N$ particles of the effective scattering length $a$ and the mass $m$ (cf. \cite{r13-3}). The sign of the coefficient $g$ of the nonlinearity in (\ref{1.3}) depends on the species of the atom used for the Bose-Einstein condensation. The case $g>0$ (e.g., for $^{87}$Rb and $^{3}$Na atoms) corresponds to repulsive interatomic interactions, in which the nonlinearity in the Gross-Pitaevskii equation is defocusing (repulsive), and the case $g<0$ (e.g., for
$^{7}$Li and $^{85}$Rb atoms) corresponds to attractive interatomic interactions, in which the corresponding nonlinearity is
focusing (attractive) \cite{focusinganddefocusing}. A generalized version of this equation of the following form:
\begin{equation}\label{1.4}
i\hbar\varphi_{t}+\frac{\hbar^{2}}{2}\Delta \varphi-V(x)\varphi+|\varphi|^{p-2}\varphi=0,
\end{equation}
was also considered in \cite{r1}.  If $p=4$, (\ref{1.4}) turns to the Gross-Pitaevskii equation. To describe the Bose-Einstein condensation with both the two-body interaction and three-body
interaction, the Gross-Pitaevskii equation takes the following form:
\begin{equation}\label{1.5}
i\hbar\varphi_{t}+\frac{\hbar^{2}}{2m}\Delta \varphi-V_{ext}\varphi+k_{1}|\varphi|^{2}\varphi+k_{2}|\varphi|^{4}\varphi=0,
\end{equation}
with two nonlinear terms, where the coefficients of the nonlinear terms $k_{1}$ and $k_{2}$ denote, respectively, the strength of the
two-body and three-body effective interactions \cite{r13-3-0001,r13-3-0002}. By considering the nonlinear source term of a more general form $f(\varphi)$ instead of the above special forms, the following  equation:
\begin{equation}\label{1.6}
i\hbar\varphi_{t}+\frac{\hbar^{2}}{2}\Delta \varphi-V(x)\varphi+f(\varphi)=0
\end{equation}
was considered in \cite{r2, r4, r4-1}. In (\ref{1.6}), the general nonlinearity satisfies $f(e^{i\theta}s)=e^{i\theta}f(s)$ for $s\in\mathbb{R}$, which includes the combined power-type nonlinearities
\begin{align*}
f(\varphi)=\lambda_{1}|\varphi|^{p-2}\varphi+\lambda_{2}|\varphi|^{q-2}\varphi
\end{align*}
considered in the present paper as special cases.

\subsection{Semilinear Schr\"{o}dinger equation.}\label{sec1.2}
Next, we recall some known results about the semilinear Schr\"{o}dinger equation with the power-type nonlinearity. According to the sign (positivity or negativity) of the nonlinear source term, we classify the nonlinear term of the Schr\"{o}dinger equation describing Bose-Einstein condensation into either the focusing version or the defocusing version, as introduced in Subsection \ref{sec1.1}. As the present paper is concerned with the focusing case of the Schr\"{o}dinger equation, we shall first briefly review some related works   on the defocusing case.

For the defocusing nonlinear Schr\"{o}dinger equation
\begin{align}\label{defocusing}
i\varphi_{t}+\Delta \varphi-|\varphi|^{p-2}\varphi=0,
\end{align}
in the case $2<p<\frac{2n}{n-2}$, Ginibre and Velo \cite{r4-1-2-0} proved that the solution exists globally for any initial data in $H^{1}(\mathbb{R}^{n})$. For the case $p=\frac{2n}{n-2}$,
  Colliander,  Keel, Staffilani, Takaoka,  and Tao  \cite{Annals-2008} and Ryckman and Visan \cite{AJM-2007} proved that the solution to the Cauchy problem of (\ref{defocusing}) exists globally for any initial data with finite energy
 when  $n=3$ and $n=4$, respectively.
Schlag \cite{Annals-2009} proved the global existence of the solution and scattering for $p=4$ and $n=3$. For the mass-critical case $p=2+\frac{4}{n}$, Dodson \cite{JAMS-2012} showed the global well-posedness and scattering for the solution to the Cauchy problem of (\ref{defocusing}) with the initial data $\varphi_{0}\in L^{2}(\mathbb{R}^{n})$.
In the paper \cite{Annals-2008} mentioned above, Bourgain \cite{57-4} and Grillakis \cite{57-20} the energy-critical case $p=\frac{2n}{n-2}$ was considered and the global existence of the solution and scattering for the Cauchy problem of (\ref{defocusing}) were proved.  Chen and Holmer \cite{Invention-2019} studied (\ref{defocusing}) with $p=\frac{2n}{n-2}$ and $n=3$ and proved the uniqueness of large GP solutions.
 Deift and Zhou \cite{Acta-2002} studied the global existence of the solution and scattering for a perturbation problem of the defocusing nonlinear Schr\"{o}dinger equation:
\begin{align}
i\varphi_{t}+\Delta \varphi-2|\varphi|^{p-2}\varphi-\epsilon|\varphi|^{q-2}\varphi=0,\ \epsilon>0.
\end{align}
The above works for the defocusing semilinear Schr\"{o}dinger equation were all about the Cauchy problem in $\mathbb{R}^{n}$.
Burq, G\'erard, and Tzvetkov \cite{Invention-2005} extended the investigation for (\ref{defocusing}) on $\mathbb{R}^{n}$ to a compact Riemannian manifold $(M,g)$ of dimension 2 without
boundary. They proved that (\ref{defocusing}) with $p=4$ is uniformly well-posed in $H^{s}$ for $s>\frac{1}{4}$ on the sphere $\mathbb{S}^{2}$ and more generally on Zoll surfaces.
Colliander,  Keel, Staffilani, Takaoka,  and Tao \cite{Invention-2010} considered the periodic defocusing cubic nonlinear Schr\"{o}dinger equation $-i\varphi_{t}+\Delta \varphi-|\varphi|^{2}\varphi=0$, where the spatial variable $x$ lies
in the torus $\mathbb{T}^{2}:=\mathbb{R}^{2}/(2\mathbb{Z})^{2}$ and proved the existence of the global smooth solution.
For the defocusing nonlinear Schr\"{o}dinger equation (\ref{defocusing}), whose corresponding energy functional is composed of the non-negative terms, in the energy supercritical case, i.e., $p>\frac{2n}{n-2}$, Merle, Rapha\"el, Rodnianski, Szeftel \cite{Invention-2022} proved the finite time blowup for some suitable initial data.

For the focusing nonlinear Schr\"{o}dinger equation:
\begin{equation}\label{1.8}
  i\varphi_{t}+\Delta\varphi+|\varphi|^{p-2}\varphi=0,
\end{equation}
for  $2\leq p<2+\frac{4}{n}$,  Ginibre and  Velo \cite{r4-1-2-0} showed that the solution to the corresponding Cauchy problem exists globally for any initial data in $H^{1}(\mathbb{R}^{n})$. Different from the above, when $p\geq2+\frac{4}{n}$  the finite time blowup may occur.
 Merle and Rapha\"{e}l \cite{Annals-2005} investigated (\ref{1.8}) with $p=2+\frac{4}{n}$  and proved the finite time blowup for negative initial energy and positive initial energy and gave the blowup rate estimates for $n=1$.
 Under some special assumptions, they also extended these blowup results to $n\geq2$.
 In \cite{Invention-2004,JAMS-2005}  Merle and Rapha\"{e}l  gave the lower bound on the blowup rate of the solution to the Cauchy problem of (\ref{1.8}) by assuming some spectral
 properties.
 For the cases $2+\frac{4}{n}<p<\frac{2n}{n-2}$, $n\geq 3$ and $2+\frac{4}{n}<p<6$, $n=2$,
Ogawa and Tsutsumi \cite{r4-1-2-00} proved that the solution to the Cauchy problem of (\ref{1.8}) blows up in finite time for the radially symmetric initial data $\varphi_{0}$ in $H^{1}(\mathbb{R}^{n})$ with negative energy.
Banica,  Luc\`a, Tzvetkov, and Vega \cite{CMP-2024} considered the case $2+\frac{4}{n}<p$, $n=1$,  and proved that for a special class of initial data in the form $\varphi_{0}(x)=e^{i\frac{x^{2}}{4}}h(x)$, where $h(x)$ is a periodic function in $H^{s}(0,4\pi)$ and $s\in(0,\frac{1}{2})$, the corresponding solution to the Cauchy problem of (\ref{1.8}) blows up in a finite time.
 For the case $p\geq2+\frac{4}{n}$, the global solution may also exist.  For $p>2+\frac{4}{n}$ and $n=1$, Krieger and Schlag \cite{JAMS-2006} proved the global existence and scattering for the solution to the Cauchy problem of (\ref{1.8}). For the case $p=\frac{2n}{n-2}$, Kenig and Merle \cite{Invention-2006} considered (\ref{1.8}) and investigated the global existence, scattering and blowup of the solution. The above works studied the homogeneous semilinear Schr\"{o}dinger equation (\ref{1.8}).
 For the inhomogeneous version of (\ref{1.8}), Rapha\"{e}l and Szeftel \cite{JAMS-2011} considered the model equation $i\varphi_{t}+\Delta \varphi+k(x)|\varphi|^{p-2}\varphi=0$,
and investigated the existence and uniqueness of minimal blowup solutions.
Faou, Germain, and Hani \cite{JAMS-2015} considered the focusing semilinear Schr\"{o}dinger equation
$i\varphi_{t}-\Delta \varphi-|\varphi|^{2}\varphi=0$ on the spatial domain
$\mathbb{T}^{2}=[0,L]\times[0,L]$ of size $L$ with periodic boundary conditions. They uncovered new and structured or coherent
aspects of the out-of-equilibrium behavior, and also conducted this investigation for the defocusing version of this model. For the vectorial nonlinear Schr\"{o}dinger
equation, Masmoudi and Nakanishi \cite{Invention-2008} studied the model $2iu_{t}-\Delta u-|u|^{2}u=0$, where $u=(u_{1},u_{2})$ and investigated the simultaneous high-frequency and
the subsonic limit for this model.

The above results indicate that the focusing case of the energy-subcritical power type nonlinearity leads to the finite time blowup for some suitable initial conditions, while the defocusing case does not. For the Schr\"{o}dinger equations with two power-type nonlinearities, i.e., the case of combined power-type nonlinearities, if the focusing nonlinearity and defocusing nonlinearity both appear in the model equation, it is interesting to investigate whether the finite time blowup occurs and the corresponding conditions. Tao, Visan, and Zhang  \cite{r4-1-2-5} undertook a comprehensive study of the Schr\"{o}dinger equations with combined power-type nonlinearities:
 \begin{align}\label{juyouzuheyuan}
i\varphi_{t}+\Delta \varphi+\lambda_{1}|\varphi|^{p-2}\varphi+\lambda_{2}|\varphi|^{q-2}\varphi=0, \ p>q
\end{align}
and proved the global existence and finite time blowup of the solution and scattering. More precisely,
they proved that if the nonlinearity with the power $p$ is defocusing ($\lambda_{1}<0$) or the nonlinearity with $2\leq q<p<2+\frac{4}{n}$ is focusing ($\lambda_{1}>0$), the solution to the Cauchy problem of (\ref{juyouzuheyuan}) exists globally in time for any initial data in $H^{1}(\mathbb{R}^{n})$.
If the nonlinearity with the power $2+\frac{4}{n}<p\leq\frac{2n}{n-2}$ is focusing ($\lambda_{1}>0$) and the nonlinearity with the power $q$ is defocusing ($\lambda_{2}<0$) or the nonlinearity with $2+\frac{4}{n}<q<p$ is focusing ($\lambda_{2}>0$), the solution blows up in a finite time in the negative energy case.
And if two nonlinearities are both focusing ($\lambda_{1}>0$, $\lambda_{2}>0$) with $2\leq q\leq 2+\frac{4}{n}<p\leq\frac{2n}{n-2}$, the solution blows up in  a finite time when the energy is suitably  negative.
For the case of the Schr\"{o}dinger equation with the potential and the nonlinear terms, Ionescu and  Kenig \cite{Acta-2004} considered the Cauchy problem of the following  equation:
\begin{align}\label{acta}
i\varphi_{t}+\Delta \varphi-V\varphi-F(\varphi)=0,
\end{align}
where they showed $L^{p}$ Carleman inequalities and proved the uniqueness of the solution. For the case $V=0$ and a class of nonlinear term $F$ covering the combined power-type nonlinearities, Kenig, Ponce, and Vega \cite{55-9} proved the uniqueness of the sufficiently smooth solution; and
Mari\c{s} \cite{Annals-2013} proved the existence of nontrivial finite energy traveling waves.

In this paper, we consider the Schr\"{o}dinger equation \eqref{1.1} with harmonic potential and the focusing combined power-type nonlinearities,  aiming to classify the initial data
to distinguish  the global existence or finite time blowup of solution.
An important method to solve this problem is to apply the potential well theory and to use the sign of the Nehari functional
to partition the space of the initial data.
The classical form of the Nehari functional arises from the variational argument on the potential energy and is linked to the weak solution of the steady-state equation.
But this form can not be applied to the Schr\"{o}dinger equation directly to prove the finite time blowup
of the solution because this form can not be used to verify an important auxiliary functional $J''(t)<0$, a key step of the proof, which will be  discussed   in the following Subsection \ref{sec1.3}.
Zhang in \cite{r14} proposed a strategy for the Cauchy problem of (\ref{1.8}) to select and combine the terms in the energy functional properly and adjust their coefficients to establish an effective cross-constrained variational problem.
This method considers the interaction between the variational structure terms and the influence of different combinations with different weight for each term on the variational structure.
The goal of this paper is to develop further the cross-constrained variational theory and apply it to the semilinear Schr\"{o}dinger equation with a more complicated structure in order to obtain a deep understanding and comprehensive solutions for the related problems. We now briefly introduce the cross-constrained variational theory.

\subsection{Cross-constrained variational theory.}\label{sec1.3}
We now review the  cross-constrained variational approach introduced in \cite{r14}.
To facilitate the upcoming discussions, we denote the norm of $L^{r}(\mathbb{R}^{n})$ by $\left\|\cdot\right\|_{r}$ for $1\leq r<\infty$.
In order to study the Cauchy problem of (\ref{1.8}) for the case $p\geq2+\frac{4}{n}$,  Zhang in \cite{r14} proposed the following functionals:
\begin{align}
&K_{(a)}(\varphi):=\frac{1}{2}\|\nabla \varphi\|^{2}_{2}+\frac{1}{2}\|\varphi\|^{2}_{2}-\frac{1}{p}\|\varphi\|^{p}_{p},\label{Ka} \\
&S_{1,(a)}(\varphi):=\|\nabla \varphi\|^{2}_{2}+\|\varphi\|^{2}_{2}-\|\varphi\|^{p}_{p},\label{S1a}\\
&S_{2,(a)}(\varphi):=\|\nabla \varphi\|^{2}_{2}-\frac{n(p-2)\gamma}{2p}\|\varphi\|^{p}_{p},\label{S2a}
\end{align}
and the constrained variational problems
\begin{align}
d_{\mathcal{N}_{1,(a)}}:=\inf_{\varphi\in \mathcal{N}_{1,(a)}}K_{(a)}(\varphi),\label{d1a}
\end{align}
\begin{align}
d_{\mathcal{N}_{2,(a)}}:=\inf_{\varphi\in \mathcal{N}_{2,(a)}}K_{(a)}(\varphi),\label{d2a}
\end{align}
where
\begin{align*}
&\mathcal{N}_{1,(a)}:=\{\varphi\in H^{1}(\mathbb{R}^{n})\backslash\{0\}|\  S_{1,(a)}(\varphi)=0\},\\
&\mathcal{N}_{2,(a)}:=\{\varphi\in H^{1}(\mathbb{R}^{n})|\ S_{2,(a)}(\varphi)=0, \ S_{1,(a)}(\varphi)<0\},
\end{align*}
Zhang considered the case $\gamma=1$, and $d_{\mathcal{N}_{1,(a)}}\leq d_{\mathcal{N}_{2,(a)}}$ holds for $\gamma=1$.
The variational problem (\ref{d2a}) is called the cross-constrained variational problem due to the two constrained conditions: $S_{2,(a)}(\varphi)=0$ and $S_{1,(a)}(\varphi)<0$.
Based on the sign of the functional $S_{1,(a)}(\varphi)$, the initial data $\{\varphi\in H^{1}(\mathbb{R}^{n})\backslash\{0\}|K_{(a)}(\varphi)<d_{\mathcal{N}_{1,(a)}}\}$ is divided into two parts
\begin{align*}
\mathcal{W}_{(a)}:=\{\varphi\in H^{1}(\mathbb{R}^{n})|K_{(a)}(\varphi)<d_{\mathcal{N}_{1,(a)}}, S_{1,(a)}(\varphi)>0\},
\end{align*}
and
\begin{align*}
\mathcal{V}_{(a)}:=\{\varphi\in H^{1}(\mathbb{R}^{n})|K_{(a)}(\varphi)<d_{\mathcal{N}_{1,(a)}}, S_{1,(a)}(\varphi)<0\},
\end{align*}
 where $\mathcal{W}_{(a)}$ and $\mathcal{V}_{(a)}$ can be proved to be the invariant manifolds of the variational problem (\ref{d1a}), and
 $$\{\varphi\in H^{1}(\mathbb{R}^{n})\backslash\{0\}|K_{(a)}(\varphi)<d_{\mathcal{N}_{1,(a)}}\}=\mathcal{W}_{(a)}\cup\mathcal{V}_{(a)}.$$
  The manifold $\mathcal{W}_{(a)}$ is used to investigate the global existence of the solution, while the manifold $\mathcal{V}_{(a)}$ is further divided into the following two parts
\begin{align*}
\mathcal{W}_{(a),\rm{cross}}:=\{\varphi\in H^{1}(\mathbb{R}^{n})|K_{(a)}(\varphi)<d_{\mathcal{N}_{1,(a)}}, S_{1,(a)}(\varphi)<0, S_{2,(a)}(\varphi)>0\},
\end{align*}
and
\begin{align*}
\mathcal{V}_{(a),\rm{cross}}:=\{\varphi\in H^{1}(\mathbb{R}^{n})|K_{(a)}(\varphi)<d_{\mathcal{N}_{1,(a)}}, S_{1,(a)}(\varphi)<0, S_{2,(a)}(\varphi)<0\}
\end{align*}
according to the sign of the functional $S_{2,(a)}(\varphi)$, where $\mathcal{W}_{(a),\rm{cross}}$ and $\mathcal{V}_{(a),\rm{cross}}$ are the invariant manifolds of the cross-constrained variational problem (\ref{d2a}), and $$\mathcal{V}_{(a)}=\mathcal{W}_{(a),\rm{cross}}\cup\mathcal{V}_{(a),\rm{cross}}.$$
 Here, the manifold $\mathcal{V}_{(a),\rm{cross}}$ is used to prove the finite time blowup of the solution, while the manifold $\mathcal{W}_{(a),\rm{cross}}$ is used to show the global existence of the solution.
We use the following diagrammatic sketch in Figure 1 to show the classification of the initial data
\begin{align*}
\{\varphi\in H^{1}(\mathbb{R}^{n})\backslash\{0\}|K_{(a)}(\varphi)<d_{\mathcal{N}_{1,(a)}}\}=\mathcal{W}_{(a)}\cup\mathcal{W}_{(a),\rm{cross}}\cup\mathcal{V}_{(a),\rm{cross}}
\end{align*}
with the aid of the above invariant manifolds.
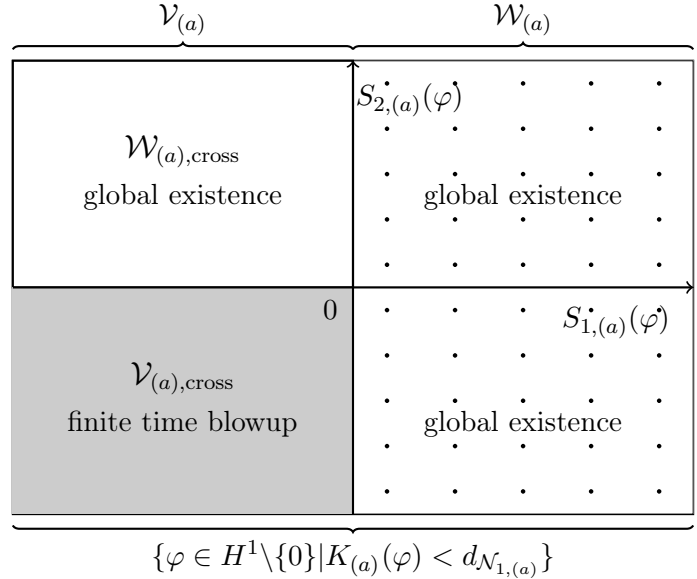
\begin{figure}[htbp]
\begin{center}
\begin{tikzpicture}[scale=3]

  \draw[thick] (-1.5,-1) rectangle (1.5,1);

  \fill[white] (0,-1) rectangle (1.5,1);

  \foreach \x in {0.15,0.45,0.75,1.05,1.35}
    \foreach \y in {-0.9,-0.7,-0.5,-0.3,-0.1,0.1,0.3,0.5,0.7,0.9}
      \fill[black] (\x,\y) circle (0.01);

  \draw (0,0) rectangle (1.5,1);

  \draw[line width=0.2pt] (0,-1) rectangle (1.5,0);

  \filldraw[fill=white, draw=black, thick] (-1.5,0) rectangle (0,1);

  \fill[gray!40] (-1.5,-1) rectangle (0,0);

  \draw[->, thick] (-1.5,0) -- (1.5,0);
  \node at (0.25, 0.83) {$S_{2,(a)}(\varphi)$};
  \draw[->, thick] (0,-1) -- (0,1);
  \node at (1.16, -0.15){$S_{1,(a)}(\varphi)$};
  \node at (-0.75, 0.6) {\large $\mathcal{W}_{(a),\text{cross}}$};
  \node at (-0.75, 0.4) {\text{global existence}};
  \node at (0.75, 0.4) {\text{global existence}};
  \node at (-0.75, -0.4) { \large $\mathcal{V}_{(a),\text{cross}}$};
  \node at (-0.75, -0.6) {\text{finite time blowup}};
  \node at (0.75, -0.6) {\text{global existence}};
  \node at (-0.1,-0.1) {$0$};
  \node at (-0.75,1.18) {\bf $\mathcal{V}_{(a)}$};
  \node at (0.75,1.18) {\bf $\mathcal{W}_{(a)}$};

  \node at (0, -1.2) {$\{\varphi\in H^{1}\backslash\{0\} | K_{(a)}(\varphi)<d_{\mathcal{N}_{1,(a)}}\}$};

  \draw[thick,decorate,decoration={brace}] (-1.5,1.05) -- (0,1.05);

  \draw[thick,decorate,decoration={brace}] (0,1.05) -- (1.5,1.05);

  \draw[thick,decorate,decoration={brace}] (1.5,-1.05) -- (-1.5,-1.05);
\end{tikzpicture}
\caption{The classification of the initial data $\{\varphi\in H^{1}\backslash\{0\}|K_{(a)}(\varphi)<d_{\mathcal{N}_{1,(a)}}\}$}
\label{a}
\end{center}
\end{figure}
With the help of the above classification of the initial data, for the Cauchy problem of (\ref{1.8}), Zhang \cite{r14}
showed that the solution exists globally for the initial data $\varphi_{0}\in\mathcal{W}_{(a)}\cup\mathcal{W}_{(a),\rm{cross}}$, and the solution blows up in finite time for the initial data $\varphi_{0}\in\mathcal{V}_{(a),\rm{cross}}$ and $|x|\varphi_{0}\in L^{2}(\mathbb{R}^{n})$.

To apply the cross-constrained variational theory in \cite{r14} for treating (\ref{1.8}), the core technique is to establish the cross-constrained variational problem (\ref{d2a}) to divide the invariant manifold $\mathcal{V}_{(a)}$ used in the classical potential well theory to investigate the finite time blowup of the solution. We shall explain this as follows. For the Cauchy problem of (\ref{1.8}), in the frame of the classical potential well theory defined by (\ref{d1a}) without introducing the cross-constrained variational problem (\ref{d2a}), we have to rely on the invariant manifolds $\mathcal{W}_{(a)}$ and $\mathcal{V}_{(a)}$ to classify the initial data for the global solution and finite time blowup solution. Here, the manifold $\mathcal{W}_{(a)}$ works for the global existence of the solution. Unfortunately, the manifold $\mathcal{V}_{(a)}$ fails in proving the finite time blowup of the solution. More precisely, for the initial data $\varphi_{0}\in \mathcal{V}_{(a)}$, the condition $S_{1,(a)}(\varphi_{0})<0$ is not sufficient to prove the finite time blowup of the solution by using the usual methods proposed in Ogawa and  Tsutsumi \cite{r4-1-2-00} and Glassey \cite{r4-1-2-00-1}. The method in \cite{r4-1-2-00} works for the case of the radially symmetric initial data with negative energy, instead of the general initial data $\varphi_{0}$ in $\mathcal{V}_{(a)}$ with positive initial energy without the radial symmetry assumption. The method proposed in \cite{r4-1-2-00-1}   is a subtle application of the auxiliary functional $J(t):=\||x|\varphi\|^{2}_{2}$ to investigate the finite time blowup of the solution, although the author of   \cite{r4-1-2-00-1}  considered a Cauchy problem of a different model $i \varphi_{t}-\Delta\varphi-|\varphi|^{p-2}\varphi=0$. By also using this method, Weinstein \cite{r4-1-2-00-2} investigated the following model
\begin{align}\label{1.14}
i \varphi_{t}+\frac{1}{2}\Delta\varphi+\frac{1}{2}|\varphi|^{p-2}\varphi=0,
\end{align}
and proved that the solution to the Cauchy problem of (\ref{1.14}) blows up in finite time for some sufficiently large  initial data. When applying the auxiliary functional $J_{(a)}(t)=J(t)=\||x|\varphi\|^{2}_{2}$ to treat (\ref{1.8}), the key step is to verify $J''_{(a)}(t)<0$, which can not be derived by the condition $\varphi_{0}\in \mathcal{V}_{(a)}$, i.e., $S_{1,(a)}(\varphi)<0$. This difficulty is overcome by introducing the so-called cross-constrained variational problem (\ref{d2a}), which divides $\mathcal{V}_{(a)}$ into $\mathcal{W}_{(a),\rm{cross}}$ and $\mathcal{V}_{(a),\rm{cross}}$, and the invariant manifold $\mathcal{V}_{(a),\rm{cross}}$ is used to prove the finite time blowup of the solution. The above manifold $\mathcal{V}_{(a),\rm{cross}}$ is defined by a new functional $S_{2,(a)}(\varphi)$, which is constructed by carefully adjusting the coefficients of the terms in the energy functional
\begin{align}\label{1.15}
E_{(a)}(t):=\frac{1}{2}\|\nabla\varphi\|^{2}_{2}-\frac{1}{p}\|\varphi\|^{p}_{p}
\end{align}
to ensure $S_{2,(a)}(\varphi)=\frac{1}{8}J''_{(a)}(t)$. Then one can deduce $J''_{(a)}(t)<0$ from $S_{2,(a)}(\varphi)<0$, i.e., $\varphi_{0}\in\mathcal{V}_{(a),\rm{cross}}$, which successfully proves the finite time blowup of the solution. From the above careful analysis of the cross-constrained variational theory, we notice that
when using this method to study a model with a complex energy structure composed of more than one structural term, such as (\ref{1.8}), the key step is to select and combine these terms by adjusting their coefficients to give an effective combination of these terms in functional form, for example, the above combination functional $S_{2,(a)}(\varphi)$ to study the finite time blowup of the solution to the Cauchy problem of (\ref{1.8}). In fact, it is usually not easy to find such effective combinations, and the difficulty is also reflected in the uniqueness of the combinations. The uniqueness depends on both the energy structure of the problem and its influence on the properties of the solution. More precisely, the uniqueness of the combination allows us to apply such combination to classify the initial data completely and clearly.
If the combinations are not unique, meaning there is more than one effective combination to classify the initial data, we should strive to find as many effective combinations as possible. This will provide a clearer and more comprehensive description of the initial data, leading to a finite time blowup.
It is important to compare the differences between these combinations and discover their relationships. This is especially true for energy structures with more than two structural terms, such as those defined for (\ref{1.16}). Since the energy functional for (\ref{1.1}) in this paper consists of four structural terms, these combinations are complex. Applying cross-constrained variational theory to study the dynamic behavior of the solution seems interesting yet challenging. To explain the motivation for this research, the research background, and the difficulties we need to overcome, we first provide a summary of the results using the cross-constrained variational theory in the following subsection.

\subsection{The results by the cross-constrained variational theory}\label{sec1.4}
   There has been a lot of of works in the literature applying the cross-constrained variational approach in  \cite{r14} directly or transforming it for different situations based on the idea of this method.
   To better illustrate the mathematical mechanics of using and adapting this method in the relevant results, we first need to focus on the relationship between the energy structure and the variational structure.
  By comparing the model equation (\ref{1.8}) and its corresponding energy (\ref{1.15}), it is easy to see that the structural terms as the cross-constrained variational functionals are from the model equation (\ref{1.8}) and its energy by considering the mass term $\|\varphi\|^{2}_{2}$ additionally. Here, although the mass conservation term $\|\varphi\|^{2}_{2}$ does not appear in the model equation, it is necessary to consider this term in the cross-constrained variational argument as it is the base norm of the $H^{1}$ norm $\|\varphi\|_{H^{1}}$ for the case of the unbounded domain $\mathbb{R}^{n}$. Different from the case of the model equation (\ref{1.8}) with two structural terms in the variational problem, Zhang \cite{r4-1-2-1}
studied further  the case with three structural terms, that is, the following semilinear Schr\"{o}dinger equation with the harmonic potential:
\begin{align}\label{1.16}
  i\varphi_{t}+\alpha\Delta\varphi-\beta|x|^{2}\varphi+|\varphi|^{p-2}\varphi=0,
\end{align}
 which takes the energy functional
\begin{align*}
E_{(b)}(t):=\frac{\alpha}{2}\|\nabla\varphi\|^{2}_{2}+\frac{\beta}{2}\||x|\varphi\|^{2}_{2}-\frac{1}{p}\|\varphi\|^{p}_{p}.
\end{align*}
By selecting and combining the energy structure terms of the above model equation (\ref{1.16}), the author established the corresponding cross-constrained variational problem following the idea introduced in \cite{r14} to give a classification of the initial data. As the model equation (\ref{1.16}) has one more structural term than (\ref{1.8}) does, there will be much more cross-constrained combinations for the case of (\ref{1.16}) compared with the case of (\ref{1.8}), which makes the corresponding cross-constrained problem much more complicated. Hence the main task here is to find out as many effective combinations as possible to classify the initial data for the global existence and finite time blowup of solution.
Some efforts in this direction have been made in \cite{r4-1-2-2,r4-1-2-2,r4-1-2-4,r4-1-2-4-1}.

\subsection{Motivation and contributions.}\label{sec1.5}
In view of the characteristic structure of the Schr\"{o}dinger equation with two conservation laws (mass conservation and energy conservation), the cross-constrained variational theory is a powerful tool for classifying the initial data leading to the global solutions or finite time blowup solutions to such model equations.
The cross-constrained variational theory also thoroughly examines the energy structure of the model equation by exploring as many combinations of the energy structure terms as possible.
It is obvious that if the model equation has many terms making its structure complicated, there exist a lot of possible combinations of the energy structure terms, which causes many difficulties and challenges. Recently, some works have been devoted to studying such complex structures caused by the potential terms and nonlinear source terms under the framework of the cross-constrained variational theory. Shu and Zhang \cite{r4-1-2-4-4} and Gan et al and Zhang \cite{r4-1-2-4-5} investigated the Davey-Stewartson system \cite{r4-1-2-4-4-000}:
$$i\varphi_{t}+\Delta\varphi+a|\varphi|^{p-2}\varphi+bE_{1}(|\varphi|^{2})\varphi=0,$$ which was introduced to describe the evolution of weakly nonlinear water waves with one predominant direction of travel. Chen and Guo in \cite{r4-1-2-4-3} considered the potential term of the special form appearing in the nonlocal Schr\"{o}dinger equation
$$i \varphi_{t}+\Delta \varphi+\left(\frac{1}{|x|^{\alpha}}\ast|\varphi|^{p}\right)|\varphi|^{p-2}\varphi=0.$$
Wang in \cite{r4-1-2-4-2} considered the combination of two different types of potential terms in the model
$$i \varphi_{t}+\Delta \varphi-|x|^{2}\varphi+\left(V\ast|\varphi|^{2}\right)\varphi=0.$$
 The works above demonstrate the profuse combinations of potential terms and nonlinear source terms in energy structures, leading to a variety of complex variational structures.

 Compared with the above models, the equation in the Cauchy problem (\ref{1.1}) considered in this paper also presents the combination of the potential term and two power-type nonlinear source terms, which come from the Gross-Pitaevskii equation in the form (\ref{1.5}) describing the Bose-Einstein condensation mentioned in Subsection \ref{sec1.1}. Even without the potential term, the NLS model equations with the combined power-type nonlinearities introduced to describe the propagation of the light in some lossless nonlinear medium \cite{r4-1-2-4-2-0000} and the effect of the saturation of nonlinear refractive index \cite{r4-1-2-4-2-00001,r4-1-2-4-2-00002} also have the profound physical background and mathematical conclusions \allowdisplaybreaks\cite{r4-1-2-5-01,r4-1-2-5-09,r4-1-2-5-13,r4-1-2-5-03,r4-1-2-5-08,r4-1-2-5-02,r4-1-2-5-05,r4-1-2-5-11,r4-1-2-5-10,r4-1-2-5-04,r4-1-2-5-06,r4-1-2-5-07,r4-1-2-5-12,r4-1-2-5},
 mainly due to the difficulties and challenges caused by the combined power-type nonlinearities.
 We note that for the model equation in problem (\ref{1.1}), the difficulties arise from both the harmonic potential term and the combined power-type nonlinearities.
 The combined power-type nonlinearities, meaning the combination of two power-type nonlinear source terms, result in a complex energy structure. Additionally, the harmonic potential term further complicates the problem significantly. As a result, we need to address the significant difficulties and challenges posed by these three terms simultaneously.
 In the theoretical framework of the cross-constrained variational theory, the main task is to create the cross-constrained variational functionals to divide the space to which the initial data belong. These functionals directly determine the characteristics of the division and the description of the initial data.
 As previously discussed, these functionals include structural terms related to energy and "mass" with varying weights. However, these combinations are not standardized, and the weights of each term are also unpredictable. The validity of these settings will be determined by assessing global existence and non-existence. Therefore, an equation with more structural terms requires consideration of more combinations, the development of complex strategies for different weights, and testing and verification of a greater number of scenarios.
 In the process of discussing this problem, we need to divide the discussion into three cases: (a) negative initial energy case, (b) non-negative low initial energy case, and (c) high initial energy case.
In the case of high initial energy, the behavior of the standard nonlinear Schr\"{o}dinger equation and the Schr\"{o}dinger equation with both the harmonic potential and combined power-type nonlinearities remains an open problem. Our current paper aims to address this issue. We will study the high initial energy case for the problem and our arguments also apply to the standard nonlinear Schr\"{o}dinger equation $i\varphi_{t}+\Delta\varphi+|\varphi|^{p-2}\varphi=0$. This is the first attempt to address the high initial energy case of Schr\"{o}dinger-like equations. Additionally, we will summarize and analyze the results obtained by applying different cross-constrained variational functionals to understand their relationships. This task is not only challenging but also meaningful and interesting.

In summary, the significant contributions of this paper can be summarized as follows.
\begin{enumerate}[leftmargin=*, label={(\arabic*)}] \addtolength{\itemsep}{0.3\baselineskip}
\item \textbf{Cross-constrained variational combinations at the sub-critical/critical initial energy level.} To classify the initial data for the global-in-time solution and the finite time blowup solution, we need to build up the variational structures by selecting and combining, with different coefficients as different weights, the four energy structural terms of the model equation (\ref{1.1}) as well as the ``mass'' term, as we discussed above.
The different  cross-constrained variational combinations  reveal various relationships between the initial data and the solution. Therefore, it is important to identify as many effective combinations as possible, rather than relying on just one, which is challenging. This paper aims to comprehensively study and identify as many effective combinations as possible within the developed cross-constrained variational theory. The objective is to systematically determine and compare the dependency relationships between the initial data and the solution at the sub-critical/critical initial energy level.
\item \textbf{Sharp conditions for high initial energy.} Classifying the initial data and applying the variational theory for the NLS at a high initial energy level has remained an open problem for decades  (cf.  \cite{rz1,r4-1-2-4-5,NA-Zhang,r4-1-2-1}).
The objective of this paper is to solve this open problem and provide the sharp conditions for the global existence and finite time blowup of a solution at a high initial energy level. We aim to draw inspiration from effective combinations for sub-critical/critical initial energy levels and apply them to the high energy case, despite the additional challenges this presents.
\item \textbf{High initial energy sharp conditions for classical NLS.}
Solving the open problem for the standard semilinear Schr\"{o}dinger equation is also one of the goals of this paper, for which we need to derive the sharp condition of global existence and finite time blowup of solution at the high initial energy level. Although this may seem similar to our previous point, our approach involves considering a combination of different structural terms, thus it is not clear that the problem can be solved using the method established in the previous point.
\end{enumerate}

The rest of this paper is organized as follows.
In Section \ref{sec2}, in order to study the classifications of the initial data under the framework of the cross-constrained variational theory, we first propose eight different cross-constrained variational functionals by the technical selection and combination of the energy structural terms, including mass term, to construct the corresponding cross-constrained problems. Then we recall the local existence theory, mass conservation law and energy conservation law of problem (\ref{1.1}). Subsequently, we prove the positivity of the potential well depth  in order to classify the initial data for three initial energy cases, i.e., negative initial energy, non-negative low initial energy and high initial energy. Meanwhile, we also prove the invariance of manifolds obtained by classifying the initial data for non-negative low initial energy.
In Section \ref{sec3}, for an important auxiliary functional, we calculate its second-order derivative to prepare to prove the finite time blowup of solution for all three initial energy cases. As a direct and quick application of this result, we prove that the negative initial energy leads to the finite time blowup, which means we need not classify the initial data for negative initial energy case.
In Section \ref{sec4}, we focus on the non-negative low initial energy and apply eight cross-constrained variational functionals to give eight classifications
 of the initial data to distinguish the global solution and finite time blowup solution.
 In Section \ref{sec5}, we study the problem of classifying the initial data for high initial energy case to try our best to solve this open problem. And we analyze the relationship between these different results obtained by applying eight different cross-constrained variational functionals. In Section \ref{sec6}, we shall treat standard nonlinear Schr\"{o}dinger equation to solve the open problem of classifying the initial data in the corresponding high initial energy case.

 \section{Completed cross-constrained variational theory}\label{sec2}
In this section, we present the energy structure of the equation with both harmonic potential term and combined power-type nonlinearities in problem (\ref{1.1}) and then establish the cross-constrained variational structures in order to classify the initial data. As the preliminary work, we define the following energy space
\begin{align*}
H:=\left\{{\varphi\in H^{1}(\mathbb{R}^n)\bigg|\int_{\mathbb{R}^n} |x|^{2}|\varphi|^{2} {\rm d}x<+\infty}\right\}\nonumber
\end{align*}
for the equation in problem (\ref{1.1}), whose inner product and norm are defined by
\begin{align*}
<\varphi,\phi>_{H}:=\int_{\mathbb{R}^n} \nabla\varphi\nabla\overline{\phi}+\varphi\overline{\phi}+|x|^{2}\varphi\overline{\phi}\ {\rm d}x\nonumber
\end{align*}
and
\begin{align*}
\|\varphi\|_{H}:=\left(\int_{\mathbb{R}^{n}} |\nabla\varphi|^{2}+|\varphi|^{2}+|x|^{2}|\varphi|^{2} {\rm d}x\right)^{\frac{1}{2}},\nonumber
\end{align*}
respectively.

According to Theorem 9.2.6 (i) of Chapter 9 in \cite{rz1}\ (see also \cite{rz1-1,rz1-2}), we know the following result about the local existence of the solution to the Cauchy problem (\ref{1.1}).
\begin{lemma}\label{lem2.1}
(Local-in-time existence \cite{rz1})
Let $\varphi_{0}\in H$, $p$ and $q$ satisfy (\ref{1.2}). Then there exists $T>0$ such that the Cauchy problem (\ref{1.1}) admits a unique local solution $\varphi\in C([0,T),H)\cap C^{1}([0,T),H^{*})$, where $H^{*}$ denotes the dual space of $H$, and $T$ is the existence time of the local solution. Moreover, if the maximum existence time $T_{\rm{max}}=+\infty$, the solution exists globally. If $T_{\rm{max}}<+\infty$, the solution blows up in finite time, that is,
\begin{align*}
\lim\limits_{t\rightarrow T_{\rm{max}}} \|\varphi\|_{H}=+\infty.
\end{align*}
\end{lemma}
Then, we define the following energy functional for the equation in problem (\ref{1.1}) in $H$
\begin{align}\label{2.1}
E(t):=\frac{1}{2}\|\nabla \varphi\|^{2}_{2}+\frac{1}{2}\||x|\varphi\|^{2}_{2}-\frac{\lambda_{1}}{p}\|\varphi\|^{p}_{p}-\frac{\lambda_{2}}{q}\|\varphi\|^{q}_{q}.
\end{align}
The energy (\ref{2.1}), as well as the term $\|\varphi\|_{2}$, are the conserved quantities according to Theorem 9.2.6 (ii) of Chapter 9 in \cite{rz1}, which is presented by the following Lemma \ref{lem2.2}.
\begin{lemma}\label{lem2.2}(Conservation laws \cite{rz1})
Suppose $\varphi_{0}\in H$, $p$ and $q$ satisfy (\ref{1.2}). If $\varphi$ is the solution of (\ref{1.1}), there exist the following two conservation laws
\begin{equation}\label{2.2}
\|\varphi\|_{2}=\|\varphi_{0}\|_{2}
\end{equation}
and
\begin{equation}\label{2.3}
E(t)=E(0).
\end{equation}
\end{lemma}

The above Lemma \ref{lem2.2} reveals two important conservation laws for the nonlinear Schr\"{o}dinger equations, i.e., the mass conservation law \cite{rz1-3,rz1-4,rz1-5,r4-1-2-5}, also called the conservation law of charge \cite{rz1,rz1-7,rz1-6}, reflected by (\ref{2.2}), and the energy conservation law shown by (\ref{2.3}). Especially for the unbounded domain $\mathbb{R}^n$, the $L^2$ norm in the mass/charge conservation law is the ``base'' norm of the $H^1$ norm, which helps overcome the difficulties caused by the missing Poincar\'{e} inequality to make the corresponding Sobolev embedding work. Hence it is necessary to combine these two conservation laws to introduce a new functional as follows
\begin{align}\label{2.4}
K(\varphi):=\frac{1}{2}\|\nabla \varphi\|^{2}_{2}+\frac{1}{2}\|\varphi\|^{2}_{2}+\frac{1}{2}\||x|\varphi\|^{2}_{2}-\frac{\lambda_{1}}{p}\|\varphi\|^{p}_{p}-\frac{\lambda_{2}}{q}\|\varphi\|^{q}_{q}
\end{align}
as the ``energy'' to construct the cross-constrained variational framework.

The core idea of applying the cross-constrained variational theory is to construct effective and available cross-constrained variational functionals, i.e., $S_{j}(\varphi)$ we shall define later, which will be used to define the Nehari manifolds, the cross-constrained variational problems, and divide the space of the initial data. These proper constructed cross-constrained variational functionals should be in the form of combinations of the selected terms with proper weights from the ``energy'' functional $K(\varphi)$. Hence we explain this mechanism roughly here. First, we clarify that among all the terms in $K(\varphi)$, $\|\nabla \varphi\|^{2}_{2}$ and $\|\varphi\|^{p}_{p}$ are mandatory, while others are optional. The necessity of $\|\nabla \varphi\|^{2}_{2}$ will be carefully explained in Remark \ref{rem4.2} later by explaining that it helps prove the finite time blowup, although the proof of global existence can be separated from it. The term $\|\varphi\|^{p}_{p}$ is negative and should be balanced by the positive terms, which is the spirit of the potential well theory, and we shall also explain in Theorem \ref{th4.1} and Remark \ref{rem4.3} that as $p>q$, the term $\|\varphi\|^{p}_{p}$ should be controlled as a mandatory term, but the term $\|\varphi\|^{q}_{q}$ is optional. Although these two mandatory terms are essential, their combination is unique, so these three optional terms indeed construct various complicated cross-constrained variational problems. As shown in Table 1, we have eight different combinations as the cross-constrained functionals by considering the two mandatory and three optional terms. In the present paper, we shall conduct a comprehensive study of all these eight combinations, which seems the first try in applying the cross-constrained variational theory.
\begin{table}[H]
  \centering
  \caption{The cross-constrained variational combinations}\label{tab1}
  \vspace{0.3cm}
  \tiny
  \begin{threeparttable}
  \begin{footnotesize}
  \begin{tabular}{p{2.8cm}<{\centering}p{1.8cm}<{\centering}p{1.7cm}<{\centering}p{1.6cm}<{\centering}p{1.4cm}<{\centering}p{1.4cm}<{\centering}}

    \toprule
    \ &$\|\nabla \varphi\|^{2}_{2}$ &$\||x|\varphi\|^{2}_{2}$ &$\|\varphi\|^{2}_{2}$ &$\|\varphi\|^{p}_{p}$ &$\|\varphi\|^{q}_{q}$\\
    \midrule
    Case 1\ \ \ \ $S_{1}(\varphi)$&$\bullet$ &$\circ$ &$\circ$ &$\bullet$  &$\circ$ \\
    Case 2\ \ \ \ $S_{2}(\varphi)$&$\bullet$ &$\bullet$ &$\circ$ &$\bullet$  &$\circ$ \\
    Case 3\ \ \ \ $S_{3}(\varphi)$&$\bullet$ &$\circ$ &$\bullet$ &$\bullet$  &$\circ$ \\
    Case 4\ \ \ \ $S_{4}(\varphi)$&$\bullet$ &$\bullet$ &$\bullet$ &$\bullet$  &$\circ$ \\
    Case 5\ \ \ \ $S_{5}(\varphi)$&$\bullet$ &$\circ$ &$\circ$ &$\bullet$  &$\bullet$ \\
    Case 6\ \ \ \ $S_{6}(\varphi)$&$\bullet$ &$\bullet$ &$\circ$ &$\bullet$  &$\bullet$ \\
    Case 7\ \ \ \ $S_{7}(\varphi)$&$\bullet$ &$\circ$ &$\bullet$ &$\bullet$  &$\bullet$ \\
    Case 8\ \ \ \ $S_{8}(\varphi)$&$\bullet$ &$\bullet$ &$\bullet$ &$\bullet$  &$\bullet$ \\
    \bottomrule
    \end{tabular}
    \begin{tablenotes}
    \item $\bullet=$ choose \ \ \ $\circ=$ not choose
    \end{tablenotes}
  \end{footnotesize}
  \end{threeparttable}
\end{table}
\noindent Now we write down these eight combinations, i.e., Case 1-Case 8, as follows
\begin{align*}
S_{1}(\varphi):=\|\nabla \varphi\|^{2}_{2}-\frac{n(p-2)\lambda_{1}}{2p}\|\varphi\|^{p}_{p},
\end{align*}
\begin{align*}
S_{2}(\varphi):=\|\nabla \varphi\|^{2}_{2}+\||x|\varphi\|^{2}_{2}-\frac{n(p-2)\lambda_{1}}{2p}\|\varphi\|^{p}_{p},
\end{align*}
\begin{align*}
S_{3}(\varphi):=\|\nabla \varphi\|^{2}_{2}+c\|\varphi\|^{2}_{2}-\frac{n(p-2)\lambda_{1}}{2p}\|\varphi\|^{p}_{p},
\end{align*}
\begin{align*}
S_{4}(\varphi):=\|\nabla \varphi\|^{2}_{2}+c\|\varphi\|^{2}_{2}+\||x|\varphi\|^{2}_{2}-\frac{n(p-2)\lambda_{1}}{2p}\|\varphi\|^{p}_{p},
\end{align*}
\begin{align*}
S_{5}(\varphi):=\|\nabla \varphi\|^{2}_{2}-\frac{n(p-2)\lambda_{1}}{2p}\|\varphi\|^{p}_{p}-\frac{n(q-2)\lambda_{2}}{2q}\|\varphi\|^{q}_{q},
\end{align*}
\begin{align*}
S_{6}(\varphi):=\|\nabla \varphi\|^{2}_{2}+\||x|\varphi\|^{2}_{2}-\frac{n(p-2)\lambda_{1}}{2p}\|\varphi\|^{p}_{p}-\frac{n(q-2)\lambda_{2}}{2q}\|\varphi\|^{q}_{q},
\end{align*}
\begin{align*}
S_{7}(\varphi):=\|\nabla \varphi\|^{2}_{2}+\|\varphi\|^{2}_{2}-\frac{n(p-2)\lambda_{1}}{2p}\|\varphi\|^{p}_{p}-\frac{n(q-2)\lambda_{2}}{2q}\|\varphi\|^{q}_{q},
\end{align*}
and
\begin{align*}
S_{8}(\varphi):=\|\nabla \varphi\|^{2}_{2}+\|\varphi\|^{2}_{2}+\||x|\varphi\|^{2}_{2}-\frac{n(p-2)\lambda_{1}}{2p}\|\varphi\|^{p}_{p}-\frac{n(q-2)\lambda_{2}}{2q}\|\varphi\|^{q}_{q},
\end{align*}where $c>0$ is chosen such that
\begin{align}\label{cdetiaojian}
\frac{1}{2}-\left(\frac{2}{n(p-2)}+\frac{2\lambda_{2}p}{\lambda_{1}qn(p-2)}\right)c-\frac{\lambda_{2}}{q}>0
\end{align}
for the case $j=3,4$, and $\frac{1}{2}-\frac{\lambda_{2}}{q}>0$ is assumed in (\ref{cdetiaojian}).

By the functionals $K(\varphi)$ and $S_{j}(\varphi)$, $j=1,2,\cdots,8$, constructed above, we introduce the following cross-constrained problems
\begin{equation}\label{2.5}
d_{\mathcal{N}_{j}}:=\inf_{\varphi\in \mathcal{N}_{j}}K(\varphi),
\end{equation}
 where
\begin{align*}
\mathcal{N}_{j}:=\left\{\varphi\in H\backslash\{0\}|S_{j}(\varphi)=0\right\}, j=1,2,\cdots,8.\nonumber
\end{align*}
Subsequently, by applying the variational problems (\ref{2.5}), for $j=1,2,\cdots,8$, we divide the space of the initial data
\begin{align}\label{2.6}
\mathcal{H}_{j}:=\left\{\varphi\in H\big|0\leq K(\varphi)<d_{\mathcal{N}_{j}}\right\}
\end{align}
into
\begin{align}\label{2.7}
\mathcal{W}_{j}:=\left\{\varphi\in H\big|0\leq K(\varphi)<d_{{\mathcal{N}_{j}}},S_{j}(\varphi)>0\right\}\cup\{0\}
\end{align}
and
\begin{align}\label{2.8}
\mathcal{V}_{j}:=\left\{\varphi\in H\big|0\leq K(\varphi)<d_{{\mathcal{N}_{j}}},S_{j}(\varphi)<0\right\},
\end{align}
that is, $\mathcal{H}_{j}=\mathcal{W}_{j}\cup\mathcal{V}_{j}$, which are indeed the eight classifications of the initial data we shall show later. We shall use these eight classifications of the initial data to give eight sharp criteria to distinguish the global solution and finite time blowup solution to the Cauchy problem (\ref{1.1}) for the non-negative low initial energy case in following Section \ref{sec4}.

Next, we shall prove the positivity of $d_{_{\mathcal{N}_{j}}}$ defined by (\ref{2.5}), that is, $d_{_{\mathcal{N}_{j}}}>0$ for $j=1,2,\cdots,8$.
The present paper aims to figure out the threshold/sharp condition of the initial data to classify the global existence and finite time blowup solution. Hence it is necessary to consider the non-negative initial energy case $0\leq K(\varphi_{0})<d_{\mathcal{N}_{j}}$, which means that we require $d_{\mathcal{N}_{j}}>0$. For $K(\varphi_{0})<0$, i.e., the negative initial energy case, the solution must blow up in finite time according to Theorem \ref{th3.1}, which is a comparably simple case and not necessarily to be dealt with by the cross-constrained variational theory.
\begin{lemma}\label{lem2.3}
(Positivity of $d_{_{\mathcal{N}_{j}}}$, $j=1,2,\cdots,8$)
Let $p$ and $q$ satisfy (\ref{1.2}), $\lambda_{1}>0$, and $\lambda_{2}>0$.

(i)\ For the cases $j=1,2,3,4$, if $\lambda_{1}$ and $\lambda_{2}$ satisfy
\begin{align}\label{2.9}
\lambda_{2}<\left(\frac{1}{2}-\frac{2}{n(p-2)}\right)\frac{qn(p-2)\lambda_{1}}{2p}
\end{align}
and
\begin{align}\label{2.10}
\lambda_{2}<\frac{q}{2},
\end{align}
then $d_{_{\mathcal{N}_{j}}}>0$;

(ii)\ For the cases $j=5,6,7,8$, $d_{_{\mathcal{N}_{j}}}>0$.
\begin{proof}
 To prove $d_{_{\mathcal{N}_{j}}}>0$, $j=1,2,\cdots,8$, we need to control the negative terms $\|\varphi\|^{p}_{p}$ and $\|\varphi\|^{q}_{q}$ in the functional $K(\varphi)$. Here, the corresponding strategies depend on the structures of the cross-constrained variational functionals $S_{j}(\varphi)$, $j=1,2,\cdots,8$. Although we have eight cases to consider, it is not necessary to propose eight different approaches. According to the nature of these eight cases, we can divide them into three groups, i.e., Case \textrm{I} ($j=1,2$), Case \textrm{II} ($j=3,4$), and Case \textrm{III} ($j=5,6,7,8$). Next, we shall prove these three cases one by one.\\
{\bf{Case \textrm{I}: \bm{$j=1,2.$}}}\\
 To facilitate our analysis, we divide the space $\mathbb{R}^{n}$ into two parts, i.e.,
\begin{equation}
\Omega_{1}:=\{{x\in\mathbb{R}^{n}||\varphi|\geq1,\varphi\in H}\}\nonumber
\end{equation}
and
\begin{equation}
\Omega_{2}:=\{{x\in\mathbb{R}^{n}||\varphi|<1,\varphi\in H}\},\nonumber
\end{equation}
that is, $\mathbb{R}^{n}=\Omega_{1}\cup\Omega_{2}$.
Then, we know
\begin{align}\label{2.11}
\begin{split}
K(\varphi)=&\frac{1}{2}\|\nabla \varphi\|^{2}_{2}+\frac{1}{2}\|\varphi\|^{2}_{2}+\frac{1}{2}\||x|\varphi\|^{2}_{2}-\frac{\lambda_{1}}{p}\|\varphi\|^{p}_{p}-\frac{\lambda_{2}}{q}\|\varphi\|^{q}_{q}\\
=&\underbrace{\int_{\Omega_{1}} \left(\frac{1}{2}|\nabla \varphi|^2+\frac{1}{2}|\varphi|^{2}+\frac{1}{2}|x|^{2}|\varphi|^{2}-\frac{\lambda_{1}}{p}|\varphi|^{p}-\frac{\lambda_{2}}{q}|\varphi|^{q}\right){\rm d}x}_{:=K_{1}}\\
&+\underbrace{\int_{\Omega_{2}} \left(\frac{1}{2}|\nabla \varphi|^2+\frac{1}{2}|\varphi|^{2}+\frac{1}{2}|x|^{2}|\varphi|^{2}-\frac{\lambda_{1}}{p}|\varphi|^{p}-\frac{\lambda_{2}}{q}|\varphi|^{q}\right){\rm d}x}_{:=K_{2}}.
\end{split}
\end{align}
First, we estimate $K_{1}$. For $x\in\Omega_{1}$, according to $p>q$ assumed in (\ref{1.2}), we deduce
\begin{align}\label{2.12}
\frac{\lambda_{2}}{q}|\varphi|^{q}\leq\frac{\lambda_{2}}{q}|\varphi|^{p}.
\end{align}
Substituting (\ref{2.12}) into $K_{1}$, we know
\begin{align*}
K_{1}\geq&\int_{\Omega_{1}} \left(\frac{1}{2}|\nabla \varphi|^2+\frac{1}{2}|\varphi|^{2}+\frac{1}{2}|x|^{2}|\varphi|^{2}\right){\rm d}x-\int_{\Omega_{1}}\left(\frac{\lambda_{1}}{p}+\frac{\lambda_{2}}{q}\right)|\varphi|^{p}{\rm d}x,
\end{align*}
which means
\begin{align}\label{2.13}
\begin{split}
K_{1}>&\int_{\Omega_{1}} \left(\frac{1}{2}|\nabla \varphi|^2+\left(\frac{1}{2}-\frac{\lambda_{2}}{q}\right)|\varphi|^{2}+\frac{1}{2}|x|^{2}|\varphi|^{2}\right){\rm d}x\\
&-\int_{\Omega_{1}}\left(\frac{\lambda_{1}}{p}+\frac{\lambda_{2}}{q}\right)|\varphi|^{p}{\rm d}x.
\end{split}
\end{align}
Now, we estimate $K_{2}$. We notice
\begin{align}\label{2.14}
\frac{\lambda_{2}}{q}|\varphi|^{q}\leq\frac{\lambda_{2}}{q}|\varphi|^{2}
\end{align}
 for $x\in\Omega_{2}$. By substituting (\ref{2.14}) into $K_{2}$, we know
\begin{align*}
K_{2}\geq&\int_{\Omega_{2}} \left(\frac{1}{2}|\nabla \varphi|^2+\frac{1}{2}|\varphi|^{2}+\frac{1}{2}|x|^{2}|\varphi|^{2}\right){\rm d}x-\int_{\Omega_{2}}\left(\frac{\lambda_{1}}{p}|\varphi|^{p}+\frac{\lambda_{2}}{q}|\varphi|^{2}\right){\rm d}x,
\end{align*}
which implies
\begin{align}\label{2.15}
\begin{split}
K_{2}>&\int_{\Omega_{2}} \left(\frac{1}{2}|\nabla \varphi|^2+\left(\frac{1}{2}-\frac{\lambda_{2}}{q}\right)|\varphi|^{2}+\frac{1}{2}|x|^{2}|\varphi|^{2}\right){\rm d}x\\
&-\int_{\Omega_{2}}\left(\frac{\lambda_{1}}{p}+\frac{\lambda_{2}}{q}\right)|\varphi|^{p}{\rm d}x.
\end{split}
\end{align}
By substituting (\ref{2.13}) and (\ref{2.15}) into (\ref{2.11}), we obtain
\begin{align}\label{2.16}
\begin{split}
K(\varphi):=&K_{1}+K_{2}\\
>&\frac{1}{2}\|\nabla \varphi\|^{2}_{2}+\left(\frac{1}{2}-\frac{\lambda_{2}}{q}\right)\|\varphi\|^{2}_{2}+\frac{1}{2}\||x|\varphi\|^{2}_{2}-\left(\frac{\lambda_{1}}{p}+\frac{\lambda_{2}}{q}\right)\|\varphi\|^{p}_{p}.
\end{split}
\end{align}
Next, in order to estimate $K(\varphi)$ for $\varphi\in\mathcal{N}_{1}$ or $\varphi\in\mathcal{N}_{2}$ based on (\ref{2.16}), we analyze the properties of the function $\varphi\in\mathcal{N}_{1}\cup\mathcal{N}_{2}$.
For $\varphi\in \mathcal{N}_{1}$, i.e., $S_{1}(\varphi)=0$, $\varphi\neq0$, we have
\begin{align*}
\|\varphi\|^{p}_{p}=\frac{2p}{n(p-2)\lambda_{1}}\|\nabla \varphi\|^{2}_{2}<\frac{2p}{n(p-2)\lambda_{1}}\left(\|\nabla \varphi\|^{2}_{2}+\||x|\varphi\|^{2}_{2}\right).
\end{align*}
And for $\varphi\in \mathcal{N}_{2}$, i.e., $S_{2}(\varphi)=0$, $\varphi\neq0$, we have
\begin{align*}
\|\varphi\|^{p}_{p}=\frac{2p}{n(p-2)\lambda_{1}}\left(\|\nabla \varphi\|^{2}_{2}+\||x|\varphi\|^{2}_{2}\right).
\end{align*}
Thus, for $\varphi\in \mathcal{N}_{1}\cup\mathcal{N}_{2}$, we know
\begin{align}\label{2.17}
\|\varphi\|^{p}_{p}\leq\frac{2p}{n(p-2)\lambda_{1}}\left(\|\nabla \varphi\|^{2}_{2}+\||x|\varphi\|^{2}_{2}\right).
\end{align}
By substituting (\ref{2.17}) into (\ref{2.16}), we know
\begin{align}\label{2.18}
\begin{split}
K(\varphi)\geq&\frac{1}{2}\|\nabla \varphi\|^{2}_{2}+\left(\frac{1}{2}-\frac{\lambda_{2}}{q}\right)\|\varphi\|^{2}_{2}+\frac{1}{2}\||x|\varphi\|^{2}_{2}\\
&-\left(\frac{\lambda_{1}}{p}+\frac{\lambda_{2}}{q}\right)\frac{2p}{n(p-2)\lambda_{1}}\left(\|\nabla\varphi\|^{2}_{2}+\||x|\varphi\|^{2}_{2}\right)\\
=&\left(\frac{1}{2}-\frac{2}{n(p-2)}-\frac{2\lambda_{2}p}{\lambda_{1}qn(p-2)}\right)\left(\|\nabla \varphi\|^{2}_{2}+\||x|\varphi\|^{2}_{2}\right)\\
&+\left(\frac{1}{2}-\frac{\lambda_{2}}{q}\right)\|\varphi\|^{2}_{2}.
\end{split}
\end{align}
Noting that (\ref{1.2}) gives
\begin{align}\label{2.19}
\frac{1}{2}-\frac{2}{n(p-2)}>0,
\end{align}
which ensures that the condition (\ref{2.9}) makes sense for $\lambda_{1}>0$ and $\lambda_{2}>0$. And (\ref{2.9}) ensures
\begin{align}\label{2.20}
\frac{1}{2}-\frac{2}{n(p-2)}-\frac{2\lambda_{2}p}{\lambda_{1}qn(p-2)}>0,
\end{align}
which makes (\ref{2.18}) turn to
\begin{align}\label{2.21}
K(\varphi)\geq\left(\frac{1}{2}-\frac{2}{n(p-2)}-\frac{2\lambda_{2}p}{\lambda_{1}qn(p-2)}\right)\|\nabla \varphi\|^{2}_{2}+\left(\frac{1}{2}-\frac{\lambda_{2}}{q}\right)\|\varphi\|^{2}_{2}.
\end{align}
In order to further estimate $K(\varphi)$ in (\ref{2.21}), we also need
\begin{align}\label{2.22}
\frac{1}{2}-\frac{\lambda_{2}}{q}>0,
\end{align}
which is ensured by the condition (\ref{2.10}).
 More precisely, according to (\ref{2.20}) and (\ref{2.22}), to estimate $K(\varphi)$, we only need to deal with the terms $\|\nabla \varphi\|^{2}_{2}$ and $\|\varphi\|^{2}_{2}$ in (\ref{2.21}). Applying Gagliardo-Nirenberg inequality
\begin{align}
&\|\varphi\|^{p}_{p}\leq C_{1}\|\nabla\varphi\|^{p\mu}_{2}\|\varphi\|^{p(1-\mu)}_{2}\label{GN-budengshi}
\end{align}
by choosing $\mu=\frac{n(p-2)}{2p}$, i.e., $\frac{1}{p}=\mu\left(\frac{1}{2}-\frac{1}{n}\right)+\frac{1-\mu}{2}$, and using the inequality
\begin{align}\label{2.24}
\frac{2p}{n(p-2)\lambda_{1}}\|\nabla \varphi\|^{2}_{2}\leq\|\varphi\|^{p}_{p}
\end{align}
by $S_{1}(\varphi)=0$ or $S_{2}(\varphi)=0$, we know
\begin{align}\label{2.25}
\frac{2p}{n(p-2)\lambda_{1}C_{1}}\leq\|\nabla\varphi\|^{\frac{n(p-2)-4}{2}}_{2}\|\varphi\|^{p-\frac{n(p-2)}{2}}_{2}.
\end{align}
In (\ref{2.25}), we see $\frac{n(p-2)-4}{2}>0$ and $p-\frac{n(p-2)}{2}>0$ obtained from $2+\frac{4}{n}<p$ and $p<\frac{2n}{n-2}$, respectively, which are both assumed in (\ref{1.2}). According to (\ref{2.25}) and Cauchy-Schwartz inequality, we obtain
\begin{align}\label{2.26}
\frac{2p}{n(p-2)\lambda_{1}C_{1}}\leq\frac{1}{2}\|\nabla\varphi\|^{n(p-2)-4}_{2}+\frac{1}{2}\|\varphi\|^{2(p+n)-np}_{2},
\end{align}
which means
\begin{align}\label{2.27}
\frac{2p}{n(p-2)\lambda_{1}C_{1}}\leq&\|\nabla\varphi\|^{n(p-2)-4}_{2}
\end{align}
or
\begin{align}\label{2.28}
\frac{2p}{n(p-2)\lambda_{1}C_{1}}\leq&\|\varphi\|^{2(p+n)-np}_{2}.
\end{align}
By substituting (\ref{2.27}) or (\ref{2.28}) into (\ref{2.21}), we can deduce
\begin{align}\label{2.29}
K(\varphi)\geq\left(\frac{1}{2}-\frac{2}{n(p-2)}-\frac{2\lambda_{2}p}{\lambda_{1}qn(p-2)}\right)\left(\frac{2p}{n(p-2)\lambda_{1}C_{1}}\right)^{\frac{2}{n(p-2)-4}}>0
\end{align}
or
\begin{align}\label{2.30}
K(\varphi)\geq\left(\frac{1}{2}-\frac{\lambda_{2}}{q}\right)\left(\frac{2p}{n(p-2)\lambda_{1}C_{1}}\right)^{\frac{2}{2(p+n)-np}}>0,
\end{align}
respectively, which implies $d_{_{\mathcal{N}_{j}}}>0$, $j=1,2$.\\
{\bf{Case \textrm{II}: \bm{$j=3,4.$}}}

In this case, for $\varphi\in \mathcal{N}_{3}$ or $\varphi\in \mathcal{N}_{4}$, we also need to apply (\ref{2.16}) to estimate the functional $K(\varphi)$. Like Case \textrm{I}, we also need to deal with the negative term $\left(\frac{\lambda_{1}}{p}+\frac{\lambda_{2}}{q}\right)\|\varphi\|^{p}_{p}$ in (\ref{2.16}).
For $\varphi\in \mathcal{N}_{3}$, i.e., $S_{3}(\varphi)=0$, $\varphi\neq0$, we know
\begin{align*}
\|\varphi\|^{p}_{p}=&\frac{2p}{n(p-2)\lambda_{1}}\left(\|\nabla \varphi\|^{2}_{2}+c\|\varphi\|^{2}_{2}\right)\\
<&\frac{2p}{n(p-2)\lambda_{1}}\left(\|\nabla \varphi\|^{2}_{2}+c\|\varphi\|^{2}_{2}+\||x|\varphi\|^{2}_{2}\right).
\end{align*}
 And for $\varphi\in \mathcal{N}_{4}$, i.e., $S_{4}(\varphi)=0$, $\varphi\neq0$, we get
\begin{align*}
\|\varphi\|^{p}_{p}=\frac{2p}{n(p-2)\lambda_{1}}\left(\|\nabla \varphi\|^{2}_{2}+c\|\varphi\|^{2}_{2}+\||x|\varphi\|^{2}_{2}\right).
\end{align*}
 Thus,
\begin{align}\label{2.31}
\|\varphi\|^{p}_{p}\leq\frac{2p}{n(p-2)\lambda_{1}}\left(\|\nabla \varphi\|^{2}_{2}+c\|\varphi\|^{2}_{2}+\||x|\varphi\|^{2}_{2}\right)
\end{align}
holds for $\varphi\in \mathcal{N}_{3}\cup\mathcal{N}_{4}$. By substituting (\ref{2.31}) into (\ref{2.16}), we deduce
\begin{align}\label{2.32}
\begin{split}
K(\varphi)\geq&\frac{1}{2}\|\nabla \varphi\|^{2}_{2}+\left(\frac{1}{2}-\frac{\lambda_{2}}{q}\right)\|\varphi\|^{2}_{2}+\frac{1}{2}\||x|\varphi\|^{2}_{2}\\
&-\left(\frac{\lambda_{1}}{p}+\frac{\lambda_{2}}{q}\right)\frac{2p}{n(p-2)\lambda_{1}}\left(\|\nabla\varphi\|^{2}_{2}+c\|\varphi\|^{2}_{2}+\||x|\varphi\|^{2}_{2}\right)\\
=&\left(\frac{1}{2}-\frac{2}{n(p-2)}-\frac{2\lambda_{2}p}{\lambda_{1}qn(p-2)}\right)\left(\|\nabla \varphi\|^{2}_{2}+\||x|\varphi\|^{2}_{2}\right)\\
&+\left(\frac{1}{2}-\left(\frac{2}{n(p-2)}+\frac{2\lambda_{2}p}{\lambda_{1}qn(p-2)}\right)c-\frac{\lambda_{2}}{q}\right)\| \varphi\|^{2}_{2}.
\end{split}
\end{align}
The condition (\ref{2.19}) ensures that the condition (\ref{2.9}) makes sense. And (\ref{2.9}) ensures
\begin{align}\label{2.33}
\frac{1}{2}-\frac{2}{n(p-2)}-\frac{2\lambda_{2}p}{\lambda_{1}qn(p-2)}>0.
\end{align}
And (\ref{2.10}) ensures
\begin{align}\label{2.34}
\frac{1}{2}-\frac{\lambda_{2}}{q}>0.
\end{align}
Thus, $c>0$ can be chosen to make (\ref{cdetiaojian})
hold.
These facts make (\ref{2.32}) turn to
\begin{align}\label{2.35}
\begin{split}
K(\varphi)
\geq&\left(\frac{1}{2}-\frac{2}{n(p-2)}-\frac{2\lambda_{2}p}{\lambda_{1}qn(p-2)}\right)\| \nabla\varphi\|^{2}_{2}\\
&+\left(\frac{1}{2}-\left(\frac{2}{n(p-2)}+\frac{2\lambda_{2}p}{\lambda_{1}qn(p-2)}\right)c-\frac{\lambda_{2}}{q}\right)\|\varphi\|^{2}_{2}\\
\geq&M\left(\|\nabla\varphi\|^{2}_{2}+\|\varphi\|^{2}_{2}\right),
\end{split}
\end{align}where
$$M:=\min\left\{\frac{1}{2}-\frac{2}{n(p-2)}-\frac{2\lambda_{2}p}{\lambda_{1}qn(p-2)},
\frac{1}{2}-\left(\frac{2}{n(p-2)}+\frac{2\lambda_{2}p}{\lambda_{1}qn(p-2)}\right)c-\frac{\lambda_{2}}{q}\right\}.$$
Next, we estimate $\|\nabla\varphi\|^{2}_{2}+\|\varphi\|^{2}_{2}$ in (\ref{2.35}). According to $S_{3}(\varphi)=0$ or $S_{4}(\varphi)=0$, we see
\begin{equation}\label{2.36}
\min\{1,c\}\left(\|\nabla\varphi\|^{2}_{2}+\|\varphi\|^{2}_{2}\right)\leq\|\nabla\varphi\|^{2}_{2}+c\|\varphi\|^{2}_{2}\leq\frac{n(p-2)\lambda_{1}}{2p}\|\varphi\|^{p}_{p}.
\end{equation}
Then, by the Sobolev embedding inequality $\|\varphi\|_{p}\leq C_{2}\left(\|\nabla\varphi\|^{2}_{2}+\|\varphi\|^{2}_{2}\right)^{\frac{1}{2}}$ for the embedding $H^{1}(\mathbb{R}^{n})\hookrightarrow L^{p}(\mathbb{R}^{n})$, we obtain
\begin{align}\label{2.37}
\frac{n(p-2)\lambda_{1}}{2p}\|\varphi\|^{p}_{p}\leq C^{p}_{2}\frac{n(p-2)\lambda_{1}}{2p}\left(\|\nabla\varphi\|^{2}_{2}+\|\varphi\|^{2}_{2}\right)^{\frac{p}{2}},
\end{align}
where $C_{2}$ is the embedding constant of $H^{1}(\mathbb{R}^{n})\hookrightarrow L^{p}(\mathbb{R}^{n})$. According to (\ref{2.36}) and (\ref{2.37}), we know
\begin{align*}
\|\nabla\varphi\|^{2}_{2}+\|\varphi\|^{2}_{2}\leq C^{p}_{2}\frac{n(p-2)\lambda_{1}}{2p\min\{1,c\}}\left(\|\nabla\varphi\|^{2}_{2}+\|\varphi\|^{2}_{2}\right)^{\frac{p}{2}},
\end{align*}
which means
\begin{align}\label{2.38}
\|\nabla\varphi\|^{2}_{2}+\|\varphi\|^{2}_{2}\geq \left(\frac{2p\min\{1,c\}}{\lambda_{1}C_{2}^{p}n(p-2)}\right)^{\frac{2}{p-2}}
\end{align}
due to $\varphi\neq 0$ for $\varphi\in \mathcal{N}_{3}\cup\mathcal{N}_{4}$. Then, by substituting (\ref{2.38}) into (\ref{2.35}), we have
\begin{align}\label{2.39}
K(\varphi)\geq M\left(\frac{2p\min\{1,c\}}{\lambda_{1}C_{2}^{p}n(p-2)}\right)^{\frac{2}{p-2}}>0,
\end{align}
 which implies $d_{_{\mathcal{N}_{j}}}>0$, $j=3,4$.\\
{\bf{Case \textrm{III}: \bm{$j=5,6,7,8.$}}}

According to $1<\frac{n(p-2)}{n(q-2)}$ deduced from (\ref{1.2}), we get
\begin{align}\label{2.40}
\begin{split}
K(\varphi)=&\frac{1}{2}\|\nabla \varphi\|^{2}_{2}+\frac{1}{2}\|\varphi\|^{2}_{2}+\frac{1}{2}\||x|\varphi\|^{2}_{2}-\frac{\lambda_{1}}{p}\|\varphi\|^{p}_{p}-\frac{\lambda_{2}}{q}\|\varphi\|^{q}_{q}\\
\geq&\frac{1}{2}\|\nabla \varphi\|^{2}_{2}+\frac{1}{2}\|\varphi\|^{2}_{2}+\frac{1}{2}\||x|\varphi\|^{2}_{2}-\frac{n(p-2)\lambda_{1}}{n(q-2)p}\|\varphi\|^{p}_{p}-\frac{n(q-2)\lambda_{2}}{n(q-2)q}
\|\varphi\|^{q}_{q}\\
=&\frac{1}{2}\|\nabla\varphi\|^{2}_{2}+\frac{1}{2}\|\varphi\|^{2}_{2}+\frac{1}{2}\||x|\varphi\|^{2}_{2}-\frac{2}{n(q-2)}\left(\frac{n(p-2)\lambda_{1}}{2p}\|\varphi\|^{p}_{p}
+\frac{n(q-2)\lambda_{2}}{2q}\|\varphi\|^{q}_{q}\right).
\end{split}
\end{align}
Next, we estimate the functional $K(\varphi)$ for $\varphi\in \mathcal{N}_{j}$, $j=5,6,7,8$. Based on the structures of the cross-constrained variational functionals $S_{j}(\varphi)$, $j=5,6,7,8$, we require the same conditions on $\lambda_{1}$ and $\lambda_{2}$ for the cases $j=5,6,7,8$. Thus, we consider these four cases in this Case \textrm{III}. As the strategies for the cases $j=5,6$ are similar, and so are the cases $j=7,8$, we shall finish the proof of this Case \textrm{III} by considering two sub-cases, i.e., Case \textrm{III}(a) and Case \textrm{III}(b).\\
{\bf{Case \textrm{III}(a): \bm{$j=5,6.$}}}

For $\varphi\in \mathcal{N}_{5}$, i.e., $S_{5}(\varphi)=0$, $\varphi\neq0$, we know
\begin{align*}
\frac{n(p-2)\lambda_{1}}{2p}\|\varphi\|^{p}_{p}+\frac{n(q-2)\lambda_{2}}{2q}\|\varphi\|^{q}_{q}&=\|\nabla\varphi\|^{2}_{2}<\|\nabla\varphi\|^{2}_{2}+\||x|\varphi\|^{2}_{2}.
\end{align*}
And for $\varphi\in \mathcal{N}_{6}$, i.e., $S_{6}(\varphi)=0$, $\varphi\neq0$, we have
\begin{align*}
\frac{n(p-2)\lambda_{1}}{2p}\|\varphi\|^{p}_{p}+\frac{n(q-2)\lambda_{2}}{2q}\|\varphi\|^{q}_{q}=\|\nabla\varphi\|^{2}_{2}+\||x|\varphi\|^{2}_{2}.
\end{align*}
Thus, for $\varphi\in \mathcal{N}_{5}\cup\mathcal{N}_{6}$, we have
\begin{align}\label{2.41}
\frac{n(p-2)\lambda_{1}}{2p}\|\varphi\|^{p}_{p}+\frac{n(q-2)\lambda_{2}}{2q}\|\varphi\|^{q}_{q}\leq\|\nabla\varphi\|^{2}_{2}+\||x|\varphi\|^{2}_{2}.
\end{align}
By substituting (\ref{2.41}) into (\ref{2.40}), we can deduce
\begin{align}\label{2.42}
\begin{split}
K(\varphi)\geq&\frac{1}{2}\|\nabla\varphi\|^{2}_{2}+\frac{1}{2}\|\varphi\|^{2}_{2}+\frac{1}{2}\||x|\varphi\|^{2}_{2}-\frac{2}{n(q-2)}\left(\|\nabla\varphi\|^{2}_{2}
+\||x|\varphi\|^{2}_{2}\right)\\
=&\left(\frac{1}{2}-\frac{2}{n(q-2)}\right)\left(\|\nabla \varphi\|^{2}_{2}+\||x|\varphi\|^{2}_{2}\right)+\frac{1}{2}\|\varphi\|^{2}_{2}.
\end{split}
\end{align}
Here, the condition (\ref{1.2}), i.e., $2+\frac{4}{n}<q$, ensures
\begin{align}\label{2.43}
\frac{1}{2}-\frac{2}{n(q-2)}>0,
\end{align}
which makes (\ref{2.42}) turn to
\begin{align}\label{2.44}
K(\varphi)\geq&\left(\frac{1}{2}-\frac{2}{n(q-2)}\right)\|\nabla \varphi\|^{2}_{2}+\frac{1}{2}\|\varphi\|^{2}_{2}.
\end{align}
Next, we estimate the terms $\|\nabla \varphi\|^{2}_{2}$ and $\|\varphi\|^{2}_{2}$ in (\ref{2.44}).
According to $S_{5}(\varphi)=0$ or $S_{6}(\varphi)=0$, we know
\begin{equation}
\frac{n(p-2)\lambda_{1}}{2p}\|\varphi\|^{p}_{p}+\frac{n(q-2)\lambda_{2}}{2q}\|\varphi\|^{q}_{q}\geq\|\nabla\varphi\|^{2}_{2},\nonumber
\end{equation}
which means
\begin{align}\label{2.45}
&\frac{n(p-2)\lambda_{1}}{2p}\|\varphi\|^{p}_{p}\geq\frac{1}{2}\|\nabla\varphi\|^{2}_{2}
\end{align}
or
\begin{align}\label{2.46}
&\frac{n(q-2)\lambda_{2}}{2q}\|\varphi\|^{q}_{q}\geq\frac{1}{2}\|\nabla\varphi\|^{2}_{2}.
\end{align}

If (\ref{2.45}) holds, a similar process to the arguments from (\ref{GN-budengshi}) to (\ref{2.26}) gives
\begin{align}\label{2.47}
\begin{split}
\frac{p}{n(p-2)\lambda_{1}C_{1}}\leq&\|\nabla\varphi\|^{\frac{n(p-2)-4}{2}}_{2}\|\varphi\|^{p-\frac{n(p-2)}{2}}_{2}\\
\leq&\frac{1}{2}\|\nabla\varphi\|^{n(p-2)-4}_{2}+\frac{1}{2}\|\varphi\|^{2(p+n)-np}_{2},
\end{split}
\end{align}
which implies
\begin{align}\label{2.48}
\frac{p}{n(p-2)\lambda_{1}C_{1}}\leq&\|\nabla\varphi\|^{n(p-2)-4}_{2}
\end{align}
or
\begin{align}\label{2.49}
\frac{p}{n(p-2)\lambda_{1}C_{1}}\leq&\|\varphi\|^{2(p+n)-np}_{2}.
\end{align}
By substituting (\ref{2.48}) or (\ref{2.49}) into (\ref{2.44}), we can obtain
\begin{align}\label{2.50}
K(\varphi)\geq \left(\frac{1}{2}-\frac{2}{n(q-2)}\right)\left(\frac{p}{n(p-2)\lambda_{1}C_{1}}\right)^{\frac{2}{n(p-2)-4}}>0
\end{align}
or
\begin{align}\label{2.51}
K(\varphi)\geq \frac{1}{2}\left(\frac{p}{n(p-2)\lambda_{1}C_{1}}\right)^{\frac{2}{2(p+n)-np}}>0,
\end{align}
respectively.

If (\ref{2.46}) holds, the process of dealing with this case is similar to that of treating the case that (\ref{2.45}) holds, also the arguments from (\ref{GN-budengshi}) to (\ref{2.26}). For this case, we need to employ a different version of Gagliardo-Nirenberg inequality in the following form
\begin{align}
&\|\varphi\|^{q}_{q}\leq C_{3}\|\nabla\varphi\|^{q\theta}_{2}\|\varphi\|^{q(1-\theta)}_{2}\label{GN-budengshi1}
\end{align}
to deal with (\ref{2.46}) by choosing $\theta=\frac{n(q-2)}{2q}$, i.e., $\frac{1}{q}=\theta\left(\frac{1}{2}-\frac{1}{n}\right)+\frac{1-\theta}{2}$. More precisely, by (\ref{GN-budengshi1}) and applying the inequality $\frac{q}{n(q-2)\lambda_{2}}\|\nabla \varphi\|^{2}_{2}\leq\|\varphi\|^{q}_{q}$ as another version of (\ref{2.46}), we get
\begin{align*}
\frac{q}{n(q-2)\lambda_{2}C_{3}}\leq\|\nabla\varphi\|^{\frac{n(q-2)-4}{2}}_{2}\|\varphi\|^{q-\frac{n(q-2)}{2}}_{2}.
\end{align*}
Then, by the similar process of inferring the validity of (\ref{2.27}) or (\ref{2.28}), we obtain
\begin{align}\label{2.53}
\frac{q}{n(q-2)\lambda_{2}C_{3}}\leq&\|\nabla\varphi\|^{n(q-2)-4}_{2}
\end{align}
or
\begin{align}\label{2.54}
\frac{q}{n(q-2)\lambda_{2}C_{3}}\leq&\|\varphi\|^{2(q+n)-nq}_{2}.
\end{align}
By substituting (\ref{2.53}) or (\ref{2.54}) into (\ref{2.44}), we can obtain
\begin{align}\label{2.55}
K(\varphi)\geq \left(\frac{1}{2}-\frac{2}{n(q-2)}\right)\left(\frac{q}{n(q-2)\lambda_{2}C_{3}}\right)^{\frac{2}{n(q-2)-4}}>0
\end{align}
or
\begin{align}\label{2.56}
K(\varphi)\geq \frac{1}{2}\left(\frac{q}{n(q-2)\lambda_{2}C_{3}}\right)^{\frac{2}{2(q+n)-nq}}>0,
\end{align}
respectively. According to (\ref{2.50}), (\ref{2.51}), (\ref{2.55}), and (\ref{2.56}), we know $d_{_{\mathcal{N}_{j}}}>0$, $j=5,6$.\\
{\bf{Case \textrm{III}(b): \bm{$j=7,8$}.}}

For $\varphi\in \mathcal{N}_{7}$, i.e., $S_{7}(\varphi)=0$, $\varphi\neq0$, we know
\begin{align*}
&\frac{n(p-2)\lambda_{1}}{2p}\|\varphi\|^{p}_{p}+\frac{n(q-2)\lambda_{2}}{2q}\|\varphi\|^{q}_{q}\\
=&\|\nabla\varphi\|^{2}_{2}+\|\varphi\|^{2}_{2}\\
<&\|\nabla\varphi\|^{2}_{2}+\|\varphi\|^{2}_{2}+\||x|\varphi\|^{2}_{2}.
\end{align*}
And for $\varphi\in \mathcal{N}_{8}$, i.e., $S_{8}(\varphi)=0$, $\varphi\neq0$, we have
\begin{align*}
\frac{n(p-2)\lambda_{1}}{2p}\|\varphi\|^{p}_{p}+\frac{n(q-2)\lambda_{2}}{2q}\|\varphi\|^{q}_{q}=\|\nabla\varphi\|^{2}_{2}+\|\varphi\|^{2}_{2}+\||x|\varphi\|^{2}_{2}.
\end{align*}
Therefore, for $\varphi\in \mathcal{N}_{7}\cup\mathcal{N}_{8}$, we have
\begin{align}\label{2.57}
\frac{n(p-2)\lambda_{1}}{2p}\|\varphi\|^{p}_{p}+\frac{n(q-2)\lambda_{2}}{2q}\|\varphi\|^{q}_{q}\leq\|\nabla\varphi\|^{2}_{2}+\|\varphi\|^{2}_{2}+\||x|\varphi\|^{2}_{2}.
\end{align}
Substituting (\ref{2.57}) into (\ref{2.40}), we see
\begin{align}\label{2.58}
\begin{split}
K(\varphi)\geq&\frac{1}{2}\|\nabla\varphi\|^{2}_{2}+\frac{1}{2}\|\varphi\|^{2}_{2}+\frac{1}{2}\||x|\varphi\|^{2}_{2}-\frac{2}{n(q-2)}\left(\|\nabla\varphi\|^{2}_{2}+
\|\varphi\|^{2}_{2}+\||x|\varphi\|^{2}_{2}\right)\\
=&\left(\frac{1}{2}-\frac{2}{n(q-2)}\right)\left(\|\nabla\varphi\|^{2}_{2}+\|\varphi\|^{2}_{2}+\||x|\varphi\|^{2}_{2}\right).
\end{split}
\end{align}
Due to (\ref{2.43}), (\ref{2.58}) turns to
\begin{align}\label{2.59}
K(\varphi)\geq&\left(\frac{1}{2}-\frac{2}{n(q-2)}\right)\left(\|\nabla\varphi\|^{2}_{2}+\|\varphi\|^{2}_{2}\right).
\end{align}
 Next, we estimate $\|\nabla\varphi\|^{2}_{2}+\|\varphi\|^{2}_{2}$ in (\ref{2.59}). According to $S_{7}(\varphi)=0$ or $S_{8}(\varphi)=0$, we see
\begin{align*}
\frac{n(p-2)\lambda_{1}}{2p}\|\varphi\|^{p}_{p}+\frac{n(q-2)\lambda_{2}}{2q}\|\varphi\|^{q}_{q}\geq\|\nabla\varphi\|^{2}_{2}+\|\varphi\|^{2}_{2},
\end{align*}
which means
\begin{align}\label{2.60}
&\frac{n(p-2)\lambda_{1}}{2p}\|\varphi\|^{p}_{p}\geq\frac{1}{2}\left(\|\nabla\varphi\|^{2}_{2}+\|\varphi\|^{2}_{2}\right)
\end{align}
or
\begin{align}\label{2.61}
&\frac{n(q-2)\lambda_{2}}{2q}\|\varphi\|^{q}_{q}\geq\frac{1}{2}\left(\|\nabla\varphi\|^{2}_{2}+\|\varphi\|^{2}_{2}\right).
\end{align}

If (\ref{2.60}) holds, by Sobolev embedding inequality of $H^{1}(\mathbb{R}^{n})\hookrightarrow L^{p}(\mathbb{R}^{n})$, we retrieve (\ref{2.37}), which makes (\ref{2.60}) become
\begin{align*}
\frac{1}{2}\left(\|\nabla\varphi\|^{2}_{2}+\|\varphi\|^{2}_{2}\right)\leq C^{p}_{2}\frac{n(p-2)\lambda_{1}}{2p}(\|\nabla\varphi\|^{2}_{2}+\|\varphi\|^{2}_{2})^{\frac{p}{2}},
\end{align*}
i.e.,
\begin{align}\label{2.62}
\|\nabla\varphi\|^{2}_{2}+\|\varphi\|^{2}_{2}\geq \left(\frac{p}{\lambda_{1}C_{2}^{p}n(p-2)}\right)^{\frac{2}{p-2}}
\end{align}
due to $\varphi\neq 0$ for $\varphi\in \mathcal{N}_{7}\cup\mathcal{N}_{8}$.

If (\ref{2.61}) holds, by using Sobolev embedding inequality $\|\varphi\|_{q}\leq C_{4}\left(\|\nabla\varphi\|^{2}_{2}+\|\varphi\|^{2}_{2}\right)^{\frac{1}{2}}$ for the embedding $H^{1}(\mathbb{R}^{n})\hookrightarrow L^{q}(\mathbb{R}^{n})$, where $C_{4}$ is the embedding constant of $H^{1}(\mathbb{R}^{n})\hookrightarrow L^{q}(\mathbb{R}^{n})$, (\ref{2.61}) becomes
\begin{align*}
\frac{1}{2}\left(\|\nabla\varphi\|^{2}_{2}+\|\varphi\|^{2}_{2}\right)\leq C^{q}_{4}\frac{n(q-2)\lambda_{2}}{2q}(\|\nabla\varphi\|^{2}_{2}+\|\varphi\|^{2}_{2})^{\frac{q}{2}},
\end{align*}
i.e.,
\begin{align}\label{2.63}
\|\nabla\varphi\|^{2}_{2}+\|\varphi\|^{2}_{2}\geq \left(\frac{q}{\lambda_{2}C_{4}^{q}n(q-2)}\right)^{\frac{2}{q-2}}
\end{align}
due to $\varphi\neq 0$ for $\varphi\in \mathcal{N}_{7}\cup\mathcal{N}_{8}$.
By substituting (\ref{2.62}) or (\ref{2.63}) into (\ref{2.59}), we have
\begin{align*}
K(\varphi)\geq\left(\frac{1}{2}-\frac{2}{n(q-2)}\right)\left(\frac{p}{\lambda_{1}C_{2}^{p}n(p-2)}\right)^{\frac{2}{p-2}}>0
\end{align*}
or
\begin{align*}
K(\varphi)\geq\left(\frac{1}{2}-\frac{2}{n(q-2)}\right)\left(\frac{q}{\lambda_{2}C_{4}^{q}n(q-2)}\right)^{\frac{2}{q-2}}>0,
\end{align*}
respectively, which implies $d_{_{\mathcal{N}_{j}}}>0$, $j=7,8$.
\end{proof}
\end{lemma}
\noindent\begin{remark}\label{rem2.1}(The necessity of the assumptions (\ref{2.9}) and (\ref{2.10}) in Lemma \ref{lem2.3})
As we focus on the non-negative initial energy case $0<K(\varphi_{0})<d_{\mathcal{N}_{j}}$, $d_{\mathcal{N}_{j}}>0$ for $j=1,2,\cdots,8$, the conditions given in Lemma \ref{lem2.3} for $d_{\mathcal{N}_{j}}>0$ are also applied to the main conclusions to classify the initial data to distinguish the global solution and finite time blowup solution of (\ref{1.1}), i.e., Theorem \ref{th4.1} and Theorem \ref{th4.2}. In discussing the positivity of the depth of the potential well $d_{\mathcal{N}_{j}}$, we find that the strategies of the proofs depend on the structures of the cross-constrained variational functionals, i.e., $S_{j}(\varphi)$, $j=1,2,\cdots,8$. For the cases $j=5,6,7,8$, the very loose restrictions $\lambda_{1}>0$ and $\lambda_{2}>0$ are enough, but for the cases $j=1,2,3,4$, we have to add some additional restrictions on $\lambda_{1}$ and $\lambda_{2}$, like (\ref{2.9}) and (\ref{2.10}), to overcome the difficulties arising in the estimate, and these restrictions are also required later for the main conclusions. Although we believe that these conditions, i.e., (\ref{2.9}) and (\ref{2.10}), can be improved or even removed, we have to say that these conditions are the best we can achieve in this work. From this perspective in this remark, we analyze the roles of these additional conditions and the necessity. We hope such efforts can be helpful for the further improvement of these additional conditions until to remove them.\\

The assumptions (\ref{2.9}) and (\ref{2.10}) as the additional restrictions on $\lambda_{1}$ and $\lambda_{2}$ are required in the proof of Case \textrm{I}, i.e., $j=1,2$, and Case \textrm{II}, i.e., $j=3,4$. In contrast, these restrictions are not required for Case \textrm{III}, i.e., $j=5,6,7,8$. This is mainly because the cross-constrained variational functionals $S_{j}(\varphi)$, $j=1,2,3,4$, have the term $\|\varphi\|^{p}_{p}$ and do not have the term $\|\varphi\|^{q}_{q}$, while $S_{j}(\varphi)$, $j=5,6,7,8$, have both these two terms, which leads to different strategies of the proofs. More precisely, since the terms $\|\varphi\|^{p}_{p}$ and $\|\varphi\|^{q}_{q}$ are two negative ones in the functional $K(\varphi)$, in order to prove $d_{\mathcal{N}_{j}}>0$, in the cases $j=5,6,7,8$, by applying the condition $S_{j}(\varphi)=0$, i.e., $\varphi\in \mathcal{N}_{j}$, both the terms $\|\varphi\|^{p}_{p}$ and $\|\varphi\|^{q}_{q}$ can be controlled, as shown by (\ref{2.41}) for $j=5,6$ and (\ref{2.57}) for $j=7,8$. However, in the cases $j=1,2,3,4$, by applying the condition $S_{j}(\varphi)=0$, only the term $\|\varphi\|^{p}_{p}$ can be controlled, as shown by (\ref{2.17}) for $j=1,2$ and (\ref{2.31}) for $j=3,4$. Thus, we need to deal with the term $\|\varphi\|^{q}_{q}$ additionally, which needs the assumptions (\ref{2.9}) and (\ref{2.10}) for $j=1,2,3,4$.

Next, we explain the necessity of the assumptions (\ref{2.9}) and (\ref{2.10}) for the cases $j=1,2$. The analysis for the case $j=3,4$ is similar. In the arguments for treating the term $\|\varphi\|^{q}_{q}$ in the cases $j=1,2$, by dividing the space $\mathbb{R}^n$ into $\Omega_{1}$ (the domain where $|\varphi|\geq1$) and $\Omega_{2}$ (the domain where $|\varphi|<1$), the term $\frac{\lambda_{2}}{q}\|\varphi\|^{q}_{q}$ in $K(\varphi)$ is rewritten as
\begin{align*}
\frac{\lambda_{2}}{q}\|\varphi\|^{q}_{q}=\frac{\lambda_{2}}{q}\int_{\Omega_{1}} |\varphi|^{q}{\rm d}x+\frac{\lambda_{2}}{q}\int_{\Omega_{2}} |\varphi|^{q}{\rm d}x,
\end{align*}
as shown by (\ref{2.11}).

For $\frac{\lambda_{2}}{q}\int_{\Omega_{1}} |\varphi|^{q}{\rm d}x$, due to $p>q$ assumed in this paper, we know $|\varphi|^{q}\leq|\varphi|^{p}$ for $x\in\Omega_{1}$, which implies
\begin{align}\label{2.64}
\frac{\lambda_{2}}{q}\int_{\Omega_{1}} |\varphi|^{q}{\rm d}x\leq\frac{\lambda_{2}}{q}\int_{\Omega_{1}} |\varphi|^{p}{\rm d}x\leq\frac{\lambda_{2}}{q}\|\varphi\|^{p}_{p}.
\end{align}
Thus, we can deal with $\frac{\lambda_{2}}{q}\int_{\Omega_{1}} |\varphi|^{q}{\rm d}x$ by controlling $\frac{\lambda_{2}}{q}\|\varphi\|^{p}_{p}$ in (\ref{2.64}). In fact, except for $\frac{\lambda_{2}}{q}\|\varphi\|^{p}_{p}$, the negative term $\frac{\lambda_{1}}{p}\|\varphi\|^{p}_{p}$ in $K(\varphi)$ is also needed to be controlled. This means we need to control
\begin{align*}
\left(\frac{\lambda_{1}}{p}+\frac{\lambda_{2}}{q}\right)\|\varphi\|^{p}_{p}
\end{align*}
in (\ref{2.16}), that is to control
\begin{align}\label{2.65}
\left(\frac{\lambda_{1}}{p}+\frac{\lambda_{2}}{q}\right)\frac{2p}{n(p-2)\lambda_{1}}\left(\|\nabla \varphi\|^{2}_{2}+\||x|\varphi\|^{2}_{2}\right)
\end{align}
 according to (\ref{2.17}). Here, we can only use the term
\begin{align}\label{2.66}
\frac{1}{2}\left(\|\nabla\varphi\|^{2}_{2}+\||x|\varphi\|^{2}_{2}\right)
\end{align}
in the functional $K(\varphi)$ to control (\ref{2.65}). Thus, we require
\begin{align*}
\left(\frac{\lambda_{1}}{p}+\frac{\lambda_{2}}{q}\right)\frac{2p}{n(p-2)\lambda_{1}}<\frac{1}{2},
\end{align*}
i.e., the assumption (\ref{2.9}).

For $\frac{\lambda_{2}}{q}\int_{\Omega_{2}} |\varphi|^{q}{\rm d}x$, according to $q>2$ assumed in this paper, we can deduce $|\varphi|^{q}\leq|\varphi|^{2}$ for $x\in\Omega_{2}$, which means
\begin{align}\label{2.67}
\frac{\lambda_{2}}{q}\int_{\Omega_{2}} |\varphi|^{q}{\rm d}x\leq\frac{\lambda_{2}}{q}\int_{\Omega_{2}} |\varphi|^{2}{\rm d}x\leq\frac{\lambda_{2}}{q}\|\varphi\|^{2}_{2}.
\end{align}
Thus, we deal with $\frac{\lambda_{2}}{q}\int_{\Omega_{2}} |\varphi|^{q}{\rm d}x$ by controlling the term
$\frac{\lambda_{2}}{q}\|\varphi\|^{2}_{2}$ in (\ref{2.67}).
Here, we can only use the term
$\frac{1}{2}\|\varphi\|^{2}_{2}$ in $K(\varphi)$ to control $\frac{\lambda_{2}}{q}\|\varphi\|^{2}_{2}$ in (\ref{2.67}), which means that we need to require $$\frac{\lambda_{2}}{q}<\frac{1}{2},$$ i.e., the assumption (\ref{2.10}).\\

\end{remark}

Next, we shall give Lemma \ref{lem2.4} to state the invariant manifolds associated with the problem considered in the present paper. Constructing and proving the invariant manifolds is the crucial step to classify the initial data, especially to derive the sharp condition on the initial data to distinguish the initial data leading to global-in-time solution and finite time blowup solution since the invariant manifolds can ensure the solution keeps the properties we impose on the initial data. Here, for $j=1,2,\cdots,8$, we need to divide the discussion of this problem into three cases, i.e., negative initial energy case $K(\varphi_{0})<0$, non-negative low initial energy case $0\leq K(\varphi_{0})<d_{\mathcal{N}_{j}}$, and high initial energy case $K(\varphi_{0})\geq d_{\mathcal{N}_{j}}$. For negative initial energy case, it is easy to verify that the corresponding solution blows up in finite time, as we showed in Section \ref{sec3}. And the efforts to classify the initial data for this case are not necessary. For non-negative low initial energy case, we require the conclusions $d_{_{\mathcal{N}_{j}}}>0$, $j=1,2,\cdots,8$, given by Lemma \ref{lem2.3} to ensure the conditions $0\leq K(\varphi_{0})<d_{\mathcal{N}_{j}}$ make sense. And high initial energy case will be discussed in Section \ref{sec5}. Next, we shall show the invariance of the manifolds for non-negative low initial energy $0\leq K(\varphi_{0})<d_{\mathcal{N}_{j}}$, $j=1,2,\cdots,8$, by the following Lemma \ref{lem2.4}. Based on the above analysis, we will see that Lemma \ref{lem2.4} will naturally use the conditions of Lemma \ref{lem2.3}.

\begin{lemma}\label{lem2.4}
(The invariance of the manifolds $\mathcal{W}_{j}$ and $\mathcal{V}_{j}$ for non-negative low initial energy $0\leq K(\varphi_{0})<d_{\mathcal{N}_{j}}$, $j=1,2,\cdots,8$)
For $j=1,2,\cdots,8$, let $p$ and $q$ satisfy (\ref{1.2}), $\lambda_{1}>0$ and $\lambda_{2}>0$. For the cases $j=1,2,3,4$, $\lambda_{1}$ and $\lambda_{2}$ satisfy the additional restrictions (\ref{2.9}) and (\ref{2.10}). Then, for $\varphi_{0}\in \mathcal{W}_{j}$, the corresponding solution $\varphi$ of (\ref{1.1}) belongs to $\mathcal{W}_{j}$. And for $\varphi_{0}\in \mathcal{V}_{j}$, the corresponding solution $\varphi$ of (\ref{1.1}) belongs to $\mathcal{V}_{j}$.
\begin{proof}
Although we require the different conditions for the different cases, i.e., $j=1,2,3,4$, and $j=5,6,7,8$, we need not propose different approaches to deal with them since these different conditions are imposed to make the conclusions $d_{_{\mathcal{N}_{j}}}>0$ hold for these different cases according to Lemma \ref{lem2.3}. Hence, we only need to prove the case $j=1$. And the proofs of other cases, i.e., $j=2,3,\cdots,8$, are similar.

 For the initial data $\varphi_{0}\in\mathcal{W}_{1}$, according to Lemma \ref{lem2.1}, we know that there exists the unique local solution $\varphi\in C([0,T),H)\cap C^{1}([0,T),H^{*})$ and also the continuity of $S_{1}(\varphi(t))$ with respect to time $t$, where $T$ is the existence time of the local solution. The condition $\varphi_{0}\in\mathcal{W}_{1}$ means the case $\varphi_{0}=0$ or the case $0\leq K(\varphi_{0})<d_{\mathcal{N}_{1}}$ and $S_{1}(\varphi_{0})>0$. For $\varphi_{0}=0$, we can obtain $\varphi(t)=0$ for $t\in(0,T)$ according to (\ref{2.2}), which means $\varphi(t)\in \mathcal{W}_{1}$ for $t\in(0,T)$. For $0\leq K(\varphi_{0})<d_{\mathcal{N}_{1}}$ and $S_{1}(\varphi_{0})>0$, i.e., $\varphi_{0}\neq0$, we know that
 \begin{align}\label{2.68}
0\leq K(\varphi(t))<d_{\mathcal{N}_{1}}
\end{align}
 holds for all $t\in(0,T)$ due to
\begin{align*}
K(\varphi(t))=E(t)+\frac{1}{2}\|\varphi(t)\|^{2}_{2}=E(0)+\frac{1}{2}\|\varphi_{0}\|^{2}_{2}=K(\varphi_{0})
\end{align*}
 given by (\ref{2.2}) and (\ref{2.3}). Next, we claim $S_{1}(\varphi(t))>0$ for all $t\in(0,T)$. Arguing by contradiction, according to the continuity of $S_{1}(\varphi(t))$ with respect to time $t$, we suppose that there exists the first time $t_{0}\in(0,T)$ to make $S_{1}(\varphi(t_{0}))=0$ and $S_{1}(\varphi(t))>0$ for $t\in(0,t_{0})$. Here, we know $\varphi(t_{0})\neq 0$ due to $\varphi_{0}\neq0$ and (\ref{2.2}). Thus, we have $\varphi(t_{0})\in \mathcal{N}_{1}$, which means $K(\varphi(t_{0}))\geq d_{_{\mathcal{N}_{1}}}$ according to the definition of $d_{_{\mathcal{N}_{1}}}$. This contradicts $K(\varphi(t_{0}))<d_{_{\mathcal{N}_{1}}}$ shown by (\ref{2.68}). Therefore $\varphi(t)\in \mathcal{W}_{1}$ for $t\in(0,T)$.

 By the similar arguments, we can also prove that the solution $\varphi$ of (\ref{1.1}) belongs to $\mathcal{V}_{1}$ if the initial data $\varphi_{0}\in \mathcal{V}_{1}$.
\end{proof}
\end{lemma}

 \section[Calculation]{The calculation of $J''(t)$ and finite time blowup for negative initial energy $K(\varphi_{0})<0$}\label{sec3}

In the present paper, the most interesting part is to investigate the finite time blowup. Here, we shall consider three cases, i.e., negative initial energy $K(\varphi_{0})<0$, non-negative low initial energy $0\leq K(\varphi_{0})<d_{\mathcal{N}_{j}}$, and high initial energy $K(\varphi_{0})\geq d_{\mathcal{N}_{j}}$, $j=1,2,\cdots,8$. The strategy to deal with the above three cases is to apply an auxiliary functional $J(t)=\||x|\varphi\|^{2}_{2}$ introduced in Section \ref{sec1}, and to verify $J''(t)<0$ to prove finite time blowup for the above three cases. In this section, we calculate $J''(t)$ and prove the finite time blowup for negative initial energy case, i.e., $K(\varphi_{0})<0$, by using the fact $J''(t)<0$ deduced from the condition $K(\varphi_{0})<0$. Here, although the calculation of $J''(t)$ looks somewhat elementary and lengthy, we dare not omit one step in this calculation since these details can help us avoid mistakes in calculation, and omitting them will cause unnecessary confusion in reading.
\begin{lemma}\label{lem3.1}
(The calculation of $J''(t)$)
For $J(t)=\||x|\varphi\|^{2}_{2}$, if $\varphi$ is the solution to the Cauchy problem (\ref{1.1}) with the initial data $\varphi_{0}\in H$, then
\begin{equation}\label{3.1}
J''(t)=8\left(\|\nabla\varphi\|^2_{2}-\||x|\varphi\|^{2}_{2}-\frac{n(p-2)\lambda_{1}}{2p}\|\varphi\|^{p}_{p}-\frac{n(q-2)\lambda_{2}}{2q}\|\varphi\|^{q}_{q}\right).
\end{equation}
\end{lemma}
\begin{proof}
In order to calculate $J''(t)$, we first calculate $J'(t)$.
\begin{align}\label{3.2}
\begin{split}
J'(t)=&\frac{{\rm d}}{{\rm d}t}\int_{\mathbb{R}^n} |x|^{2}|\varphi|^{2} {\rm d}x=\frac{{\rm d}}{{\rm d}t}\int_{\mathbb{R}^n} |x|^{2}\varphi \overline{\varphi} {\rm d}x\\
=&\int_{\mathbb{R}^n} |x|^{2}\left(\varphi_{t} \overline{\varphi}+\varphi \overline{\varphi_{t}}\right){\rm d}x=\int_{\mathbb{R}^n} |x|^{2}\left(\varphi_{t} \overline{\varphi}+\overline{\varphi_{t}\overline{\varphi}}\right){\rm d}x\\
=&2{\rm Re}\int_{\mathbb{R}^n} |x|^{2}\varphi_{t} \overline{\varphi}{\rm d}x.
\end{split}
\end{align}
We multiply both sides of the equation in problem (\ref{1.1}) by $-i\overline{\varphi}$ to get
\begin{align}\label{3.3}
\overline{\varphi}\varphi_{t}=i\left(\overline{\varphi}\Delta\varphi-|x|^{2}|\varphi|^{2}+\lambda_{1}|\varphi|^{p}+\lambda_{2}|\varphi|^{q}\right).
\end{align}
By substituting (\ref{3.3}) into (\ref{3.2}), we get
\begin{align}\label{3.4}
\begin{split}
J'(t)=&2{\rm Re}\int_{\mathbb{R}^n} i|x|^{2}\left(\overline{\varphi}\Delta\varphi-|x|^{2}|\varphi|^{2}+\lambda_{1}|\varphi|^{p}+\lambda_{2}|\varphi|^{q}\right) {\rm d}x\\
=&-2{\rm Im}\int_{\mathbb{R}^n} |x|^{2}\left(\overline{\varphi}\Delta\varphi-|x|^{2}|\varphi|^{2}+\lambda_{1}|\varphi|^{p}+\lambda_{2}|\varphi|^{q}\right) {\rm d}x\\
=&-2{\rm Im}\int_{\mathbb{R}^n} |x|^{2}\overline{\varphi}\Delta\varphi {\rm d}x\\
=&2{\rm Im}\int_{\mathbb{R}^n} |x|^{2}\varphi\Delta\overline{\varphi} {\rm d}x.
\end{split}
\end{align}
Next, we calculate $J''(t)$ based on (\ref{3.4}) as
\begin{align}\label{3.5}
\begin{split}
J''(t)=&\frac{{\rm d}}{{\rm d}t}J'(t)\\
=&2\frac{{\rm d}}{{\rm d}t}{\rm Im}\int_{\mathbb{R}^n} |x|^{2}\varphi\Delta\overline{\varphi} {\rm d}x\\
=&2{\rm Im}\int_{\mathbb{R}^n}|x|^{2}\varphi_{t}\Delta\overline{\varphi}{\rm d}x+2{\rm Im}\int_{\mathbb{R}^n} |x|^{2}\varphi\Delta\overline{\varphi_{t}}{\rm d}x\\
=&2{\rm Im}\int_{\mathbb{R}^n}|x|^{2}\varphi_{t}\Delta\overline{\varphi}{\rm d}x+2{\rm Im}\int_{\mathbb{R}^n} \overline{\varphi_{t}}\Delta\left(|x|^{2}\varphi\right){\rm d}x\\
=&2{\rm Im}\int_{\mathbb{R}^n}\left(|x|^{2}\varphi_{t}\Delta\overline{\varphi}+\overline{\varphi_{t}}\left(\sum_{j=1}^{n}\frac{\partial^{2}}{\partial x_{j}^{2}}\left(|x|^{2}\varphi\right)\right)\right){\rm d}x\\
=&2{\rm Im}\int_{\mathbb{R}^n}\left(|x|^{2}\varphi_{t}\Delta\overline{\varphi}+\overline{\varphi_{t}}\left(\sum_{j=1}^{n}\frac{\partial}{\partial x_{j}}\left(|x|^{2}\frac{\partial \varphi}{\partial x_{j}}+2x_{j}\varphi\right)\right)\right){\rm d}x\\
=&2{\rm Im}\int_{\mathbb{R}^n}\left(|x|^{2}\varphi_{t}\Delta\overline{\varphi}+\overline{\varphi_{t}}\left(\sum_{j=1}^{n}\left(|x|^{2}\frac{\partial^{2} \varphi}{\partial x^{2}_{j}}+4x_{j}\frac{\partial \varphi}{\partial x_{j}}+2\varphi\right)\right)\right){\rm d}x\\
=&2{\rm Im}\int_{\mathbb{R}^n}\left(|x|^{2}\varphi_{t}\Delta\overline{\varphi}+\overline{\varphi_{t}}\left(|x|^{2}\Delta\varphi+4x\cdot\nabla\varphi+2n\varphi\right)\right){\rm d}x\\
=&2{\rm Im}\int_{\mathbb{R}^n}|x|^{2}\left(\varphi_{t}\Delta\overline{\varphi}+\overline{\varphi_{t}\Delta\overline{\varphi}}\right){\rm d}x+4{\rm Im}\int_{\mathbb{R}^n}\overline{\varphi_{t}}\left(n\varphi+2x\cdot\nabla\varphi\right){\rm d}x\\
=&4{\rm Im}\int_{\mathbb{R}^n}\overline{\varphi_{t}}\left(n\varphi+2x\cdot\nabla\varphi\right){\rm d}x.
\end{split}
\end{align}
We take the conjugate of both sides of the equation in problem (\ref{1.1}) and multiply both sides by $-i$ to obtain
\begin{align}\label{3.6}
\overline{\varphi_{t}}=-i\left(\Delta\overline{\varphi}-|x|^{2}\overline{\varphi}+\lambda_{1}|\varphi|^{p-2}\overline{\varphi}+\lambda_{2}|\varphi|^{q-2}\overline{\varphi}\right).
\end{align}
By substituting (\ref{3.6}) into (\ref{3.5}), we obtain
\begin{align}\label{3.7}
\begin{split}
J''(t)=&4{\rm Im}\int_{\mathbb{R}^n}-i\left(\Delta\overline{\varphi}-|x|^{2}\overline{\varphi}+\lambda_{1}|\varphi|^{p-2}\overline{\varphi}+\lambda_{2}|\varphi|^{q-2}\overline{\varphi}\right)\left(n\varphi+2x\cdot\nabla\varphi\right){\rm d}x\\
=&-4{\rm Re}\int_{\mathbb{R}^n}\left(\Delta\overline{\varphi}-|x|^{2}\overline{\varphi}+\lambda_{1}|\varphi|^{p-2}\overline{\varphi}+\lambda_{2}|\varphi|^{q-2}\overline{\varphi}\right)\left(n\varphi+2x\cdot\nabla\varphi\right){\rm d}x\\
=&-4\underbrace{{\rm Re}\int_{\mathbb{R}^n}\left(n\varphi+2x\cdot\nabla\varphi\right)\Delta\overline{\varphi}{\rm d}x}_{:=J_{1}}+4\underbrace{{\rm Re}\int_{\mathbb{R}^n}\left(n\varphi+2x\cdot\nabla\varphi\right)|x|^{2}\overline{\varphi}{\rm d}x}_{:=J_{2}}\\
&-4\underbrace{{\rm Re}\int_{\mathbb{R}^n}\left(n\varphi+2x\cdot\nabla\varphi\right)\left(\lambda_{1}|\varphi|^{p-2}\overline{\varphi}+\lambda_{2}|\varphi|^{q-2}\overline{\varphi}\right){\rm d}x}_{:=J_{3}}.
\end{split}
\end{align}
Next, we calculate $J_{j}\ (j=1,2,3)$ separately.
For $J_{1}$, we have
\begin{align}\label{3.8}
J_{1}=&{\rm Re}\int_{\mathbb{R}^n}n\varphi\Delta\overline{\varphi}{\rm d}x+{\rm Re}\int_{\mathbb{R}^n}2\left(x\cdot\nabla\varphi\right)\Delta\overline{\varphi}{\rm d}x\nonumber\\
=&-{\rm Re}\int_{\mathbb{R}^n}n\nabla\varphi\cdot\nabla\overline{\varphi}{\rm d}x-{\rm Re}\int_{\mathbb{R}^n}2\nabla \overline{\varphi}\cdot\nabla\left(x\cdot\nabla\varphi\right){\rm d}x\nonumber\\
=&-n\int_{\mathbb{R}^n}|\nabla\varphi|^{2}{\rm d}x-2{\rm Re}\int_{\mathbb{R}^n}\sum_{j=1}^{n}\frac{\partial\overline{\varphi}}{\partial x_{j}}\frac{\partial}{\partial x_{j}}\left(\sum_{l=1}^{n}x_{l}\frac{\partial\varphi}{\partial x_{l}}\right){\rm d}x\nonumber\\
=&-n\int_{\mathbb{R}^n}|\nabla\varphi|^{2}{\rm d}x-2{\rm Re}\int_{\mathbb{R}^n}\sum_{j=1}^{n}\frac{\partial\overline{\varphi}}{\partial x_{j}}\left(\sum_{l=1}^{n}\frac{\partial}{\partial x_{j}}\left(x_{l}\frac{\partial\varphi}{\partial x_{l}}\right)\right){\rm d}x\nonumber\\
=&-n\int_{\mathbb{R}^n}|\nabla\varphi|^{2}{\rm d}x-2{\rm Re}\int_{\mathbb{R}^n}\sum_{j=1}^{n}\frac{\partial\overline{\varphi}}{\partial x_{j}}\left(\sum_{l=1}^{n}\frac{\partial x_{l}}{\partial x_{j}}\frac{\partial\varphi}{\partial x_{l}}+\sum_{l=1}^{n}x_{l}\frac{\partial^{2}\varphi}{\partial x_{l}\partial x_{j}}\right){\rm d}x\nonumber\\
=&-n\int_{\mathbb{R}^n}|\nabla\varphi|^{2}{\rm d}x-2{\rm Re}\int_{\mathbb{R}^n}\sum_{j=1}^{n}\frac{\partial\overline{\varphi}}{\partial x_{j}}\left(\frac{\partial\varphi}{\partial x_{j}}+\sum_{l=1}^{n}x_{l}\frac{\partial^{2}\varphi}{\partial x_{l}\partial x_{j}}\right){\rm d}x\nonumber\\
=&-n\int_{\mathbb{R}^n}|\nabla\varphi|^{2}{\rm d}x-2{\rm Re}\int_{\mathbb{R}^n}\nabla\varphi\nabla\overline{\varphi}{\rm d}x-2{\rm Re}\int_{\mathbb{R}^n}\sum_{j=1}^{n}\sum_{l=1}^{n}x_{l}\frac{\partial^{2}\varphi}{\partial x_{l}\partial x_{j}}\frac{\partial\overline{\varphi}}{\partial x_{j}}{\rm d}x\nonumber\\
=&-(n+2)\int_{\mathbb{R}^n}|\nabla\varphi|^{2}{\rm d}x-\int_{\mathbb{R}^n}\sum_{j=1}^{n}\sum_{l=1}^{n}x_{l}\left(\frac{\partial^{2}\varphi}{\partial x_{l}\partial x_{j}}\frac{\partial\overline{\varphi}}{\partial x_{j}}+\overline{\frac{\partial^{2}\varphi}{\partial x_{l}\partial x_{j}}\frac{\partial\overline{\varphi}}{\partial x_{j}}}\right){\rm d}x\nonumber\\
=&-(n+2)\int_{\mathbb{R}^n}|\nabla\varphi|^{2}{\rm d}x-\int_{\mathbb{R}^n}\sum_{j=1}^{n}\sum_{l=1}^{n}x_{l}\left(\frac{\partial^{2}\varphi}{\partial x_{l}\partial x_{j}}\frac{\partial\overline{\varphi}}{\partial x_{j}}+\frac{\partial^{2}\overline{\varphi}}{\partial x_{l}\partial x_{j}}\frac{\partial\varphi}{\partial x_{j}}\right){\rm d}x\nonumber\\
=&-(n+2)\int_{\mathbb{R}^n}|\nabla\varphi|^{2}{\rm d}x-\int_{\mathbb{R}^n}\sum_{j=1}^{n}\sum_{l=1}^{n}x_{l}\frac{\partial}{\partial x_{l}}\left(\frac{\partial\varphi}{\partial x_{j}}\frac{\partial\overline{\varphi}}{\partial x_{j}}\right){\rm d}x\nonumber\\
=&-(n+2)\int_{\mathbb{R}^n}|\nabla\varphi|^{2}{\rm d}x-\int_{\mathbb{R}^n}\sum_{l=1}^{n}x_{l}\frac{\partial}{\partial x_{l}}\left(\sum_{j=1}^{n}\left(\frac{\partial\varphi}{\partial x_{j}}\frac{\partial\overline{\varphi}}{\partial x_{j}}\right)\right){\rm d}x\nonumber\\
=&-(n+2)\int_{\mathbb{R}^n}|\nabla\varphi|^{2}{\rm d}x-\sum_{l=1}^{n}\int_{\mathbb{R}^n}x_{l}\frac{\partial}{\partial x_{l}}\left(|\nabla\varphi|^{2}\right){\rm d}x\nonumber\\
=&-(n+2)\int_{\mathbb{R}^n}|\nabla\varphi|^{2}{\rm d}x+\sum_{l=1}^{n}\int_{\mathbb{R}^n}\frac{\partial x_{l}}{\partial x_{l}}|\nabla\varphi|^{2}{\rm d}x\nonumber\\
=&-(n+2)\int_{\mathbb{R}^n}|\nabla\varphi|^{2}{\rm d}x+\int_{\mathbb{R}^n}n|\nabla\varphi|^{2}{\rm d}x\nonumber\\
=&-2\|\nabla\varphi\|^2_{2}.
\end{align}
For $J_{2}$, we know
\begin{align}\label{3.9}
J_{2}=&{\rm Re}\int_{\mathbb{R}^n}n|x|^{2}|\varphi|^{2}{\rm d}x+2{\rm Re}\int_{\mathbb{R}^n}|x|^{2}\overline{\varphi}\left(x\cdot\nabla\varphi\right){\rm d}x\nonumber\\
=&\int_{\mathbb{R}^n}n|x|^{2}|\varphi|^{2}{\rm d}x+\int_{\mathbb{R}^n}\left(|x|^{2}\overline{\varphi}\left(x\cdot\nabla\varphi\right)+\overline{|x|^{2}\overline{\varphi}\left(x\cdot\nabla\varphi\right)}\right){\rm d}x\nonumber\\
=&\int_{\mathbb{R}^n}n|x|^{2}|\varphi|^{2}{\rm d}x+\int_{\mathbb{R}^n}|x|^{2}x\cdot\left(\overline{\varphi}\nabla\varphi+\varphi\nabla\overline{\varphi}\right){\rm d}x\nonumber\\
=&\int_{\mathbb{R}^n}n|x|^{2}|\varphi|^{2}{\rm d}x+\int_{\mathbb{R}^n}|x|^{2}x\cdot\nabla\left(\varphi\overline{\varphi}\right){\rm d}x\nonumber\\
=&\int_{\mathbb{R}^n}n|x|^{2}|\varphi|^{2}{\rm d}x+\int_{\mathbb{R}^n}|x|^{2}\left(\sum_{j=1}^{n}x_{j}\frac{\partial}{\partial x_{j}}\left(|\varphi|^{2}\right)\right){\rm d}x\nonumber\\
=&\int_{\mathbb{R}^n}n|x|^{2}|\varphi|^{2}{\rm d}x-\sum_{j=1}^{n}\int_{\mathbb{R}^n}\frac{\partial}{\partial x_{j}}\left(|x|^{2}x_{j}\right)|\varphi|^{2}{\rm d}x\nonumber\\
=&\int_{\mathbb{R}^n}n|x|^{2}|\varphi|^{2}{\rm d}x-\int_{\mathbb{R}^n}|\varphi|^{2}\left(\sum_{j=1}^{n}\left(\frac{\partial|x|^{2}}{\partial x_{j}}x_{j}+|x|^{2}\right)\right){\rm d}x\nonumber\\
=&\int_{\mathbb{R}^n}n|x|^{2}|\varphi|^{2}{\rm d}x-\int_{\mathbb{R}^n}\left(2|\varphi|^{2}\left(\sum_{j=1}^{n}x_{j}^{2}\right)+n|x|^{2}|\varphi|^{2}\right){\rm d}x\nonumber\\
=&-2\||x|\varphi\|^{2}_{2}.
\end{align}
For $J_{3}$, we have
\begin{align}\label{3.10}
J_{3}=&{\rm Re}\int_{\mathbb{R}^n}n\left(\lambda_{1}|\varphi|^{p}+\lambda_{2}|\varphi|^{q}\right){\rm d}x
+2{\rm Re}\int_{\mathbb{R}^n}\lambda_{1}|\varphi|^{p-2}\overline{\varphi}\left(x\cdot\nabla\varphi\right){\rm d}x\nonumber\\
&+2{\rm Re}\int_{\mathbb{R}^n}\lambda_{2}|\varphi|^{q-2}\overline{\varphi}\left(x\cdot\nabla\varphi\right){\rm d}x\nonumber\\
=&\int_{\mathbb{R}^n}n\left(\lambda_{1}|\varphi|^{p}+\lambda_{2}|\varphi|^{q}\right){\rm d}x+\int_{\mathbb{R}^n}\lambda_{1}|\varphi|^{p-2}\left(x\cdot\left(\overline{\varphi}\nabla\varphi+\overline{\overline{\varphi}\nabla\varphi}\right)\right){\rm d}x\nonumber\\
&+\int_{\mathbb{R}^n}\lambda_{2}|\varphi|^{q-2}\left(x\cdot\left(\overline{\varphi}\nabla\varphi+\overline{\overline{\varphi}\nabla\varphi}\right)\right){\rm d}x\nonumber\\
=&\int_{\mathbb{R}^n}n\left(\lambda_{1}|\varphi|^{p}+\lambda_{2}|\varphi|^{q}\right){\rm d}x+\int_{\mathbb{R}^n}\lambda_{1}|\varphi|^{p-2}\left(x\cdot\left(\overline{\varphi}\nabla\varphi+\varphi\nabla\overline{\varphi}\right)\right){\rm d}x\nonumber\\
&+\int_{\mathbb{R}^n}\lambda_{2}|\varphi|^{q-2}\left(x\cdot\left(\overline{\varphi}\nabla\varphi+\varphi\nabla\overline{\varphi}\right)\right){\rm d}x\nonumber\\
=&\int_{\mathbb{R}^n}n\left(\lambda_{1}|\varphi|^{p}+\lambda_{2}|\varphi|^{q}\right){\rm d}x+\int_{\mathbb{R}^n}\lambda_{1}|\varphi|^{p-2}\left(x\cdot\nabla\left(\varphi\overline{\varphi}\right)\right){\rm d}x\nonumber\\
&+\int_{\mathbb{R}^n}\lambda_{2}|\varphi|^{q-2}\left(x\cdot\nabla\left(\varphi\overline{\varphi}\right)\right){\rm d}x\nonumber\\
=&\int_{\mathbb{R}^n}n\left(\lambda_{1}|\varphi|^{p}+\lambda_{2}|\varphi|^{q}\right){\rm d}x\nonumber\\
&+\int_{\mathbb{R}^n}x\cdot\left(\lambda_{1}\left(|\varphi|^{2}\right)^{\frac{p-2}{2}}+\lambda_{2}\left(|\varphi|^{2}\right)^{\frac{q-2}{2}}\right)\nabla\left(|\varphi|^{2}\right){\rm d}x\nonumber\\
=&\int_{\mathbb{R}^n}n\left(\lambda_{1}|\varphi|^{p}+\lambda_{2}|\varphi|^{q}\right){\rm d}x+\int_{\mathbb{R}^n}x\cdot\nabla\left(\frac{2\lambda_{1}}{p}\left(|\varphi|^{2}\right)^{\frac{p}{2}}+\frac{2\lambda_{2}}{q}\left(|\varphi|^{2}\right)^{\frac{q}{2}}\right){\rm d}x\nonumber\\
=&\int_{\mathbb{R}^n}n\left(\lambda_{1}|\varphi|^{p}+\lambda_{2}|\varphi|^{q}\right){\rm d}x+\int_{\mathbb{R}^n}\sum_{j=1}^{n}x_{j}\frac{\partial}{\partial x_{j}}\left(\frac{2\lambda_{1}}{p}|\varphi|^{p}+\frac{2\lambda_{2}}{q}|\varphi|^{q}\right){\rm d}x\nonumber\\
=&\int_{\mathbb{R}^n}n\left(\lambda_{1}|\varphi|^{p}+\lambda_{2}|\varphi|^{q}\right){\rm d}x+\sum_{j=1}^{n}\int_{\mathbb{R}^n}x_{j}\frac{\partial}{\partial x_{j}}\left(\frac{2\lambda_{1}}{p}|\varphi|^{p}+\frac{2\lambda_{2}}{q}|\varphi|^{q}\right){\rm d}x\nonumber\\
=&\int_{\mathbb{R}^n}n\left(\lambda_{1}|\varphi|^{p}+\lambda_{2}|\varphi|^{q}\right){\rm d}x-\sum_{j=1}^{n}\int_{\mathbb{R}^n}\frac{\partial x_{j}}{\partial x_{j}}\left(\frac{2\lambda_{1}}{p}|\varphi|^{p}+\frac{2\lambda_{2}}{q}|\varphi|^{q}\right){\rm d}x\nonumber\\
=&\int_{\mathbb{R}^n}n\left(\lambda_{1}|\varphi|^{p}+\lambda_{2}|\varphi|^{q}\right){\rm d}x-\frac{2n\lambda_{1}}{p}\int_{\mathbb{R}^n}|\varphi|^{p}{\rm d}x-\frac{2n\lambda_{2}}{q}\int_{\mathbb{R}^n}|\varphi|^{q}{\rm d}x\nonumber\\
=&\frac{n(p-2)\lambda_{1}}{p}\|\varphi\|^{p}_{p}+\frac{n(q-2)\lambda_{2}}{q}\|\varphi\|^{q}_{q}.
\end{align}
By substituting (\ref{3.8}), (\ref{3.9}), and (\ref{3.10}) into (\ref{3.7}), we obtain (\ref{3.1}).
\end{proof}
Next, as a direct application of above calculation of $J''(t)$, we prove that negative initial energy $K(\varphi_{0})<0$ leads to the finite time blowup of the solution. In this approach, we also need to apply the following inequality introduced in \cite{r4-1-2-00-2}.
\begin{lemma}\label{lem3.2}
(\cite{r4-1-2-00-2} )
Let $\varphi\in H$. Then one has
\begin{align}
\|\varphi\|^{2}_{2}
\leq\frac{2}{n}\|\nabla\varphi\|_{2}\||x|\varphi\|_{2}.\nonumber
\end{align}
\end{lemma}

\begin{theorem}\label{th3.1}
(Finite time blowup for negative initial energy $K(\varphi_{0})<0$)
Let $p$ and $q$ satisfy (\ref{1.2}), $\lambda_{1}>0$ and $\lambda_{2}>0$. For the initial data $\varphi_{0}\in H$ satisfying $K(\varphi_{0})<0$, the corresponding solution to the Cauchy problem (\ref{1.1}) blows up in finite time.
\begin{proof}
According to Lemma \ref{lem2.1}, we only need to prove $T_{{\rm max}}<+\infty$. Arguing by contradiction, we suppose $T_{{\rm max}}=+\infty$. According to the assumption (\ref{1.2}), we have $\frac{n(p-2)}{4}>1$ and $\frac{n(q-2)}{4}>1$ for any $n\in\mathbb{N}^{+}$, which means $$\frac{n(p-2)\lambda_{1}}{4p}>\frac{\lambda_{1}}{p}$$ and $$\frac{n(q-2)\lambda_{2}}{4q}>\frac{\lambda_{2}}{q}.$$ Then, based on these two inequalities and Lemma \ref{lem3.1}, we know
\begin{align}\label{3.11}
\begin{split}
J''(t)=&16\left(\frac{1}{2}\|\nabla\varphi\|^{2}_{2}-\frac{1}{2}\||x|\varphi\|^{2}_{2}-\frac{n(p-2)\lambda_{1}}{4p}\|\varphi\|^{p}_{p}-\frac{n(q-2)\lambda_{2}}{4q}\|\varphi\|^{q}_{q}\right)\\
\leq&16\left(\frac{1}{2}\|\nabla\varphi\|^{2}_{2}-\frac{1}{2}\||x|\varphi\|^{2}_{2}-\frac{\lambda_{1}}{p}\|\varphi\|^{p}_{p}-\frac{\lambda_{2}}{q}\|\varphi\|^{q}_{q}\right)\\
\leq&16\left(\frac{1}{2}\|\nabla\varphi\|^{2}_{2}+\frac{1}{2}\|\varphi\|^{2}_{2}+\frac{1}{2}\||x|\varphi\|^{2}_{2}-\frac{\lambda_{1}}{p}\|\varphi\|^{p}_{p}
-\frac{\lambda_{2}}{q}\|\varphi\|^{q}_{q}\right)\\
=&16K(\varphi).
\end{split}
\end{align}
According to $K(\varphi)=K(\varphi_{0})$ given by (\ref{2.2}) and (\ref{2.3}), and the condition $K(\varphi_{0})<0$, we notice that (\ref{3.11}) turns to
\begin{align}\label{3.12}
J''(t)\leq16K(\varphi_{0})<0,
\end{align}
which implies
\begin{align*}
&\int^{t}_{0}J''(\tau){\rm d}\tau\leq\int^{t}_{0}16K(\varphi_{0}){\rm d}\tau,\nonumber
\end{align*}
i.e.,
\begin{align}\label{3.13}
&J'(t)-J'(0)\leq16K(\varphi_{0})t,
\end{align}
for $t\in \left(0,T_{{\rm max}}\right)$. Subsequently, integrating (\ref{3.13}) from 0 to $t$, we have
\begin{align*}
&\int^{t}_{0}J'(\tau){\rm d}\tau\leq\int^{t}_{0}\left(J'(0)+16K(\varphi_{0})\tau\right){\rm d}\tau,\nonumber
\end{align*}
which means
\begin{equation}\label{3.14}
J(t)\leq J(0)+J'(0)t+8K(\varphi_{0})t^{2}
\end{equation}
 for $t\in \left(0, T_{{\rm max}}\right)$. In (\ref{3.14}), due to $K(\varphi_{0})<0$ and $T_{{\rm max}}=+\infty$ assumed, we can find large enough $\overline{t}\in \left(0, T_{{\rm max}}\right)$ to make
\begin{equation}\label{3.15}
J(0)+J'(0)t+8K(\varphi_{0})t^{2}<0
\end{equation}
hold for $t\in\left[\overline{t},+\infty\right)$. And for $J(t)$, we notice $J(0)=\||x|\varphi_{0}\|^{2}_{2}>0$ due to $K(\varphi_{0})<0$, i.e., $\varphi_{0}\neq 0$. Thus, according to this fact, (\ref{3.14}), (\ref{3.15}), and the continuity of $J(t)$ shown by Lemma \ref{lem2.1}, we can find $\widetilde{t}\in (0,\overline{t})$ to make
\begin{align}\label{3.16}
\lim\limits_{t\rightarrow \widetilde{t}} J(t)=0
\end{align}
hold. According to Lemma \ref{lem3.2} and (\ref{2.2}), we see
\begin{align}\label{3.17}
\|\varphi_{0}\|^{2}_{2}\leq\frac{2}{n}\|\nabla\varphi\|_{2}\left(J(t)\right)^{\frac{1}{2}}.
\end{align}
We notice $\|\varphi_{0}\|^{2}_{2}>0$ in (\ref{3.17}) due to $K(\varphi_{0})<0$, i.e., $\varphi_{0}\neq 0$. Based on this fact, (\ref{3.16}), and (\ref{3.17}), we obtain
\begin{align*}
\lim_{t\rightarrow \widetilde{t}}\|\nabla\varphi\|_{2}=+\infty,
\end{align*}
which implies
\begin{align*}
\lim_{t\rightarrow \widetilde{t}}\|\varphi\|_{H}=+\infty.
\end{align*}
 This contradicts the assumption $T_{\rm{max}}=+\infty$. Thus, we have $T_{\rm{max}}<+\infty$, which implies the finite time blowup according to Lemma \ref{lem2.1}.
\end{proof}
\end{theorem}

\begin{remark}\label{rem3.1}($E(0)$ v.s. $K(\varphi_{0})$)
Experience and intuition tell us that the negative total initial energy, i.e., $E(0)<0$, leads to a finite time blowup of the solution, which has been verified by Theorem \ref{th3.1}. Hence, on the surface, the condition $E(0)<0$ is weaker than $K(\varphi_{0})<0$, to ensure the finite time blowup. However, $E(0)<0$ is the sufficient condition for the finite time blowup, especially no sharp conditions or threshold results like Theorems \ref{th4.1}, \ref{th4.2}, \ref{th5.1}, \ref{th5.2} can be obtained under the condition $E(0)<0$, and so does $E(0)\geq0$. This explains why we use the functional $K(\varphi_{0})$ instead of $E(0)$.
\end{remark}
 \section{Global existence and finite time blowup for non-negative low initial energy $0\leq K(\varphi_{0})<d_{\mathcal{N}_{j}}$, $j=1,2,\cdots,8$}\label{sec4}
In this section, we focus on non-negative low initial energy $0\leq K(\varphi_{0})<d_{\mathcal{N}_{j}}$, $j=1,2,\cdots,8$, to classify the initial data leading to global existence or finite time blowup. More precisely, we give eight classifications of the initial data to distinguish the global solution and finite time blowup solution under the framework of the cross-constrained variational theory. Here, we require the conclusions $d_{\mathcal{N}_{j}}>0$, $j=1,2,\cdots,8$, given by Lemma \ref{lem2.3} to ensure that the conditions $0<K(\varphi_{0})<d_{\mathcal{N}_{j}}$ make sense. Thus, we naturally require the conditions in Lemma \ref{lem2.3} for the following Theorem $\ref{th4.1}$ and Theorem \ref{th4.2}.
\subsection{Global existence $\&$ finite time blowup}\label{sec4.1}

We prove the global existence and finite time blowup of the solution to the Cauchy problem (\ref{1.1}) for the non-negative low initial energy $0\leq K(\varphi_{0})<d_{\mathcal{N}_{j}}$, $j=1,2,\cdots,8$.
\begin{theorem}\label{th4.1}(Global existence for the non-negative low initial energy $0\leq K(\varphi_{0})<d_{\mathcal{N}_{j}}$, $j=1,2,\cdots,8$)
Let $p$ and $q$ satisfy (\ref{1.2}), and $\lambda_{1}>0$, $\lambda_{2}>0$ satisfy the following conditions

(i)\ for $j=1,2,3,4$, $\lambda_{1}$ and $\lambda_{2}$ additionally satisfy (\ref{2.9}) and (\ref{2.10});

(ii)\ for $j=5,6,7,8$, no additional restrictions on $\lambda_{1}$ and $\lambda_{2}$.\\
Then the corresponding solution $\varphi$ to problem (\ref{1.1}) exists globally in time provided $\varphi_{0}\in \mathcal{W}_{j}$, $j=1,2,\cdots,8$.
\begin{proof}
  For the initial data $\varphi_{0}$ belonging to $\mathcal{W}_{j}$, $j=1,2,\cdots,8$, we shall prove $T_{\rm{max}}=+\infty$. Here, a comparably simple case is $\varphi_{0}=0$, which implies $\varphi(t)=0$ for $t\in(0,+\infty)$ according to (\ref{2.2}), and thus it is the global solution. Hence, we shall only focus on the case of $K(\varphi_{0})<d_{\mathcal{N}_{j}}$ and $S_{j}(\varphi_{0})>0$, i.e., $\varphi_{0}\neq0$, to prove the global existence, which requires us to obtain the boundedness estimate of the norm $\|\varphi\|_{H}$. The strategies of estimating the norm $\|\varphi\|_{H}$ depend on the structures of the cross-constrained variational functionals $S_{j}(\varphi)$, $j=1,2,\cdots,8$. Although we need to consider eight cases, it is not necessary for us to propose eight different methods to estimate $\|\varphi\|_{H}$. According to the structures of the cross-constrained variational functionals $S_{j}(\varphi)$, $j=1,2,\cdots,8$, we shall finish the proof by considering two cases: Case \textrm{I} ($j=1,2,3,4$) and Case \textrm{II} ($j=5,6,7,8$), as the strategies for $S_{j}(\varphi)$, $j=1,2,3,4$, with the term $\|\varphi\|_{p}^{p}$ are similar, and the strategies for $S_{j}(\varphi)$, $j=5,6,7,8$, with the both terms $\|\varphi\|_{p}^{p}$ and $\|\varphi\|_{q}^{q}$ are similar.\\
{\bf{Case \textrm{I}: \bm{$j=1,2,3,4.$}}}

For the initial data $\varphi_{0}$ satisfying $K(\varphi_{0})<d_{\mathcal{N}_{j}}$ and $S_{j}(\varphi_{0})>0$, i.e., $\varphi_{0}\in\mathcal{W}_{j}$, we know that the corresponding solution $\varphi(t)\in\mathcal{W}_{j}$ for all $t\in [0,T_{\rm{max}})$ according to Lemma \ref{lem2.4}. Here the case $\varphi_{0}\neq0$ leads to $\varphi(t)\neq0$ according to (\ref{2.2}). We know $K(\varphi(t))<d_{\mathcal{N}_{j}}$ and $S_{j}(\varphi(t))>0$, where $t\in [0,T_{\rm{max}})$. According to $K(\varphi)<d_{\mathcal{N}_{j}}$ and the definition of $K(\varphi)$, i.e., (\ref{2.4}), we know
\begin{align}\label{4.1}
  K(\varphi)=\frac{1}{2}\|\varphi\|^{2}_{H}-\frac{\lambda_{1}}{p}\|\varphi\|^{p}_{p}-\frac{\lambda_{2}}{q}\|\varphi\|^{q}_{q}<d_{\mathcal{N}_{j}}.
\end{align}
   To estimate $\|\varphi\|^{2}_{H}$, we need to deal with the negative terms in (\ref{4.1}), i.e., $-\frac{\lambda_{1}}{p}\|\varphi\|^{p}_{p}$ and $-\frac{\lambda_{2}}{q}\|\varphi\|^{q}_{q}$.
Noticing that the cross-constrained variational functionals $S_{j}(\varphi)$ have the term $\|\varphi\|^{p}_{p}$ and do not have the term $\|\varphi\|^{q}_{q}$, by applying $S_{j}(\varphi)>0$ we get
\begin{align}\label{4.2}
\|\varphi\|^{p}_{p}<\frac{2p}{n(p-2)\lambda_{1}}\left(\|\nabla \varphi\|^{2}_{2}+c\|\varphi\|^{2}_{2}+\||x|\varphi\|^{2}_{2}\right),
\end{align}where $c>0$ is chosen such that (\ref{cdetiaojian}) holds.
Meanwhile, since the assumptions on $p$, $q$, $\lambda_{1}$ and $\lambda_{2}$ we adopted here are same as Lemma \ref{lem2.3}, in order to deal with $-\frac{\lambda_{2}}{q}\|\varphi\|^{q}_{q}$, we retrieve (\ref{2.16}), which can be rewritten as
\begin{align}\label{4.3}
K(\varphi)\geq\frac{1}{2}\| \varphi\|^{2}_{H}-\left(\frac{\lambda_{1}}{p}+\frac{\lambda_{2}}{q}\right)\|\varphi\|^{p}_{p}-\frac{\lambda_{2}}{q}\|\varphi\|^{2}_{2}
\end{align}
in terms of $\|\varphi\|^{2}_{H}$.
Next, we shall deal with the negative terms in (\ref{4.3}), i.e., $-\frac{\lambda_{2}}{q}\|\varphi\|^{2}_{2}$ and $-\left(\frac{\lambda_{1}}{p}+\frac{\lambda_{2}}{q}\right)\|\varphi\|^{p}_{p}$. Noticing that $\|\varphi\|^{2}_{2}$ is the conserved quantity according to (\ref{2.2}), to deal with the negative term $-\frac{\lambda_{2}}{q}\|\varphi\|^{2}_{2}$, we plus $\left(\frac{\lambda_{2}}{q}+\left(\frac{\lambda_{1}}{p}+\frac{\lambda_{2}}{q}\right)\frac{2p c}{n(p-2)\lambda_{1}}\right)\|\varphi\|^{2}_{2}$ on both sides of (\ref{4.3}) to get
 \begin{align}\label{4.4}
 \begin{split}
&K(\varphi)+\left(\frac{\lambda_{2}}{q}+\left(\frac{\lambda_{1}}{p}+\frac{\lambda_{2}}{q}\right)\frac{2p c}{n(p-2)\lambda_{1}}\right)\|\varphi\|^{2}_{2}\\
\geq&\frac{1}{2}\| \varphi\|^{2}_{H}-\left(\frac{\lambda_{1}}{p}+\frac{\lambda_{2}}{q}\right)\|\varphi\|^{p}_{p}+\left(\frac{\lambda_{1}}{p}+\frac{\lambda_{2}}{q}\right)\frac{2p c}{n(p-2)\lambda_{1}}\|\varphi\|^{2}_{2}.
\end{split}
\end{align}
According to (\ref{2.2}), (\ref{2.3}) and $K(\varphi_{0})<d_{\mathcal{N}_{j}}$, (\ref{4.4}) turns to
\begin{align}\label{4.5}
\begin{split}
&d_{\mathcal{N}_{j}}+\left(\frac{\lambda_{2}}{q}+\left(\frac{\lambda_{1}}{p}+\frac{\lambda_{2}}{q}\right)\frac{2p c}{n(p-2)\lambda_{1}}\right)\|\varphi_{0}\|^{2}_{2}\\
>&K(\varphi)+\left(\frac{\lambda_{2}}{q}+\left(\frac{\lambda_{1}}{p}+\frac{\lambda_{2}}{q}\right)\frac{2p c}{n(p-2)\lambda_{1}}\right)\|\varphi\|^{2}_{2}\\
\geq&\frac{1}{2}\| \varphi\|^{2}_{H}-\left(\frac{\lambda_{1}}{p}+\frac{\lambda_{2}}{q}\right)\|\varphi\|^{p}_{p}+\left(\frac{\lambda_{1}}{p}+\frac{\lambda_{2}}{q}\right)\frac{2p c}{n(p-2)\lambda_{1}}\|\varphi\|^{2}_{2}.
\end{split}
\end{align}
By substituting (\ref{4.2}) into (\ref{4.5}), we have
\begin{align}\label{4.6}
d_{\mathcal{N}_{j}}+\left(\frac{\lambda_{2}}{q}+\left(\frac{\lambda_{1}}{p}+\frac{\lambda_{2}}{q}\right)\frac{2p c}{n(p-2)\lambda_{1}}\right)\|\varphi_{0}\|^{2}_{2}>\frac{1}{2}\| \varphi\|^{2}_{H}-\left(\frac{\lambda_{1}}{p}+\frac{\lambda_{2}}{q}\right)\frac{2p}{n(p-2)\lambda_{1}}\| \varphi\|^{2}_{H}.
\end{align}
From the assumption (\ref{2.9}) for $j=1,2,3,4$, we see
$$\frac{1}{2}-\left(\frac{\lambda_{1}}{p}+\frac{\lambda_{2}}{q}\right)\frac{2p}{n(p-2)\lambda_{1}}>0.$$
Thus, from (\ref{4.6}) we have
\begin{align}\label{4.7}
\|\varphi\|^{2}_{H}
\leq\frac{2qn(p-2)\lambda_{1}}{qn(p-2)\lambda_{1}-4\lambda_{1}q-4\lambda_{2}p} \left(d_{\mathcal{N}_{j}}+\left(\frac{\lambda_{2}}{q}+\left(\frac{\lambda_{1}}{p}+\frac{\lambda_{2}}{q}\right)\frac{2p c}{n(p-2)\lambda_{1}}\right)\|\varphi_{0}\|^{2}_{2}\right)
\end{align}
for all $t\in(0,T_{\rm{max}})$ and $j=1,2,3,4$. Therefore, according to Lemma \ref{lem2.1}, we know $T_{\rm{max}}=+\infty$. Here, we need the conclusions $d_{\mathcal{N}_{j}}>0$, $j=1,2,3,4$, to ensure that the upper bound of $\|\varphi\|_{H}$ in the estimate (\ref{4.7}) is positive, which also explains the additional restrictions (\ref{2.9}) and (\ref{2.10}) for $j=1,2,3,4$.\\
{\bf{Case \textrm{II}: \bm{$j=5,6,7,8.$}}}

Similar to Case \textrm{I}, we also know that the corresponding solution $\varphi(t)\in\mathcal{W}_{j}$ for all $t\in (0,T_{\rm{max}})$ according to Lemma \ref{2.4}, and we have $K(\varphi(t))<d_{\mathcal{N}_{j}}$ and $S_{j}(\varphi(t))>0$ for all $t\in [0,T_{\rm{max}})$.
 For this case, since the cross-constrained variational functionals $S_{j}(\varphi)$ have both the terms $\|\varphi\|^{p}_{p}$ and $\|\varphi\|^{q}_{q}$, we can apply $S_{j}(\varphi)>0$ to control both the terms $\|\varphi\|^{p}_{p}$ and $\|\varphi\|^{q}_{q}$ in (\ref{4.1})
as
\begin{align}\label{4.8}
\frac{n(p-2)\lambda_{1}}{2p}\|\varphi\|^{p}_{p}+\frac{n(q-2)\lambda_{2}}{2q}\|\varphi\|^{q}_{q}<\|\varphi\|^{2}_{H}.
\end{align}
Due to $p>q$ assumed in this paper, we know
\begin{align}\label{4.9}
\begin{split}
&\frac{n(p-2)\lambda_{1}}{2p}\|\varphi\|^{p}_{p}+\frac{n(q-2)\lambda_{2}}{2q}\|\varphi\|^{q}_{q}\\
>&\frac{n(q-2)\lambda_{1}}{2p}\|\varphi\|^{p}_{p}+\frac{n(q-2)\lambda_{2}}{2q}\|\varphi\|^{q}_{q}\\
=&\frac{n(q-2)}{2}\left(\frac{\lambda_{1}}{p}\|\varphi\|^{p}_{p}+\frac{\lambda_{2}}{q}\|\varphi\|^{q}_{q}\right).
\end{split}
\end{align}
Combining (\ref{4.8}) and (\ref{4.9}), we know
\begin{align}\label{4.10}
\frac{\lambda_{1}}{p}\|\varphi\|^{p}_{p}+\frac{\lambda_{2}}{q}\|\varphi\|^{q}_{q}<\frac{2}{n(q-2)}\|\varphi\|^{2}_{H}.
\end{align}
By substituting (\ref{4.10}) and $K(\varphi)<d_{\mathcal{N}_{j}}$ into $K(\varphi)$, we can obtain
\begin{align}\label{4.11}
d_{\mathcal{N}_{j}}>\left(\frac{1}{2}-\frac{2}{n(q-2)}\right)\|\varphi\|^{2}_{H},
\end{align}
which gives the boundness estimate of the solution, i.e.,
\begin{align}\label{4.12}
\|\varphi\|^{2}_{H}\leq\frac{2n(q-2)d_{\mathcal{N}_{j}}}{n(q-2)-4}.
\end{align}
In (\ref{4.11}), we notice that $\frac{1}{2}-\frac{2}{n(q-2)}>0$ can be ensured by (\ref{1.2}), i.e., $2+\frac{4}{n}<q$, and the fact $d_{\mathcal{N}_{j}}>0$, $j=5,6,7,8$, can be ensured by Lemma \ref{lem2.3}.
According to (\ref{4.12}) and Lemma \ref{lem2.1}, we know $T_{\rm{max}}=+\infty$.

\end{proof}
\end{theorem}

According to Theorem \ref{th4.1}, we also have the following different conditions for judging the global existence for non-negative low initial energy $0\leq K(\varphi_{0})<d_{\mathcal{N}_{j}}$, $j=1,2,\cdots,8$, in Corollary \ref{cor4.1}.
\begin{corollary}\label{cor4.1}(Other conditions of the global existence for non-negative low initial energy $0\leq K(\varphi_{0})<d_{\mathcal{N}_{j}}$, $j=1,2,\cdots,8$)
Let $p$ and $q$ satisfy (\ref{1.2}), and $\lambda_{1}>0$, $\lambda_{2}>0$ satisfy the following conditions

(i)\ for $j=1,2,3,4$, $\lambda_{1}$ and $\lambda_{2}$ additionally satisfy (\ref{2.9}) and (\ref{2.10});

(ii)\ for $j=5,6,7,8$, no additional restrictions on $\lambda_{1}$ and $\lambda_{2}$.\\
If the initial data satisfy $\|\varphi_{0}\|^{2}_{H}<2d_{\mathcal{N}_{j}}$, $j=1,2,\cdots,8$, and $K(\varphi_{0})\geq0$, the solution of (\ref{1.1}) exists globally in time.
\begin{proof}
We only write down the details of the proof for the case $j=1$ since the processes of dealing with the cases $j=2,\cdots,8$ are similar. A comparably simple case is $\varphi_{0}=0$, which means $\varphi=0$ according to (\ref{2.2}). Next, we assume $\varphi_{0}\neq0$. It is clear that $K(\varphi_{0})<d_{\mathcal{N}_{1}}$ according to $\|\varphi_{0}\|^{2}_{H}<2d_{\mathcal{N}_{1}}$. Next, we prove $S_{1}(\varphi_{0})>0$. Arguing by contradiction, we suppose $S_{1}(\varphi_{0})\leq0$. Then we can choose $0<\nu\leq 1$ to make
\begin{align*}
S_{1}(\nu\varphi_{0})=\nu^{2}\|\nabla \varphi_{0}\|^{2}_{2}-\frac{n(p-2)\lambda_{1}\nu^{p}}{2p}\|\varphi_{0}\|^{p}_{p}=0.
\end{align*}
 Due to $\varphi_{0}\neq0$, i.e., $\nu\varphi_{0}\neq0$, we see
$K(\nu\varphi_{0})\geq d_{\mathcal{N}_{1}}$.
 However, we have
\begin{align*}
&K(\nu\varphi_{0})\leq\frac{1}{2}\|\nu \varphi_{0}\|^{2}_{H}=\frac{\nu^{2}}{2}\|\varphi_{0}\|^{2}_{H}<\nu^{2}d_{\mathcal{N}_{1}}\leq d_{\mathcal{N}_{1}},
\end{align*}
which contradicts $K(\nu\varphi_{0})\geq d_{\mathcal{N}_{1}}$. Thus we can see $\varphi_{0}\in \mathcal{W}_{1}$, which implies that the solution of (\ref{1.1}) exists globally according to Theorem \ref{th4.1}.
\end{proof}
\end{corollary}

Next, we prove the finite time blowup of the solution to the Cauchy problem (\ref{1.1}) by applying the manifolds $\mathcal{V}_{j}$, $j=1,2,\cdots,8$. Here, as we narrate, we investigate the finite time blowup when the conclusions $d_{\mathcal{N}_{j}}>0$, $j=1,2,\cdots,8$, hold. Here, for $j=1,2,3,4$, we also need to apply the additional restrictions (\ref{2.9}) and (\ref{2.10}) on $\lambda_{1}$ and $\lambda_{2}$ given in Lemma \ref{lem2.3} to ensure $d_{\mathcal{N}_{j}}>0$, which are not required for the cases $j=5,6,7,8$.
\begin{theorem}\label{th4.2}(Finite time blowup for the non-negative low initial energy $0\leq K(\varphi_{0})<d_{\mathcal{N}_{j}}$, $j=1,2,\cdots,8$)
Let $p$ and $q$ satisfy (\ref{1.2}), and $\lambda_{1}>0$, $\lambda_{2}>0$ satisfy the following conditions

(i)\ for $j=1,2,3,4$, $\lambda_{1}$ and $\lambda_{2}$ additionally satisfy (\ref{2.9}) and (\ref{2.10});

(ii)\ for $j=5,6,7,8$, no additional restrictions on $\lambda_{1}$ and $\lambda_{2}$.\\
For the initial data $\varphi_{0}\in \mathcal{V}_{j}$, $j=1,2,\cdots,8$, the corresponding solution $\varphi$ of (\ref{1.1}) blows up in finite time.
\begin{proof}
We apply the auxiliary functional $J(t)=\||x|\varphi\|^{2}_{2}$ to prove the finite time blowup of the solution.
As we discussed in the ``Introduction'' section of the present paper, the auxiliary functional $$J(t)=\||x|\varphi\|^{2}_{2}$$ always plays an important role in proving the finite time blowup of the solution, and we also use this auxiliary functional in this proof. According to Lemma \ref{lem2.1}, we only need to prove $T_{\rm{max}}<+\infty$. Arguing by contradiction, we suppose $T_{\rm{max}}=+\infty$. For $j=1,2,\cdots,8$, we shall apply the condition $S_{j}(\varphi)<0$, i.e., $\varphi\in \mathcal{V}_{j}$ ensured by Lemma \ref{lem2.4} to verify $J''(t)<0$, which is the key step of this proof. Here, the strategies of this process depend on the structures of the cross-constrained variational functionals $S_{j}(\varphi)$, $j=1,2,\cdots,8$. We shall divide the proof into two cases, i.e., Case \textrm{I} ($j=3,4,7,8$, here the cross-constrained variational functionals $S_{j}(\varphi)$ have the term $\|\varphi\|^{2}_{2}$) and Case \textrm{II} ($j=1,2,5,6$, here $S_{j}(\varphi)$ do not have the term $\|\varphi\|^{2}_{2}$).\\
{\bf{Case \textrm{I}: \bm{$j=3,4,7,8.$}}}

For the cases $j=3,4,7,8$, the cross-constrained variational functionals $S_{j}(\varphi)$ having the term $\|\varphi\|^{2}_{2}$, by $S_{j}(\varphi)<0$ and $\lambda_{2}>0$ we know
\begin{equation}
\|\nabla\varphi\|^{2}_{2}+\min\left\{c,1\right\}\|\varphi\|^{2}_{2}<\frac{n(p-2)}{2p}\lambda_{1}\|\varphi\|^{p}_{p}+\frac{n(q-2)}{2q}\lambda_{2}\|\varphi\|^{q}_{q},\nonumber\\
\end{equation}
i.e.,
\begin{equation}\label{4.13}
\|\nabla\varphi\|^{2}_{2}+\min\left\{c,1\right\}\|\varphi\|^{2}_{2}-\frac{n(p-2)}{2p}\lambda_{1}\|\varphi\|^{p}_{p}-\frac{n(q-2)}{2q}\lambda_{2}\|\varphi\|^{q}_{q}<0,
\end{equation}where $c>0$ is chosen such that (\ref{cdetiaojian}) holds.
Next, according to (\ref{4.13}), we deal with $J''(t)$ shown by (\ref{3.1}) as follows
\begin{align}\label{4.14}
\begin{split}
J''(t)=&8\|\nabla\varphi\|^{2}_{2}-8\||x|\varphi\|^{2}_{2}-\frac{8n(p-2)}{2p}\lambda_{1}\|\varphi\|^{p}_{p}-\frac{8n(q-2)}{2q}\lambda_{2}\|\varphi\|^{q}_{q}\\
=&8\left(\|\nabla\varphi\|^{2}_{2}+\min\left\{c,1\right\}\|\varphi\|^{2}_{2}-\frac{n(p-2)}{2p}\lambda_{1}\|\varphi\|^{p}_{p}-\frac{n(q-2)}{2q}\lambda_{2}\|\varphi\|^{q}_{q}\right)
\\&-8\||x|\varphi\|^{2}_{2}-8\min\left\{c,1\right\}\|\varphi\|^{2}_{2}\\
<&-8\min\left\{c,1\right\}\|\varphi\|^{2}_{2},
\end{split}
\end{align}
where $t\in(0,+\infty)$. By substituting (\ref{2.2}), i.e., the mass conservation law, into (\ref{4.14}), we know
\begin{align}\label{4.15}
J''(t)<-8\min\left\{c,1\right\}\|\varphi\|^{2}_{2}=-8\min\left\{c,1\right\}\|\varphi_{0}\|^{2}_{2}<0
\end{align}
due to $S_{j}(\varphi_{0})<0$, i.e., $\varphi_{0}\neq0$. By the similar process of applying (\ref{3.12}) to complete the proof of Theorem \ref{th3.1}, we can also easily apply (\ref{4.15}) to prove $T_{\rm{max}}<+\infty$, which means the finite time blowup of the solution according to Lemma \ref{lem2.1}.\\
{\bf{Case \textrm{II}: \bm{$j=1,2,5,6.$}}}

For the cases $j=1,2,5,6$, the cross-constrained variational functionals $S_{j}(\varphi)$ not having the term $\|\varphi\|^{2}_{2}$, from $S_{j}(\varphi)<0$ and $\lambda_{2}>0$, we can deduce
\begin{equation}
\|\nabla\varphi\|^{2}_{2}<\frac{n(p-2)}{2p}\lambda_{1}\|\varphi\|^{p}_{p}+\frac{n(q-2)}{2q}\lambda_{2}\|\varphi\|^{q}_{q},\nonumber
\end{equation}
i.e.,
\begin{equation}\label{4.16}
\|\nabla\varphi\|^{2}_{2}-\frac{n(p-2)}{2p}\lambda_{1}\|\varphi\|^{p}_{p}-\frac{n(q-2)}{2q}\lambda_{2}\|\varphi\|^{q}_{q}<0.
\end{equation}
 Next, according to (\ref{4.16}), we deal with $J''(t)$ shown by (\ref{3.1}) as follows
\begin{align}\label{4.17}
\begin{split}
J''(t)=&8\left(\|\nabla\varphi\|^{2}_{2}-\frac{n(p-2)}{2p}\lambda_{1}\|\varphi\|^{p}_{p}-\frac{n(q-2)}{2q}\lambda_{2}\|\varphi\|^{q}_{q}\right)-8\||x|\varphi\|^{2}_{2}\\
<&-8\||x|\varphi\|^{2}_{2}\\
=&-8J(t)
\end{split}
\end{align}
for $t\in[0,+\infty)$. Different from Case \textrm{I} of this proof (in Case \textrm{I}, we have the conserved quantity $\|\varphi\|^{2}_{2}$ in (\ref{4.14})) we need a different strategy to deal with (\ref{4.17}). Set $$g(t):=\frac{J'(t)}{J(t)}.$$ By applying (\ref{4.17}), for $t\in[0,+\infty)$, we have
\begin{align}\label{4.18}
\begin{split}
g'(t)&=\frac{J''(t)J(t)-(J'(t))^2}{(J(t))^2}\\
&=\frac{J''(t)}{J(t)}-\left(\frac{J'(t)}{J(t)}\right)^2\\
&<-8-(g(t))^2<0,
\end{split}
\end{align}
which implies
\begin{align*}
\int^{t}_{0}g'(\tau){\rm d}\tau<\int^{t}_{0}\left(-8-\left(g(\tau)\right)^2\right){\rm d}\tau<\int^{t}_{0}\left(-8\right){\rm d}\tau,
\end{align*}
i.e.,
\begin{align}\label{4.19}
g(t)<& g(0)-8t.
\end{align}
For the case $g(0)\leq0$, (\ref{4.19}) gives $g(t)<0$ for $t\in(0,+\infty)$. And for the case $g(0)>0$, since (\ref{4.19}) holds for $(0,+\infty)$, we know that there exists a large enough $T_{1}\in(0,+\infty)$ such that $g(t)<0$ for all $t\in[T_{1},+\infty)$. However, according to (\ref{4.18}), for $t\in(T_{1},+\infty)$, we have
\begin{align*}
\frac{g'(t)}{\left(g(t)\right)^{2}}<-\frac{8}{\left(g(t)\right)^{2}}-1<-1,\nonumber
\end{align*}
which means
\begin{align*}
\int^{t}_{T_{1}}\frac{g'(\tau)}{(g(\tau))^{2}} {\rm d}\tau <\int^{t}_{T_{1}}(-1) {\rm d}\tau,\nonumber
\end{align*}
i.e.,
\begin{align}\label{4.20}
\frac{1}{g(t)}>\frac{1}{g(T_{1})}+(t-T_{1}).
\end{align}
 For (\ref{4.20}) with $\frac{1}{g(T_{1})}<0$, according to the assumption $T_{{\rm max}}=+\infty$, we can find a large enough $T_{2}\in(T_{1},+\infty)$ such that $\frac{1}{g(T_{1})}+(t-T_{1})>0$, i.e., $g(t)>0$, for $t\in(T_{2},+\infty)$, which contradicts $g(t)<0$ for all $t\in(T_{1},+\infty)$ obtained above. Thus, we have $T_{\rm{max}}<+\infty$, which means that the solution $\varphi$ blows up in finite time according to Lemma \ref{lem2.1}.
\end{proof}
\end{theorem}In Theorem \ref{th4.1} and Theorem \ref{th4.2}, we use the manifolds $\mathcal{W}_{j}$ and $\mathcal{V}_{j}$, $j=1,2,\cdots,8$, to prove the global existence and finite time blowup, respectively, we show the manifolds $\mathcal{W}_{j}$ and $\mathcal{V}_{j}$, $j=1,2,\cdots,8$, are both not empty.
\begin{remark}\label{rem4.1}(The manifolds $\mathcal{W}_{j}$ and $\mathcal{V}_{j}$, $j=1,2,\cdots,8$, are not empty)
We only discuss the case $j=1$ since the the analysis of the cases $j=2,3,\cdots,8$ are similar.
For the function $u\neq0$, we take $\iota>0$ and consider
\begin{align}\label{4.21-}\begin{split}
S_{1}(\iota u)=&\left(\frac{1}{2}\|\nabla u\|_{2}^{2}\right)\iota^{2}-\left(\frac{n(p-2)\lambda_{1}}{2p}\| u\|_{p}^{p}\right)\iota^{p}\\
=&\iota^{2}\left(\frac{1}{2}\|\nabla u\|_{2}^{2}-\left(\frac{n(p-2)\lambda_{1}}{2p}\| u\|_{p}^{p}\right)\iota^{p-2}\right).
\end{split}
\end{align}
Thus, we know that exists a unique $\iota_{0}>0$ such that
\begin{align}\label{4.22-}
\frac{1}{2}\|\nabla u\|_{2}^{2}-\left(\frac{n(p-2)\lambda_{1}}{2p}\| u\|_{p}^{p}\right)\iota_{0}^{p-2}=0,
\end{align}
i.e.,
$S_{1}(\iota_{0} u)=0$, $S_{1}(\iota u)>0$ for $0<\iota<\iota_{0}$ and $S_{1}(\iota u)<0$ for $\iota>\iota_{0}$. Thus, we have $\iota_{0} u\in\mathcal{N}_{1}$, which means $K(\iota_{0} u)\geq d_{\mathcal{N}_{1}}$. We consider
\begin{align}\label{Kdebianhuan}
\begin{split}
K(\iota u)=&\left(\frac{1}{2}\|\nabla u\|_{2}^{2}+\frac{1}{2}\|u\|_{2}^{2}+\frac{1}{2}\||x| u\|_{2}^{2}\right)\iota^{2}-\left(\frac{\lambda_{1}}{p}\| u\|_{p}^{p}\right)\iota^{p}-\left(\frac{\lambda_{2}}{q}\| u\|_{q}^{q}\right)\iota^{q}\\
=&\iota^{2}\left(\frac{1}{2}\|\nabla u\|_{2}^{2}+\frac{1}{2}\|u\|_{2}^{2}+\frac{1}{2}\||x| u\|_{2}^{2}-\left(\frac{\lambda_{1}}{p}\| u\|_{p}^{p}\right)\iota^{p-2}-\left(\frac{\lambda_{2}}{q}\| u\|_{q}^{q}\right)\iota^{q-2}\right).
\end{split}
\end{align}
Then, we can explain that the manifolds $\mathcal{W}_{1}$ and $\mathcal{V}_{1}$ are both not empty as follows.\\
{\bf{(i) The manifold $\mathcal{W}_{1}$ is not empty.}} According to $(\ref{Kdebianhuan})$, for $\iota=0$, we know $K(\iota u)=0$. Combining this fact, $K(\iota_{0} u)\geq d_{\mathcal{N}_{1}}$ and the continuity of $K(\iota u)$ with respect to $\iota$, we know that there exists $\bar{\iota}\in(0,\iota_{0})$ such that $0<K(\bar{\iota} u)<d_{\mathcal{N}_{1}}$. And due to $\bar{\iota}<\iota_{0}$, we have $S_{1}(\bar{\iota} u)>0$. Thus, we know $\bar{\iota} u\in\mathcal{W}_{1}$.\\
{\bf{(ii) The manifold $\mathcal{V}_{1}$ is not empty.}} According to $(\ref{Kdebianhuan})$, there exists $\tilde{\iota}>\iota_{0}$ large enough such that $K(\tilde{\iota} u)<0$. Combining this fact, $K(\iota_{0} u)\geq d_{\mathcal{N}_{1}}$ and the continuity of $K(\iota u)$ with respect to $\iota$, we know that there exists $\hat{\iota}\in(\iota_{0},\tilde{\iota})$ such that $0\leq K(\hat{\iota} u)<d_{\mathcal{N}_{1}}$. And due to $\hat{\iota}>\iota_{0}$, we have $S_{1}(\hat{\iota} u)<0$. Thus, we know $\hat{\iota} u\in\mathcal{V}_{1}$.
\end{remark}

\subsection{Some imperfect criteria}\label{sec4.2}

For non-negative low initial energy $0\leq K(\varphi_{0})<d_{\mathcal{N}_{j}}$, $j=1,2,\cdots,8$, in Subsection \ref{sec4.1}, we give eight sharp conditions to distinguish the global solution and finite time blowup solution to the Cauchy problem (\ref{1.1}) by eight cross-constrained variational functionals $S_{j}(\varphi)$, $j=1,2,\cdots,8$, which are constructed by selecting and combining the five terms in ``energy" $K(\varphi)$ via adjusting their coefficients, as shown in Table 1.
By observing Table 1, in constructing each of these eight cross-constrained variational functionals, we select both the terms $\|\nabla\varphi\|^{2}_{2}$ and $\|\varphi\|^{p}_{p}$ for each combination, while the other terms are optional. Next, we analyze the necessity of selecting both of these two terms by testing the cases of applying the cross-constrained variational functionals without one of these two terms. We use the following Table 2 to present the combinations of the structural terms in $K(\varphi)$ for the case that only one of these two terms is selected and the case that neither of these two terms is selected.
\begin{table}[H]
  \centering
  \caption{The cross-constrained variational combinations without one of the two mandatory terms $\|\nabla\varphi\|^{2}_{2}$ and $\|\varphi\|^{p}_{p}$}\label{tab1}
  \vspace{0.3cm}
  \tiny
  \begin{threeparttable}
  \begin{scriptsize}
  \begin{tabular}{p{2.3cm}<{\centering}p{0.8cm}<{\centering}p{0.9cm}<{\centering}p{0.6cm}<{\centering}p{0.6cm}<{\centering}p{0.7cm}<{\centering}p{2.0cm}<{\centering}p{2.5cm}<{\centering}}

    \toprule
    \ &$\|\nabla \varphi\|^{2}_{2}$ &$\||x|\varphi\|^{2}_{2}$ &$\|\varphi\|^{2}_{2}$ &$\|\varphi\|^{p}_{p}$ &$\|\varphi\|^{q}_{q}$&Structures&Results\\
    \midrule
    Case 1\ \ \ \ $S_{9}(\varphi)$&$\circ$ &$\circ$ &$\bullet$ &$\bullet$  &$\circ$ & without $\|\nabla \varphi\|^{2}_{2}$&Global existence\\
    Case 2\ \ \ \ $S_{10}(\varphi)$&$\circ$ &$\bullet$ &$\circ$ &$\bullet$  &$\circ$ & without $\|\nabla \varphi\|^{2}_{2}$&Global existence\\
    Case 3\ \ \ \ $S_{11}(\varphi)$&$\circ$ &$\bullet$ &$\bullet$ &$\bullet$  &$\circ$ & without $\|\nabla \varphi\|^{2}_{2}$&Global existence\\
    Case 4\ \ \ \ $S_{12}(\varphi)$&$\circ$ &$\circ$ &$\bullet$ &$\bullet$  &$\bullet$ & without $\|\nabla \varphi\|^{2}_{2}$&Global existence\\
    Case 5\ \ \ \ $S_{13}(\varphi)$&$\circ$ &$\bullet$ &$\circ$ &$\bullet$  &$\bullet$ & without $\|\nabla \varphi\|^{2}_{2}$&Global existence\\
    Case 6\ \ \ \ $S_{14}(\varphi)$&$\circ$ &$\bullet$ &$\bullet$ &$\bullet$  &$\bullet$ & without $\|\nabla \varphi\|^{2}_{2}$&Global existence\\
    \toprule
    Case 7\ \ \ \ $S_{15}(\varphi)$&$\bullet$ &$\circ$ &$\circ$ &$\circ$  &$\bullet$ & without $\|\varphi\|^{p}_{p}$&Finite time blowup\\
    Case 8\ \ \ \ $S_{16}(\varphi)$&$\bullet$ &$\circ$ &$\bullet$ &$\circ$  &$\bullet$ &without $\|\varphi\|^{p}_{p}$&Finite time blowup\\
    Case 9\ \ \ \ $S_{17}(\varphi)$&$\bullet$ &$\bullet$ &$\circ$ &$\circ$  &$\bullet$ &without $\|\varphi\|^{p}_{p}$&Finite time blowup\\
    Case 10\ \ \ \ $S_{18}(\varphi)$&$\bullet$ &$\bullet$ &$\bullet$ &$\circ$  &$\bullet$ &without $\|\varphi\|^{p}_{p}$&Finite time blowup\\
   \toprule
    Case 11\ \ \ \ $S_{19}(\varphi)$&$\circ$ &$\circ$ &$\bullet$ &$\circ$  &$\bullet$ &without $\|\nabla \varphi\|^{2}_{2}$ and $\|\varphi\|^{p}_{p}$&Unavailable\\
    Case 12\ \ \ \ $S_{20}(\varphi)$&$\circ$ &$\bullet$ &$\circ$ &$\circ$  &$\bullet$ &without $\|\nabla \varphi\|^{2}_{2}$ and $\|\varphi\|^{p}_{p}$&Unavailable\\
    Case 13\ \ \ \ $S_{21}(\varphi)$&$\circ$ &$\bullet$ &$\bullet$ &$\circ$  &$\bullet$ &without $\|\nabla \varphi\|^{2}_{2}$ and $\|\varphi\|^{p}_{p}$&Unavailable\\
    \bottomrule
    \end{tabular}
    \begin{tablenotes}
    \item $\bullet=$ choose \ \ \ $\circ=$ not choose
    \end{tablenotes}
  \end{scriptsize}
  \end{threeparttable}
\end{table}

For these two mandatory terms $\|\nabla\varphi\|^{2}_{2}$ and $\|\varphi\|^{p}_{p}$, we select $\|\varphi\|^{p}_{p}$ and do not select $\|\nabla\varphi\|^{2}_{2}$ to construct the following cross-constrained variational functionals, i.e., $S_{9}(\varphi)-S_{14}(\varphi)$,
\begin{align*}
S_{9}(\varphi):=c\|\varphi\|^{2}_{2}-\frac{n(p-2)\lambda_{1}}{2p}\|\varphi\|^{p}_{p},
\end{align*}
\begin{align*}
S_{10}(\varphi):=\||x|\varphi\|^{2}_{2}-\frac{n(p-2)\lambda_{1}}{2p}\|\varphi\|^{p}_{p},
\end{align*}
\begin{align*}
S_{11}(\varphi):=c\|\varphi\|^{2}_{2}+\||x|\varphi\|^{2}_{2}-\frac{n(p-2)\lambda_{1}}{2p}\|\varphi\|^{p}_{p},
\end{align*}
\begin{align*}
S_{12}(\varphi):=\|\varphi\|^{2}_{2}-\frac{n(p-2)\lambda_{1}}{2p}\|\varphi\|^{p}_{p}-\frac{n(q-2)\lambda_{2}}{2q}\|\varphi\|^{q}_{q},
\end{align*}
\begin{align*}
S_{13}(\varphi):=\||x|\varphi\|^{2}_{2}-\frac{n(p-2)\lambda_{1}}{2p}\|\varphi\|^{p}_{p}-\frac{n(q-2)\lambda_{2}}{2q}\|\varphi\|^{q}_{q},
\end{align*}
and
\begin{align*}
S_{14}(\varphi):=\|\varphi\|^{2}_{2}+\||x|\varphi\|^{2}_{2}-\frac{n(p-2)\lambda_{1}}{2p}\|\varphi\|^{p}_{p}-\frac{n(q-2)\lambda_{2}}{2q}\|\varphi\|^{q}_{q},
\end{align*}where $c>0$ is chosen to make
(\ref{cdetiaojian}) hold.
And we select $\|\nabla\varphi\|^{2}_{2}$ and do not select $\|\varphi\|^{p}_{p}$ to construct the following cross-constrained variational functionals, i.e., $S_{15}(\varphi)-S_{18}(\varphi)$,
\begin{align*}
S_{15}(\varphi):=\|\nabla \varphi\|^{2}_{2}-\frac{n(q-2)\lambda_{2}}{2q}\|\varphi\|^{q}_{q},
\end{align*}
\begin{align*}
S_{16}(\varphi):=\|\nabla \varphi\|^{2}_{2}+\|\varphi\|^{2}_{2}-\frac{n(q-2)\lambda_{2}}{2q}\|\varphi\|^{q}_{q},
\end{align*}
\begin{align*}
S_{17}(\varphi):=\|\nabla \varphi\|^{2}_{2}+\||x|\varphi\|^{2}_{2}-\frac{n(q-2)\lambda_{2}}{2q}\|\varphi\|^{q}_{q},
\end{align*}
and
\begin{align*}
S_{18}(\varphi):=\|\nabla \varphi\|^{2}_{2}+\|\varphi\|^{2}_{2}+\||x|\varphi\|^{2}_{2}-\frac{n(q-2)\lambda_{2}}{2q}\|\varphi\|^{q}_{q}.
\end{align*}
In the following, we shall show that $S_{j}(\varphi)$, $j=9,10,\cdots,14$, can derive six criteria for global existence of the solution in Theorem \ref{dinglis9-14}, however, we fail to prove the finite time blowup, which is explained in Remark \ref{rem4.2}.
And for $S_{j}(\varphi)$, $j=15,16,17,18$, we shall give four criteria for finite time blowup of the solution in Theorem \ref{th4.4}, however, we fail to prove the global existence, which is explained in Remark \ref{rem4.3}. For the cross-constrained variational functionals $S_{j}(\varphi)$, $j=9,10,\cdots,18$, without only one of two mandatory terms $\|\nabla\varphi\|^{2}_{2}$ and $\|\varphi\|^{p}_{p}$, by classifying the initial data, we can only show one result between the global existence and finite time blowup. That is, we show the one-sided result of the classification of the initial data for $S_{j}(\varphi)$, $j=9,10,\cdots,18$. Thus we call these results the imperfect criteria.

We define $d_{{\mathcal{N}_{j}}}$, $\mathcal{H}_{j}$, $\mathcal{W}_{j}$ and $\mathcal{V}_{j}$, $j=9,10,\cdots,18$, as (\ref{2.5})-(\ref{2.8}). And by the similar analysis to Remark \ref{rem4.1}, we know that the manifolds $\mathcal{W}_{j}$ and $\mathcal{V}_{j}$, $j=9,10,\cdots,18$, are both not empty.
For the case that we do not select the mandatory term $\|\nabla\varphi\|^{2}_{2}$ nor $\|\varphi\|^{p}_{p}$, we present the following cross-constrained variational functionals, i.e., $S_{19}(\varphi)-S_{21}(\varphi)$,
\begin{align*}
S_{19}(\varphi):=\|\varphi\|^{2}_{2}-\frac{n(q-2)\lambda_{2}}{2q}\|\varphi\|^{q}_{q},
\end{align*}
\begin{align*}
S_{20}(\varphi):=\||x|\varphi\|^{2}_{2}-\frac{n(q-2)\lambda_{2}}{2q}\|\varphi\|^{q}_{q},
\end{align*}
and
\begin{align*}
S_{21}(\varphi):=\|\varphi\|^{2}_{2}+\||x|\varphi\|^{2}_{2}-\frac{n(q-2)\lambda_{2}}{2q}\|\varphi\|^{q}_{q}.
\end{align*}
According to Remark \ref{rem4.2} and Remark \ref{rem4.3}, we fail to prove the global existence or the finite time blowup by the cross-constrained variational functionals $S_{j}(\varphi)$, $j=19,20,21$.

Next, we apply the cross-constrained variational functionals $S_{j}(\varphi)$, $j=9,10,\cdots,14$, without the term $\|\nabla \varphi\|^{2}_{2}$ to give the criteria for the global existence.
\begin{theorem}\label{dinglis9-14}(Global existence criteria for non-negative low initial energy $0\leq K(\varphi_{0})< d_{\mathcal{N}_{j}}$, $j=9,10,\cdots,14$, by $S_{j}(\varphi)$, $j=9,10,\cdots,14$, without the term $\|\nabla \varphi\|^{2}_{2}$)
Let $n$, $p$, $q$, $\lambda_{1}>0$, $\lambda_{2}>0$ satisfy the following conditions

(i)\ for $j=9,11$: $n=1,2$, $p$, $q$ satisfy (\ref{1.2}), $\lambda_{1}$ and $\lambda_{2}$ additionally satisfy (\ref{2.9}) and (\ref{2.10});

(ii)\ for $j=10$: $n=1$, $p>q\geq6$, $\lambda_{1}$ and $\lambda_{2}$ additionally satisfy (\ref{2.9}) and (\ref{2.10});

(iii)\ for $j=12,14$: $n=1,2$, $p$, $q$ satisfy (\ref{1.2}), with no additional restrictions on $\lambda_{1}$ and $\lambda_{2}$;

(iv)\ for $j=13$: $n=1$, $p>q\geq6$, with no additional restrictions on $\lambda_{1}$ and $\lambda_{2}$.\\
Then
the corresponding solution $\varphi$ to problem (\ref{1.1}) exists globally in time provided $\varphi_{0}\in \mathcal{W}_{j}$, $j=9,10,\cdots,14$.
\end{theorem}

\begin{proof}
For the initial data $\varphi_{0}\in\mathcal{W}_{j}$, $j=9,10,\cdots,14$, we prove $T_{\rm{max}}=+\infty$. Here, in order to ensure that  $\mathcal{W}_{j}$, $j=9,10,\cdots,14$, are well-defined, we need to prove $d_{{\mathcal{N}_{j}}}>0$, $j=9,10,\cdots,14$. After this, we shall prove the invariance of the manifolds $\mathcal{W}_{j}$, $j=9,10,\cdots,14$, which is required for proving the global existence. Thus, we divide the proof into the following three steps, i.e., Step \textrm{I} ($d_{{\mathcal{N}_{j}}}>0$, $j=9,10,\cdots,14$), Step \textrm{II} (the manifold $\mathcal{W}_{j}$, $j=9,10,\cdots,14$, is invariant) and Step \textrm{III} (global existence).\\
\textbf{Step \textrm{I}}:\ \bm{$d_{{\mathcal{N}_{j}}}>0$}, \bm{$j=9,10,\cdots,14$}.

 In order to prove $d_{_{\mathcal{N}_{j}}}>0$, $j=9,10,\cdots,14$, we need to control the negative terms $-\frac{\lambda_{1}}{p}\|\varphi\|^{p}_{p}$ and $-\frac{\lambda_{2}}{q}\|\varphi\|^{q}_{q}$ in the functional $K(\varphi)$. The strategies depend on the structures of the cross-constrained variational functionals $S_{j}(\varphi)$, $j=9,10,\cdots,14$. According to the structures of $S_{j}(\varphi)$, $j=9,10,\cdots,14$, we divide the proof into two cases: Case \textrm{I} ($j=9,10,11$, here the cross-constrained variational functionals $S_{j}(\varphi)$ have the term $\|\varphi\|^{p}_{p}$ and do not have the term $\|\varphi\|^{q}_{q}$) and Case \textrm{II} ($j=12,13,14$, here $S_{j}(\varphi)$ both have the terms $\|\varphi\|^{p}_{p}$ and $\|\varphi\|^{q}_{q}$).\\
{\bf{Case \textrm{I}: \bm{$j=9,10,11.$}}}

According to $\varphi\in\mathcal{N}_{j}$, i.e., $S_{j}(\varphi)=0$, $\varphi\neq0$, $j=9,10,11$, we get
\begin{align}
\|\varphi\|^{p}_{p}\leq\frac{2p}{n(p-2)\lambda_{1}}\left(c\|\varphi\|^{2}_{2}+\||x|\varphi\|^{2}_{2}\right).\label{4.21}
\end{align}
Meanwhile, since the assumptions we adopted here are same as Lemma \ref{lem2.3}, we can retrieve (\ref{2.16}) to deal with $-\frac{\lambda_{2}}{q}\|\varphi\|^{q}_{q}$. Substituting (\ref{4.21}) into (\ref{2.16}), we have
\begin{align}\label{4.22}
\begin{split}
K(\varphi)\geq&\frac{1}{2}\|\nabla \varphi\|^{2}_{2}+\left(\frac{1}{2}-\frac{\lambda_{2}}{q}\right)\|\varphi\|^{2}_{2}+\frac{1}{2}\||x|\varphi\|^{2}_{2}\\
&-\left(\frac{\lambda_{1}}{p}+\frac{\lambda_{2}}{q}\right)\frac{2p}{n(p-2)\lambda_{1}}\left(c\|\varphi\|^{2}_{2}+\||x|\varphi\|^{2}_{2}\right)\\
=&\frac{1}{2}\|\nabla \varphi\|^{2}_{2}+\left(\frac{1}{2}-\frac{2}{n(p-2)}-\frac{2\lambda_{2}p}{\lambda_{1}qn(p-2)}\right)\||x|\varphi\|^{2}_{2}\\
&+\left(\frac{1}{2}-\left(\frac{2}{n(p-2)}+\frac{2\lambda_{2}p}{\lambda_{1}qn(p-2)}\right)c-\frac{\lambda_{2}}{q}\right)\| \varphi\|^{2}_{2},
\end{split}
\end{align}
In (\ref{4.22}), the assumption (\ref{2.10}) ensures
\begin{align*}
\frac{1}{2}-\frac{2}{n(p-2)}-\frac{2\lambda_{2}p}{\lambda_{1}qn(p-2)}>0.
\end{align*}
And $c>0$ is chosen such that
$$\frac{1}{2}-\left(\frac{2}{n(p-2)}+\frac{2\lambda_{2}p}{\lambda_{1}qn(p-2)}\right)c-\frac{\lambda_{2}}{q}>0.$$
Next, we deal with (\ref{4.22}). Since the cross-constrained variational functionals $S_{j}(\varphi)$, $j=9,11$, have the term $\|\varphi\|^{2}_{2}$, the strategies are similar. Thus, we shall finish the proof of Case \textrm{I} by considering two sub-cases, i.e., Case \textrm{I}(a) ($j=9,11$) and Case \textrm{I}(b) ($j=10$).\\
{\bf{Case \textrm{I}(a): \bm{$j=9,11.$}}}

According to $S_{j}(\varphi)=0$, $j=9,11$, we know
\begin{equation}\label{4.24}
\frac{2p c}{n(p-2)\lambda_{1}}\|\varphi\|^{2}_{2}\leq\|\varphi\|^{p}_{p}.
\end{equation}
By applying Gagliardo-Nirenberg inequality, i.e., (\ref{GN-budengshi}), and (\ref{4.24}), we have
\begin{align}
\frac{2p c}{n(p-2)\lambda_{1}C_{1}}\leq&\|\nabla\varphi\|^{\frac{n(p-2)}{2}}_{2}\|\varphi\|^{p-\frac{n(p-2)}{2}-2}_{2}.\label{4.25}
\end{align}
According to (\ref{1.2}), i.e., $2+\frac{4}{n}<p$, we notice $\frac{n(p-2)}{2}>0$ in (\ref{4.25}). In order to obtain the estimate of $\|\nabla\varphi\|^{2}_{2}$ and $\|\varphi\|^{2}_{2}$ by (\ref{4.25}), we require $p-\frac{n(p-2)}{2}-2\geq0$, i.e., $p-2\geq\frac{(p-2)n}{2}$, which means that we need the restriction $n=1,2$ in regard of the space dimension $n$. More precisely, for $n=1$, we have $$p-\frac{n(p-2)}{2}-2=\frac{p}{2}-1>0$$ according to (\ref{1.2}), i.e., $2+\frac{4}{n}<p$. By Cauchy-Schwartz inequality, we know (\ref{4.25}) turns to
\begin{align*}
\frac{2p c}{n(p-2)\lambda_{1}C_{1}}\leq&\frac{1}{2}\|\nabla\varphi\|^{n(p-2)}_{2}+\frac{1}{2}\|\varphi\|^{2(p+n)-np-4}_{2},\ n=1,
\end{align*}
which
implies
\begin{align}
\left(\frac{2p c}{n(p-2)\lambda_{1}C_{1}}\right)^{\frac{2}{n(p-2)}}\leq&\|\nabla\varphi\|^{2}_{2}, \ n=1,\label{4.26}
\end{align}
or
\begin{align}
\left(\frac{2p c}{n(p-2)\lambda_{1}C_{1}}\right)^{\frac{2}{2(p+n)-np-4}}\leq&\|\varphi\|^{2}_{2},\ n=1.\label{4.27}
\end{align}
By substituting (\ref{4.26}) or (\ref{4.27}) into (\ref{4.22}), we can obtain
\begin{align}
K(\varphi)\geq \frac{1}{2}\left(\frac{2p c}{n(p-2)\lambda_{1}C_{1}}\right)^{\frac{2}{n(p-2)}}>0,\ n=1,\label{4.27+}
\end{align}
or
\begin{align}
K(\varphi)\geq &\left(\frac{1}{2}-\left(\frac{2}{n(p-2)}+\frac{2\lambda_{2}p}{\lambda_{1}qn(p-2)}\right)c-\frac{\lambda_{2}}{q}\right)\left(\frac{2p c}{n(p-2)\lambda_{1}C_{1}}\right)^{\frac{2}{2(p+n)-np-4}}>0,\ n=1.\label{4.27++}
\end{align}
For $n=2$, we have $p-\frac{n(p-2)}{2}-2=0$. Then, we know (\ref{4.25}) turns to
\begin{align}
\left(\frac{2p c}{n(p-2)\lambda_{1}C_{1}}\right)^{\frac{4}{n(p-2)}}\leq&\|\nabla\varphi\|^{2}_{2},\ n=2.\label{4.26+}
\end{align}
By substituting (\ref{4.26+}) into (\ref{4.22}), we have
\begin{align}
K(\varphi)\geq \frac{1}{2}\left(\frac{2p c}{n(p-2)\lambda_{1}C_{1}}\right)^{\frac{4}{n(p-2)}}>0,\ n=2.\label{4.27+++}
\end{align}
According to (\ref{4.27+}), (\ref{4.27++}) and (\ref{4.27+++}), we know $d_{{\mathcal{N}_{j}}}>0$, $j=9,11$.\\
{\bf{Case \textrm{I}(b): \bm{$j=10.$}}}

According to $S_{10}(\varphi)=0$, we know
\begin{equation}\label{4.28}
\frac{2p}{n(p-2)\lambda_{1}}\||x|\varphi\|^{2}_{2}=\|\varphi\|^{p}_{p}.
\end{equation}
By applying Gagliardo-Nirenberg inequality, i.e., (\ref{GN-budengshi}), choosing $\mu=\frac{n(p-2)}{2p}$ in (\ref{GN-budengshi}), and Lemma \ref{lem3.2}, we have
\begin{align}
&\|\varphi\|^{p}_{p}\leq C_{1}\|\nabla\varphi\|^{p\mu}_{2}\|\varphi\|^{p(1-\mu)}_{2}\leq C_{1}\|\nabla\varphi\|^{p\mu}_{2}\left(\frac{2}{n}\|\nabla\varphi\|_{2}\||x|\varphi\|_{2}\right)^{\frac{p(1-\mu)}{2}}.\label{4.29}
\end{align}
Combining (\ref{4.29}) and (\ref{4.28}), we have
\begin{align*}
&\frac{2p}{n(p-2)\lambda_{1}C_{1}}\||x|\varphi\|^{2}_{2}\leq\|\nabla\varphi\|^{p\mu}_{2}\left(\frac{2}{n}\|\nabla\varphi\|_{2}\||x|\varphi\|_{2}\right)^{\frac{p(1-\mu)}{2}},
\end{align*}
i.e.,
\begin{align}
&\frac{2p}{n(p-2)\lambda_{1}C_{1}}\left(\frac{n}{2}\right)^{\frac{p(1-\mu)}{2}}\leq\|\nabla\varphi\|^{p\mu+\frac{p(1-\mu)}{2}}_{2}\||x|\varphi\|_{2}^{\frac{p(1-\mu)}{2}-2},\label{4.30}
\end{align}
where $\mu=\frac{n(p-2)}{2p}$.
The assumption $n=1$ gives $0<\mu<1$ in (\ref{4.30}). In order to estimate $\|\nabla\varphi\|^{2}_{2}$ and $\|\varphi\|^{2}_{2}$ in (\ref{4.22}), we need $$\frac{p(1-\mu)}{2}-2\geq0,$$ i.e., $$p-\frac{pn}{2}+n\geq4,$$ which requires $n=1$ and $p\geq6$.
  If $p>6$, by Cauchy-Schwartz inequality, we know (\ref{4.30}) turns to
\begin{align}
&\frac{2p}{n(p-2)\lambda_{1}C_{1}}\left(\frac{n}{2}\right)^{\frac{p(1-\mu)}{2}}\leq\frac{1}{2}\|\nabla\varphi\|^{p(\mu+1)}_{2}+\frac{1}{2}\||x|\varphi\|_{2}^{p(1-\mu)-4},\ n=1,p>6,\label{4.31}
\end{align}which implies
\begin{align}
\left(\frac{2p}{n(p-2)\lambda_{1}C_{1}}\right)^{\frac{2}{p(\mu+1)}}\left(\frac{n}{2}\right)^{\frac{1-\mu}{1+\mu}}\leq&\|\nabla\varphi\|^{2}_{2},\ n=1,p>6,\label{4.32}
\end{align}
or
\begin{align}
\left(\frac{2p}{n(p-2)\lambda_{1}C_{1}}\right)^{\frac{2}{p(1-\mu)-4}}\left(\frac{n}{2}\right)^{\frac{p(1-\mu)}{p(1-\mu)-4}}\leq&\||x|\varphi\|^{2}_{2},\ n=1,p>6.\label{4.33}
\end{align}
 By substituting (\ref{4.32}) or (\ref{4.33}) into (\ref{4.22}), we can obtain
\begin{align}
K(\varphi)\geq \frac{1}{2}\left(\frac{2p}{n(p-2)\lambda_{1}C_{1}}\right)^{\frac{4}{n(p-2)+2p}}\left(\frac{n}{2}\right)^{\frac{2p-n(p-2)}{2p+n(p-2)}}>0,\ n=1,p>6,\label{4.34-}
\end{align}
or
\begin{align}\label{4.35-}
\begin{split}
K(\varphi)\geq & \left(\frac{1}{2}-\frac{2}{n(p-2)}-\frac{2\lambda_{2}p}{\lambda_{1}qn(p-2)}\right)
\left(\frac{2p}{n(p-2)\lambda_{1}C_{1}}\right)^{\frac{4}{2p-n(p-2)-8}}\left(\frac{n}{2}\right)^{\frac{2p-n(p-2)}{2p-n(p-2)-8}}>0,\\
&\ \ n=1,p>6.
\end{split}
\end{align}
If $p=6$, according to (\ref{4.30}), we have
\begin{align}
\left(\frac{2p}{n(p-2)\lambda_{1}C_{1}}\right)^{\frac{4}{p(\mu+1)}}\left(\frac{n}{2}\right)^{\frac{2(1-\mu)}{1+\mu}}\leq&\|\nabla\varphi\|^{2}_{2},\ n=1,p=6.\label{4.32+}
\end{align}
 By substituting (\ref{4.32+}) into (\ref{4.22}), we can obtain
\begin{align}
K(\varphi)\geq \frac{1}{2}\left(\frac{2p}{n(p-2)\lambda_{1}C_{1}}\right)^{\frac{8}{n(p-2)+2p}}\left(\frac{n}{2}\right)^{\frac{4p-2n(p-2)}{2p+n(p-2)}}>0,\ n=1,p=6.\label{4.34--}
\end{align}
According to (\ref{4.34-}), (\ref{4.35-}) and (\ref{4.34--}), we know $d_{{\mathcal{N}_{10}}}>0$.\\
{\bf{Case \textrm{II}: \bm{$j=12,13,14.$}}}

Since (\ref{1.2}) is assumed, we know that (\ref{2.40}) holds. According to the condition $\varphi\in\mathcal{N}_{j}$, i.e., $S_{j}(\varphi)=0$, $\varphi\neq0$, $j=12,13,14,$ we have
\begin{align}
\frac{n(p-2)\lambda_{1}}{2p}\|\varphi\|^{p}_{p}+\frac{n(q-2)\lambda_{2}}{2q}\|\varphi\|^{q}_{q}\leq\|\varphi\|^{2}_{2}+\||x|\varphi\|^{2}_{2}.\label{4.36}
\end{align}
By substituting (\ref{4.36}) into (\ref{2.40}), we can deduce
\begin{align}
K(\varphi)\geq&\frac{1}{2}\|\nabla \varphi\|^{2}_{2}+\frac{1}{2}\|\varphi\|^{2}_{2}+\frac{1}{2}\||x|\varphi\|^{2}_{2}-\frac{2}{n(q-2)}\left(\|\varphi\|^{2}_{2}+\||x|\varphi\|^{2}_{2}\right),\label{4.37}
\end{align}
i.e.,
\begin{align}
K(\varphi)\geq&\frac{1}{2}\|\nabla\varphi\|^{2}_{2}+\left(\frac{1}{2}-\frac{2}{n(q-2)}\right)\left(\|\varphi\|^{2}_{2}+\||x|\varphi\|^{2}_{2}\right)
.\label{4.38}
\end{align}In (\ref{4.38}), we know $\frac{1}{2}-\frac{2}{n(q-2)}>0$ obtained by (\ref{1.2}), i.e., $2+\frac{4}{n}<q$.
Next, we deal with (\ref{4.38}). For the case $j=12$ and the case $j=14$, due to $S_{j}(\varphi)$ having the term $\|\varphi\|^{2}_{2}$, the strategies are similar. Thus, we shall also finish the proof of this Case \textrm{II} by considering the following two sub-cases, i.e., Case \textrm{II}(a) ($j=12,14$) and Case \textrm{II}(b) ($j=13$).\\
{\bf{Case \textrm{II}(a): \bm{$j=12,14.$}}}

According to $S_{j}(\varphi)=0$, $j=12,14$, we know
\begin{equation}
\frac{n(p-2)\lambda_{1}}{2p}\|\varphi\|^{p}_{p}+\frac{n(q-2)\lambda_{2}}{2q}\|\varphi\|^{q}_{q}\geq\|\varphi\|^{2}_{2},\nonumber
\end{equation}
which means
\begin{align}
&\frac{n(p-2)\lambda_{1}}{2p}\|\varphi\|^{p}_{p}\geq\frac{1}{2}\|\varphi\|^{2}_{2}\label{4.39}
\end{align}
or
\begin{align}
&\frac{n(q-2)\lambda_{2}}{2q}\|\varphi\|^{q}_{q}\geq\frac{1}{2}\|\varphi\|^{2}_{2}.\label{4.40}
\end{align}

If (\ref{4.39}) holds, applying (\ref{GN-budengshi}), (\ref{4.39}) and Cauchy-Schwarz inequality, we have
\begin{align}\label{4.41}
\begin{split}
\frac{p}{n(p-2)\lambda_{1}C_{1}}\leq&\|\nabla\varphi\|^{\frac{n(p-2)}{2}}_{2}\|\varphi\|^{p-\frac{n(p-2)}{2}-2}_{2}\\
\leq&\frac{1}{2}\|\nabla\varphi\|^{n(p-2)}_{2}+\frac{1}{2}\|\varphi\|^{2(p+n)-np-4}_{2}.
\end{split}
\end{align}
According to (\ref{1.2}), i.e., $2+\frac{4}{n}<p$, we notice $\frac{n(p-2)}{2}>0$ in (\ref{4.41}). In order to obtain the estimate of $\|\nabla\varphi\|^{2}_{2}$ and $\|\varphi\|^{2}_{2}$ by (\ref{4.41}), we require $p-\frac{n(p-2)}{2}-2\geq0$, which means that we need the restriction $n=1,2$. More precisely, for $n=1$, we have $p-\frac{n(p-2)}{2}-2=\frac{p}{2}-1>0$ according to (\ref{1.2}), i.e., $2+\frac{4}{n}<p$. Thus, we notice that (\ref{4.41}) implies
\begin{align}
\left(\frac{p}{n(p-2)\lambda_{1}C_{1}}\right)^{\frac{2}{n(p-2)}}\leq&\|\nabla\varphi\|^{2}_{2},\ n=1,\label{4.42}
\end{align}
or
\begin{align}
\left(\frac{p}{n(p-2)\lambda_{1}C_{1}}\right)^{\frac{2}{2(p+n)-np-4}}\leq&\|\varphi\|^{2}_{2},\ n=1.\label{4.43}
\end{align}
By substituting (\ref{4.42}) or (\ref{4.43}) into (\ref{4.38}), we obtain
\begin{align}
K(\varphi)\geq \frac{1}{2}\left(\frac{p}{n(p-2)\lambda_{1}C_{1}}\right)^{\frac{2}{n(p-2)}}>0,\ n=1,\label{4.45}
\end{align}
or
\begin{align}
K(\varphi)\geq \left(\frac{1}{2}-\frac{2}{n(q-2)}\right)\left(\frac{p}{n(p-2)\lambda_{1}C_{1}}\right)^{\frac{2}{2(p+n)-np-4}}>0,\ n=1.\label{4.46}
\end{align}
For $n=2$, we have $p-\frac{n(p-2)}{2}-2=0$. Then, according to (\ref{4.41}), we obtain
 \begin{align}
\left(\frac{p}{n(p-2)\lambda_{1}C_{1}}\right)^{\frac{4}{n(p-2)}}\leq&\|\nabla\varphi\|^{2}_{2},\ n=2.\label{4.44}
\end{align}
By substituting (\ref{4.44}) into (\ref{4.38}), we obtain
\begin{align}
K(\varphi)\geq \frac{1}{2}\left(\frac{p}{n(p-2)\lambda_{1}C_{1}}\right)^{\frac{4}{n(p-2)}}>0,\ n=2.\label{4.47}
\end{align}

If (\ref{4.40}) holds, the process of dealing with this case is similar to that of treating the case that (\ref{4.39}) holds. For this case, we need to employ a different version of Gagliardo-Nirenberg, i.e., (\ref{GN-budengshi1}),
to deal with (\ref{4.40}) by choosing $\theta=\frac{n(q-2)}{2q}$ in (\ref{GN-budengshi1}). More precisely, by (\ref{GN-budengshi1}) and the inequality $$\frac{q}{n(q-2)\lambda_{2}}\|\varphi\|^{2}_{2}\leq\|\varphi\|^{q}_{q}$$ as another version of (\ref{4.40}), we get
\begin{align}\label{4.48}
\begin{split}
\frac{q}{n(q-2)\lambda_{2}C_{3}}\leq&\|\nabla\varphi\|^{\frac{n(q-2)}{2}}_{2}\|\varphi\|^{q-\frac{n(q-2)+4}{2}}_{2}\\
\leq&\frac{1}{2}\|\nabla\varphi\|^{n(q-2)}_{2}+\frac{1}{2}\|\varphi\|^{2(q+n)-nq-4}_{2}.
\end{split}
\end{align}Similar to the analysis for the case that (\ref{4.39}) holds, we require $n=1,2$ to ensure $q-\frac{n(q-2)+4}{2}\geq0$ in (\ref{4.48}).
Then, by the similar process of inferring the validity of (\ref{4.42}) and (\ref{4.43}) for $n=1$, we obtain
\begin{align}
\left(\frac{q}{n(q-2)\lambda_{2}C_{3}}\right)^{\frac{2}{n(q-2)}}\leq\|\nabla\varphi\|^{2}_{2},\ n=1,\label{4.49}
\end{align}
or
\begin{align}
\left(\frac{q}{n(q-2)\lambda_{2}C_{3}}\right)^{\frac{2}{2(q+n)-nq-4}}\leq&\|\varphi\|^{2}_{2},\ n=1.\label{4.50}
\end{align}
By substituting (\ref{4.49}) or (\ref{4.50}) into (\ref{4.38}), we can obtain
\begin{align}
K(\varphi)\geq \frac{1}{2}\left(\frac{q}{n(q-2)\lambda_{2}C_{3}}\right)^{\frac{2}{n(q-2)}}>0,\ n=1,\label{4.52}
\end{align}
or\begin{align}
K(\varphi)\geq \left(\frac{1}{2}-\frac{2}{n(q-2)}\right)\left(\frac{q}{n(q-2)\lambda_{2}C_{3}}\right)^{\frac{2}{2(q+n)-nq-4}}>0,\ n=1.\label{4.53}
\end{align}
By the similar process of inferring the validity of (\ref{4.44}) for $n=2$, we have
\begin{align}
\left(\frac{q}{n(q-2)\lambda_{2}C_{3}}\right)^{\frac{4}{n(q-2)}}\leq\|\nabla\varphi\|^{2}_{2},\ n=2.\label{4.51}
\end{align}
 By substituting (\ref{4.51}) into (\ref{4.38}), we can obtain
\begin{align}
K(\varphi)\geq \frac{1}{2}\left(\frac{q}{n(q-2)\lambda_{2}C_{3}}\right)^{\frac{4}{n(q-2)}}>0,\ n=2.\label{4.54}
\end{align}
According to (\ref{4.45}), (\ref{4.46}), (\ref{4.47}), (\ref{4.52}), (\ref{4.53}) and (\ref{4.54}), we know $d_{_{\mathcal{N}_{j}}}>0$, $j=12,14$.\\
 {\bf{Case \textrm{II}(b): \bm{$j=13.$}}}

According to $S_{13}(\varphi)=0$, we know
\begin{equation}
\frac{n(p-2)\lambda_{1}}{2p}\|\varphi\|^{p}_{p}+\frac{n(q-2)\lambda_{2}}{2q}\|\varphi\|^{q}_{q}=\||x|\varphi\|^{2}_{2},\nonumber
\end{equation}
which means
\begin{align}
&\frac{n(p-2)\lambda_{1}}{2p}\|\varphi\|^{p}_{p}\geq\frac{1}{2}\||x|\varphi\|^{2}_{2}\label{4.55}
\end{align}
or
\begin{align}
&\frac{n(q-2)\lambda_{2}}{2q}\|\varphi\|^{q}_{q}\geq\frac{1}{2}\||x|\varphi\|^{2}_{2}.\label{4.56}
\end{align}

If (\ref{4.55}) holds, by applying Gagliardo-Nirenberg inequality (\ref{GN-budengshi}), and Lemma \ref{lem3.2}, we can also get (\ref{4.29}).
Substituting (\ref{4.29}) into (\ref{4.55}), we have
\begin{align*}
&\frac{p}{n(p-2)\lambda_{1}C_{1}}\||x|\varphi\|^{2}_{2}\leq\|\nabla\varphi\|^{p\mu}_{2}\left(\frac{2}{n}\|\nabla\varphi\|_{2}\||x|\varphi\|_{2}\right)^{\frac{p(1-\mu)}{2}},
\end{align*}
i.e.,
\begin{align}
&\frac{p}{n(p-2)\lambda_{1}C_{1}}\left(\frac{n}{2}\right)^{\frac{p(1-\mu)}{2}}\leq\|\nabla\varphi\|^{p\mu+\frac{p(1-\mu)}{2}}_{2}\||x|\varphi\|_{2}^{\frac{p(1-\mu)}{2}-2},\label{4.57}
\end{align}
where $\mu=\frac{n(p-2)}{2p}$. The assumptions $n=1$ and $p\geq6$ gives $0<\mu<1$ and $$p(1-\mu)-4\geq0$$ in (\ref{4.57}). For $p>6$, by using Cauchy-Schwartz inequality, we know (\ref{4.57}) turns to
\begin{align}
&\frac{p}{n(p-2)\lambda_{1}C_{1}}\left(\frac{n}{2}\right)^{\frac{p(1-\mu)}{2}}\leq\frac{1}{2}\|\nabla\varphi\|^{p(\mu+1)}_{2}+\frac{1}{2}\||x|\varphi\|_{2}^{p(1-\mu)-4},\ n=1,p>6,\label{4.58}
\end{align}
which implies
\begin{align}
\left(\frac{p}{n(p-2)\lambda_{1}C_{1}}\right)^{\frac{2}{p(\mu+1)}}\left(\frac{n}{2}\right)^{\frac{1-\mu}{1+\mu}}\leq&\|\nabla\varphi\|^{2}_{2},\ n=1,p>6,\label{4.59}
\end{align}
or
\begin{align}
\left(\frac{p}{n(p-2)\lambda_{1}C_{1}}\right)^{\frac{2}{p(1-\mu)-4}}\left(\frac{n}{2}\right)^{\frac{p(1-\mu)}{p(1-\mu)-4}}\leq&\||x|\varphi\|^{2}_{2},\ n=1,p>6.\label{4.60}
\end{align}
By substituting (\ref{4.59}) or (\ref{4.60}) into (\ref{4.38}), we can obtain
\begin{align}
K(\varphi)\geq \frac{1}{2}\left(\frac{p}{n(p-2)\lambda_{1}C_{1}}\right)^{\frac{4}{n(p-2)+2p}}\left(\frac{n}{2}\right)^{\frac{2p-n(p-2)}{2p+n(p-2)}}>0,\ n=1,p>6,\label{4.62}
\end{align}
or\begin{align}
K(\varphi)\geq & \left(\frac{1}{2}-\frac{2}{n(q-2)}\right)\left(\frac{p}{n(p-2)\lambda_{1}C_{1}}\right)^{\frac{4}{2p-n(p-2)-8}}\left(\frac{n}{2}\right)^{\frac{2p-n(p-2)}{2p-n(p-2)-8}}>0,\ n=1,p>6.\label{4.63}
\end{align}
If $p=6$, according to (\ref{4.57}), we have
\begin{align}
\left(\frac{p}{n(p-2)\lambda_{1}C_{1}}\right)^{\frac{4}{p(\mu+1)}}\left(\frac{n}{2}\right)^{\frac{2(1-\mu)}{1+\mu}}\leq&\|\nabla\varphi\|^{2}_{2},\ n=1,p=6.\label{4.61}
\end{align}
 By substituting (\ref{4.61}) into (\ref{4.38}), we can obtain
\begin{align}
K(\varphi)\geq \frac{1}{2}\left(\frac{p}{n(p-2)\lambda_{1}C_{1}}\right)^{\frac{8}{n(p-2)+2p}}\left(\frac{n}{2}\right)^{\frac{4p-2n(p-2)}{2p+n(p-2)}}>0,\ n=1,p=6.\label{4.64}
\end{align}

If (\ref{4.56}) holds, the process of dealing with this case is similar to that of treating the case where (\ref{4.55}) holds. For this case, we need to employ a different version of Gagliardo-Nirenberg, i.e., (\ref{GN-budengshi1}),
to deal with (\ref{4.56}) by choosing $\theta=\frac{n(q-2)}{2q}$. More precisely, by (\ref{GN-budengshi1}) and Lemma \ref{lem3.2}, we have
\begin{align}
&\|\varphi\|^{q}_{q}\leq C_{3}\|\nabla\varphi\|^{q\theta}_{2}\|\varphi\|^{q(1-\theta)}_{2}\leq C_{3}\|\nabla\varphi\|^{q\theta}_{2}\left(\frac{2}{n}\|\nabla\varphi\|_{2}\||x|\varphi\|_{2}\right)^{\frac{q(1-\theta)}{2}}.\label{4.65}
\end{align}
Combining this fact and (\ref{4.56}), we have
\begin{align}
&\frac{q}{n(q-2)\lambda_{2}C_{3}}\left(\frac{n}{2}\right)^{\frac{q(1-\theta)}{2}}\leq\|\nabla\varphi\|^{q\theta+\frac{q(1-\theta)}{2}}_{2}
\||x|\varphi\|_{2}^{\frac{q(1-\theta)}{2}-2},\label{4.66}
\end{align}
where $\theta=\frac{n(q-2)}{2q}$. By Cauchy-Schwartz inequality, we know (\ref{4.66}) turns to
\begin{align}
&\frac{q}{n(q-2)\lambda_{2}C_{3}}\left(\frac{n}{2}\right)^{\frac{q(1-\theta)}{2}}\leq\frac{1}{2}\|\nabla\varphi\|^{q(\theta+1)}_{2}+\frac{1}{2}\||x|\varphi\|_{2}^{q(1-\theta)-4}.\label{4.67}
\end{align}
The assumptions $n=1$ and $q\geq6$ gives $0<\theta<1$ and $$q(1-\theta)-4\geq0$$ in (\ref{4.67}).
For $q>6$, (\ref{4.67}) implies
\begin{align}
\left(\frac{q}{n(q-2)\lambda_{2}C_{3}}\right)^{\frac{2}{q(\theta+1)}}\left(\frac{n}{2}\right)^{\frac{1-\theta}{1+\theta}}\leq&\|\nabla\varphi\|^{2}_{2},\ n=1,q>6,\label{4.68}
\end{align}
or
\begin{align}
\left(\frac{q}{n(q-2)\lambda_{2}C_{3}}\right)^{\frac{2}{q(1-\theta)-4}}\left(\frac{n}{2}\right)^{\frac{q(1-\theta)}{q(1-\theta)-4}}\leq&\||x|\varphi\|^{2}_{2},\ n=1,q>6.\label{4.69}
\end{align}
By substituting (\ref{4.68}) or (\ref{4.69}) into (\ref{4.38}), we can obtain
\begin{align}
K(\varphi)\geq \frac{1}{2}\left(\frac{q}{n(q-2)\lambda_{2}C_{3}}\right)^{\frac{4}{n(q-2)+2q}}\left(\frac{n}{2}\right)^{\frac{2q-n(q-2)}{2q+n(q-2)}}>0,\ n=1,q>6,\label{4.71}
\end{align}
or
\begin{align}
K(\varphi)\geq & \left(\frac{1}{2}-\frac{2}{n(q-2)}\right)\left(\frac{q}{n(q-2)\lambda_{2}C_{3}}\right)^{\frac{4}{2q-n(q-2)-8}}\left(\frac{n}{2}\right)^{\frac{2q-n(q-2)}{2q-n(q-2)-8}}>0,\ n=1,q>6.\label{4.72}
\end{align}
If $q=6$, according to (\ref{4.66}), we have
\begin{align}
\left(\frac{q}{n(q-2)\lambda_{2}C_{3}}\right)^{\frac{4}{q(\theta+1)}}\left(\frac{n}{2}\right)^{\frac{2(1-\theta)}{1+\theta}}\leq&\|\nabla\varphi\|^{2}_{2},\ n=1,q=6.\label{4.70}
\end{align}
 By substituting (\ref{4.70}) into (\ref{4.38}), we can obtain
\begin{align}
K(\varphi)\geq \frac{1}{2}\left(\frac{q}{n(q-2)\lambda_{2}C_{3}}\right)^{\frac{8}{n(q-2)+2q}}\left(\frac{n}{2}\right)^{\frac{4q-2n(q-2)}{2q+n(q-2)}}>0,\ n=1,q=6.\label{4.73}
\end{align}
According to (\ref{4.62}), (\ref{4.63}), (\ref{4.64}), (\ref{4.71}), (\ref{4.72}) and (\ref{4.73}), we know $d_{_{\mathcal{N}_{13}}}>0$.\\
\textbf{Step \textrm{II}:\ The manifold \bm{$\mathcal{W}_{j}$}, \bm{$j=9,10,\cdots,14$}, is invariant.}

The conclusion $d_{_{\mathcal{N}_{j}}}>0$, $j=9,10,\cdots,14$, ensures that the conditions $0\leq K(\varphi_{0})<d_{_{\mathcal{N}_{j}}}$ make sense. For the non-negative initial energy $0\leq K(\varphi_{0})<d_{_{\mathcal{N}_{j}}}$, $j=9,10,\cdots,14$, we can show the invariance of the manifolds $\mathcal{W}_{j}$ and $\mathcal{V}_{j}$, $j=9,10,\cdots,14$. Here, since the strategy of this proof does not depend on the structures of the cross-constrained variational functionals $S_{j}(\varphi)$, by the similar process of proving Lemma \ref{lem2.4}, we can finish the proofs. Thus, we omit the details of the proofs.\\
\textbf{Step \textrm{III}:\ Global existence.}

 For the initial data $\varphi_{0}$ belonging to $\mathcal{W}_{j}$, $j=9,10,\cdots,14$, we shall prove $T_{\rm{max}}=+\infty$. Here, for the case $\varphi_{0}=0$, we know $\varphi(t)=0$ for $t\in(0,+\infty)$ according to (\ref{2.2}), which means that $\varphi$ is the global trivial solution. Hence, we shall only focus on the case $K(\varphi_{0})<d_{\mathcal{N}_{j}}$ and $S_{j}(\varphi_{0})>0$, i.e., $\varphi_{0}\neq0$, to prove the global existence. The strategies of the proofs depend on the structures of the cross-constrained variational functionals $S_{j}(\varphi)$, $j=9,10,\cdots,14$. According to the structures of $S_{j}(\varphi)$, $j=9,10,\cdots,14$, we shall finish the proof by considering the following two cases: Case \textrm{I} ($j=9,10,11$, here $S_{j}(\varphi)$ have the term $\|\varphi\|^{p}_{p}$ and do not have the term $\|\varphi\|^{q}_{q}$) and Case \textrm{II} ($j=12,13,14$, here $S_{j}(\varphi)$ both have the terms $\|\varphi\|^{p}_{p}$ and $\|\varphi\|^{q}_{q}$).\\
{\bf{Case \textrm{I}: \bm{$j=9,10,11.$}}}

For the cross-constrained variational functionals $S_{j}(\varphi)$ having the term $\|\varphi\|^{p}_{p}$ but not having the term $\|\varphi\|^{q}_{q}$, by applying $S_{j}(\varphi)>0$, we get
\begin{align}
\|\varphi\|^{p}_{p}<\frac{2p}{n(p-2)\lambda_{1}}\left(c\| \varphi\|^{2}_{2}+\||x|\varphi\|^{2}_{2}\right).\label{4.74}
\end{align}
Since the assumptions we adopted here are stronger than Theorem \ref{th4.1}, we can retrieve (\ref{4.5}) for $j=9,10,11$.
By substituting (\ref{4.74}) and $K(\varphi)<d_{\mathcal{N}_{j}}$ into (\ref{4.5}) with $j=9,10,11$, we have
\begin{align*}
d_{\mathcal{N}_{j}}+\left(\frac{\lambda_{2}}{q}+\left(\frac{\lambda_{1}}{p}+\frac{\lambda_{2}}{q}\right)\frac{2p c}{n(p-2)\lambda_{1}}\right)\|\varphi_{0}\|^{2}_{2}>&\frac{1}{2}\| \varphi\|^{2}_{H}-\left(\frac{\lambda_{1}}{p}+\frac{\lambda_{2}}{q}\right)\frac{2p}{n(p-2)\lambda_{1}}\||x|\varphi\|^{2}_{2}\nonumber\\
\geq&\frac{1}{2}\| \varphi\|^{2}_{H}-\left(\frac{\lambda_{1}}{p}+\frac{\lambda_{2}}{q}\right)\frac{2p}{n(p-2)\lambda_{1}}\| \varphi\|^{2}_{H}.
\end{align*}
The rest proof is similar to Theorem \ref{th4.1}.\\
{\bf{Case \textrm{II}: \bm{$j=12,13,14.$}}}

Since the cross-constrained variational functionals $S_{j}(\varphi)$ have both the terms $\|\varphi\|^{p}_{p}$ and $\|\varphi\|^{q}_{q}$, we can apply $S_{j}(\varphi)>0$ to control both the terms $\|\varphi\|^{p}_{p}$ and $\|\varphi\|^{q}_{q}$
as
\begin{align}
\frac{n(p-2)\lambda_{1}}{2p}\|\varphi\|^{p}_{p}+\frac{n(q-2)\lambda_{2}}{2q}\|\varphi\|^{q}_{q}<\| \varphi\|^{2}_{2}+\||x|\varphi\|^{2}_{2}.\label{4.75u}
\end{align}
According to (\ref{4.75u}), we can retrieve (\ref{4.10}) as follows
\begin{align*}
&\frac{\lambda_{1}}{p}\|\varphi\|^{p}_{p}+\frac{\lambda_{2}}{q}\|\varphi\|^{q}_{q}
<\frac{2}{n(q-2)}\left(\| \varphi\|^{2}_{2}+\||x|\varphi\|^{2}_{2}\right)
\leq\frac{2}{n(q-2)}\|\varphi\|^{2}_{H}.
\end{align*}
The rest of the proof is similar to Theorem \ref{th4.1}.
\end{proof}
\begin{remark}\ (Fail to prove the finite time blowup by the cross-constrained variational functionals $S_{j}(\varphi)$, i.e., $j=9,10,\cdots,14$, without the term $\|\nabla\varphi\|^{2}_{2}$)\label{rem4.2}
We analyze the reason that we fail to prove finite time blowup by the cross-constrained variational functionals $S_{j}(\varphi)$, $j=9,10,\cdots,14$. In the investigation of the finite time blowup, we need to use the auxiliary functional $J(t)=\||x|\varphi\|^{2}_{2}$. The key step of proving the finite time blowup of the solution to the Cauchy problem (\ref{1.1}) is to verify $J''(t)<0$, i.e.,
\begin{align}
\|\nabla\varphi\|^2_{2}<\||x|\varphi\|^{2}_{2}+\frac{n(p-2)\lambda_{1}}{2p}\|\varphi\|^{p}_{p}+\frac{n(q-2)\lambda_{2}}{2q}\|\varphi\|^{q}_{q},\label{fil}
\end{align}
which means that we need to know the relationship between the terms $\|\nabla\varphi\|^2_{2}$ and $\||x|\varphi\|^{2}_{2}+\frac{n(p-2)\lambda_{1}}{2p}\|\varphi\|^{p}_{p}+\frac{n(q-2)\lambda_{2}}{2q}\|\varphi\|^{q}_{q}$ in (\ref{fil}).
For the initial data $\varphi_{0}\in\mathcal{V}_{j}$, i.e., $S_{j}(\varphi_{0})<0$, since $S_{j}(\varphi)$ does not have the term $\|\nabla\varphi\|^2_{2}$, we can not obtain this relationship by the condition $S_{j}(\varphi)<0$, which means that we fail to verify $J''(t)<0$.

\end{remark}
For the cross-constrained variational functionals $S_{j}(\varphi)$, $j=15,16,17,18$,
since $S_{j}(\varphi)$ have the term $\|\varphi\|^{q}_{q}$ and do not have the term $\|\varphi\|^{p}_{p}$, we fail to prove $d_{{\mathcal{N}_{j}}}>0$ and we further fail to prove global existence, which will be explained in following Remark \ref{rem4.3}. Next, we give the finite time blowup criteria for $S_{j}(\varphi)$, $j=15,16,17,18$.
\begin{theorem}\label{th4.4}(Finite time blowup criteria for $S_{j}(\varphi)$, $j=15,16,17,18$, without the term $\|\varphi\|^{p}_{p}$)
Let $p$, $q$ satisfy (\ref{1.2}), $\lambda_{1}>0$, and $\lambda_{2}>0$.
Then the corresponding solution $\varphi$ to problem (\ref{1.1}) blows up in finite time provided $K(\varphi_{0})<d_{{\mathcal{N}_{j}}}$ and $S_{j}(\varphi)<0$, $j=15,16,17,18$.

\end{theorem}
\begin{proof}
For the condition $K(\varphi_{0})<d_{{\mathcal{N}_{j}}}$, $j=15,16,17,18$, since we fail to prove $d_{{\mathcal{N}_{j}}}>0$, which will be explained in Remark \ref{rem4.3}, we need to divide the proof into two cases $d_{{\mathcal{N}_{j}}}\leq0$ and $d_{{\mathcal{N}_{j}}}>0$. If $d_{{\mathcal{N}_{j}}}\leq0$, $j=15,16,17,18$, occurs, we notice that the condition $K(\varphi_{0})<d_{{\mathcal{N}_{j}}}$ means $K(\varphi_{0})<0$, which turns to the negative initial energy case. And the finite time blowup can be given by Theorem \ref{th3.1}. For $d_{{\mathcal{N}_{j}}}>0$, $j=15,16,17,18$, by the similar process of proving Lemma \ref{lem2.4}, we can prove the solution $\varphi(t)\in\mathcal{V}_{j}$ for $t\in(0,T_{\rm{max}})$ provided the initial data $\varphi_{0}\in\mathcal{V}_{j}$. Next, we prove $T_{\rm{max}}<+\infty$ for this case. Arguing by contradiction, we suppose $T_{\rm{max}}=+\infty$. Then, we shall apply the condition $S_{j}(\varphi)<0$, i.e., $\varphi\in \mathcal{V}_{j}$, $j=15,16,17,18$, to verify $J''(t)<0$, which is the key step of this proof. Here, the strategies of this process depend on the structures of the cross-constrained variational functionals $S_{j}(\varphi)$, $j=15,16,17,18$. We shall divide the proof into two cases, i.e., Case \textrm{I} ($j=16,18$, here the cross-constrained variational functionals $S_{j}(\varphi)$ have the term $\|\varphi\|^{2}_{2}$) and Case \textrm{II} ($j=15,17$, here $S_{j}(\varphi)$ do not have the term $\|\varphi\|^{2}_{2}$).\\
{\bf{Case \textrm{I}: \bm{$j=16,18.$}}}

For the cases $j=16,18$, the cross-constrained variational functionals $S_{j}(\varphi)$ having the term $\|\varphi\|^{2}_{2}$, by $S_{j}(\varphi)<0$ and $\lambda_{1}>0$, we know
\begin{align*}
\|\nabla\varphi\|^{2}_{2}+\|\varphi\|^{2}_{2}<\frac{n(q-2)}{2q}\lambda_{2}\|\varphi\|^{q}_{q},
\end{align*}
i.e.,
\begin{align}\label{lamuda1fangsuo}
\|\nabla\varphi\|^{2}_{2}+\|\varphi\|^{2}_{2}-\frac{n(q-2)}{2q}\lambda_{2}\|\varphi\|^{q}_{q}<0.
\end{align}
By (\ref{lamuda1fangsuo}), Lemma \ref{lem3.1} and the condition $\lambda_{1}>0$, we have
\begin{align}\label{Jerjiedao}
\begin{split}
J''(t)\leq&8\|\nabla\varphi\|^{2}_{2}-8\||x|\varphi\|^{2}_{2}-\frac{8n(q-2)}{2q}\lambda_{2}\|\varphi\|^{q}_{q}\\
=&8\left(\|\nabla\varphi\|^{2}_{2}+\|\varphi\|^{2}_{2}-\frac{n(q-2)}{2q}\lambda_{2}\|\varphi\|^{q}_{q}\right)-8\||x|\varphi\|^{2}_{2}-8\|\varphi\|^{2}_{2}\\
<&-8\|\varphi\|^{2}_{2},
\end{split}
\end{align}
where $t\in(0,+\infty)$. The rest of the proof is similar to Theorem \ref{th4.2}.\\
{\bf{Case \textrm{II}: \bm{$j=15,17.$}}}

For the cases $j=15,17$, the cross-constrained variational functionals $S_{j}(\varphi)$ not having the term $\|\varphi\|^{2}_{2}$, by $S_{j}(\varphi)<0$ and $\lambda_{1}>0$, we can deduce
\begin{align*}
\|\nabla\varphi\|^{2}_{2}<\frac{n(q-2)}{2q}\lambda_{2}\|\varphi\|^{q}_{q},
\end{align*}
i.e.,
\begin{align}
\|\nabla\varphi\|^{2}_{2}-\frac{n(q-2)}{2q}\lambda_{2}\|\varphi\|^{q}_{q}<0.\label{4.16+}
\end{align}
 Next, according to (\ref{4.16+}), Lemma \ref{lem3.1} and the condition $\lambda_{1}>0$, we have
\begin{align*}
J''(t)\leq&8\left(\|\nabla\varphi\|^{2}_{2}-\frac{n(q-2)}{2q}\lambda_{2}\|\varphi\|^{q}_{q}\right)-8\||x|\varphi\|^{2}_{2}\nonumber\\
<&-8\||x|\varphi\|^{2}_{2}\nonumber\\
=&-8J(t)
\end{align*}
for $t\in[0,+\infty)$. The rest of the proof is similar to Theorem \ref{th4.2}.
\end{proof}
\begin{remark}(Fail to prove the global existence by the cross-constrained variational functionals $S_{j}(\varphi)$, i.e., $j=15,16,17,18$, without the term $\|\varphi\|^{p}_{p}$)\label{rem4.3} For the cross-constrained variational functionals $S_{j}(\varphi)$, $j=15,16,17,18$, due to these cross-constrained variational combinations without the term $\|\varphi\|^{p}_{p}$, we fail to prove the conclusion $d_{{\mathcal{N}_{j}}}>0$, $j=15,16,17,18$, which is required for the invariant manifold used to prove the global existence. Compared to the case of the finite time blowup, where we can apply the condition $K(\varphi_{0})<0$ to finish the proof without using the invariant manifold, in order to prove the global existence, we need the restriction $0\leq K(\varphi_{0})<d_{{\mathcal{N}_{j}}}$, $j=15,16,17,18$, to prove the invariance of the manifold for low initial energy $K(\varphi_{0})<d_{{\mathcal{N}_{j}}}$, which is ensured to make sense by $d_{{\mathcal{N}_{j}}}>0$, $j=15,16,17,18$. Meanwhile, since the cross-constrained variational functionals $S_{j}(\varphi)$, $j=15,16,17,18$, do not have the term $\|\varphi\|^{p}_{p}$, it seems too difficult to estimate the boundness of the solution. According to the two points above, we fail to prove the global existence. In the following analysis, we shall analyze and explain the reasons for failing to obtain $d_{{\mathcal{N}_{j}}}>0$, $j=15,16,17,18$, and the boundedness estimate of the solution.\\
{\bf{(i) Why do we fail to prove \bm{$d_{{\mathcal{N}_{j}}}>0$}, \bm{$j=15,16,17,18$}?}} In order to prove $d_{{\mathcal{N}_{j}}}>0$, $j=15,16,17,18$, we need to control the negative terms in $K(\varphi)$, i.e., $-\frac{\lambda_{1}}{p}\|\varphi\|^{p}_{p}$ and $-\frac{\lambda_{2}}{q}\|\varphi\|^{q}_{q}$, in $\mathcal{N}_{j}$ which also means $S_{j}(\varphi)=0$, $j=15,16,17,18$.
Here, the cross-constrained variational functionals $S_{j}(\varphi)$, $j=15,16,17,18$, have the term $\|\varphi\|^{q}_{q}$ and do not have the term $\|\varphi\|^{p}_{p}$. By the fact $S_{j}(\varphi)=0$, $j=15,16,17,18$, we have
\begin{align}
\|\varphi\|^{q}_{q}\leq\frac{2q}{n(q-2)\lambda_{2}}\| \varphi\|^{2}_{H},\label{kongzhiq}
\end{align}
which means that only the term $\|\varphi\|^{q}_{q}$ is controlled. Thus, we need to deal with the term $\|\varphi\|^{p}_{p}$ additionally, where $p>q$. Compared to the case of the cross-constrained variational functionals $S_{j}(\varphi)$, $j=1,2,3,4$, with the term $\|\varphi\|^{p}_{p}$ and without the term $\|\varphi\|^{q}_{q}$, in which we can control the term $\|\varphi\|^{q}_{q}$ when the term $\|\varphi\|^{p}_{p}$ is controlled by applying $S_{j}(\varphi)=0$, $j=1,2,3,4$, according to the assumption $p>q$, as shown in Lemma \ref{lem2.3}, we do not have the appropriate method to control the term $\|\varphi\|^{p}_{p}$ even the term $\|\varphi\|^{q}_{q}$ is controlled for the case $j=15,16,17,18$ due to $p>q$, so we fail to prove $d_{{\mathcal{N}_{j}}}>0$, $j=15,16,17,18$.\\
{\bf{(ii) Why do we fail to obtain the boundedness estimate of the solution?}} Even if we suppose that we have good enough conditions to obtain the conclusion $d_{{\mathcal{N}_{j}}}>0$, $j=15,16,17,18$, and then prove the invariance of the manifold $\mathcal{W}_{j}$, $j=15,16,17,18$, in order to obtain the boundedness estimate of $\|\varphi\|_{H}$ to finish the proof of the global existence, we need to control the negative terms in $K(\varphi)$, i.e., $-\frac{\lambda_{1}}{p}\|\varphi\|^{p}_{p}$ and $-\frac{\lambda_{2}}{q}\|\varphi\|^{q}_{q}$, in the invariant manifold $\mathcal{W}_{j}$, $j=15,16,17,18$, i.e., $S_{j}(\varphi)>0$, $j=15,16,17,18$. However, for the cross-constrained variational functionals $S_{j}(\varphi)$, $j=15,16,17,18$, having the term $\|\varphi\|^{q}_{q}$ and not having the term $\|\varphi\|^{p}_{p}$, by the condition $S_{j}(\varphi)>0$, $j=15,16,17,18$, we know
\begin{align*}
\|\varphi\|^{q}_{q}<\frac{2q}{n(q-2)\lambda_{2}}\| \varphi\|^{2}_{H},
\end{align*}
which means that only the term $\|\varphi\|^{q}_{q}$ is controlled, and thus we also need to deal with the term $\|\varphi\|^{p}_{p}$ additionally. Therefore, the similar difficulty in proving $d_{{\mathcal{N}_{j}}}>0$, $j=15,16,17,18$, also arises in obtaining the boundedness estimate of $\|\varphi\|_{H}$, hence we fail to prove the global existence by the cross-constrained variational functionals $S_{j}(\varphi)$, $j=15,16,17,18$.
\end{remark}

 \section{Global existence and finite time blowup for high initial energy $K(\varphi_{0})\geq d_{\mathcal{N}_{j}}$, $j=1,2,\cdots,8$}\label{sec5}
In this section, we aim to focus on the high initial energy $K(\varphi_{0})\geq d_{\mathcal{N}_{j}}$, $j=1,2,\cdots,8$, to prove the global existence and finite time blowup of the solution to Cauchy problem (\ref{1.1}). Here, the so-called high initial energy case means the extension from the low initial energy $K(\varphi_{0})<d_{\mathcal{N}_{j}}$, $j=1,2,\cdots,8$, to the case $K(\varphi_{0})\geq d_{\mathcal{N}_{j}}$, $j=1,2,\cdots,8$. The related problem has been continuously attempted in this research field, but has not been successfully solved due to some critical difficulties. More precisely, for the case $K(\varphi_{0})\geq d_{\mathcal{N}_{j}}$, $j=1,2,\cdots,8$, we can not use the strategy of dealing with low initial energy to prove the invariance of the manifolds used to prove the global existence and finite time blowup, which is to use the condition $K(\varphi_{0})<d_{\mathcal{N}_{j}}$ and definition of $d_{\mathcal{N}_{j}}$ to finish the proofs via argument by contradiction. In order to deal with high initial energy, we need to find the additional conditions since the conditions $S_{j}(\varphi_{0})>0$ and $S_{j}(\varphi_{0})<0$ are not enough to prove the global existence and finite time blowup. Thus, a lot of research has focused on the non-negative low initial energy or negative initial energy \allowdisplaybreaks\cite{r4-1-2-00-1,r4-1-2-5,r14,r4-1-2-1,r4-1-2-4-5}. Meanwhile, since we use the cross-constrained variational theory to treat the Cauchy problem (\ref{1.1}), for the constructed cross-constrained problems (\ref{2.5}), we need to deal with eight cases for high initial energy, i.e., $K(\varphi_{0})\geq d_{\mathcal{N}_{j}}$, $j=1,2,\cdots,8$, which makes this job very difficult. In this section, we shall introduce the conditions $K(\varphi_{0})\leq R_{j}(\|\varphi_{0}\|_{2})$ to show the eight classifications of the initial data to distinguish the global existence and finite time blowup for high initial energy $K(\varphi_{0})\geq d_{\mathcal{N}_{j}}$, $j=1,2,\cdots,8$, where $R_{j}(\|\varphi_{0}\|_{2})$, $j=1,2,\cdots,8$, are defined by (\ref{Rf1})-(\ref{Rf8}), respectively. For the cases $j=9,10,\cdots,14$, we give six global existence criteria for $K(\varphi_{0})\geq d_{\mathcal{N}_{j}}$, $j=9,10,\cdots,14$. Due to the lack of the mandatory term $\|\nabla\varphi\|^{2}_{2}$ for the cross-constrained variational combinations $S_{j}(\varphi)$, $j=9,10,\cdots,14$, we fail to prove the finite time blowup for $K(\varphi_{0})\geq d_{\mathcal{N}_{j}}$, $j=9,10,\cdots,14$.

\subsection{Global existence $\&$ finite time blowup for high initial energy $K(\varphi_{0})\geq d_{{\mathcal{N}_{j}}}$, $j=1,2,\cdots,8$}\label{sec5.1}
Next, we focus on high initial energy $K(\varphi_{0})\geq d_{\mathcal{N}_{j}}$, $j=1,2,\cdots,8$, to present eight classifications of the initial data to distinguish the global existence and finite time blowup. More precisely, for $j=1,2,\cdots,8$, we define the space of the initial data as follows\begin{align}
\mathcal{H}_{R_{j}}:=\left\{\varphi\in H\Big|d_{_{\mathcal{N}_{j}}}\leq K(\varphi)\leq R_{j}(\|\varphi\|_{2})\right\},\label{HR}
\end{align}
where
\begin{align}\label{Rf1}
\begin{split}
R_{1}(\|\varphi\|_{2}):=&\left(\frac{1}{2}-\frac{\lambda_{2}}{q}+\frac{n}{2}\left(\frac{1}{2}-\left(\frac{\lambda_{1}}{p}+\frac{\lambda_{2}}{q}\right)\frac{2p}{n(p-2)\lambda_{1}}\right)^{\frac{1}{2}}\right)
\|\varphi\|^{2}_{2}\\
&+\frac{1}{2}\left(\frac{1}{2}-\left(\frac{\lambda_{1}}{p}+\frac{\lambda_{2}}{q}\right)\frac{2p}{n(p-2)\lambda_{1}}\right) M_{1}(\|\varphi\|_{2}),
\end{split}
\end{align}
\begin{align}\label{Rf2}
\begin{split}
R_{2}(\|\varphi\|_{2}):=&\left(\frac{1}{2}-\frac{\lambda_{2}}{q}+\frac{n}{2^{\frac{1}{2}}}\left(\frac{1}{2}-\left(\frac{\lambda_{1}}{p}+\frac{\lambda_{2}}{q}\right)
\frac{2p}{n(p-2)\lambda_{1}}\right)\right)\|\varphi\|^{2}_{2}\\
&+\frac{1}{2}\left(\frac{1}{2}-\left(\frac{\lambda_{1}}{p}+\frac{\lambda_{2}}{q}\right)\frac{2p}{n(p-2)\lambda_{1}}\right)
\\&\times\frac{M_{1}(\|\varphi\|_{2})+M_{2}(\|\varphi\|_{2})
+\left|M_{1}(\|\varphi\|_{2})-M_{2}(\|\varphi\|_{2})\right|}{2},
\end{split}
\end{align}
\begin{align}\label{Rf3}
\begin{split}
R_{3}(\|\varphi\|_{2}):=&\left(\frac{1}{2}-\left(\frac{\lambda_{1}}{p}+\frac{\lambda_{2}}{q}\right)\frac{2p c}{n(p-2)\lambda_{1}}-\frac{\lambda_{2}}{q}+\frac{n}{2}\left(\frac{1}{2}-\left(\frac{\lambda_{1}}{p}+\frac{\lambda_{2}}{q}\right)\frac{2p}{n(p-2)\lambda_{1}}\right)^{\frac{1}{2}}
\right)\|\varphi\|^{2}_{2}\\
&+\frac{1}{2}\left(\frac{1}{2}-\left(\frac{\lambda_{1}}{p}+\frac{\lambda_{2}}{q}\right)\frac{2p}{n(p-2)\lambda_{1}}\right)
\\&\times\frac{M_{1}(\|\varphi\|_{2})+M_{3}(\|\varphi\|_{2})
+\left|M_{1}(\|\varphi\|_{2})-M_{3}(\|\varphi\|_{2})\right|}{2},
\end{split}
\end{align}
\begin{align}\label{Rf4}
\begin{split}
R_{4}(\|\varphi\|_{2}):=&\left(\frac{1}{2}-\left(\frac{\lambda_{1}}{p}+\frac{\lambda_{2}}{q}\right)\frac{2p c}{n(p-2)\lambda_{1}}-\frac{\lambda_{2}}{q}+\frac{n}{2^{\frac{1}{2}}}\left(\frac{1}{2}-\left(\frac{\lambda_{1}}{p}+\frac{\lambda_{2}}{q}\right)\frac{2p}{n(p-2)
\lambda_{1}}\right)\right)\|\varphi\|^{2}_{2}\\
&+
\frac{1}{2}\left(\frac{1}{2}-\left(\frac{\lambda_{1}}{p}+\frac{\lambda_{2}}{q}\right)\frac{2p}{n(p-2)\lambda_{1}}\right)
\\&\times\frac{M_{1}(\|\varphi\|_{2})+M_{4}(\|\varphi\|_{2})
+\left|M_{1}(\|\varphi\|_{2})-M_{4}(\|\varphi\|_{2})\right|}{2},
\end{split}
\end{align}
\begin{align}\label{Rf5}
\begin{split}
R_{5}(\|\varphi\|_{2}):=&\frac{1}{2}\left(\frac{1}{2}-\frac{2}{n(q-2)}\right)N_{1}(\|\varphi\|_{2})+\left(\frac{1}{2}+\frac{n}{2}\left(\frac{1}{2}-\frac{2}{n(q-2)}\right)^{\frac{1}{2}}\right)\|\varphi\|^{2}_{2},
\end{split}\end{align}
\begin{align}\label{Rf6}
\begin{split}
R_{6}(\|\varphi\|_{2}):=&\frac{1}{2}\left(\frac{1}{2}-\frac{2}{n(q-2)}\right)\frac{N_{1}(\|\varphi\|_{2})+N_{2}(\|\varphi\|_{2})+\left|N_{1}(\|\varphi\|_{2})-
N_{2}(\|\varphi\|_{2})\right|}{2}\\
&+\left(\frac{1}{2}+\frac{n}{2^{\frac{1}{2}}}\left(\frac{1}{2}-\frac{2}{n(q-2)}\right)\right)\|\varphi\|^{2}_{2},
\end{split}
\end{align}
\begin{align}\label{Rf7}
\begin{split}
R_{7}(\|\varphi\|_{2}):=&\frac{1}{2}\left(\frac{1}{2}-\frac{2}{n(q-2)}\right)\frac{N_{1}(\|\varphi\|_{2})+N_{3}(\|\varphi\|_{2})+\left|N_{1}(\|\varphi\|_{2})-N_{3}(\|\varphi\|_{2})\right|}{2}\\
&+\left(\frac{1}{2}-\frac{2}{n(q-2)}+\frac{n}{2}\left(\frac{1}{2}-\frac{2}{n(q-2)}\right)^{\frac{1}{2}}\right)\|\varphi\|^{2}_{2},
\end{split}
\end{align}
\begin{align}\label{Rf8}
\begin{split}
R_{8}(\|\varphi\|_{2}):=&\frac{1}{2}\left(\frac{1}{2}-\frac{2}{n(q-2)}\right)\frac{N_{1}(\|\varphi\|_{2})+N_{4}(\|\varphi\|_{2})
+\left|N_{1}(\|\varphi\|_{2})-N_{4}(\|\varphi\|_{2})\right|}{2}\\
&
+\left(\frac{1}{2}-\frac{2}{n(q-2)}+\frac{n}{2^{\frac{1}{2}}}\left(\frac{1}{2}-\frac{2}{n(q-2)}\right)\right)\|\varphi\|^{2}_{2},
\end{split}
\end{align}
\begin{align}\label{M1}
M_{1}(\|\varphi\|_{2}):=\left(\frac{2p}{n(p-2) \lambda_{1}C_{1}\|\varphi\|^{p-\frac{n(p-2)}{2}}_{2}}\right)^{\frac{4}{n(p-2)-4}},
\end{align}
\begin{align}\label{M2}
M_{2}(\|\varphi\|_{2}):=\left(\frac{2p}{(p-2)\lambda_{1} C_{1}\|\varphi\|^{p-\frac{n(p-2)}{2}-2}_{2}}\right)^{\frac{4}{n(p-2)}},
\end{align}
\begin{align}\label{M3}
M_{3}(\|\varphi\|_{2}):=\left(\frac{4p c^{\frac{1}{2}}}{n(p-2) \lambda_{1}C_{1}\|\varphi\|^{p-\frac{n(p-2)+2}{2}}_{2}}\right)^{\frac{4}{n(p-2)-2}},
\end{align}
\begin{align}\label{M4}
M_{4}(\|\varphi\|_{2}):=\left(\frac{2p(n+c)}{n(p-2)\lambda_{1} C_{1}\|\varphi\|^{p-\frac{n(p-2)}{2}-2}_{2}}\right)^{\frac{4}{n(p-2)}},
\end{align}
\begin{align}\label{N1}
N_{1}(\|\varphi\|_{2}):=\frac{A_{1}(\|\varphi\|_{2})+B_{1}(\|\varphi\|_{2})-\left|A_{1}(\|\varphi\|_{2})-B_{1}(\|\varphi\|_{2})\right|}{2},
\end{align}
\begin{align}\label{N2}
N_{2}(\|\varphi\|_{2}):=\frac{A_{2}(\|\varphi\|_{2})+B_{2}(\|\varphi\|_{2})-\left|A_{2}(\|\varphi\|_{2})-B_{2}(\|\varphi\|_{2})\right|}{2}
,
\end{align}
\begin{align}\label{N3}
N_{3}(\|\varphi\|_{2}):=\frac{A_{3}(\|\varphi\|_{2})+B_{3}(\|\varphi\|_{2})-\left|A_{3}(\|\varphi\|_{2})-B_{3}(\|\varphi\|_{2})\right|}{2},
\end{align}
\begin{align}\label{N4}
N_{4}(\|\varphi\|_{2}):=\frac{A_{4}(\|\varphi\|_{2})+B_{4}(\|\varphi\|_{2})-\left|A_{4}(\|\varphi\|_{2})-B_{4}(\|\varphi\|_{2})\right|}{2},
\end{align}
\begin{align}\label{A1}
A_{1}(\|\varphi\|_{2}):=\left(\frac{p}{n(p-2)C_{1}\lambda_{1}\|\varphi\|^{p-\frac{n(p-2)}{2}}_{2}}\right)^{\frac{4}{n(p-2)-4}},
\end{align}
\begin{align}\label{B1}
B_{1}(\|\varphi\|_{2}):=\left(\frac{q}{n(q-2)C_{3}\lambda_{2}\|\varphi\|^{q-\frac{n(q-2)}{2}}_{2}}\right)^{\frac{4}{n(q-2)-4}}.
\end{align}
\begin{align}\label{A2}
A_{2}(\|\varphi\|_{2}):=\left(\frac{p}{(p-2)C_{1}\lambda_{1}\|\varphi\|^{p-\frac{n(p-2)}{2}-2}_{2}}\right)^{\frac{4}{n(p-2)}}
\end{align}
\begin{align}\label{B2}
B_{2}(\|\varphi\|_{2}):=\left(\frac{q}{(q-2)C_{3}\lambda_{2}\|\varphi\|^{q-\frac{n(q-2)}{2}-2}_{2}}\right)^{\frac{4}{n(q-2)}}.
\end{align}

\begin{align}\label{A3}
A_{3}(\|\varphi\|_{2}):=\left(\frac{2p}{n(p-2)\lambda_{1}C_{1}\|\varphi\|^{p-\frac{n(p-2)}{2}-1}_{2}}\right)^{\frac{4}{n(p-2)-2}},
\end{align}
\begin{align}\label{B3}
B_{3}(\|\varphi\|_{2}):=\left(\frac{2q}{n(q-2)\lambda_{2}C_{3}\|\varphi\|^{q-\frac{n(q-2)}{2}-1}_{2}}\right)^{\frac{4}{n(q-2)-2}},
\end{align}

\begin{align}\label{A4}
A_{4}(\|\varphi\|_{2}):=\left(\frac{p(n+1)}{n(p-2)C_{1}\lambda_{1}\|\varphi\|^{p-\frac{n(p-2)}{2}-2}_{2}}\right)^{\frac{4}{n(p-2)}},
\end{align}
\begin{align}\label{B4}
B_{4}(\|\varphi\|_{2}):=\left(\frac{q(n+1)}{n(q-2)C_{3}\lambda_{2}\|\varphi\|^{q-\frac{n(q-2)}{2}-2}_{2}}\right)^{\frac{4}{n(q-2)}},
\end{align} and $c>0$ is chosen such that (\ref{cdetiaojian}) holds.

We divide $\mathcal{H}_{R_{j}}$ into
\begin{align}
\mathcal{W}_{R_{j}}:=\left\{\varphi\in H\Big|d_{_{\mathcal{N}_{j}}}\leq K(\varphi)\leq R_{j}(\|\varphi\|_{2}),S_{j}(\varphi)>0\right\},\label{WR}
\end{align}
and
\begin{align}
\mathcal{V}_{R_{j}}:=\left\{\varphi\in H\Big|d_{_{\mathcal{N}_{j}}}\leq K(\varphi)\leq R_{j}(\|\varphi\|_{2}),S_{j}(\varphi)<0\right\},\label{VR}
\end{align}
that is, $\mathcal{H}_{R_{j}}=\mathcal{W}_{R_{j}}\cup\mathcal{V}_{R_{j}}$, $j=1,2,\cdots,8$.
We shall use the manifolds $\mathcal{W}_{R_{j}}$ and $\mathcal{V}_{R_{j}}$, $j=1,2,\cdots,8$, to prove the global existence and finite time blowup, respectively. We first need to prove the invariance of the manifolds $\mathcal{W}_{R_{j}}$ and $\mathcal{V}_{R_{j}}$, $j=1,2,\cdots,8$.
Here, we add the condition $ K(\varphi_{0})\leq R_{j}(\|\varphi_{0}\|_{2})$, $j=1,2,\cdots,8$, when we consider the case of $K(\varphi_{0})\geq d_{_{\mathcal{N}_{j}}}$, $j=1,2,\cdots,8$. We shall explain the necessity of this condition in Remark \ref{rem5.2}.
And for high initial energy case $K(\varphi_{0})\geq d_{{\mathcal{N}_{j}}}$, $j=1,2,\cdots,8$, we also need the conclusion $d_{{\mathcal{N}_{j}}}>0$, $j=1,2,\cdots,8$, since negative initial energy $K(\varphi_{0})<0$ and non-negative low initial energy $0\leq K(\varphi_{0})<d_{{\mathcal{N}_{j}}}$, $j=1,2,\cdots,8$, are discussed in Section \ref{sec3} and Section \ref{sec4}, respectively. For the cases $j=1,2,3,4$, we require the additional restrictions (\ref{2.9}) and (\ref{2.10}), which are not required for $j=5,6,7,8$, as we analyze in Remark \ref{rem2.1}. Thus, we naturally add the additional restrictions (\ref{2.9}) and (\ref{2.10}) when we prove the global existence and finite time blowup for $K(\varphi_{0})\geq d_{{\mathcal{N}_{j}}}$, $j=1,2,3,4$.

\begin{lemma}\label{lem5.1}
(The invariance of the manifolds $\mathcal{W}_{R_{j}}$ and $\mathcal{V}_{R_{j}}$ for high initial energy $K(\varphi_{0})\geq d_{\mathcal{N}_{j}}$, $j=1,2,\cdots,8$)
For $j=1,2,\cdots,8$, let $p$ and $q$ satisfy (\ref{1.2}), and $\lambda_{1}>0$, $\lambda_{2}>0$ satisfy the following conditions

(i)\ for $j=1,2,3,4$: $\lambda_{1}$ and $\lambda_{2}$ additionally satisfy (\ref{2.9}) and (\ref{2.10});

(ii)\ for $j=5,6,7,8$: no additional restrictions on $\lambda_{1}$ and $\lambda_{2}$.\\
Then, for $\varphi_{0}\in \mathcal{W}_{R_{j}}$, the corresponding solution $\varphi$ of (\ref{1.1}) belongs to $\mathcal{W}_{R_{j}}$. And for $\varphi_{0}\in \mathcal{V}_{R_{j}}$, the corresponding solution $\varphi$ of (\ref{1.1}) belongs to $\mathcal{V}_{R_{j}}$.

\begin{proof}
 We only write down the details of proving the invariance of the manifold $\mathcal{W}_{R_{j}}$, $j=1,2,\cdots,8$, since the proof of proving the invariance of the manifold $\mathcal{V}_{R_{j}}$, $j=1,2,\cdots,8$, is similar.

For the initial data $\varphi_{0}\in\mathcal{W}_{R_{j}}$, $j=1,2,\cdots,8$, we know that there exists a unique solution $\varphi\in C([0,T_{\rm{max}}),H)$ to the Cauchy problem (\ref{1.1}) according to Lemma \ref{lem2.1}, which means the continuity of $S_{j}(\varphi(t))$, $j=1,2,\cdots,8$, with respect to time $t$. For $\varphi_{0}\in\mathcal{W}_{R_{j}}$, $j=1,2,\cdots,8$, we know
\begin{align}
d_{_{\mathcal{N}_{j}}}\leq K(\varphi_{0})\leq R_{j}(\|\varphi_{0}\|_{2}),\ j=1,2,\cdots,8,\label{K0--detiaoj1}
\end{align}
and $S_{j}(\varphi_{0})>0$, $j=1,2,\cdots,8$. According to (\ref{2.2}) and
\begin{align}
K(\varphi(t))=K(\varphi_{0})\label{K0shouheng}
\end{align}
indicated by (\ref{2.2}) and (\ref{2.3}) for all $t\in(0,T_{\rm{max}})$, we know
\begin{align}
d_{_{\mathcal{N}_{j}}}\leq K(\varphi(t))<R_{j}(\|\varphi(t)\|_{2}),\ j=1,2,\cdots,8.\label{Kfaitdayud}
\end{align}
Next, we only need to prove $S_{j}(\varphi(t))>0$, $j=1,2,\cdots,8$, for $t\in(0,T_{\rm{max}})$, which means $\varphi(t)\in\mathcal{W}_{R_{j}}$, $j=1,2,\cdots,8$, for all $t\in(0,T_{\rm{max}})$.
Here, the corresponding strategies depend on the structures of the cross-constrained variational functionals $S_{j}(\varphi)$, $j=1,2,\cdots,8$. We need to use the eight different strategies to deal with these eight cases $j=1,2,\cdots,8$. Thus, we treat these eight cases one by one.\\
{\bf{Case \textrm{I}: \bm{$j=1.$}}}

We prove $S_{1}(\varphi(t))>0$ for $t\in(0,T_{\rm{max}})$. Arguing by contradiction, we suppose that there exists the first time $\tau_{1}\in(0,T_{\rm{max}})$ such that $S_{1}(\varphi(\tau_{1}))=0$ and $S_{1}(\varphi(t))>0$ for $t\in(0,\tau_{1})$, and we know $\varphi(\tau_{1})\neq 0$ due to (\ref{2.2}) and $\varphi_{0}\neq0$. Next, we estimate $K(\varphi(\tau_{1}))$.
According to $S_{1}(\varphi(\tau_{1}))=0$, we have
\begin{align}
\|\varphi(\tau_{1})\|^{p}_{p}&=\frac{2p}{n(p-2)\lambda_{1}}\|\nabla \varphi(\tau_{1})\|^{2}_{2}.\label{jdengyu1}
\end{align}
By substituting (\ref{jdengyu1}) into (\ref{2.16}), we have
\begin{align}\label{2.17-pjwei1}
\begin{split}
K(\varphi(\tau_{1}))>&\left(\frac{1}{2}-\left(\frac{\lambda_{1}}{p}+\frac{\lambda_{2}}{q}\right)\frac{2p}{n(p-2)\lambda_{1}}\right)\|\nabla \varphi(\tau_{1})\|^{2}_{2}+\frac{1}{2}\||x|\varphi(\tau_{1})\|^{2}_{2}+\left(\frac{1}{2}-\frac{\lambda_{2}}{q}\right)\|\varphi(\tau_{1})\|^{2}_{2}\\
=&\frac{1}{2}\left(\frac{1}{2}-\left(\frac{\lambda_{1}}{p}+\frac{\lambda_{2}}{q}\right)\frac{2p}{n(p-2)\lambda_{1}}\right)\|\nabla \varphi(\tau_{1})\|^{2}_{2}+\frac{1}{2}\||x|\varphi(\tau_{1})\|^{2}_{2}+\left(\frac{1}{2}-\frac{\lambda_{2}}{q}\right)\|\varphi(\tau_{1})\|^{2}_{2}\\
&+\frac{1}{2}\left(\frac{1}{2}-\left(\frac{\lambda_{1}}{p}+\frac{\lambda_{2}}{q}\right)\frac{2p}{n(p-2)\lambda_{1}}\right)\|\nabla \varphi(\tau_{1})\|^{2}_{2}.
\end{split}
\end{align}
In (\ref{2.17-pjwei1}), the assumptions (\ref{2.9}) and (\ref{2.10}) ensure $$\frac{1}{2}-\left(\frac{\lambda_{1}}{p}+\frac{\lambda_{2}}{q}\right)\frac{2p}{n(p-2)\lambda_{1}}>0$$ and $$\frac{1}{2}-\frac{\lambda_{2}}{q}>0,$$respectively. By Cauchy-Schwartz inequality and Lemma \ref{lem3.2}, we know that (\ref{2.17-pjwei1}) turns to
\begin{align}\label{2.17-pjwei1-1}
\begin{split}
K(\varphi(\tau_{1}))>&\left(\frac{1}{2}-\left(\frac{\lambda_{1}}{p}+\frac{\lambda_{2}}{q}\right)\frac{2p}{n(p-2)\lambda_{1}}\right)^{\frac{1}{2}}
\|\nabla \varphi(\tau_{1})\|_{2}\||x|\varphi(\tau_{1})\|_{2}\\
&+\frac{1}{2}\left(\frac{1}{2}-\left(\frac{\lambda_{1}}{p}+\frac{\lambda_{2}}{q}\right)\frac{2p}{n(p-2)\lambda_{1}}\right)\|\nabla \varphi(\tau_{1})\|^{2}_{2}+\left(\frac{1}{2}-\frac{\lambda_{2}}{q}\right)\|\varphi(\tau_{1})\|^{2}_{2}\\
\geq&\frac{n}{2}\left(\frac{1}{2}-\left(\frac{\lambda_{1}}{p}+\frac{\lambda_{2}}{q}\right)\frac{2p}{n(p-2)\lambda_{1}}\right)^{\frac{1}{2}}
\|\varphi(\tau_{1})\|^{2}_{2}\\
&+\frac{1}{2}\left(\frac{1}{2}-\left(\frac{\lambda_{1}}{p}+\frac{\lambda_{2}}{q}\right)\frac{2p}{n(p-2)\lambda_{1}}\right)\|\nabla \varphi(\tau_{1})\|^{2}_{2}+\left(\frac{1}{2}-\frac{\lambda_{2}}{q}\right)\|\varphi(\tau_{1})\|^{2}_{2}.
\end{split}
\end{align}
 Then, we estimate the term $\|\nabla \varphi(\tau_{1})\|^{2}_{2}$ in (\ref{2.17-pjwei1-1}). According to $S_{1}(\varphi(\tau_{1}))=0$, i.e.,
(\ref{jdengyu1}), and Gagliardo-Nirenberg inequality (\ref{GN-budengshi}), we know
\begin{align}
\frac{2p}{n(p-2)\lambda_{1} C_{1}}\leq\|\nabla\varphi(\tau_{1})\|^{\frac{n(p-2)-4}{2}}_{2}\|\varphi(\tau_{1})\|^{p-\frac{n(p-2)}{2}}_{2}.\label{Gnbudengshi1}
\end{align}
In (\ref{Gnbudengshi1}), we know $\frac{n(p-2)-4}{2}>0$ and $p-\frac{n(p-2)}{2}>0$ according to (\ref{1.2}), i.e., $2+\frac{4}{n}<p$ and $p<\frac{2n}{n-2}$. Due to (\ref{2.2}), we know (\ref{Gnbudengshi1}) turns to
\begin{align*}
\frac{2p}{n(p-2)\lambda_{1}C_{1}}\leq\|\nabla\varphi(\tau_{1})\|^{\frac{n(p-2)-4}{2}}_{2}\|\varphi_{0}\|^{p-\frac{n(p-2)}{2}}_{2},
\end{align*}
i.e.,
\begin{align}
\|\nabla\varphi(\tau_{1})\|^{2}_{2}\geq M_{1}(\|\varphi_{0}\|_{2}),\label{Gnbudengshi1-1}
\end{align}where $M_{1}(\|\varphi_{0}\|_{2})$ is defined by (\ref{M1}).
Due to (\ref{Gnbudengshi1-1}), (\ref{2.2}) and (\ref{2.3}), we know that (\ref{2.17-pjwei1-1}) turns into
\begin{align*}
K(\varphi_{0})
>&\left(\frac{1}{2}-\frac{\lambda_{2}}{q}+\frac{n}{2}\left(\frac{1}{2}-\left(\frac{\lambda_{1}}{p}+\frac{\lambda_{2}}{q}\right)\frac{2p}{n(p-2)\lambda_{1}}\right)^{\frac{1}{2}}\right)
\|\varphi_{0}\|^{2}_{2}\nonumber\\
&+\frac{1}{2}\left(\frac{1}{2}-\left(\frac{\lambda_{1}}{p}+\frac{\lambda_{2}}{q}\right)\frac{2p}{n(p-2)\lambda_{1}}\right) M_{1}(\|\varphi_{0}\|_{2}),
\end{align*}
which contradicts (\ref{K0--detiaoj1}) for $j=1$.
Thus, we have $S_{1}(\varphi(t))>0$ for $t\in(0,T_{\rm{max}})$.\\
{\bf{Case \textrm{II}: \bm{$j=2.$}}}

We prove $S_{2}(\varphi(t))>0$ for $t\in(0,T_{\rm{max}})$. Arguing by contradiction, we suppose that there exists the first time $\tau_{2}\in(0,T_{\rm{max}})$ such that $S_{2}(\varphi(\tau_{2}))=0$ and $S_{2}(\varphi(t))>0$ for $t\in(0,\tau_{2})$, and we know $\varphi(\tau_{2})\neq 0$ due to (\ref{2.2}) and $\varphi_{0}\neq0$. Next, we estimate $K(\varphi(\tau_{2}))$.
According to $S_{2}(\varphi(\tau_{2}))=0$, we have
\begin{align}
\|\varphi(\tau_{2})\|^{p}_{p}=\frac{2p}{n(p-2)\lambda_{1}}\left(\|\nabla \varphi(\tau_{2})\|^{2}_{2}+\||x|\varphi(\tau_{2})\|^{2}_{2}\right).\label{2jdengyu2-2}
\end{align}
By substituting (\ref{2jdengyu2-2}) into (\ref{2.16}), we have
\begin{align}\label{2.17-jdengyu2}
\begin{split}
K(\varphi(\tau_{2}))>&\left(\frac{1}{2}-\left(\frac{\lambda_{1}}{p}+\frac{\lambda_{2}}{q}\right)\frac{2p}{n(p-2)\lambda_{1}}\right)\left(\|\nabla \varphi(\tau_{2})\|^{2}_{2}+\||x|\varphi(\tau_{2})\|^{2}_{2}\right)+\left(\frac{1}{2}-\frac{\lambda_{2}}{q}\right)\|\varphi(\tau_{2})\|^{2}_{2}\\
=&\frac{1}{2}\left(\frac{1}{2}-\left(\frac{\lambda_{1}}{p}+\frac{\lambda_{2}}{q}\right)\frac{2p}{n(p-2)\lambda_{1}}\right)\|\nabla \varphi(\tau_{2})\|^{2}_{2}\\
&+\left(\frac{1}{2}-\left(\frac{\lambda_{1}}{p}+\frac{\lambda_{2}}{q}\right)\frac{2p}{n(p-2)\lambda_{1}}\right)\||x|\varphi(\tau_{2})\|^{2}_{2}\\
&+\frac{1}{2}\left(\frac{1}{2}-\left(\frac{\lambda_{1}}{p}+\frac{\lambda_{2}}{q}\right)\frac{2p}{n(p-2)\lambda_{1}}\right)\|\nabla \varphi(\tau_{2})\|^{2}_{2}+\left(\frac{1}{2}-\frac{\lambda_{2}}{q}\right)\|\varphi(\tau_{2})\|^{2}_{2}.
\end{split}
\end{align}
In (\ref{2.17-jdengyu2}), the assumptions (\ref{2.9}) and (\ref{2.10}) ensure $$\frac{1}{2}-\left(\frac{\lambda_{1}}{p}+\frac{\lambda_{2}}{q}\right)\frac{2p}{n(p-2)\lambda_{1}}>0$$ and $$\frac{1}{2}-\frac{\lambda_{2}}{q}>0,$$respectively. By Cauchy-Schwartz inequality and Lemma \ref{lem3.2}, we know that (\ref{2.17-jdengyu2}) turns to
\begin{align}\label{2.17-jdengyu2-2}
\begin{split}
K(\varphi(\tau_{2}))>&\frac{2}{2^{\frac{1}{2}}}\left(\frac{1}{2}-\left(\frac{\lambda_{1}}{p}+\frac{\lambda_{2}}{q}\right)
\frac{2p}{n(p-2)\lambda_{1}}\right)\|\nabla \varphi(\tau_{2})\|_{2}\||x|\varphi(\tau_{2})\|_{2}\\
&+\frac{1}{2}\left(\frac{1}{2}-\left(\frac{\lambda_{1}}{p}+\frac{\lambda_{2}}{q}\right)\frac{2p}{n(p-2)\lambda_{1}}\right)\|\nabla \varphi(\tau_{2})\|^{2}_{2}+\left(\frac{1}{2}-\frac{\lambda_{2}}{q}\right)\|\varphi(\tau_{2})\|^{2}_{2}\\
\geq&\frac{n}{2^{\frac{1}{2}}}\left(\frac{1}{2}-\left(\frac{\lambda_{1}}{p}+\frac{\lambda_{2}}{q}\right)\frac{2p}{n(p-2)\lambda_{1}}\right)\| \varphi(\tau_{2})\|_{2}^{2}\\
&+\frac{1}{2}\left(\frac{1}{2}-\left(\frac{\lambda_{1}}{p}+\frac{\lambda_{2}}{q}\right)\frac{2p}{n(p-2)\lambda_{1}}\right)\|\nabla \varphi(\tau_{2})\|^{2}_{2}+\left(\frac{1}{2}-\frac{\lambda_{2}}{q}\right)\|\varphi(\tau_{2})\|^{2}_{2}.
\end{split}
\end{align}
Meanwhile, from Cauchy-Schwartz inequality and Lemma \ref{lem3.2}, we also have
\begin{align}
\|\nabla \varphi(\tau_{2})\|^{2}_{2}+\||x|\varphi(\tau_{2})\|^{2}_{2}
\geq 2\|\nabla \varphi(\tau_{2})\|_{2}\||x|\varphi(\tau_{2})\|_{2}
\geq n\|\varphi(\tau_{2})\|^{2}_{2}.\label{erfanshude1-2}
\end{align}
Combining (\ref{erfanshude1-2}) and (\ref{2jdengyu2-2}), we have
\begin{align}
&\|\varphi(\tau_{2})\|^{p}_{p}\geq \frac{2p}{(p-2)\lambda_{1}}\|\varphi(\tau_{2})\|^{2}_{2}.\label{erfanshude1-1}
\end{align}
Then, according to (\ref{erfanshude1-1}) and Gagliardo-Nirenberg inequality (\ref{GN-budengshi}), we know
\begin{align}
\frac{2p}{(p-2)\lambda_{1} C_{1}}\leq\|\nabla\varphi(\tau_{2})\|^{\frac{n(p-2)}{2}}_{2}\|\varphi(\tau_{2})\|^{p-\frac{n(p-2)}{2}-2}_{2}.\label{2Gnbudengshi-1-1}
\end{align}
Due to (\ref{2.2}), we know (\ref{2Gnbudengshi-1-1}) turns to
\begin{align}
\frac{2p}{(p-2)\lambda_{1}C_{1}}\leq\|\nabla\varphi(\tau_{2})\|^{\frac{n(p-2)}{2}}_{2}\|\varphi_{0}\|^{p-\frac{n(p-2)}{2}-2}_{2},\label{2Gnbudengshi2}
\end{align}
i.e.,
\begin{align}
\|\nabla\varphi(\tau_{2})\|^{2}_{2}\geq M_{2}(\|\varphi_{0}\|_{2}),\label{2Gnbudengshi2-1-2}
\end{align}where $M_{2}(\|\varphi_{0}\|_{2})$ is defined by (\ref{M2}).
Meanwhile, (\ref{2jdengyu2-2}) gives
\begin{align}\label{shizi12}
\|\varphi(\tau_{2})\|^{p}_{p}>\frac{2p}{n(p-2)\lambda_{1}}\|\nabla \varphi(\tau_{2})\|^{2}_{2}.
\end{align}
By the similar process of giving (\ref{Gnbudengshi1-1}), we can use (\ref{shizi12}) to obtain
\begin{align}
\|\nabla\varphi(\tau_{2})\|^{2}_{2}\geq M_{1}(\|\varphi_{0}\|_{2}),\label{Gnbudengshi1-1-2}
\end{align}where $M_{1}(\|\varphi_{0}\|_{2})$ is defined by (\ref{M1}).
From (\ref{2Gnbudengshi2-1-2}) and (\ref{Gnbudengshi1-1-2}), we have
\begin{align*}
\|\nabla\varphi(\tau_{2})\|^{2}_{2}\geq \frac{M_{1}(\|\varphi_{0}\|_{2})+M_{2}(\|\varphi_{0}\|_{2})
+\left|M_{1}(\|\varphi_{0}\|_{2})-M_{2}(\|\varphi_{0}\|_{2})\right|}{2},
\end{align*}
which means (\ref{2.17-jdengyu2-2}) becomes
\begin{align}\label{zhunbeimaodun-2+2-1}
\begin{split}
K(\varphi(\tau_{2}))>&\left(\frac{1}{2}-\frac{\lambda_{2}}{q}+\frac{n}{2^{\frac{1}{2}}}\left(\frac{1}{2}-\left(\frac{\lambda_{1}}{p}+\frac{\lambda_{2}}{q}\right)
\frac{2p}{n(p-2)\lambda_{1}}\right)\right)\|\varphi(\tau_{2})\|^{2}_{2}\\
&+\frac{1}{2}\left(\frac{1}{2}-\left(\frac{\lambda_{1}}{p}+\frac{\lambda_{2}}{q}\right)\frac{2p}{n(p-2)\lambda_{1}}\right)
\\&\times\frac{M_{1}(\|\varphi_{0}\|_{2})
+M_{2}(\|\varphi_{0}\|_{2})
+\left|M_{1}(\|\varphi_{0}\|_{2})-M_{2}(\|\varphi_{0}\|_{2})\right|}{2}.
\end{split}
\end{align}
According to (\ref{2.2}) and (\ref{2.3}), we know (\ref{zhunbeimaodun-2+2-1}) turns to
\begin{align*}
K(\varphi_{0})>&\left(\frac{1}{2}-\frac{\lambda_{2}}{q}+\frac{n}{2^{\frac{1}{2}}}\left(\frac{1}{2}-\left(\frac{\lambda_{1}}{p}+\frac{\lambda_{2}}{q}\right)
\frac{2p}{n(p-2)\lambda_{1}}\right)\right)\|\varphi_{0}\|^{2}_{2}\nonumber\\
&+\frac{1}{2}\left(\frac{1}{2}-\left(\frac{\lambda_{1}}{p}+\frac{\lambda_{2}}{q}\right)\frac{2p}{n(p-2)\lambda_{1}}\right)
\\&\times\frac{M_{1}(\|\varphi_{0}\|_{2})+M_{2}(\|\varphi_{0}\|_{2})
+\left|M_{1}(\|\varphi_{0}\|_{2})-M_{2}(\|\varphi_{0}\|_{2})\right|}{2},
\end{align*}
which contradicts (\ref{K0--detiaoj1}) for $j=2$. Thus, we have $S_{2}(\varphi(t))>0$ for $t\in(0,T_{\rm{max}})$.\\
{\bf{Case \textrm{III}: \bm{$j=3.$}}}

We prove $S_{3}(\varphi(t))>0$ for all $t\in(0,T_{\rm{max}})$. Arguing by contradiction, we suppose that there exists the first time $\tau_{3}\in(0,T_{\rm{max}})$ to make $S_{3}(\varphi(\tau_{3}))=0$, i.e.,
\begin{align}
\|\varphi(\tau_{3})\|^{p}_{p}=&\frac{2p}{n(p-2)\lambda_{1}}\left(\|\nabla \varphi(\tau_{3})\|^{2}_{2}+c\|\varphi(\tau_{3})\|^{2}_{2}\right),\label{p-gujij3}
\end{align}and $S_{3}(\varphi(t))>0$ for $t\in(0,\tau_{3})$,
and we know $\varphi(\tau_{3})\neq 0$ due to (\ref{2.2}) and $\varphi_{0}\neq0$.
Next, we estimate $K(\varphi(\tau_{3}))$.
By substituting (\ref{p-gujij3}) into (\ref{2.16}), we have
\begin{align}\label{p-3zhunbeimaodun+1}
\begin{split}
K(\varphi(\tau_{3}))>&\left(\frac{1}{2}-\left(\frac{\lambda_{1}}{p}+\frac{\lambda_{2}}{q}\right)\frac{2p}{n(p-2)\lambda_{1}}\right)\|\nabla \varphi(\tau_{3})\|^{2}_{2}+\frac{1}{2}\||x|\varphi(\tau_{3})\|^{2}_{2}\\
&+\left(\frac{1}{2}-\left(\frac{\lambda_{1}}{p}+\frac{\lambda_{2}}{q}\right)\frac{2p c}{n(p-2)\lambda_{1}}-\frac{\lambda_{2}}{q}\right)\|\varphi(\tau_{3})\|^{2}_{2}\\
=&\frac{1}{2}\left(\frac{1}{2}-\left(\frac{\lambda_{1}}{p}+\frac{\lambda_{2}}{q}\right)\frac{2p}{n(p-2)\lambda_{1}}\right)\|\nabla \varphi(\tau_{3})\|^{2}_{2}+\frac{1}{2}\||x|\varphi(\tau_{3})\|^{2}_{2}\\
&+\frac{1}{2}\left(\frac{1}{2}-\left(\frac{\lambda_{1}}{p}+\frac{\lambda_{2}}{q}\right)\frac{2p}{n(p-2)\lambda_{1}}\right)\|\nabla \varphi(\tau_{3})\|^{2}_{2}\\
&+\left(\frac{1}{2}-\left(\frac{\lambda_{1}}{p}+\frac{\lambda_{2}}{q}\right)\frac{2p c}{n(p-2)\lambda_{1}}-\frac{\lambda_{2}}{q}\right)\|\varphi(\tau_{3})\|^{2}_{2}.
\end{split}
\end{align}
In (\ref{p-3zhunbeimaodun+1}), the assumption (\ref{2.9}) ensures $$\frac{1}{2}-\left(\frac{\lambda_{1}}{p}+\frac{\lambda_{2}}{q}\right)\frac{2p}{n(p-2)\lambda_{1}}>0.$$
And the assumption (\ref{2.10}) ensures $$\frac{1}{2}-\frac{\lambda_{2}}{q}>0.$$We choose $c>0$ to make $$\frac{1}{2}-\left(\frac{\lambda_{1}}{p}+\frac{\lambda_{2}}{q}\right)\frac{2p c}{n(p-2)\lambda_{1}}-\frac{\lambda_{2}}{q}>0,$$i.e.,
(\ref{cdetiaojian}), hold.
According to Cauchy-Schwartz inequality and Lemma \ref{lem3.2}, we know (\ref{p-3zhunbeimaodun+1}) turns to
\begin{align}\label{p-3zhunbeimaodun+1-1}
\begin{split}
K(\varphi(\tau_{3}))>
&\left(\frac{1}{2}-\left(\frac{\lambda_{1}}{p}+\frac{\lambda_{2}}{q}\right)\frac{2p}{n(p-2)\lambda_{1}}\right)^{\frac{1}{2}}\|\nabla \varphi(\tau_{3})\|_{2}\||x|\varphi(\tau_{3})\|_{2}\\
&+\frac{1}{2}\left(\frac{1}{2}-\left(\frac{\lambda_{1}}{p}+\frac{\lambda_{2}}{q}\right)\frac{2p}{n(p-2)\lambda_{1}}\right)\|\nabla \varphi(\tau_{3})\|^{2}_{2}\\
&+\left(\frac{1}{2}-\left(\frac{\lambda_{1}}{p}+\frac{\lambda_{2}}{q}\right)\frac{2p c}{n(p-2)\lambda_{1}}-\frac{\lambda_{2}}{q}\right)\|\varphi(\tau_{3})\|^{2}_{2}\\
\geq
&\frac{n}{2}\left(\frac{1}{2}-\left(\frac{\lambda_{1}}{p}+\frac{\lambda_{2}}{q}\right)\frac{2p}{n(p-2)\lambda_{1}}\right)^{\frac{1}{2}}\|\varphi(\tau_{3})\|^{2}_{2}\\
&+\frac{1}{2}\left(\frac{1}{2}-\left(\frac{\lambda_{1}}{p}+\frac{\lambda_{2}}{q}\right)\frac{2p}{n(p-2)\lambda_{1}}\right)\|\nabla \varphi(\tau_{3})\|^{2}_{2}\\
&+\left(\frac{1}{2}-\left(\frac{\lambda_{1}}{p}+\frac{\lambda_{2}}{q}\right)\frac{2p c}{n(p-2)\lambda_{1}}-\frac{\lambda_{2}}{q}\right)\|\varphi(\tau_{3})\|^{2}_{2}.
\end{split}
\end{align}
Meanwhile, by Cauchy-Schwartz inequality, we know (\ref{p-gujij3}) turns to
\begin{align}\label{erfanshude1p3}
\begin{split}
\|\varphi(\tau_{3})\|^{p}_{p}
\geq&\frac{4p c^{\frac{1}{2}}}{n(p-2)\lambda_{1}}\|\nabla \varphi(\tau_{3})\|_{2}\|\varphi(\tau_{3})\|_{2}.
\end{split}
\end{align}
Then, according to (\ref{erfanshude1p3}) and Gagliardo-Nirenberg inequality (\ref{GN-budengshi}), we know
\begin{align*}
\frac{4p c^{\frac{1}{2}}}{n(p-2)\lambda_{1}C_{1}}\|\nabla \varphi(\tau_{3})\|_{2}\|\varphi(\tau_{3})\|_{2}\leq\|\nabla\varphi(\tau_{3})\|^{\frac{n(p-2)}{2}}_{2}\|\varphi(\tau_{3})\|^{p-\frac{n(p-2)}{2}}_{2},
\end{align*}
which means
\begin{align}
\frac{4p c^{\frac{1}{2}}}{n(p-2)\lambda_{1}C_{1}}\leq\|\nabla\varphi(\tau_{3})\|^{\frac{n(p-2)-2}{2}}_{2}\|\varphi(\tau_{3})\|^{p-\frac{n(p-2)+2}{2}}_{2}.\label{p3Gnbudengshi-1-1}
\end{align}
Due to (\ref{2.2}), we know (\ref{p3Gnbudengshi-1-1}) turns to
\begin{align*}
\frac{4p c^{\frac{1}{2}}}{n(p-2)\lambda_{1}C_{1}}\leq\|\nabla\varphi(\tau_{3})\|^{\frac{n(p-2)-2}{2}}_{2}\|\varphi_{0}\|^{p-\frac{n(p-2)+2}{2}}_{2},
\end{align*}
i.e.,
\begin{align}
\|\nabla\varphi(\tau_{3})\|^{2}_{2}\geq M_{3}(\|\varphi_{0}\|_{2}),\label{jdengyu3dexiajie}
\end{align}where $M_{3}(\|\varphi_{0}\|_{2})$ is defined by (\ref{M3}).
Meanwhile, (\ref{p-gujij3}) gives
\begin{align}
\|\varphi(\tau_{3})\|^{p}_{p}>\frac{2p}{n(p-2)\lambda_{1}}\|\nabla \varphi(\tau_{3})\|^{2}_{2}.\label{shizi13}
\end{align}
By the similar process of giving (\ref{Gnbudengshi1-1}), we can use (\ref{shizi13}) to obtain
\begin{align}
\|\nabla\varphi(\tau_{3})\|^{2}_{2}\geq M_{1}(\|\varphi_{0}\|_{2}),\label{Gnbudengshi1-1-3}
\end{align}where $M_{1}(\|\varphi_{0}\|_{2})$ is defined by (\ref{M1}). Due to (\ref{jdengyu3dexiajie}) and (\ref{Gnbudengshi1-1-3}), we have
\begin{align*}
\|\nabla\varphi(\tau_{3})\|^{2}_{2}\geq \frac{M_{1}(\|\varphi_{0}\|_{2})+M_{3}(\|\varphi_{0}\|_{2})
+\left|M_{1}(\|\varphi_{0}\|_{2})-M_{3}(\|\varphi_{0}\|_{2})\right|}{2},
\end{align*}
which means (\ref{p-3zhunbeimaodun+1-1}) turns to
\begin{align}\label{p-3zhunbeimaodun+1-1-1}
\begin{split}
K(\varphi(\tau_{3}))
>
&\frac{n}{2}\left(\frac{1}{2}-\left(\frac{\lambda_{1}}{p}+\frac{\lambda_{2}}{q}\right)\frac{2p}{n(p-2)\lambda_{1}}\right)^{\frac{1}{2}}\|\varphi(\tau_{3})\|^{2}_{2}\\
&+\frac{1}{2}\left(\frac{1}{2}-\left(\frac{\lambda_{1}}{p}+\frac{\lambda_{2}}{q}\right)\frac{2p}{n(p-2)\lambda_{1}}\right)
\\&\times\frac{M_{1}(\|\varphi_{0}\|_{2})+M_{3}(\|\varphi_{0}\|_{2})
+\left|M_{1}(\|\varphi_{0}\|_{2})-M_{3}(\|\varphi_{0}\|_{2})\right|}{2}\\
&+\left(\frac{1}{2}-\left(\frac{\lambda_{1}}{p}+\frac{\lambda_{2}}{q}\right)\frac{2p c}{n(p-2)\lambda_{1}}-\frac{\lambda_{2}}{q}\right)\|\varphi(\tau_{3})\|^{2}_{2}.
\end{split}
\end{align}
According to (\ref{2.2}) and (\ref{2.3}), we know (\ref{p-3zhunbeimaodun+1-1-1}) turns to
\begin{align*}
K(\varphi_{0})
>&\left(\frac{1}{2}-\left(\frac{\lambda_{1}}{p}+\frac{\lambda_{2}}{q}\right)\frac{2p c}{n(p-2)\lambda_{1}}-\frac{\lambda_{2}}{q}+\frac{n}{2}\left(\frac{1}{2}-\left(\frac{\lambda_{1}}{p}+\frac{\lambda_{2}}{q}\right)\frac{2p}{n(p-2)\lambda_{1}}\right)^{\frac{1}{2}}
\right)\|\varphi_{0}\|^{2}_{2}\\
&+\frac{1}{2}\left(\frac{1}{2}-\left(\frac{\lambda_{1}}{p}+\frac{\lambda_{2}}{q}\right)\frac{2p}{n(p-2)\lambda_{1}}\right)
\\&\times\frac{M_{1}(\|\varphi_{0}\|_{2})+M_{3}(\|\varphi_{0}\|_{2})
+\left|M_{1}(\|\varphi_{0}\|_{2})-M_{3}(\|\varphi_{0}\|_{2})\right|}{2},
\end{align*}
which contradicts (\ref{K0--detiaoj1}) for $j=3$.
Thus, we have $S_{3}(\varphi(t))>0$ for $t\in(0,T_{\rm{max}})$.\\
{\bf{Case \textrm{IV}: \bm{$j=4.$}}}

We prove $S_{4}(\varphi(t))>0$ for all $t\in(0,T_{\rm{max}})$. Arguing by contradiction, we suppose that there exists the first time $\tau_{4}\in(0,T_{\rm{max}})$ such that $S_{4}(\varphi(\tau_{4}))=0$, i.e.,
\begin{align}
\|\varphi(\tau_{4})\|^{p}_{p}=&\frac{2p}{n(p-2)\lambda_{1}}\left(\|\nabla \varphi(\tau_{4})\|^{2}_{2}+c\|\varphi(\tau_{4})\|^{2}_{2}+\||x|\varphi(\tau_{4})\|^{2}_{2}\right),\label{gujij4-p}
\end{align} and $S_{4}(\varphi(t))>0$ for $t\in(0,\tau_{4})$, and we know $\varphi(\tau_{4})\neq 0$ due to (\ref{2.2}) and $\varphi_{0}\neq0$.
Next, we estimate $K(\varphi(\tau_{4}))$.
By substituting (\ref{gujij4-p}) into (\ref{2.16}), we have
\begin{align}\label{p4---zhunbeimaodun+1}
\begin{split}
K(\varphi(\tau_{4}))>&\left(\frac{1}{2}-\left(\frac{\lambda_{1}}{p}+\frac{\lambda_{2}}{q}\right)\frac{2p}{n(p-2)\lambda_{1}}\right)\left(\|\nabla \varphi(\tau_{4})\|^{2}_{2}+\||x|\varphi(\tau_{4})\|^{2}_{2}\right)\\
&+\left(\frac{1}{2}-\left(\frac{\lambda_{1}}{p}+\frac{\lambda_{2}}{q}\right)\frac{2p c}{n(p-2)\lambda_{1}}-\frac{\lambda_{2}}{q}\right)\|\varphi(\tau_{4})\|^{2}_{2}\\
=&\frac{1}{2}\left(\frac{1}{2}-\left(\frac{\lambda_{1}}{p}+\frac{\lambda_{2}}{q}\right)\frac{2p}{n(p-2)\lambda_{1}}\right)\|\nabla \varphi(\tau_{4})\|^{2}_{2}\\
&+\left(\frac{1}{2}-\left(\frac{\lambda_{1}}{p}+\frac{\lambda_{2}}{q}\right)\frac{2p}{n(p-2)\lambda_{1}}\right)
\||x|\varphi(\tau_{4})\|^{2}_{2}\\
&+
\frac{1}{2}\left(\frac{1}{2}-\left(\frac{\lambda_{1}}{p}+\frac{\lambda_{2}}{q}\right)\frac{2p}{n(p-2)\lambda_{1}}\right)\|\nabla \varphi(\tau_{4})\|^{2}_{2}\\
&+\left(\frac{1}{2}-\left(\frac{\lambda_{1}}{p}+\frac{\lambda_{2}}{q}\right)\frac{2p c}{n(p-2)\lambda_{1}}-\frac{\lambda_{2}}{q}\right)\|\varphi(\tau_{4})\|^{2}_{2}.
\end{split}
\end{align}
In (\ref{p4---zhunbeimaodun+1}), the assumption (\ref{2.9}) ensures $$\frac{1}{2}-\left(\frac{\lambda_{1}}{p}+\frac{\lambda_{2}}{q}\right)\frac{2p}{n(p-2)\lambda_{1}}>0.$$
And the assumption (\ref{2.10}) ensures $$\frac{1}{2}-\frac{\lambda_{2}}{q}>0.$$We choose $c>0$ to make $$\frac{1}{2}-\left(\frac{\lambda_{1}}{p}+\frac{\lambda_{2}}{q}\right)\frac{2p c}{n(p-2)\lambda_{1}}-\frac{\lambda_{2}}{q}>0,$$ i.e., (\ref{cdetiaojian}), hold.
According to Cauchy-Schwartz inequality and Lemma \ref{lem3.2}, we know (\ref{p4---zhunbeimaodun+1}) turns to
\begin{align}\label{p4--+-zhunbeimaodun+1-}
\begin{split}
K(\varphi(\tau_{4}))>&\frac{2}{2^{\frac{1}{2}}}\left(\frac{1}{2}-\left(\frac{\lambda_{1}}{p}+\frac{\lambda_{2}}{q}\right)\frac{2p}{n(p-2)
\lambda_{1}}\right)\|\nabla \varphi(\tau_{4})\|_{2}\||x|\varphi(\tau_{4})\|_{2}\\
&+
\frac{1}{2}\left(\frac{1}{2}-\left(\frac{\lambda_{1}}{p}+\frac{\lambda_{2}}{q}\right)\frac{2p}{n(p-2)\lambda_{1}}\right)\|\nabla \varphi(\tau_{4})\|^{2}_{2}\\
&+\left(\frac{1}{2}-\left(\frac{\lambda_{1}}{p}+\frac{\lambda_{2}}{q}\right)\frac{2p c}{n(p-2)\lambda_{1}}-\frac{\lambda_{2}}{q}\right)\|\varphi(\tau_{4})\|^{2}_{2}\\
\geq&\frac{n}{2^{\frac{1}{2}}}\left(\frac{1}{2}-\left(\frac{\lambda_{1}}{p}+\frac{\lambda_{2}}{q}\right)\frac{2p}{n(p-2)
\lambda_{1}}\right)\|\varphi(\tau_{4})\|^{2}_{2}\\
&+
\frac{1}{2}\left(\frac{1}{2}-\left(\frac{\lambda_{1}}{p}+\frac{\lambda_{2}}{q}\right)\frac{2p}{n(p-2)\lambda_{1}}\right)\|\nabla \varphi(\tau_{4})\|^{2}_{2}\\
&+\left(\frac{1}{2}-\left(\frac{\lambda_{1}}{p}+\frac{\lambda_{2}}{q}\right)\frac{2p c}{n(p-2)\lambda_{1}}-\frac{\lambda_{2}}{q}\right)\|\varphi(\tau_{4})\|^{2}_{2}.
\end{split}
\end{align}
Meanwhile, by Cauchy-Schwartz inequality, Lemma \ref{lem3.2} and (\ref{gujij4-p}), we have
\begin{align}\label{p4-erfanshude1-1}
\begin{split}
\|\varphi(\tau_{4})\|^{p}_{p}
\geq&\frac{2p}{n(p-2)\lambda_{1}}\left(2\|\nabla \varphi(\tau_{4})\|_{2}\||x|\varphi(\tau_{4})\|_{2}+c\|\varphi(\tau_{4})\|^{2}_{2}\right)\\
\geq&\frac{2p}{n(p-2)\lambda_{1}}\left(n\|\varphi(\tau_{4})\|^{2}_{2}+c\|\varphi(\tau_{4})\|^{2}_{2}\right).
\end{split}
\end{align}
Then, according to (\ref{p4-erfanshude1-1}) and Gagliardo-Nirenberg inequality (\ref{GN-budengshi}), we know
\begin{align}
\frac{2p(n+c)}{n(p-2)\lambda_{1} C_{1}}\leq\|\nabla\varphi(\tau_{4})\|^{\frac{n(p-2)}{2}}_{2}\|\varphi(\tau_{4})\|^{p-\frac{n(p-2)}{2}-2}_{2}.\label{p4--Gnbudengshi-1-1}
\end{align}
Due to (\ref{2.2}), we know (\ref{p4--Gnbudengshi-1-1}) turns to
\begin{align*}
\frac{2p(n+c)}{n(p-2)\lambda_{1} C_{1}}\leq\|\nabla\varphi(\tau_{4})\|^{\frac{n(p-2)}{2}}_{2}\|\varphi_{0}\|^{p-\frac{n(p-2)}{2}-2}_{2},
\end{align*}
i.e.,
\begin{align}
\|\nabla\varphi(\tau_{4})\|^{2}_{2}\geq M_{4}(\|\varphi_{0}\|_{2}),\label{p-4-Gnbudengshi2-1-2}
\end{align}where $M_{4}(\|\varphi_{0}\|_{2})$ is defined by (\ref{M4}).
Meanwhile, (\ref{gujij4-p}) gives
\begin{align}\label{shizi14}
\|\varphi(\tau_{4})\|^{p}_{p}>\frac{2p}{n(p-2)\lambda_{1}}\|\nabla \varphi(\tau_{4})\|^{2}_{2}.
\end{align}By the similar process of giving (\ref{Gnbudengshi1-1}), we can use (\ref{shizi14}) to obtain
\begin{align}
\|\nabla\varphi(\tau_{4})\|^{2}_{2}\geq M_{1}(\|\varphi_{0}\|_{2}),\label{Gnbudengshi1-1-4}
\end{align}where $M_{1}(\|\varphi_{0}\|_{2})$ is defined by (\ref{M1}).
By (\ref{p-4-Gnbudengshi2-1-2}) and (\ref{Gnbudengshi1-1-4}), we have
\begin{align*}
\|\nabla\varphi(\tau_{4})\|^{2}_{2}\geq \frac{M_{1}(\|\varphi_{0}\|_{2})+M_{4}(\|\varphi_{0}\|_{2})
+\left|M_{1}(\|\varphi_{0}\|_{2})-M_{4}(\|\varphi_{0}\|_{2})\right|}{2},
\end{align*}which means (\ref{p4--+-zhunbeimaodun+1-}) becomes
\begin{align}\label{p4--+-zhunbeimaodun+1}
\begin{split}
K(\varphi(\tau_{4}))
>&\frac{n}{2^{\frac{1}{2}}}\left(\frac{1}{2}-\left(\frac{\lambda_{1}}{p}+\frac{\lambda_{2}}{q}\right)\frac{2p}{n(p-2)
\lambda_{1}}\right)\|\varphi(\tau_{4})\|^{2}_{2}\\
&+
\frac{1}{2}\left(\frac{1}{2}-\left(\frac{\lambda_{1}}{p}+\frac{\lambda_{2}}{q}\right)\frac{2p}{n(p-2)\lambda_{1}}\right)
\\&\times\frac{M_{1}(\|\varphi_{0}\|_{2})+M_{4}(\|\varphi_{0}\|_{2})
+\left|M_{1}(\|\varphi_{0}\|_{2})-M_{4}(\|\varphi_{0}\|_{2})\right|}{2}\\
&+\left(\frac{1}{2}-\left(\frac{\lambda_{1}}{p}+\frac{\lambda_{2}}{q}\right)\frac{2p c}{n(p-2)\lambda_{1}}-\frac{\lambda_{2}}{q}\right)\|\varphi(\tau_{4})\|^{2}_{2}.
\end{split}
\end{align}
According to (\ref{p4--+-zhunbeimaodun+1}), (\ref{2.2}) and (\ref{2.3}), we know
\begin{align*}
K(\varphi_{0})
>&\left(\frac{1}{2}-\left(\frac{\lambda_{1}}{p}+\frac{\lambda_{2}}{q}\right)\frac{2p c}{n(p-2)\lambda_{1}}-\frac{\lambda_{2}}{q}+\frac{n}{2^{\frac{1}{2}}}\left(\frac{1}{2}-\left(\frac{\lambda_{1}}{p}+\frac{\lambda_{2}}{q}\right)\frac{2p}{n(p-2)
\lambda_{1}}\right)\right)\|\varphi_{0}\|^{2}_{2}\nonumber\\
&+
\frac{1}{2}\left(\frac{1}{2}-\left(\frac{\lambda_{1}}{p}+\frac{\lambda_{2}}{q}\right)\frac{2p}{n(p-2)\lambda_{1}}\right)
\\&\times\frac{M_{1}(\|\varphi_{0}\|_{2})+M_{4}(\|\varphi_{0}\|_{2})
+\left|M_{1}(\|\varphi_{0}\|_{2})-M_{4}(\|\varphi_{0}\|_{2})\right|}{2},\end{align*}
which contradicts (\ref{K0--detiaoj1}) for $j=4$. Thus, we have $S_{4}(\varphi(t))>0$ for $t\in(0,T_{\rm{max}})$.\\
{\bf{Case \textrm{V}: \bm{$j=5.$}}}

We prove $S_{5}(\varphi(t))>0$ for $t\in(0,T_{\rm{max}})$. Arguing by contradiction, we suppose that there exists the first time $\tau_{5}\in(0,T_{\rm{max}})$ to make $S_{5}(\varphi(\tau_{5}))=0$ and $S_{5}(\varphi(t))>0$ for $t\in(0,\tau_{5})$, and we know $\varphi(\tau_{5})\neq 0$ due to (\ref{2.2}) and $\varphi_{0}\neq0$. Next, we estimate $K(\varphi(\tau_{5}))$.
According to $S_{5}(\varphi(\tau_{5}))=0$, we have
\begin{align}
\frac{n(p-2)\lambda_{1}}{2p}\|\varphi(\tau_{5})\|^{p}_{p}+\frac{n(q-2)\lambda_{2}}{2q}\|\varphi(\tau_{5})\|^{q}_{q}=\|\nabla \varphi(\tau_{5})\|^{2}_{2}.\label{jdengyu5}
\end{align}
Due to $p>q$ assumed in this paper, we know
\begin{align}\label{chuli1-p-q}
\begin{split}
&\frac{n(p-2)\lambda_{1}}{2p}\|\varphi(\tau_{5})\|^{p}_{p}+\frac{n(q-2)\lambda_{2}}{2q}\|\varphi(\tau_{5})\|^{q}_{q}\\
>&\frac{n(q-2)\lambda_{1}}{2p}\|\varphi(\tau_{5})\|^{p}_{p}+\frac{n(q-2)\lambda_{2}}{2q}\|\varphi(\tau_{5})\|^{q}_{q}\\
=&\frac{n(q-2)}{2}\left(\frac{\lambda_{1}}{p}\|\varphi(\tau_{5})\|^{p}_{p}+\frac{\lambda_{2}}{q}\|\varphi(\tau_{5})\|^{q}_{q}\right).
\end{split}
\end{align}
Combining (\ref{jdengyu5}) and (\ref{chuli1-p-q}), we have
\begin{align}
&\frac{\lambda_{1}}{p}\|\varphi(\tau_{5})\|^{p}_{p}+\frac{\lambda_{2}}{q}\|\varphi(\tau_{5})\|^{q}_{q}
<\frac{2}{n(q-2)}\|\nabla\varphi(\tau_{5})\|^{2}_{2}.\label{pq5gujijgaoneng3+}
\end{align}
By substituting (\ref{pq5gujijgaoneng3+}) into $K(\varphi(\tau_{5}))$, we have
\begin{align}\label{pq5zhunbeimaodun-1-5}
\begin{split}
K(\varphi(\tau_{5}))>&\left(\frac{1}{2}-\frac{2}{n(q-2)}\right)\|\nabla \varphi(\tau_{5})\|^{2}_{2}+\frac{1}{2}\||x|\varphi(\tau_{5})\|^{2}_{2}+\frac{1}{2}\|\varphi(\tau_{5})\|^{2}_{2}\\
=&\frac{1}{2}\left(\frac{1}{2}-\frac{2}{n(q-2)}\right)\|\nabla \varphi(\tau_{5})\|^{2}_{2}+\frac{1}{2}\left(\frac{1}{2}-\frac{2}{n(q-2)}\right)\|\nabla \varphi(\tau_{5})\|^{2}_{2}+\frac{1}{2}\||x|\varphi(\tau_{5})\|^{2}_{2}
\\&+\frac{1}{2}\|\varphi(\tau_{5})\|^{2}_{2}.
\end{split}
\end{align}
In (\ref{pq5zhunbeimaodun-1-5}), the assumption (\ref{1.2}) ensures $$\frac{1}{2}-\frac{2}{n(q-2)}>0.$$ By Cauchy-Schwartz inequality and Lemma \ref{lem3.2}, we know (\ref{pq5zhunbeimaodun-1-5}) turns to
\begin{align}\label{pq5zhunbeimaodun-1}
\begin{split}
K(\varphi(\tau_{5}))\geq&\frac{1}{2}\left(\frac{1}{2}-\frac{2}{n(q-2)}\right)\|\nabla \varphi(\tau_{5})\|^{2}_{2}+\left(\frac{1}{2}-\frac{2}{n(q-2)}\right)^{\frac{1}{2}}\|\nabla \varphi(\tau_{5})\|_{2}\||x|\varphi(\tau_{5})\|_{2}\\&+\frac{1}{2}\|\varphi(\tau_{5})\|^{2}_{2}\\
\geq&\frac{1}{2}\left(\frac{1}{2}-\frac{2}{n(q-2)}\right)\|\nabla \varphi(\tau_{5})\|^{2}_{2}+\frac{n}{2}\left(\frac{1}{2}-\frac{2}{n(q-2)}\right)^{\frac{1}{2}}\|\varphi(\tau_{5})\|^{2}_{2}+\frac{1}{2}\|\varphi(\tau_{5})\|^{2}_{2}.
\end{split}
\end{align}
Meanwhile, according to $(\ref{jdengyu5})$, (\ref{GN-budengshi}) and (\ref{GN-budengshi1}), we have
\begin{align*}
\|\nabla \varphi(\tau_{5})\|^{2}_{2}=&\frac{n(p-2)\lambda_{1}}{2p}\|\varphi(\tau_{5})\|^{p}_{p}+\frac{n(q-2)\lambda_{2}}{2q}\|\varphi(\tau_{5})\|^{q}_{q}\nonumber\\
\leq&\frac{n(p-2)C_{1}\lambda_{1}}{2p}\|\nabla\varphi(\tau_{5})\|^{\frac{n(p-2)}{2}}_{2}\|\varphi(\tau_{5})\|^{p-\frac{n(p-2)}{2}}_{2}\nonumber\\
&+\frac{n(q-2)C_{3}\lambda_{2}}{2q}\|\nabla\varphi(\tau_{5})\|^{\frac{n(q-2)}{2}}_{2}\|\varphi(\tau_{5})\|^{q-\frac{n(q-2)}{2}}_{2},
\end{align*}
i.e.,
\begin{align}\label{jdengyu5-5}
\begin{split}
1
\leq&\frac{n(p-2)C_{1}\lambda_{1}}{2p}\|\nabla\varphi(\tau_{5})\|^{\frac{n(p-2)-4}{2}}_{2}\|\varphi(\tau_{5})\|^{p-\frac{n(p-2)}{2}}_{2}\\
&+\frac{n(q-2)C_{3}\lambda_{2}}{2q}\|\nabla\varphi(\tau_{5})\|^{\frac{n(q-2)-4}{2}}_{2}\|\varphi(\tau_{5})\|^{q-\frac{n(q-2)}{2}}_{2},
\end{split}
\end{align}
which means
\begin{align*}
\frac{1}{2}
\leq&\frac{n(p-2)C_{1}\lambda_{1}}{2p}\|\nabla\varphi(\tau_{5})\|^{\frac{n(p-2)-4}{2}}_{2}\|\varphi(\tau_{5})\|^{p-\frac{n(p-2)}{2}}_{2},
\end{align*}
i.e.,
\begin{align}
\left(\frac{p}{n(p-2)C_{1}\lambda_{1}\|\varphi(\tau_{5})\|^{p-\frac{n(p-2)}{2}}_{2}}\right)^{\frac{4}{n(p-2)-4}}
\leq&\|\nabla\varphi(\tau_{5})\|^{2}_{2},\label{jdengyu5-1}
\end{align}
or
\begin{align*}
\frac{1}{2}
\leq&\frac{n(q-2)C_{3}\lambda_{2}}{2q}\|\nabla\varphi(\tau_{5})\|^{\frac{n(q-2)-4}{2}}_{2}\|\varphi(\tau_{5})\|^{q-\frac{n(q-2)}{2}}_{2},
\end{align*}
i.e.,
\begin{align}
\left(\frac{q}{n(q-2)C_{3}\lambda_{2}\|\varphi(\tau_{5})\|^{q-\frac{n(q-2)}{2}}_{2}}\right)^{\frac{4}{n(q-2)-4}}
\leq&\|\nabla\varphi(\tau_{5})\|^{2}_{2}.\label{jdengyu5-2}
\end{align}
According to (\ref{2.2}), (\ref{jdengyu5-1}) and (\ref{jdengyu5-2}) turn to
\begin{align}
\left(\frac{p}{n(p-2)C_{1}\lambda_{1}\|\varphi_{0}\|^{p-\frac{n(p-2)}{2}}_{2}}\right)^{\frac{4}{n(p-2)-4}}
\leq&\|\nabla\varphi(\tau_{5})\|^{2}_{2}\label{jdengyu5-1-1}
\end{align}
and
\begin{align}
\left(\frac{q}{n(q-2)C_{3}\lambda_{2}\|\varphi_{0}\|^{q-\frac{n(q-2)}{2}}_{2}}\right)^{\frac{4}{n(q-2)-4}}
\leq&\|\nabla\varphi(\tau_{5})\|^{2}_{2},\label{jdengyu5-2-1}
\end{align}
respectively. From (\ref{jdengyu5-1-1}) and (\ref{jdengyu5-2-1}), we have
\begin{align}
\|\nabla\varphi(\tau_{5})\|^{2}_{2}\geq&N_{1}(\|\varphi_{0}\|_{2}),\label{jdengyu5-2-2}
\end{align}
where $N_{1}(\|\varphi_{0}\|_{2})$ is defined (\ref{N1}).
According to (\ref{jdengyu5-2-2}), we know (\ref{pq5zhunbeimaodun-1}) turns to
\begin{align}\label{zhunbeimaodun-1-5-}
\begin{split}
K(\varphi(\tau_{5}))
>&\frac{1}{2}\left(\frac{1}{2}-\frac{2}{n(q-2)}\right)N_{1}(\|\varphi_{0}\|_{2})+\frac{n}{2}\left(\frac{1}{2}-\frac{2}{n(q-2)}\right)^{\frac{1}{2}}\|\varphi(\tau_{5})\|^{2}_{2}+\frac{1}{2}\|\varphi(\tau_{5})\|^{2}_{2}.
\end{split}
\end{align}
Due to (\ref{2.2}) and (\ref{2.3}), we know (\ref{zhunbeimaodun-1-5-}) turns to
\begin{align*}
K(\varphi_{0})
>&\frac{1}{2}\left(\frac{1}{2}-\frac{2}{n(q-2)}\right)N_{1}(\|\varphi_{0}\|_{2})+\left(\frac{1}{2}+\frac{n}{2}\left(\frac{1}{2}-\frac{2}{n(q-2)}\right)^{\frac{1}{2}}\right)\|\varphi_{0}\|^{2}_{2},
\end{align*}
which contradicts (\ref{K0--detiaoj1}) for $j=5$. Thus, we have $S_{5}(\varphi(t))>0$ for $t\in(0,T_{\rm{max}})$.\\
{\bf{Case \textrm{VI}: \bm{$j=6.$}}}

We prove $S_{6}(\varphi(t))>0$ for $t\in(0,T_{\rm{max}})$. Arguing by contradiction, we suppose that there exists the first time $\tau_{6}\in(0,T_{\rm{max}})$ to make $S_{6}(\varphi(\tau_{6}))=0$ and $S_{6}(\varphi(t))>0$ for $t\in(0,\tau_{6})$, and we know $\varphi(\tau_{6})\neq 0$ due to (\ref{2.2}) and $\varphi_{0}\neq0$. Next, we estimate $K(\varphi(\tau_{6}))$.
According to $S_{6}(\varphi(\tau_{6}))=0$, we have
\begin{align}
\frac{n(p-2)\lambda_{1}}{2p}\|\varphi(\tau_{6})\|^{p}_{p}+\frac{n(q-2)\lambda_{2}}{2q}\|\varphi(\tau_{6})\|^{q}_{q}=\|\nabla \varphi(\tau_{6})\|^{2}_{2}+\||x|\varphi(\tau_{6})\|^{2}_{2}.\label{pq6jdengyu6}
\end{align}
Using (\ref{pq6jdengyu6}) and retrieving (\ref{chuli1-p-q}) for $j=6$, we know
\begin{align}
&\frac{\lambda_{1}}{p}\|\varphi(\tau_{6})\|^{p}_{p}+\frac{\lambda_{2}}{q}\|\varphi(\tau_{6})\|^{q}_{q}
<\frac{2}{n(q-2)}\left(\|\nabla\varphi(\tau_{6})\|^{2}_{2}+\||x|\varphi(\tau_{6})\|^{2}_{2}\right).\label{gujijgaoneng3+}
\end{align}
By substituting (\ref{gujijgaoneng3+}) into $K(\varphi(\tau_{6}))$, we have
\begin{align}\label{pq6zhunbeimaodun-1-5}
\begin{split}
K(\varphi(\tau_{6}))>&\left(\frac{1}{2}-\frac{2}{n(q-2)}\right)\left(\|\nabla \varphi(\tau_{6})\|^{2}_{2}+\||x|\varphi(\tau_{6})\|^{2}_{2}\right)+\frac{1}{2}\|\varphi(\tau_{6})\|^{2}_{2}\\
=&\frac{1}{2}\left(\frac{1}{2}-\frac{2}{n(q-2)}\right)\|\nabla \varphi(\tau_{6})\|^{2}_{2}+\frac{1}{2}\left(\frac{1}{2}-\frac{2}{n(q-2)}\right)\|\nabla \varphi(\tau_{6})\|^{2}_{2}\\
&+\left(\frac{1}{2}-\frac{2}{n(q-2)}\right)\||x|\varphi(\tau_{6})\|^{2}_{2}+\frac{1}{2}\|\varphi(\tau_{6})\|^{2}_{2}.
\end{split}
\end{align}
In (\ref{pq6zhunbeimaodun-1-5}), the assumption (\ref{1.2}) ensures $$\frac{1}{2}-\frac{2}{n(q-2)}>0.$$ By Cauchy-Schwartz inequality and Lemma \ref{lem3.2}, we know (\ref{pq6zhunbeimaodun-1-5}) turns to
\begin{align}\label{pq6-zhunbeimaodun-1}
\begin{split}
K(\varphi(\tau_{6}))>&\frac{1}{2}\left(\frac{1}{2}-\frac{2}{n(q-2)}\right)\|\nabla \varphi(\tau_{6})\|^{2}_{2}+\frac{2}{2^{\frac{1}{2}}}\left(\frac{1}{2}-\frac{2}{n(q-2)}\right)\|\nabla \varphi(\tau_{6})\|_{2}\||x|\varphi(\tau_{6})\|_{2}
\\
&+\frac{1}{2}\|\varphi(\tau_{6})\|^{2}_{2}\\
\geq&\frac{1}{2}\left(\frac{1}{2}-\frac{2}{n(q-2)}\right)\|\nabla \varphi(\tau_{6})\|^{2}_{2}+\frac{n}{2^{\frac{1}{2}}}\left(\frac{1}{2}-\frac{2}{n(q-2)}\right)\|\varphi(\tau_{6})\|^{2}_{2}+\frac{1}{2}\|\varphi(\tau_{6})\|^{2}_{2}.
\end{split}
\end{align}
Meanwhile, by using $(\ref{pq6jdengyu6})$ and retrieving (\ref{erfanshude1-2}) for $j=6$, we have
\begin{align}
\frac{n(p-2)\lambda_{1}}{2p}\|\varphi(\tau_{6})\|^{p}_{p}+\frac{n(q-2)\lambda_{2}}{2q}\|\varphi(\tau_{6})\|^{q}_{q}\geq n\|\varphi(\tau_{6})\|^{2}_{2}.\label{pq6jdengyu6-1}
\end{align}
According to (\ref{pq6jdengyu6-1}), (\ref{GN-budengshi}) and (\ref{GN-budengshi1}), we have
\begin{align*}
n\|\varphi(\tau_{6})\|^{2}_{2}\leq&\frac{n(p-2)\lambda_{1}}{2p}\|\varphi(\tau_{6})\|^{p}_{p}+\frac{n(q-2)\lambda_{2}}{2q}\|\varphi(\tau_{6})\|^{q}_{q}\nonumber\\
\leq&\frac{n(p-2)C_{1}\lambda_{1}}{2p}\|\nabla\varphi(\tau_{6})\|^{\frac{n(p-2)}{2}}_{2}\|\varphi(\tau_{6})\|^{p-\frac{n(p-2)}{2}}_{2}\nonumber\\
&+\frac{n(q-2)C_{3}\lambda_{2}}{2q}\|\nabla\varphi(\tau_{6})\|^{\frac{n(q-2)}{2}}_{2}\|\varphi(\tau_{6})\|^{q-\frac{n(q-2)}{2}}_{2},
\end{align*}
i.e.,
\begin{align}
n
\leq&\frac{n(p-2)C_{1}\lambda_{1}}{2p}\|\nabla\varphi(\tau_{6})\|^{\frac{n(p-2)}{2}}_{2}\|\varphi(\tau_{6})\|^{p-\frac{n(p-2)}{2}-2}_{2}\nonumber\\
&+\frac{n(q-2)C_{3}\lambda_{2}}{2q}\|\nabla\varphi(\tau_{6})\|^{\frac{n(q-2)}{2}}_{2}\|\varphi(\tau_{6})\|^{q-\frac{n(q-2)}{2}-2}_{2},
\label{jdengyu5-5}
\end{align}
which means
\begin{align*}
\frac{n}{2}
\leq&\frac{n(p-2)C_{1}\lambda_{1}}{2p}\|\nabla\varphi(\tau_{6})\|^{\frac{n(p-2)}{2}}_{2}\|\varphi(\tau_{6})\|^{p-\frac{n(p-2)}{2}-2}_{2},
\end{align*}
i.e.,
\begin{align}
\left(\frac{p}{(p-2)C_{1}\lambda_{1}\|\varphi(\tau_{6})\|^{p-\frac{n(p-2)}{2}-2}_{2}}\right)^{\frac{4}{n(p-2)}}
\leq&\|\nabla\varphi(\tau_{6})\|^{2}_{2},\label{jdengyu6-1}
\end{align}
or
\begin{align*}
\frac{n}{2}
\leq&\frac{n(q-2)C_{3}\lambda_{2}}{2q}\|\nabla\varphi(\tau_{6})\|^{\frac{n(q-2)}{2}}_{2}\|\varphi(\tau_{6})\|^{q-\frac{n(q-2)}{2}-2}_{2},
\end{align*}
i.e.,
\begin{align}
\left(\frac{q}{(q-2)C_{3}\lambda_{2}\|\varphi(\tau_{6})\|^{q-\frac{n(q-2)}{2}-2}_{2}}\right)^{\frac{4}{n(q-2)}}
\leq&\|\nabla\varphi(\tau_{6})\|^{2}_{2}.\label{jdengyu6-2}
\end{align}
According to (\ref{2.2}), (\ref{jdengyu6-1}) and (\ref{jdengyu6-2}) turn to
\begin{align}
\left(\frac{p}{(p-2)C_{1}\lambda_{1}\|\varphi_{0}\|^{p-\frac{n(p-2)}{2}-2}_{2}}\right)^{\frac{4}{n(p-2)}}
\leq&\|\nabla\varphi(\tau_{6})\|^{2}_{2}\label{jdengyu6-1-1}
\end{align}
and
\begin{align}
\left(\frac{q}{(q-2)C_{3}\lambda_{2}\|\varphi_{0}\|^{q-\frac{n(q-2)}{2}-2}_{2}}\right)^{\frac{4}{n(q-2)}}
\leq&\|\nabla\varphi(\tau_{6})\|^{2}_{2},\label{jdengyu6-2-1}
\end{align}
respectively. From (\ref{jdengyu6-1-1}) and (\ref{jdengyu6-2-1}), we have
\begin{align}
\|\nabla\varphi(\tau_{6})\|^{2}_{2}\geq&N_{2}(\|\varphi_{0}\|_{2})
,\label{jdengyu6-2-2-6}
\end{align}
where $N_{2}(\|\varphi_{0}\|_{2})$ is defined by (\ref{N2}).
Meanwhile, (\ref{pq6jdengyu6}) also gives
\begin{align}\label{shizi16}
&\frac{n(p-2)\lambda_{1}}{2p}\|\varphi(\tau_{6})\|^{p}_{p}+\frac{n(q-2)\lambda_{2}}{2q}\|\varphi(\tau_{6})\|^{q}_{q}
>\|\nabla\varphi(\tau_{6})\|^{2}_{2}.
\end{align}By the similar process of giving (\ref{jdengyu5-2-2}), we can use (\ref{shizi16}) to obtain
\begin{align}
\|\nabla\varphi(\tau_{6})\|^{2}_{2}\geq&N_{1}(\|\varphi_{0}\|_{2}),\label{jdengyu5-2-2-6}
\end{align}
where $N_{1}(\|\varphi_{0}\|_{2})$ is defined by (\ref{N1}).
From (\ref{jdengyu6-2-2-6}) and (\ref{jdengyu5-2-2-6}), we know
\begin{align*}
\|\nabla\varphi(\tau_{6})\|^{2}_{2}\geq&\frac{N_{1}(\|\varphi_{0}\|_{2})+N_{2}(\|\varphi_{0}\|_{2})+\left|N_{1}(\|\varphi_{0}\|_{2})-
N_{2}(\|\varphi_{0}\|_{2})\right|}{2},
\end{align*}
which means
(\ref{pq6-zhunbeimaodun-1}) turns to
\begin{align}\label{pq6--zhunbeimaodun-1-5}
\begin{split}
K(\varphi(\tau_{6}))
>&\frac{1}{2}\left(\frac{1}{2}-\frac{2}{n(q-2)}\right)\frac{N_{1}(\|\varphi_{0}\|_{2})+N_{2}(\|\varphi_{0}\|_{2})+\left|N_{1}(\|\varphi_{0}\|_{2})-
N_{2}(\|\varphi_{0}\|_{2})\right|}{2}\\
&+\left(\frac{1}{2}+\frac{n}{2^{\frac{1}{2}}}\left(\frac{1}{2}-\frac{2}{n(q-2)}\right)\right)\|\varphi(\tau_{6})\|^{2}_{2}.
\end{split}
\end{align}
Due to (\ref{2.2}) and (\ref{2.3}), we know (\ref{pq6--zhunbeimaodun-1-5}) turns to
\begin{align*}
K(\varphi_{0})>&\frac{1}{2}\left(\frac{1}{2}-\frac{2}{n(q-2)}\right)\frac{N_{1}(\|\varphi_{0}\|_{2})+N_{2}(\|\varphi_{0}\|_{2})
+\left|N_{1}(\|\varphi_{0}\|_{2})-
N_{2}(\|\varphi_{0}\|_{2})\right|}{2}\\
&+\left(\frac{1}{2}+\frac{n}{2^{\frac{1}{2}}}\left(\frac{1}{2}-\frac{2}{n(q-2)}\right)\right)\|\varphi_{0}\|^{2}_{2},
\end{align*}
which contradicts (\ref{K0--detiaoj1}) for $j=6$. Thus, we have $S_{6}(\varphi(t))>0$ for $t\in(0,T_{\rm{max}})$.\\
{\bf{Case \textrm{VI}: \bm{$j=7.$}}}

We prove $S_{7}(\varphi(t))>0$ for $t\in(0,T_{\rm{max}})$. Arguing by contradiction, we suppose that there exists the first time $\tau_{7}\in(0,T_{\rm{max}})$ to make $S_{7}(\varphi(\tau_{7}))=0$ and $S_{7}(\varphi(t))>0$ for $t\in(0,\tau_{7})$, and we know $\varphi(\tau_{7})\neq 0$ due to (\ref{2.2}) and $\varphi_{0}\neq0$. Next, we estimate $K(\varphi(\tau_{7}))$.
According to $S_{7}(\varphi(\tau_{7}))=0$, we have
\begin{align}
\frac{n(p-2)\lambda_{1}}{2p}\|\varphi(\tau_{7})\|^{p}_{p}+\frac{n(q-2)\lambda_{2}}{2q}\|\varphi(\tau_{7})\|^{q}_{q}=\|\nabla \varphi(\tau_{7})\|^{2}_{2}+\|\varphi(\tau_{7})\|^{2}_{2}.\label{jdengyu7}
\end{align}
Using (\ref{jdengyu7}) and retrieving (\ref{chuli1-p-q}) for $j=7$, we know
\begin{align}
&\frac{\lambda_{1}}{p}\|\varphi(\tau_{7})\|^{p}_{p}+\frac{\lambda_{2}}{q}\|\varphi(\tau_{7})\|^{q}_{q}
<\frac{2}{n(q-2)}\left(\|\nabla\varphi(\tau_{7})\|^{2}_{2}+\|\varphi(\tau_{7})\|^{2}_{2}\right).\label{gujijgaoneng3-7+}
\end{align}
By substituting (\ref{gujijgaoneng3-7+}) into $K(\varphi(\tau_{7}))$, we have
\begin{align}\label{pq7zhunbeimaodun-1-5}
\begin{split}
K(\varphi(\tau_{7}))>&\left(\frac{1}{2}-\frac{2}{n(q-2)}\right)\|\nabla \varphi(\tau_{7})\|^{2}_{2}+\frac{1}{2}\||x|\varphi(\tau_{7})\|^{2}_{2}+\left(\frac{1}{2}-\frac{2}{n(q-2)}\right)\| \varphi(\tau_{7})\|^{2}_{2}\\
=&\frac{1}{2}\left(\frac{1}{2}-\frac{2}{n(q-2)}\right)\|\nabla \varphi(\tau_{7})\|^{2}_{2}+\frac{1}{2}\left(\frac{1}{2}-\frac{2}{n(q-2)}\right)\|\nabla \varphi(\tau_{7})\|^{2}_{2}\\
&+\frac{1}{2}\||x|\varphi(\tau_{7})\|^{2}_{2}
+\left(\frac{1}{2}-\frac{2}{n(q-2)}\right)\|\varphi(\tau_{7})\|^{2}_{2}.
\end{split}
\end{align}
In (\ref{pq7zhunbeimaodun-1-5}), the assumption (\ref{1.2}) ensures $$\frac{1}{2}-\frac{2}{n(q-2)}>0.$$ By Cauchy-Schwartz inequality and Lemma \ref{lem3.2}, we know (\ref{pq7zhunbeimaodun-1-5}) turns to
\begin{align}\label{pq7-zhunbeimaodun-7-1}
\begin{split}
K(\varphi(\tau_{7}))>&\frac{1}{2}\left(\frac{1}{2}-\frac{2}{n(q-2)}\right)\|\nabla \varphi(\tau_{7})\|^{2}_{2}+\left(\frac{1}{2}-\frac{2}{n(q-2)}\right)^{\frac{1}{2}}\|\nabla \varphi(\tau_{7})\|_{2}\||x|\varphi(\tau_{7})\|_{2}\\&+\left(\frac{1}{2}-\frac{2}{n(q-2)}\right)\|\varphi(\tau_{7})\|^{2}_{2}\\
\geq&\frac{1}{2}\left(\frac{1}{2}-\frac{2}{n(q-2)}\right)\|\nabla \varphi(\tau_{7})\|^{2}_{2}
+\frac{n}{2}\left(\frac{1}{2}-\frac{2}{n(q-2)}\right)^{\frac{1}{2}}\|\varphi(\tau_{7})\|^{2}_{2}\\&+\left(\frac{1}{2}-\frac{2}{n(q-2)}\right)\|\varphi(\tau_{7})\|^{2}_{2}.
\end{split}
\end{align}
Meanwhile, according to $(\ref{jdengyu7})$ and Cauchy-Schwartz inequality, we have
\begin{align}
&\frac{n(p-2)\lambda_{1}}{2p}\|\varphi(\tau_{7})\|^{p}_{p}+\frac{n(q-2)\lambda_{2}}{2q}\|\varphi(\tau_{7})\|^{q}_{q}
=\|\nabla \varphi(\tau_{7})\|^{2}_{2}+\|\varphi(\tau_{7})\|^{2}_{2}
\geq2\|\nabla \varphi(\tau_{7})\|_{2}\|\varphi(\tau_{7})\|_{2}.\label{jdengyu7-1}
\end{align}
According to (\ref{jdengyu7-1}), (\ref{GN-budengshi}) and (\ref{GN-budengshi1}), we have
\begin{align*}
2\|\nabla \varphi(\tau_{7})\|_{2}\|\varphi(\tau_{7})\|_{2}\leq&\frac{n(p-2)\lambda_{1}}{2p}\|\varphi(\tau_{7})\|^{p}_{p}
+\frac{n(q-2)\lambda_{2}}{2q}\|\varphi(\tau_{7})\|^{q}_{q}\nonumber\\
\leq&\frac{n(p-2)\lambda_{1}C_{1}}{2p}\|\nabla\varphi(\tau_{7})\|^{\frac{n(p-2)}{2}}_{2}\|\varphi(\tau_{7})\|^{p-\frac{n(p-2)}{2}}_{2}\nonumber\\
&+\frac{n(q-2)\lambda_{2}C_{3}}{2q}\|\nabla\varphi(\tau_{7})\|^{\frac{n(q-2)}{2}}_{2}\|\varphi(\tau_{7})\|^{q-\frac{n(q-2)}{2}}_{2},
\end{align*}
which means
\begin{align}\label{jieguo6}
\begin{split}
2
\leq&\frac{n(p-2)\lambda_{1}C_{1}}{2p}\|\nabla\varphi(\tau_{7})\|^{\frac{n(p-2)}{2}-1}_{2}\|\varphi(\tau_{7})\|^{p-\frac{n(p-2)}{2}-1}_{2}\\
&+\frac{n(q-2)\lambda_{2}C_{3}}{2q}\|\nabla\varphi(\tau_{7})\|^{\frac{n(q-2)}{2}-1}_{2}\|\varphi(\tau_{7})\|^{q-\frac{n(q-2)}{2}-1}_{2}.
\end{split}
\end{align}
According to (\ref{jieguo6}), we have
\begin{align}
1
\leq\frac{n(p-2)\lambda_{1}C_{1}}{2p}\|\nabla\varphi(\tau_{7})\|^{\frac{n(p-2)}{2}-1}_{2}\|\varphi(\tau_{7})\|^{p-\frac{n(p-2)}{2}-1}_{2}
\label{jieguo6-p}
\end{align}
or
\begin{align}
1
\leq\frac{n(q-2)\lambda_{2}C_{3}}{2q}\|\nabla\varphi(\tau_{7})\|^{\frac{n(q-2)}{2}-1}_{2}\|\varphi(\tau_{7})\|^{q-\frac{n(q-2)}{2}-1}_{2}.\label{jieguo6-q}
\end{align}
According to (\ref{2.2}), (\ref{jieguo6-p}) and (\ref{jieguo6-q}) turn to
\begin{align}
\left(\frac{2p}{n(p-2)\lambda_{1}C_{1}\|\varphi_{0}\|^{p-\frac{n(p-2)}{2}-1}_{2}}\right)^{\frac{4}{n(p-2)-2}}
\leq&\|\nabla\varphi(\tau_{7})\|^{2}_{2}\label{jdengyu7-1-1}
\end{align}
and
\begin{align}
\left(\frac{2q}{n(q-2)\lambda_{2}C_{3}\|\varphi_{0}\|^{q-\frac{n(q-2)}{2}-1}_{2}}\right)^{\frac{4}{n(q-2)-2}}
\leq&\|\nabla\varphi(\tau_{7})\|^{2}_{2},\label{jdengyu7-1-2}
\end{align}
respectively. From (\ref{jdengyu7-1-1}) and (\ref{jdengyu7-1-2}), we have
\begin{align}
\|\nabla\varphi(\tau_{7})\|^{2}_{2}\geq&N_{3}(\|\varphi_{0}\|_{2}),
\label{jdengyu6-2-2-7}
\end{align}
where $N_{3}(\|\varphi_{0}\|_{2})$ is defined by (\ref{N3}).
Meanwhile, (\ref{jdengyu7}) also gives
\begin{align}
\frac{n(p-2)\lambda_{1}}{2p}\|\varphi(\tau_{7})\|^{p}_{p}+\frac{n(q-2)\lambda_{2}}{2q}\|\varphi(\tau_{7})\|^{q}_{q}>\|\nabla \varphi(\tau_{7})\|^{2}_{2}.\label{shizi17}
\end{align}By the similar process of giving (\ref{jdengyu5-2-2}), we can use (\ref{shizi17}) to obtain
\begin{align}
\|\nabla\varphi(\tau_{7})\|^{2}_{2}\geq&N_{1}(\|\varphi_{0}\|_{2}),\label{jdengyu5-2-2-7}
\end{align}
where $N_{1}(\|\varphi_{0}\|_{2})$ are defined by (\ref{N1}).
From (\ref{jdengyu6-2-2-7}) and (\ref{jdengyu5-2-2-7}), we know
\begin{align*}
\|\nabla\varphi(\tau_{7})\|^{2}_{2}\geq&\frac{N_{1}(\|\varphi_{0}\|_{2})+N_{3}(\|\varphi_{0}\|_{2})+\left|N_{1}(\|\varphi_{0}\|_{2})-N_{3}(\|\varphi_{0}\|_{2})\right|}{2},
\end{align*}
which means
(\ref{pq7-zhunbeimaodun-7-1}) turns to
\begin{align}\label{pq-7-zhunbeimaodun-7-1}
\begin{split}
K(\varphi(\tau_{7}))
>&\frac{1}{2}\left(\frac{1}{2}-\frac{2}{n(q-2)}\right)\frac{N_{1}(\|\varphi_{0}\|_{2})+N_{3}(\|\varphi_{0}\|_{2})+\left|N_{1}(\|\varphi_{0}\|_{2})-N_{3}(\|\varphi_{0}\|_{2})\right|}{2}\\
&+\frac{n}{2}
\left(\frac{1}{2}-\frac{2}{n(q-2)}\right)^{\frac{1}{2}}\|\varphi(\tau_{7})\|^{2}_{2}+\left(\frac{1}{2}-\frac{2}{n(q-2)}\right)\|\varphi(\tau_{7})\|^{2}_{2}.
\end{split}
\end{align}
According to (\ref{2.2}) and (\ref{2.3}), (\ref{pq-7-zhunbeimaodun-7-1}) becomes
\begin{align*}
K(\varphi_{0})
>&\frac{1}{2}\left(\frac{1}{2}-\frac{2}{n(q-2)}\right)\frac{N_{1}(\|\varphi_{0}\|_{2})+N_{3}(\|\varphi_{0}\|_{2})+\left|N_{1}(\|\varphi_{0}\|_{2})-N_{3}(\|\varphi_{0}\|_{2})\right|}{2}\nonumber\\
&+\left(\frac{1}{2}-\frac{2}{n(q-2)}+\frac{n}{2}\left(\frac{1}{2}-\frac{2}{n(q-2)}\right)^{\frac{1}{2}}\right)\|\varphi_{0}\|^{2}_{2},
\end{align*}
which contradicts (\ref{K0--detiaoj1}) for $j=7$. Thus, we have $S_{7}(\varphi(t))>0$ for $t\in(0,T_{\rm{max}})$.\\
{\bf{Case \textrm{VIII}: \bm{$j=8.$}}}

We prove $S_{8}(\varphi(t))>0$ for $t\in(0,T_{\rm{max}})$. Arguing by contradiction, we suppose that there exists the first time $\tau_{8}\in(0,T_{\rm{max}})$ to make $S_{8}(\varphi(\tau_{8}))=0$ and $S_{8}(\varphi(t))>0$ for $t\in(0,\tau_{8})$, and we know $\varphi(\tau_{8})\neq 0$ due to (\ref{2.2}) and $\varphi_{0}\neq0$. Next, we estimate $K(\varphi(\tau_{8}))$.
According to $S_{8}(\varphi(\tau_{8}))=0$, we have
\begin{align}
\frac{n(p-2)\lambda_{1}}{2p}\|\varphi(\tau_{8})\|^{p}_{p}+\frac{n(q-2)\lambda_{2}}{2q}\|\varphi(\tau_{8})\|^{q}_{q}=\|\nabla\varphi(\tau_{8})\|^{2}_{2}+\|\varphi(\tau_{8})\|^{2}_{2}+\||x|\varphi(\tau_{8})\|^{2}_{2}.\label{jdengyu6}
\end{align}
Using (\ref{jdengyu6}) and  retrieving (\ref{chuli1-p-q}) for $j=8$, we know
\begin{align}
&\frac{\lambda_{1}}{p}\|\varphi(\tau_{8})\|^{p}_{p}+\frac{\lambda_{2}}{q}\|\varphi(\tau_{8})\|^{q}_{q}<\frac{2}{n(q-2)}
\left(\|\nabla\varphi(\tau_{8})\|^{2}_{2}+\|\varphi(\tau_{8})\|^{2}_{2}+\||x|\varphi(\tau_{8})\|^{2}_{2}\right).\label{gujijgaoneng3-8+}
\end{align}
By substituting (\ref{gujijgaoneng3-8+}) into $K(\varphi(\tau_{8}))$, we have
\begin{align}\label{zhunbeimaodun-1-5}
\begin{split}
K(\varphi(\tau_{8}))>&\left(\frac{1}{2}-\frac{2}{n(q-2)}\right)
\left(\|\nabla \varphi(\tau_{8})\|^{2}_{2}+\| \varphi(\tau_{8})\|^{2}_{2}+\||x|\varphi(\tau_{8})\|^{2}_{2}\right)\\
=&\frac{1}{2}\left(\frac{1}{2}-\frac{2}{n(q-2)}\right)\|\nabla \varphi(\tau_{8})\|^{2}_{2}+\frac{1}{2}\left(\frac{1}{2}-\frac{2}{n(q-2)}\right)\|\nabla \varphi(\tau_{8})\|^{2}_{2}\\
&+\left(\frac{1}{2}-\frac{2}{n(q-2)}\right)\||x|\varphi(\tau_{8})\|^{2}_{2}+\left(\frac{1}{2}-\frac{2}{n(q-2)}\right)\|\varphi(\tau_{8})\|^{2}_{2}.
\end{split}
\end{align}
In (\ref{zhunbeimaodun-1-5}), the assumption (\ref{1.2}) ensures $$\frac{1}{2}-\frac{2}{n(q-2)}>0.$$ By Cauchy-Schwartz inequality and Lemma \ref{lem3.2}, we know (\ref{zhunbeimaodun-1-5}) turns to
\begin{align}\label{zhunbeimaodun-7-1}
\begin{split}
K(\varphi(\tau_{8}))>&\frac{1}{2}\left(\frac{1}{2}-\frac{2}{n(q-2)}\right)\|\nabla \varphi(\tau_{8})\|^{2}_{2}+\frac{2}{2^{\frac{1}{2}}}\left(\frac{1}{2}-\frac{2}{n(q-2)}\right)\|\nabla \varphi(\tau_{8})\|_{2}\||x|\varphi(\tau_{8})\|_{2}\\
&+\left(\frac{1}{2}-\frac{2}{n(q-2)}\right)\|\varphi(\tau_{8})\|^{2}_{2}\\
\geq&\frac{1}{2}\left(\frac{1}{2}-\frac{2}{n(q-2)}\right)\|\nabla \varphi(\tau_{8})\|^{2}_{2}+\frac{n}{2^{\frac{1}{2}}}\left(\frac{1}{2}-\frac{2}{n(q-2)}\right)\|\varphi(\tau_{8})\|^{2}_{2}
\\&+\left(\frac{1}{2}-\frac{2}{n(q-2)}\right)\|\varphi(\tau_{8})\|^{2}_{2}.
\end{split}
\end{align}
Meanwhile, according to $(\ref{jdengyu6})$, Cauchy-Schwartz inequality and Lemma \ref{lem3.2}, we have
\begin{align}\label{zhunbeimaodun-8-1}
\begin{split}
&\frac{n(p-2)\lambda_{1}}{2p}\|\varphi(\tau_{8})\|^{p}_{p}+\frac{n(q-2)\lambda_{2}}{2q}\|\varphi(\tau_{8})\|^{q}_{q}\\=&\|\nabla \varphi(\tau_{8})\|^{2}_{2}+\||x|\varphi(\tau_{8})\|^{2}_{2}+\|\varphi(\tau_{8})\|^{2}_{2}\\
\geq&2\|\nabla \varphi(\tau_{8})\|_{2}\||x|\varphi(\tau_{8})\|_{2}+\|\varphi(\tau_{8})\|^{2}_{2}\\
\geq&(n+1)\|\varphi(\tau_{8})\|^{2}_{2}.
\end{split}
\end{align}
According to (\ref{zhunbeimaodun-8-1}), (\ref{GN-budengshi}) and (\ref{GN-budengshi1}), we have
\begin{align*}
(n+1)\|\varphi(\tau_{8})\|^{2}_{2}\leq&\frac{n(p-2)\lambda_{1}}{2p}\|\varphi(\tau_{8})\|^{p}_{p}
+\frac{n(q-2)\lambda_{2}}{2q}\|\varphi(\tau_{8})\|^{q}_{q}\nonumber\\
\leq&\frac{n(p-2)C_{1}\lambda_{1}}{2p}\|\nabla\varphi(\tau_{8})\|^{\frac{n(p-2)}{2}}_{2}\|\varphi(\tau_{8})\|^{p-\frac{n(p-2)}{2}}_{2}\nonumber\\
&+\frac{n(q-2)C_{3}\lambda_{2}}{2q}\|\nabla\varphi(\tau_{8})\|^{\frac{n(q-2)}{2}}_{2}\|\varphi(\tau_{8})\|^{q-\frac{n(q-2)}{2}}_{2},
\end{align*}
i.e.,
\begin{align*}
n+1
\leq&\frac{n(p-2)C_{1}\lambda_{1}}{2p}\|\nabla\varphi(\tau_{8})\|^{\frac{n(p-2)}{2}}_{2}\|\varphi(\tau_{8})\|^{p-\frac{n(p-2)}{2}-2}_{2}\nonumber\\
&+\frac{n(q-2)C_{3}\lambda_{2}}{2q}\|\nabla\varphi(\tau_{8})\|^{\frac{n(q-2)}{2}}_{2}\|\varphi(\tau_{8})\|^{q-\frac{n(q-2)}{2}-2}_{2},
\end{align*}
which means
\begin{align}
\frac{n+1}{2}
\leq\frac{n(p-2)C_{1}\lambda_{1}}{2p}\|\nabla\varphi(\tau_{8})\|^{\frac{n(p-2)}{2}}_{2}\|\varphi(\tau_{8})\|^{p-\frac{n(p-2)}{2}-2}_{2},\label{jieguo7}
\end{align}
i.e.,
\begin{align}
\left(\frac{p(n+1)}{n(p-2)C_{1}\lambda_{1}\|\varphi(\tau_{8})\|^{p-\frac{n(p-2)}{2}-2}_{2}}\right)^{\frac{4}{n(p-2)}}
\leq\|\nabla\varphi(\tau_{8})\|^{2}_{2},\label{jieguo8-p-}
\end{align}
or
\begin{align}
\frac{n+1}{2}
\leq\frac{n(q-2)C_{3}\lambda_{2}}{2q}\|\nabla\varphi(\tau_{8})\|^{\frac{n(q-2)}{2}}_{2}\|\varphi(\tau_{8})\|^{q-\frac{n(q-2)}{2}-2}_{2},\label{jieguo7}
\end{align}
i.e.,
\begin{align}
\left(\frac{q(n+1)}{n(q-2)C_{3}\lambda_{2}\|\varphi(\tau_{8})\|^{q-\frac{n(q-2)}{2}-2}_{2}}\right)^{\frac{4}{n(q-2)}}
\leq\|\nabla\varphi(\tau_{8})\|^{2}_{2}.\label{jieguo8-q-}
\end{align}
According to (\ref{2.2}), (\ref{jieguo8-p-}) and (\ref{jieguo8-q-}) turn to
\begin{align}
\left(\frac{p(n+1)}{n(p-2)C_{1}\lambda_{1}\|\varphi_{0}\|^{p-\frac{n(p-2)}{2}-2}_{2}}\right)^{\frac{4}{n(p-2)}}
\leq\|\nabla\varphi(\tau_{8})\|^{2}_{2}\label{jieguo8-p}
\end{align}
and
\begin{align}
\left(\frac{q(n+1)}{n(q-2)C_{3}\lambda_{2}\|\varphi_{0}\|^{q-\frac{n(q-2)}{2}-2}_{2}}\right)^{\frac{4}{n(q-2)}}
\leq\|\nabla\varphi(\tau_{8})\|^{2}_{2},\label{jieguo8-q}
\end{align}
respectively. From (\ref{jieguo8-p}) and (\ref{jieguo8-q}), we have
\begin{align}
\|\nabla\varphi(\tau_{8})\|^{2}_{2}\geq&N_{4}(\|\varphi_{0}\|_{2}),
\label{jdengyu6-2-2}
\end{align}
where $N_{4}(\|\varphi_{0}\|_{2})$ is defined by (\ref{N4}).
Meanwhile, (\ref{jdengyu6}) also gives
\begin{align}
\frac{n(p-2)\lambda_{1}}{2p}\|\varphi(\tau_{8})\|^{p}_{p}+\frac{n(q-2)\lambda_{2}}{2q}\|\varphi(\tau_{8})\|^{q}_{q}>\|\nabla \varphi(\tau_{8})\|^{2}_{2}.\label{shizi18}
\end{align}By the similar process of giving (\ref{jdengyu5-2-2}), we can use (\ref{shizi18}) to obtain
\begin{align}
\|\nabla\varphi(\tau_{8})\|^{2}_{2}\geq&N_{1}(\|\varphi_{0}\|_{2}),\label{jdengyu5-2-2-8}
\end{align}
where $N_{1}(\|\varphi_{0}\|_{2})$ is defined by (\ref{N1}).
 According to (\ref{jdengyu6-2-2}) and (\ref{jdengyu5-2-2-8}), we have
 \begin{align*}
\|\nabla\varphi(\tau_{8})\|^{2}_{2}\geq&\frac{N_{1}(\|\varphi_{0}\|_{2})+N_{4}(\|\varphi_{0}\|_{2})+\left|N_{1}(\|\varphi_{0}\|_{2})-N_{4}(\|\varphi_{0}
\|_{2})\right|}{2},
\end{align*}which means (\ref{zhunbeimaodun-7-1}) turns to
\begin{align}\label{zhunbeimaodun-7-1-8}
\begin{split}
K(\varphi(\tau_{8}))
>&\frac{1}{2}\left(\frac{1}{2}-\frac{2}{n(q-2)}\right)\frac{N_{1}(\|\varphi_{0}\|_{2})+N_{4}(\|\varphi_{0}\|_{2})+\left|N_{1}(\|\varphi_{0}\|_{2})-N_{4}(\|\varphi_{0}
\|_{2})\right|}{2}
\\&+\frac{n}{2^{\frac{1}{2}}}\left(\frac{1}{2}-\frac{2}{n(q-2)}\right)\|\varphi(\tau_{8})\|^{2}_{2}
+\left(\frac{1}{2}-\frac{2}{n(q-2)}\right)\|\varphi(\tau_{8})\|^{2}_{2}.
\end{split}
\end{align}
According to (\ref{2.2}) and (\ref{2.3}), (\ref{zhunbeimaodun-7-1-8}) turns to
\begin{align*}
K(\varphi_{0})
>&\frac{1}{2}\left(\frac{1}{2}-\frac{2}{n(q-2)}\right)\frac{N_{1}(\|\varphi_{0}\|_{2})+N_{4}(\|\varphi_{0}\|_{2})+\left|N_{1}(\|\varphi_{0}\|_{2})-N_{4}(\|\varphi_{0}\|_{2})\right|}{2}\nonumber\\
&
+\left(\frac{1}{2}-\frac{2}{n(q-2)}+\frac{n}{2^{\frac{1}{2}}}\left(\frac{1}{2}-\frac{2}{n(q-2)}\right)\right)\|\varphi_{0}\|^{2}_{2},
\end{align*}
which contradicts (\ref{K0--detiaoj1}) for $j=8$. Thus, we have $S_{8}(\varphi(t))>0$ for $t\in(0,T_{\rm{max}})$.
\end{proof}
\end{lemma}

Next, we use the invariant manifolds $\mathcal{W}_{R_{j}}$ and $\mathcal{V}_{R_{j}}$, $j=1,2,\cdots,8$, to prove the global existence and finite time blowup, respectively, for high initial energy $K(\varphi_{0})\geq d_{\mathcal{N}_{j}}$, $j=1,2,\cdots,8$.
\begin{theorem}\label{th5.1}(Global existence for high initial energy $K(\varphi_{0})\geq d_{\mathcal{N}_{j}}$, $j=1,2,\cdots,8$)
For $j=1,2,\cdots,8$, let $p$ and $q$ satisfy (\ref{1.2}), and $\lambda_{1}>0$, $\lambda_{2}>0$ satisfy the following conditions

(i)\ for $j=1,2,3,4$, $\lambda_{1}$ and $\lambda_{2}$ additionally satisfy (\ref{2.9}) and (\ref{2.10});

(ii)\ for $j=5,6,7,8$, no additional restrictions on $\lambda_{1}$ and $\lambda_{2}$.\\
 Then the solution $\varphi$ to Cauchy problem (\ref{1.1}) exists globally provided $\varphi_{0}\in \mathcal{W}_{R_{j}}$.\begin{proof}
For the initial data $\varphi_{0}\in \mathcal{W}_{R_{j}}$, $j=1,2,\cdots,8$, we know the corresponding solution $\varphi\in \mathcal{W}_{R_{j}}$, i.e., $S_{j}(\varphi)>0$, according to Lemma \ref{lem5.1}. Then, by the similar process of proving (\ref{4.6}) for $j=1,2,3,4$ and (\ref{4.11}) for $j=5,6,7,8$, we can use the condition $S_{j}(\varphi)>0$, $j=1,2,3,4$, to obtain
\begin{align}\label{shi1}
\begin{split}
&K(\varphi_{0})+\left(\frac{\lambda_{2}}{q}+\left(\frac{\lambda_{1}}{p}+\frac{\lambda_{2}}{q}\right)\frac{2p c}{n(p-2)\lambda_{1}}\right)\|\varphi_{0}\|^{2}_{2}\\
\geq&\left(\frac{1}{2}-\left(\frac{\lambda_{1}}{p}+\frac{\lambda_{2}}{q}\right)\frac{2p}{n(p-2)\lambda_{1}}\right)\| \varphi\|^{2}_{H}
\end{split}\end{align}
and the condition $S_{j}(\varphi)>0$, $j=5,6,7,8$, to obtain
\begin{align}\label{shi2}
K(\varphi_{0})\geq
&\left(\frac{1}{2}-\frac{2}{n(q-2)}\right)\|\varphi\|^{2}_{H},
\end{align}
 where the assumptions (\ref{2.9}) and (\ref{1.2}) ensure
$$\frac{1}{2}-\left(\frac{\lambda_{1}}{p}+\frac{\lambda_{2}}{q}\right)\frac{2p}{n(p-2)\lambda_{1}}>0$$
and $$\frac{1}{2}-\frac{2}{n(q-2)}>0,$$ respectively, $c>0$ is chosen such that (\ref{cdetiaojian}) holds.
From (\ref{shi1}) and (\ref{shi2}), we have
\begin{align}
\|\varphi\|^{2}_{H}
\leq\frac{2qn(p-2)\lambda_{1}}{qn(p-2)\lambda_{1}-4\lambda_{1}q-4\lambda_{2}p} \left(K(\varphi_{0})+\left(\frac{\lambda_{2}}{q}+\left(\frac{\lambda_{1}}{p}+\frac{\lambda_{2}}{q}\right)\frac{2p c}{n(p-2)\lambda_{1}}\right)\|\varphi_{0}\|^{2}_{2}\right)\label{youjie1}
\end{align}for $j=1,2,3,4$
and
\begin{align}
\|\varphi\|^{2}_{H}\leq\frac{2n(q-2)K(\varphi_{0})}{n(q-2)-4}\label{youjie2}
\end{align}for $j=5,6,7,8$, respectively,
 which mean $T_{\rm{max}}=+\infty$ according to Lemma \ref{lem2.1}.
\end{proof}
\end{theorem}

\begin{theorem}\label{th5.2}(Finite time blowup for high initial energy $K(\varphi_{0})\geq d_{\mathcal{N}_{j}}$, $j=1,2,\cdots,8$)
For $j=1,2,\cdots,8$, let $p$ and $q$ satisfy (\ref{1.2}), and $\lambda_{1}>0$, $\lambda_{2}>0$ satisfy the following conditions

(i)\ for $j=1,2,3,4$, $\lambda_{1}$ and $\lambda_{2}$ additionally satisfy (\ref{2.9}) and (\ref{2.10});

(ii)\ for $j=5,6,7,8$, no additional restrictions on $\lambda_{1}$ and $\lambda_{2}$.\\
Then the solution $\varphi$ to Cauchy problem (\ref{1.1}) blows up in finite time provided $\varphi_{0}\in \mathcal{V}_{R_{j}}$.
\begin{proof}

For the initial data $\varphi_{0}\in \mathcal{V}_{R_{j}}$, $j=1,2,\cdots,8$, we know the corresponding solution $\varphi\in \mathcal{V}_{R_{j}}$, i.e., $S_{j}(\varphi)<0$, according to Lemma \ref{lem5.1}. Then, by the similar process of proving Theorem \ref{th4.2}, we can apply the auxiliary functional $J(t)=\||x|\varphi\|^{2}_{2}$ and use the condition $S_{j}(\varphi)<0$, $j=1,2,5,6$, to obtain (\ref{4.17}) and $S_{j}(\varphi)<0$, $j=3,4,7,8$, to give (\ref{4.15}) to complete the proof of $T_{\rm{max}}<+\infty$, which means the finite time blowup of the solution according to Lemma \ref{lem2.1}.
\end{proof}
\end{theorem}

\begin{remark}\label{rem5.2}(The necessity of the condition $K(\varphi_{0})\leq R_{j}(\|\varphi_{0}\|_{2})$, $j=1,2,\cdots,8$, for proving the global existence and finite time blowup of high initial energy case $K(\varphi_{0})\geq d_{{\mathcal{N}_{j}}}$)
When considering the initial conditions to classify the global existence and the finite time blowup of the solution for the high initial energy case $K(\varphi_{0})\geq d_{{\mathcal{N}_{j}}}$, $j=1,2,\cdots,8$, condition $K(\varphi_{0})\leq R_{j}(\|\varphi_{0}\|_{2})$, $j=1,2,\cdots,8$, seems technical but necessary because condition $K(\varphi_{0})\leq R_{j}(\|\varphi_{0}\|_{2})$, $j=1,2,\cdots,8$, plays a crucial role in proving the invariant manifolds for the high initial energy case $K(\varphi_{0})\geq d_{{\mathcal{N}_{j}}}$, $j=1,2,\cdots,8$, including the so-called stable manifold and unstable manifold. The situation here is very different from the bounded non-negative energy case $0\leq K(\varphi_{0})<d_{{\mathcal{N}_{j}}}$, $j=1,2,\cdots,8$, for which there is a natural contradiction between the sign-change of $S_{j}(\varphi)$ and the assumption $K(\varphi_{0})<d_{{\mathcal{N}_{j}}}$, $j=1,2,\cdots,8$, however, such contradiction will not happen for the high initial energy case $K(\varphi_{0})\geq d_{{\mathcal{N}_{j}}}$, $j=1,2,\cdots,8$, anymore. Hence, we find this condition $K(\varphi_{0})\leq R_{j}(\|\varphi_{0}\|_{2})$, $j=1,2,\cdots,8$, to retrieve a new contradiction to finish the proof of the invariant manifolds.

\end{remark}

\begin{remark}\label{Manifoldsfeikong}(The examples of the initial data in the invariant manifolds $\mathcal{W}_{R_{j}}$ and $\mathcal{V}_{R_{j}}$, $j=1,2,\cdots,8$, used to prove the global existence and the finite time blowup, respectively, for high initial energy $K(\varphi_{0})\geq d_{{\mathcal{N}_{j}}}$, $j=1,2,\cdots,8$)
For high initial energy $K(\varphi_{0})\geq d_{{\mathcal{N}_{j}}}$, $j=1,2,\cdots,8$, in order to prove the global existence and the finite time blowup, we need the invariant manifolds $\mathcal{W}_{R_{j}}$ and $\mathcal{V}_{R_{j}}$, $j=1,2,\cdots,8$, respectively, as shown in Theorem \ref{th5.1} and Theorem \ref{th5.2}, respectively. Next, we give some examples of the initial data in the invariant manifolds $\mathcal{W}_{R_{j}}$ and $\mathcal{V}_{R_{j}}$, $j=1,2,\cdots,8$, to show that these two manifolds are not empty, which means the conditions in Theorem \ref{th5.1} and Theorem \ref{th5.2} make sense. We first focus on the case $j=1$ .\\
(1)\ Example of the initial data belonging to the manifold $\mathcal{W}_{R_{1}}$.

Find a function $\zeta_{1}\in H$ satisfying $\zeta_{1}\neq0$ and
\begin{align}
d_{_{\mathcal{N}_{1}}}< \frac{1}{2}\left(\frac{1}{2}-\left(\frac{\lambda_{1}}{p}+\frac{\lambda_{2}}{q}\right)\frac{2p}{n(p-2)\lambda_{1}}\right)M_{1}(\|\zeta_{1}\|_{2}),\label{hanshu}
 \end{align}where $M_{1}(\|\zeta_{1}\|_{2})$ is defined by (\ref{M1}).
 Here, we shall use functions $\zeta_{j}$, $j=1,2,\cdots,8$, to find the examples of the initial data belonging to $\mathcal{W}_{R_{j}}$ and $\mathcal{V}_{R_{j}}$, $j=1,2,\cdots,8$, respectively. Hence for the case $j=1$, we use the lower index $1$ to denote this function $\zeta_{1}$.
 For a constant $\sigma>0$, define $$w(x):=\sigma^{\frac{n}{2}}\zeta_{1}(\sigma x).$$ And we denote $$y=\left(y_{1},y_{2},\cdots,y_{n}\right):=\sigma x=\left(\sigma x_{1},\sigma x_{2},\cdots,\sigma x_{n}\right)$$ and
\begin{align*}
\nabla_{y}:=\left(\frac{\partial }{\partial y_{1}},\frac{\partial }{\partial y_{2}},\cdots,\frac{\partial }{\partial y_{n}}\right).
\end{align*}
Then we have
\begin{align}
\|\nabla w\|^{2}_{2}=&\left\|\nabla\left(\sigma^{\frac{n}{2}}\zeta_{1}(\sigma x)\right)\right\|^{2}_{2}
=\int_{\mathbb{R}^n} \sigma^{n}|\nabla\zeta_{1}(\sigma x)|^{2}{\rm d}x
=\int_{\mathbb{R}^n} |\nabla_{y}\zeta_{1}(y)|^{2}\sigma^{2}{\rm d}y
=\sigma^{2}\|\nabla\zeta_{1}\|^{2}_{2},\label{daihuan}
\end{align}
\begin{align}\label{daihuanw--}
\||x|w\|^{2}_{2}=&\int_{\mathbb{R}^n} \frac{\sigma^{n}}{\sigma^{2}}||\sigma x|\zeta_{1}(\sigma x)|^{2}{\rm d}x
=\int_{\mathbb{R}^n} \frac{1}{\sigma^{2}}||y|\zeta_{1}(y)|^{2}{\rm d}y
=\frac{1}{\sigma^{2}}\||x|\zeta_{1}\|^{2}_{2},
\end{align}
\begin{align}
\|w\|^{2}_{2}=&\int_{\mathbb{R}^n} \sigma^{n}|\zeta_{1}(\sigma x)|^{2}{\rm d}x
=\int_{\mathbb{R}^n} |\zeta_{1}(y)|^{2}{\rm d}y
=\|\zeta_{1}\|^{2}_{2},\label{daihuan2}
\end{align}
and
\begin{align}
\|w\|^{p}_{p}=\int_{\mathbb{R}^n} \left|\sigma^{\frac{n}{2}}\zeta_{1}(\sigma x)\right|^{p}{\rm d}x
=\int_{\mathbb{R}^n} \frac{\sigma^{\frac{pn}{2}}}{\sigma^{n}}|\zeta_{1}(y)|^{p}{\rm d}y
=\sigma^{\frac{n(p-2)}{2}}\|\zeta_{1}\|^{p}_{p}.\label{daihuan3}
\end{align}
Then, we have
\begin{align}\label{S1w}
\begin{split}
S_{1}(w)=&\|\nabla w\|_{2}^{2}-\frac{n(p-2)\lambda_{1}}{2p}\| w\|_{p}^{p}\\
=&\sigma^{2}\|\nabla\zeta_{1}\|^{2}_{2}-\frac{n(p-2)\lambda_{1}}{2p}\sigma^{\frac{n(p-2)}{2}}\|\zeta_{1}\|^{p}_{p}\\
=&\sigma^{2}\left(\|\nabla\zeta_{1}\|^{2}_{2}-\left(\frac{n(p-2)\lambda_{1}}{2p}\|\zeta_{1}\|^{p}_{p}\right)\sigma^{\frac{n(p-2)}{2}-2}\right).
\end{split}
\end{align}
The assumption (\ref{1.2}), i.e., $2+\frac{4}{n}<p$, ensures
$$\frac{n(p-2)}{2}-2>0.$$
Thus, there exists a unique $\sigma_{0}>0$ such that $S_{1}(\sigma_{0}^{\frac{n}{2}}\zeta_{1}(\sigma_{0} x))=0$, $S_{1}(\sigma^{\frac{n}{2}}\zeta_{1}(\sigma x))>0$ for $0<\sigma<\sigma_{0}$ and $S_{1}(\sigma^{\frac{n}{2}}\zeta_{1}(\sigma x))<0$ for $\sigma>\sigma_{0}$. We denote \begin{align}\label{w0x}
w_{0}(x):=\sigma_{0}^{\frac{n}{2}}\zeta_{1}(\sigma_{0} x).
\end{align}
Due to $\|w_{0}\|^{2}_{2}=\|\zeta_{1}\|^{2}_{2}$ given by (\ref{daihuan2}), we know $w_{0}$ also satisfies (\ref{hanshu}).
And from $S_{1}(w_{0})=0$, by similar process of dealing with $j=1$ in the proof of Lemma \ref{lem5.1}, we have
\begin{align}
K(w_{0})> R_{1}(\|w_{0}\|_{2}).\label{Kw0}\end{align}
For $s>0$, we have
\begin{align}\label{S1sw}
\begin{split}
S_{1}(sw_{0})=&\left(\|\nabla w_{0}\|_{2}^{2}\right)s^{2}-\left(\frac{n(p-2)\lambda_{1}}{2p}\| w_{0}\|_{p}^{p}\right)s^{p}\\
=&s^{2}\left(\|\nabla w_{0}\|^{2}_{2}-\left(\frac{n(p-2)\lambda_{1}}{2p}\|w_{0}\|^{p}_{p}\right)s^{p-2}\right).
\end{split}
\end{align}
Then, according to $S_{1}(w_{0})=0$, we know $S_{1}(sw_{0})>0$ for $0<s<1$ and $S_{1}(sw_{0})<0$ for $s>1$. Meanwhile, we have
\begin{align}\label{Ks-}
\begin{split}
K(sw_{0})=&\left(\frac{1}{2}\|\nabla w_{0}\|_{2}^{2}+\frac{1}{2}\|w_{0}\|_{2}^{2}+\frac{1}{2}\||x|w_{0}\|_{2}^{2}\right)s^{2}-\left(\frac{\lambda_{1}}{p}\| w_{0}\|_{p}^{p}\right)s^{p}-\left(\frac{\lambda_{2}}{q}\| w_{0}\|_{q}^{q}\right)s^{q}\\
=&s^{2}\left(\frac{1}{2}\|\nabla w_{0}\|_{2}^{2}+\frac{1}{2}\|w_{0}\|_{2}^{2}+\frac{1}{2}\||x|w_{0}\|_{2}^{2}-\left(\frac{\lambda_{1}}{p}\| w_{0}\|_{p}^{p}\right)s^{p-2}
-\left(\frac{\lambda_{2}}{q}\| w_{0}\|_{q}^{q}\right)s^{q-2}\right).
\end{split}\end{align}
From $(\ref{Ks-})$, for $s=0$, we know $K(s w_{0})=0$. Combining this fact, (\ref{Kw0}) and the continuity of $K(s w_{0})$ with respect to $s$, we know there exists an $\bar{s}\in(0,1)$ such that \begin{align*}
d_{\mathcal{N}_{1}}\leq K(\bar{s}w_{0})\leq & \frac{1}{2}\left(\frac{1}{2}-\left(\frac{\lambda_{1}}{p}+\frac{\lambda_{2}}{q}\right)\frac{2p}{n(p-2)\lambda_{1}}\right)\left(\frac{2p}{n(p-2)\lambda_{1} C_{1}\|w_{0}\|^{p-\frac{n(p-2)}{2}}_{2}}\right)^{\frac{4}{n(p-2)-4}}\nonumber\\
\leq& \frac{1}{2}\left(\frac{1}{2}-\left(\frac{\lambda_{1}}{p}+\frac{\lambda_{2}}{q}\right)\frac{2p}{n(p-2)\lambda_{1}}\right)\left(\frac{2p}{n(p-2) \lambda_{1}C_{1}\bar{s}^{p-\frac{n(p-2)}{2}}\|w_{0}\|^{p-\frac{n(p-2)}{2}}_{2}}\right)^{\frac{4}{n(p-2)-4}}\nonumber\\
=&\frac{1}{2}\left(\frac{1}{2}-\left(\frac{\lambda_{1}}{p}+\frac{\lambda_{2}}{q}\right)\frac{2p}{n(p-2)\lambda_{1}}\right)\left(\frac{2p}{n(p-2) \lambda_{1}C_{1}\|\bar{s}w_{0}\|^{p-\frac{n(p-2)}{2}}_{2}}\right)^{\frac{4}{n(p-2)-4}}\nonumber\\
<&R_{1}(\|\bar{s}w_{0}\|_{2})\end{align*}
due to \begin{align*}
p-\frac{n(p-2)}{2}>0
\end{align*} given by (\ref{1.2}).
Meanwhile, noticing (\ref{S1sw}) and $S_{1}(w_{0})=0$, we know $S_{1}(\bar{s}w_{0})>0$ due to $0<\bar{s}<1$. Thus, we have $\bar{s}w_{0}\in\mathcal{W}_{R_{1}}$.\\
(2)\ Example of the initial data belonging to the manifold $\mathcal{V}_{R_{1}}$.

 For the function $\zeta_{1}$ satisfying (\ref{hanshu}), by (\ref{daihuan})-(\ref{daihuan3}) and
 \begin{align*}
\|w\|^{q}_{q}=\int_{\mathbb{R}^n} \left|\sigma^{\frac{n}{2}}\zeta_{1}(\sigma x)\right|^{q}{\rm d}x
=\int_{\mathbb{R}^n} \frac{\sigma^{\frac{qn}{2}}}{\sigma^{n}}|\zeta_{1}(y)|^{q}{\rm d}y
=\sigma^{\frac{n(q-2)}{2}}\|\zeta_{1}\|^{q}_{q},
\end{align*}
 we have
\begin{align}\label{Ks}
\begin{split}
K(\sigma^{\frac{n}{2}}\zeta_{1}(\sigma x))=&\left(\frac{1}{2}\|\nabla \zeta_{1}\|_{2}^{2}\right)\sigma^{2}+\frac{1}{2}\|\zeta_{1}\|_{2}^{2}+\left(\frac{1}{2}\||x|\zeta_{1}\|_{2}^{2}\right)\sigma^{-2}-\left(\frac{\lambda_{1}}{p}\| \zeta_{1}\|_{p}^{p}\right)\sigma^{\frac{n(p-2)}{2}}\\&-\left(\frac{\lambda_{2}}{q}\|\zeta_{1}\|_{q}^{q}\right)\sigma^{\frac{n(q-2)}{2}}\\
=&\sigma^{2}\left(\frac{1}{2}\|\nabla \zeta_{1}\|_{2}^{2}+\left(\frac{1}{2}\|\zeta_{1}\|_{2}^{2}\right)\sigma^{-2}+\left(\frac{1}{2}\||x|\zeta_{1}\|_{2}^{2}\right)\sigma^{-4}
-\left(\frac{\lambda_{1}}{p}\| \zeta_{1}\|_{p}^{p}\right)\sigma^{\frac{n(p-2)-4}{2}}
\right.\\
&\left.-\left(\frac{\lambda_{2}}{q}\|\zeta_{1}\|_{q}^{q}\right)\sigma^{\frac{n(q-2)-4}{2}}\right)
.
\end{split}
\end{align}
Thus, we can find large enough $\sigma_{1}>\sigma_{0}$ such that $K(\sigma_{1}^{\frac{n}{2}}\zeta_{1}(\sigma_{1} x))<0$, where $\sigma_{0}$ satisfies $S_{1}(\sigma_{0}^{\frac{n}{2}}\zeta_{1}(\sigma_{0} x))=0$. Meanwhile, as the analysis below (\ref{w0x}), the function $\sigma_{0}^{\frac{n}{2}}\zeta_{1}(\sigma_{0} x)$ also satisfies (\ref{hanshu}), which gives $d_{\mathcal{N}_{1}}<R_{1}(\|\sigma_{0}^{\frac{n}{2}}\zeta_{1}(\sigma_{0} x)\|_{2})$. Thus, we know that there exists a $\sigma^{*}\in(\sigma_{0},\sigma_{1})$ such that
\begin{align*}
d_{\mathcal{N}_{1}}\leq K(\left(\sigma^{*}\right)^{\frac{n}{2}}\zeta_{1}(\sigma^{*} x))\leq R_{1}(\|\sigma_{0}^{\frac{n}{2}}\zeta_{1}(\sigma_{0} x)\|_{2})=R_{1}(\|\left(\sigma^{*}\right)^{\frac{n}{2}}\zeta_{1}(\sigma^{*} x)\|_{2})\end{align*}according to the continuity of $K(\sigma^{\frac{n}{2}}\zeta_{1}(\sigma x))$ with respect to $\sigma$, (\ref{Kw0}) and (\ref{daihuan2}). And by observing (\ref{S1w}), from $\sigma^{*}>\sigma_{0}$ and $S_{1}(\sigma_{0}^{\frac{n}{2}}\zeta_{1}(\sigma_{0} x))=0$, we have $S_{1}(\left(\sigma^{*}\right)^{\frac{n}{2}}\zeta_{1}(\sigma^{*} x))<0$. Then, we know
$\left(\sigma^{*}\right)^{\frac{n}{2}}\zeta_{1}(\sigma^{*} x)\in\mathcal{V}_{R_{1}}$.\\
(3)\ Examples of the initial data belonging to the manifolds $\mathcal{W}_{R_{j}}$ and $\mathcal{V}_{R_{j}}$, $j=2,3,4$.

We find functions $\zeta_{j}\in H$, $j=2,3,4$, satisfying $\zeta_{j}\neq0$ and
\begin{align}
d_{_{\mathcal{N}_{j}}}< \frac{1}{2}\left(\frac{1}{2}-\left(\frac{\lambda_{1}}{p}+\frac{\lambda_{2}}{q}\right)\frac{2p}{n(p-2)\lambda_{1}}\right)M_{1}(\|\zeta_{j}\|_{2}),\ j=2,3,4,\label{hanshu----}
 \end{align}
 where $M_{1}(\|\zeta_{j}\|_{2})$ is defined by (\ref{M1}).
We shall use functions $\zeta_{j}$, $j=2,3,4$, to deal with the cases $j=2,3,4$, respectively. We first treat the case $j=2$.
Due to
\begin{align}\label{xiaoyuzuixiaozhi}
M_{1}(\|\zeta_{2}\|_{2})\leq\frac{M_{1}(\|\zeta_{2}\|_{2})+M_{2}(\|\zeta_{2}\|_{2})
+\left|M_{1}(\|\zeta_{2}\|_{2})-M_{2}(\|\zeta_{2}\|_{2})\right|}{2},
\end{align}
where $M_{2}(\|\zeta_{j}\|_{2})$ is defined by (\ref{M2}),
we know (\ref{hanshu----}) gives
\begin{align*}
d_{_{\mathcal{N}_{2}}}<\frac{1}{2}\left(\frac{1}{2}-\left(\frac{\lambda_{1}}{p}+\frac{\lambda_{2}}{q}\right)\frac{2p}{n(p-2)\lambda_{1}}\right)\frac{M_{1}(\|\zeta_{2}\|_{2})+M_{2}(\|\zeta_{2}\|_{2})
+\left|M_{1}(\|\zeta_{2}\|_{2})-M_{2}(\|\zeta_{2}\|_{2})\right|}{2}.
\end{align*}
By the similar arguments for the manifolds $\mathcal{W}_{R_{1}}$ and $\mathcal{V}_{R_{1}}$, we can use the above function $\zeta_{2}$ to show the examples of the initial data belonging to the manifolds $\mathcal{W}_{R_{2}}$ and $\mathcal{V}_{R_{2}}$.

 By the similar process of dealing with the case $j=2$, we can use the functions $\zeta_{j}$, $j=3,4$,
 to show the examples of the initial data belonging to the manifolds $\mathcal{W}_{R_{j}}$ and $\mathcal{V}_{R_{j}}$, $j=3,4$.\\
(4)\ Examples of the initial data belonging to the manifolds $\mathcal{W}_{R_{5}}$ and $\mathcal{V}_{R_{5}}$.

 For the case $j=5$, we find a function $\zeta_{5}\in H$ satisfying $\zeta_{5}\neq0$,
\begin{align}
d_{_{\mathcal{N}_{5}}}< \frac{1}{2}\left(\frac{1}{2}-\frac{2}{n(q-2)}\right)A_{1}(\|\zeta_{5}\|_{2}),\label{hanshu---0-}
 \end{align}
 and
 \begin{align}
d_{_{\mathcal{N}_{5}}}< \frac{1}{2}\left(\frac{1}{2}-\frac{2}{n(q-2)}\right)B_{1}(\|\zeta_{5}\|_{2}),\label{hanshu---0-1}
 \end{align}
 where $A_{1}(\|\zeta_{5}\|_{2})$ and $B_{1}(\|\zeta_{5}\|_{2})$ are defined by (\ref{A1}) and (\ref{B1}), respectively,
\begin{align*}
p-\frac{n(p-2)}{2}>0,
\end{align*} and
\begin{align*}
q-\frac{n(q-2)}{2}>0
\end{align*}are given by (\ref{1.2}).
From (\ref{hanshu---0-}) and (\ref{hanshu---0-1}), we have
\begin{align}
d_{_{\mathcal{N}_{5}}}< \frac{1}{2}\left(\frac{1}{2}-\frac{2}{n(q-2)}\right)N_{1}(\|\zeta_{5}\|_{2}),\label{hanshu---0-1-}
 \end{align}
 where
$N_{1}(\|\zeta_{5}\|_{2})$ is defined by (\ref{N1}).
Then, by the similar arguments for the manifolds $\mathcal{W}_{R_{1}}$ and $\mathcal{V}_{R_{1}}$, we can use the above function $\zeta_{5}$ to show the examples of the initial data belonging to the manifolds $\mathcal{W}_{R_{5}}$ and $\mathcal{V}_{R_{5}}$.\\
(5)\ Examples of the initial data belonging to the manifolds $\mathcal{W}_{R_{j}}$ and $\mathcal{V}_{R_{j}}$, $j=6,7,8$.

We find functions $\zeta_{j}\in H$, $j=6,7,8$, satisfying $\zeta_{j}\neq0$ and
\begin{align}
d_{_{\mathcal{N}_{j}}}< \frac{1}{2}\left(\frac{1}{2}-\frac{2}{n(q-2)}\right) N_{1}(\|\zeta_{j}\|_{2}),\ j=6,7,8,\label{hanshu----==}
 \end{align}
where $N_{1}(\|\zeta_{j}\|_{2})$ is defined by (\ref{N1}).
We shall use functions $\zeta_{j}$, $j=6,7,8$, to deal with the cases $j=6,7,8$, respectively. We first treat the case $j=6$.
Due to
\begin{align*}
N_{1}(\|\zeta_{6}\|_{2})\leq\frac{N_{1}(\|\zeta_{6}\|_{2})+N_{2}(\|\zeta_{6}\|_{2})
+\left|N_{1}(\|\zeta_{6}\|_{2})-N_{2}(\|\zeta_{6}\|_{2})\right|}{2},
\end{align*}
where $N_{2}(\|\zeta_{6}\|_{2})$ is defined by (\ref{N2}).
we know (\ref{hanshu----==}) gives
\begin{align*}
d_{_{\mathcal{N}_{6}}}<\frac{1}{2}\left(\frac{1}{2}-\frac{2}{n(q-2)}\right)\frac{N_{1}(\|\zeta_{6}\|_{2})+N_{2}(\|\zeta_{6}\|_{2})
+\left|N_{1}(\|\zeta_{6}\|_{2})-N_{2}(\|\zeta_{6}\|_{2})\right|}{2}.
\end{align*}
By the similar arguments of the manifolds $\mathcal{W}_{R_{1}}$ and $\mathcal{V}_{R_{1}}$, we can use the above function $\zeta_{6}$ to show the examples of the initial data belonging to the manifolds $\mathcal{W}_{R_{6}}$ and $\mathcal{V}_{R_{6}}$.

By the similar process of dealing with the case $j=6$, we can use the functions $\zeta_{j}$, $j=7,8$, to show the examples of the initial data belonging to the manifolds $\mathcal{W}_{R_{j}}$ and $\mathcal{V}_{R_{j}}$, $j=7,8$.
\end{remark}

\subsection{Some imperfect criteria of global existence for high initial energy $K(\varphi_{0})\geq d_{{\mathcal{N}_{j}}}$, $j=9,10,\cdots,14$}\label{sec5.2}

In Subsection \ref{sec5.1}, we use the eight cross-constrained variational functionals $S_{j}(\varphi)$, $j=1,2,\cdots,8$, with both the
mandatory terms $\|\nabla\varphi\|^{2}_{2}$ and $\|\varphi\|^{p}_{p}$ to give the eight classifications of the initial data to distinguish the global solution and the finite time blowup solution for high initial energy $K(\varphi_{0})\geq d_{{\mathcal{N}_{j}}}$, $j=1,2,\cdots,8$.
In this subsection, we focus on the cross-constrained variational functionals $S_{j}(\varphi)$, $j=9,10,\cdots,14$, without the mandatory term $\|\nabla\varphi\|^{2}_{2}$ to give six global existence criteria for $K(\varphi_{0})\geq d_{{\mathcal{N}_{j}}}$, $j=9,10,\cdots,14$. In Remark \ref{rem5.3}, we shall point out that we fail to prove the finite time blowup by $S_{j}(\varphi)$, $j=9,10,\cdots,14$, for $K(\varphi_{0})\geq d_{{\mathcal{N}_{j}}}$, $j=9,10,\cdots,14$. Especially in Remark \ref{rem5.4}, we explain that we can neither prove the global existence nor the finite time blowup for $K(\varphi_{0})\geq d_{{\mathcal{N}_{j}}}$, $j=15,16,17,18$, due to the lack of the mandatory term $\|\varphi\|^{p}_{p}$ in the cross-constrained variational functional $S_{j}(\varphi)$, $j=15,16,17,18$.

Next, we focus on high initial energy $K(\varphi_{0})\geq d_{\mathcal{N}_{j}}$, $j=9,10,\cdots,14$, to present six global existence criteria. More precisely, for $j=9,10,\cdots,14$, we define the manifold $\mathcal{W}_{R_{j}}$, $j=9,10,\cdots,14$, by (\ref{WR}), i.e.,
\begin{align}\label{5.119}
\mathcal{W}_{R_{j}}=\left\{\varphi\in H\Big|d_{_{\mathcal{N}_{j}}}\leq K(\varphi)\leq R_{j}\left(\|\varphi\|_{2}\right), S_{j}(\varphi)>0\right\},\ j=9,10,\cdots,14,
\end{align}
where
\begin{align}\label{L9}
\begin{split}
R_{9}(\|\varphi\|_{2}):=&\left(\frac{n}{2^{\frac{3}{2}}}+\frac{1}{2}-\left(\frac{\lambda_{1}}{p}+\frac{\lambda_{2}}{q}\right)\frac{2p c}{n(p-2)\lambda_{1}}-\frac{\lambda_{2}}{q}\right)\|\varphi\|^{2}_{2}\\
&+\frac{1}{4}\left(\frac{2p c}{n(p-2)\lambda_{1}C_{1}\|\varphi\|^{p-\frac{n(p-2)}{2}-2}_{2}}\right)^{\frac{4}{n(p-2)}},
\end{split}
\end{align}

\begin{align*}
R_{10}(\|\varphi\|_{2}):=&\left(\frac{n}{2}\left(\frac{1}{2}-\left(\frac{\lambda_{1}}{p}+\frac{\lambda_{2}}{q}\right)\frac{2p}{n(p-2)\lambda_{1}}\right)^{\frac{1}{2}}+\frac{1}{2}
-\frac{\lambda_{2}}{q}\right)\|\varphi\|^{2}_{2}\nonumber\\
&+\frac{1}{4}\left(\frac{np}{2(p-2)\lambda_{1}C_{1}\|\varphi\|^{p-\frac{n(p-2)}{2}-4}_{2}}\right)
^{\frac{4}{n(p-2)+4}}
,\end{align*}

\begin{align*}
R_{11}(\|\varphi\|_{2}):=&\left(\frac{n}{2}\left(\frac{1}{2}-\left(\frac{\lambda_{1}}{p}+\frac{\lambda_{2}}{q}\right)\frac{2p}{n(p-2)\lambda_{1}}\right)
^{\frac{1}{2}}+\frac{1}{2}-\left(\frac{\lambda_{1}}{p}+\frac{\lambda_{2}}{q}\right)\frac{2p c}{n(p-2)\lambda_{1}}-\frac{\lambda_{2}}{q}\right)\|\varphi\|^{2}_{2}\nonumber\\
&+\frac{1}{4}\left(\frac{2p c}{n(p-2)\lambda_{1}C_{1}\|\varphi\|^{p-\frac{n(p-2)}{2}-2}_{2}}\right)^{\frac{4}{n(p-2)}},
\end{align*}

\begin{align}\label{L12}
R_{12}(\|\varphi\|_{2}):=\frac{A_{5}(\|\varphi\|_{2})+B_{5}(\|\varphi\|_{2})-\left|A_{5}(\|\varphi\|_{2})-B_{5}(\|\varphi\|_{2})\right|}{8}
+\left(\frac{n}{2^{\frac{3}{2}}}+\frac{1}{2}-\frac{2}{n(q-2)}\right)\|\varphi\|^{2}_{2},
\end{align}

\begin{align*}
R_{13}(\|\varphi\|_{2}):=&\frac{A_{6}(\|\varphi\|_{2})+B_{6}(\|\varphi\|_{2})-\left|A_{6}(\|\varphi\|_{2})-B_{6}(\|\varphi\|_{2})\right|}{8}
\\&+\left(\frac{1}{2}+\frac{n}{2}\left(\frac{1}{2}-\frac{2}{n(q-2)}\right)^{\frac{1}{2}}\right)\|\varphi\|^{2}_{2},
\end{align*}

\begin{align*}
R_{14}(\|\varphi\|_{2}):=
&\frac{A_{5}(\|\varphi\|_{2})+B_{5}(\|\varphi\|_{2})-
\left|A_{5}(\|\varphi\|_{2})-B_{5}(\|\varphi\|_{2})\right|}{8}
\nonumber\\&+\left(\frac{n}{2}\left(\frac{1}{2}-\frac{2}{n(q-2)}\right)^{\frac{1}{2}}+\frac{1}{2}-\frac{2}{n(q-2)}\right)\|\varphi\|^{2}_{2},
\end{align*}

\begin{align}\label{A5}
A_{5}(\|\varphi\|_{2}):=\left(\frac{p }{n(p-2)C_{1}\lambda_{1}\|\varphi\|^{p-\frac{n(p-2)}{2}-2}_{2}}\right)^{\frac{4}{n(p-2)}},
\end{align}
\begin{align}\label{B5}
B_{5}(\|\varphi\|_{2}):=\left(\frac{q }{n(q-2)C_{3}\lambda_{2}\|\varphi\|^{q-\frac{n(q-2)}{2}-2}_{2}}\right)^{\frac{4}{n(q-2)}},
\end{align}
\begin{align}\label{A6}
A_{6}(\|\varphi\|_{2}):=\left(\frac{np}{4(p-2)\lambda_{1}C_{1}\|\varphi\|^{p-\frac{n(p-2)}{2}-4}_{2}}\right)^{\frac{4}{n(p-2)+4}},
\end{align}and
\begin{align}\label{B6}
B_{6}(\|\varphi\|_{2}):=\left(\frac{nq}{4(q-2)\lambda_{2}C_{3}\|\varphi\|^{q-\frac{n(q-2)}{2}-4}_{2}}\right)^{\frac{4}{n(q-2)+4}}
.
\end{align}
By the similar analysis in Remark $\ref{Manifoldsfeikong}$, we know that the manifold $\mathcal{W}_{R_{j}}$, $j=9,10,\cdots,14$, is not empty.

Next,
we shall use the manifold $\mathcal{W}_{R_{j}}$, $j=9,10,\cdots,14$, to prove the global existence for $K(\varphi_{0})\geq d_{_{\mathcal{N}_{j}}}$, $j=9,10,\cdots,14$, in following Theorem $\ref{dinglis9--14}$. For the case $K(\varphi_{0})<d_{_{\mathcal{N}_{j}}}$, $j=9,10,\cdots,14$, the global existence is proved in Theorem $\ref{dinglis9-14}$, and the conditions there ensure that $d_{_{\mathcal{N}_{j}}}>0$, $j=9,10,\cdots,14$. As in Theorem $\ref{dinglis9--14}$, we also need the conclusion $d_{_{\mathcal{N}_{j}}}>0$, $j=9,10,\cdots,14$, we naturally
take the corresponding the conditions of Theorem $\ref{dinglis9-14}$ into the following Theorem $\ref{dinglis9--14}$.

\begin{theorem}\label{dinglis9--14}(Global existence criteria for high initial energy $K(\varphi_{0})\geq d_{\mathcal{N}_{j}}$, $j=9,10,\cdots,14$)
Let $n$, $p$, $q$, $\lambda_{1}>0$, $\lambda_{2}>0$ satisfy the following conditions

(i)\ for $j=9,11$: $n=1,2$, $p$, $q$ satisfy (\ref{1.2}), $\lambda_{1}$ and $\lambda_{2}$ additionally satisfy (\ref{2.9}) and (\ref{2.10});

(ii)\ for $j=10$: $n=1$, $p>q\geq6$, $\lambda_{1}$ and $\lambda_{2}$ additionally satisfy (\ref{2.9}) and (\ref{2.10});

(iii)\ for $j=12,14$: $n=1,2$, $p$, $q$ satisfy (\ref{1.2}), with no additional restrictions on $\lambda_{1}$ and $\lambda_{2}$;

(iv)\ for $j=13$: $n=1$, $p>q\geq6$, with no additional restrictions on $\lambda_{1}$ and $\lambda_{2}$.\\
Then
the corresponding solution $\varphi$ to problem (\ref{1.1}) exists globally in time provided $\varphi_{0}\in \mathcal{W}_{R_{j}}$, $j=9,10,\cdots,14$.

\end{theorem}

\begin{proof}
For the initial data $\varphi_{0}\in\mathcal{W}_{R_{j}}$, $j=9,10,\cdots,14$, we prove $T_{\rm{max}}=+\infty$. First, we need to prove the invariance of the manifold $\mathcal{W}_{R_{j}}$, $j=9,10,\cdots,14$. Thus, we divide the proof into the following two steps, i.e., Step \textrm{I} (the manifold $\mathcal{W}_{R_{j}}$, $j=9,10,\cdots,14$, is invariant) and Step \textrm{II} (global existence).\\
\textbf{Step \textrm{I}:\ The manifold \bm{$\mathcal{W}_{R_{j}}$}, \bm{$j=9,10,\cdots,14$}, is invariant.}

For the initial data $\varphi_{0}\in\mathcal{W}_{R_{j}}$, $j=9,10,\cdots,14$, we know that there exists a unique solution $\varphi\in C([0,T_{\rm{max}}),H)$ to the Cauchy problem (\ref{1.1}) according to Lemma \ref{lem2.1}, which means the continuity of $S_{j}(\varphi(t))$, $j=9,10,\cdots,14$, with respect to time $t$. For $\varphi_{0}\in\mathcal{W}_{R_{j}}$, $j=9,10,\cdots,14$, we know
\begin{align}
d_{_{\mathcal{N}_{j}}}\leq K(\varphi_{0})\leq R_{j}(\|\varphi_{0}\|_{2}),\ j=9,10,\cdots,14,\label{K0detiaoj1}
\end{align}
and $S_{j}(\varphi_{0})>0$, $j=9,10,\cdots,14$. According to (\ref{K0detiaoj1}), (\ref{2.2}) and (\ref{2.3}), we have
\begin{align}
d_{_{\mathcal{N}_{j}}}\leq K(\varphi(t))\leq R_{j}(\|\varphi(t)\|_{2}),\ j=9,10,\cdots,14.\label{Kfaitdayud}
\end{align}
Next, we only need to prove $S_{j}(\varphi(t))>0$, $j=9,10,\cdots,14$, for $t\in(0,T_{\rm{max}})$, which means $\varphi(t)\in\mathcal{W}_{R_{j}}$, $j=9,10,\cdots,14$, for all $t\in(0,T_{\rm{max}})$.
Here, the corresponding strategies depend on the structures of the cross-constrained variational functionals $S_{j}(\varphi)$, $j=9,10,\cdots,14$. We need to use six different strategies to deal with these six cases $j=9,10,\cdots,14$, respectively. Thus, we treat these six cases one by one.\\
{\bf{Case \textrm{I}: \bm{$j=9.$}}}

We prove $S_{9}(\varphi(t))>0$ for all $t\in(0,T_{\rm{max}})$. Arguing by contradiction, we suppose that there exists the first time $\tau_{9}\in(0,T_{\rm{max}})$ to make $S_{9}(\varphi(\tau_{9}))=0$ and $S_{9}(\varphi(t))>0$ for $t\in(0,\tau_{9})$, and we know $\varphi(\tau_{9})\neq 0$ from (\ref{2.2}) and $\varphi_{0}\neq0$.
Next, we estimate $K(\varphi(\tau_{9}))$.
According to $S_{9}(\varphi(\tau_{9}))=0$, we have
\begin{align}
\|\varphi(\tau_{9})\|^{p}_{p}=&\frac{2p c}{n(p-2)\lambda_{1}}\|\varphi(\tau_{9})\|^{2}_{2}.\label{p-gujij9}
\end{align}
Substituting (\ref{p-gujij9}) into (\ref{2.16}), we have
\begin{align}\label{p-9zhunbeimaodun+1}
\begin{split}
K(\varphi(\tau_{9}))>&\frac{1}{2}\|\nabla \varphi(\tau_{9})\|^{2}_{2}+\frac{1}{2}\||x|\varphi(\tau_{9})\|^{2}_{2}\\
&+\left(\frac{1}{2}-\left(\frac{\lambda_{1}}{p}+\frac{\lambda_{2}}{q}\right)\frac{2p c}{n(p-2)\lambda_{1}}-\frac{\lambda_{2}}{q}\right)\|\varphi(\tau_{9})\|^{2}_{2}\\
=&\frac{1}{4}\|\nabla \varphi(\tau_{9})\|^{2}_{2}+\frac{1}{2}\||x|\varphi(\tau_{9})\|^{2}_{2}+\frac{1}{4}\|\nabla \varphi(\tau_{9})\|^{2}_{2}\\
&+\left(\frac{1}{2}-\left(\frac{\lambda_{1}}{p}+\frac{\lambda_{2}}{q}\right)\frac{2p c}{n(p-2)\lambda_{1}}-\frac{\lambda_{2}}{q}\right)\|\varphi(\tau_{9})\|^{2}_{2}.
\end{split}
\end{align}
In (\ref{p-9zhunbeimaodun+1}), the assumption (\ref{2.10}) ensures $$\frac{1}{2}-\frac{\lambda_{2}}{q}>0.$$Then, $c>0$ is chosen to make $$\frac{1}{2}-\left(\frac{\lambda_{1}}{p}+\frac{\lambda_{2}}{q}\right)\frac{2p c}{n(p-2)\lambda_{1}}-\frac{\lambda_{2}}{q}>0,$$i.e., (\ref{cdetiaojian}), hold.
According to Cauchy-Schwartz inequality and Lemma \ref{lem3.2}, we know (\ref{p-9zhunbeimaodun+1}) turns to
\begin{align}\label{p-9zhunbeimaodun+1-1}
\begin{split}
K(\varphi(\tau_{9}))>
&\frac{1}{2^{\frac{1}{2}}}\|\nabla \varphi(\tau_{9})\|_{2}\||x|\varphi(\tau_{9})\|_{2}+\frac{1}{4}\|\nabla \varphi(\tau_{9})\|^{2}_{2}\\
&+\left(\frac{1}{2}-\left(\frac{\lambda_{1}}{p}+\frac{\lambda_{2}}{q}\right)\frac{2p c}{n(p-2)\lambda_{1}}-\frac{\lambda_{2}}{q}\right)\|\varphi(\tau_{9})\|^{2}_{2}\\
\geq
&\frac{n}{2^{\frac{3}{2}}}\|\varphi(\tau_{9})\|^{2}_{2}+\frac{1}{4}\|\nabla \varphi(\tau_{9})\|^{2}_{2}+\left(\frac{1}{2}-\left(\frac{\lambda_{1}}{p}+\frac{\lambda_{2}}{q}\right)\frac{2p c}{n(p-2)\lambda_{1}}-\frac{\lambda_{2}}{q}\right)\|\varphi(\tau_{9})\|^{2}_{2}.
\end{split}
\end{align}
Then, by Gagliardo-Nirenberg inequality (\ref{GN-budengshi}), we know (\ref{p-gujij9}) becomes
\begin{align*}
\frac{2p c}{n(p-2)\lambda_{1}C_{1}}\|\varphi(\tau_{9})\|^{2}_{2}\leq\|\nabla\varphi(\tau_{9})\|^{\frac{n(p-2)}{2}}_{2}\|\varphi(\tau_{9})\|^{p-\frac{n(p-2)}{2}}_{2},
\end{align*}
i.e.,
\begin{align}
\frac{2p c}{n(p-2)\lambda_{1}C_{1}}\leq\|\nabla\varphi(\tau_{9})\|^{\frac{n(p-2)}{2}}_{2}\|\varphi(\tau_{9})\|^{p-\frac{n(p-2)}{2}-2}_{2}.\label{p-9-zhunbeimaodun+1}
\end{align}
Due to (\ref{2.2}), we know (\ref{p-9-zhunbeimaodun+1}) turns to
\begin{align*}
\frac{2p c}{n(p-2)\lambda_{1}C_{1}}\leq\|\nabla\varphi(\tau_{9})\|^{\frac{n(p-2)}{2}}_{2}\|\varphi_{0}\|^{p-\frac{n(p-2)}{2}-2}_{2},
\end{align*}
i.e.,
\begin{align}
\left(\frac{2p c}{n(p-2)\lambda_{1}C_{1}\|\varphi_{0}\|^{p-\frac{n(p-2)}{2}-2}_{2}}\right)^{\frac{4}{n(p-2)}}\leq\|\nabla\varphi(\tau_{9})\|^{2}_{2},\label{jdengyu11}
\end{align}
which means (\ref{p-9zhunbeimaodun+1-1}) turns to
\begin{align}\label{p--9zhunbeimaodun+1-1}
\begin{split}
K(\varphi(\tau_{9}))
>
&\frac{n}{2^{\frac{3}{2}}}\|\varphi(\tau_{9})\|^{2}_{2}+\frac{1}{4}\left(\frac{2p c}{n(p-2)\lambda_{1}C_{1}\|\varphi_{0}\|^{p-\frac{n(p-2)}{2}-2}_{2}}\right)^{\frac{4}{n(p-2)}}\\
&+\left(\frac{1}{2}-\left(\frac{\lambda_{1}}{p}+\frac{\lambda_{2}}{q}\right)\frac{2p c}{n(p-2)\lambda_{1}}-\frac{\lambda_{2}}{q}\right)\|\varphi(\tau_{9})\|^{2}_{2}.
\end{split}
\end{align}
According to (\ref{2.2}) and (\ref{2.3}), we know (\ref{p--9zhunbeimaodun+1-1}) becomes
\begin{align*}
K(\varphi_{0})
>
&\left(\frac{n}{2^{\frac{3}{2}}}+\frac{1}{2}-\left(\frac{\lambda_{1}}{p}+\frac{\lambda_{2}}{q}\right)\frac{2p c}{n(p-2)\lambda_{1}}-\frac{\lambda_{2}}{q}\right)\|\varphi_{0}\|^{2}_{2}\nonumber\\
&+\frac{1}{4}\left(\frac{2p c}{n(p-2)\lambda_{1}C_{1}\|\varphi_{0}\|^{p-\frac{n(p-2)}{2}-2}_{2}}\right)^{\frac{4}{n(p-2)}},
\end{align*}
which contradicts (\ref{K0detiaoj1}) for $j=9$.
Thus, we have $S_{9}(\varphi(t))>0$ for $t\in(0,T_{\rm{max}})$.\\
{\bf{Case \textrm{II}: \bm{$j=10.$}}}

We prove $S_{10}(\varphi(t))>0$ for $t\in(0,T_{\rm{max}})$. Arguing by contradiction, we suppose that there exists the first time $\tau_{10}\in(0,T_{\rm{max}})$ to make $S_{10}(\varphi(\tau_{10}))=0$ and $S_{10}(\varphi(t))>0$ for $t\in(0,\tau_{10})$, and we know $\varphi(\tau_{10})\neq 0$ due to (\ref{2.2}) and $\varphi_{0}\neq0$. Next, we estimate $K(\varphi(\tau_{10}))$.
According to $S_{10}(\varphi(\tau_{10}))=0$, we have
\begin{align}
\|\varphi(\tau_{10})\|^{p}_{p}=\frac{2p}{n(p-2)\lambda_{1}}\||x|\varphi(\tau_{10})\|^{2}_{2}.\label{10-2jdengyu2-2}
\end{align}
By substituting (\ref{10-2jdengyu2-2}) into (\ref{2.16}), we have
\begin{align}\label{2.17--jdengyu10}
\begin{split}
K(\varphi(\tau_{10}))>&\frac{1}{2}\|\nabla \varphi(\tau_{10})\|^{2}_{2}+\left(\frac{1}{2}-\left(\frac{\lambda_{1}}{p}+\frac{\lambda_{2}}{q}\right)\frac{2p}{n(p-2)\lambda_{1}}\right)
\||x|\varphi(\tau_{10})\|^{2}_{2}+\left(\frac{1}{2}-\frac{\lambda_{2}}{q}\right)\|\varphi(\tau_{10})\|^{2}_{2}\\
=&\frac{1}{4}\|\nabla \varphi(\tau_{10})\|^{2}_{2}+\frac{1}{4}\|\nabla \varphi(\tau_{10})\|^{2}_{2}+\left(\frac{1}{2}-\left(\frac{\lambda_{1}}{p}
+\frac{\lambda_{2}}{q}\right)\frac{2p}{n(p-2)\lambda_{1}}\right)\||x|\varphi(\tau_{10})\|^{2}_{2}\\
&+\left(\frac{1}{2}-\frac{\lambda_{2}}{q}\right)\|\varphi(\tau_{10})\|^{2}_{2}.
\end{split}
\end{align}
In (\ref{2.17--jdengyu10}), the assumption (\ref{2.9}) ensures $$\frac{1}{2}-\left(\frac{\lambda_{1}}{p}+\frac{\lambda_{2}}{q}\right)\frac{2p}{n(p-2)\lambda_{1}}>0.$$ And the assumption (\ref{2.10}) ensures $$\frac{1}{2}-\frac{\lambda_{2}}{q}>0.$$By Cauchy-Schwartz inequality and Lemma \ref{3.2}, we know that (\ref{2.17--jdengyu10}) turns to
\begin{align}\label{2.17-jdengyu10}
\begin{split}
K(\varphi(\tau_{10}))
>&\frac{1}{4}\|\nabla \varphi(\tau_{10})\|^{2}_{2}+\left(\frac{1}{2}-\frac{\lambda_{2}}{q}\right)\|\varphi(\tau_{10})\|^{2}_{2}\\
&+\left(\frac{1}{2}-\left(\frac{\lambda_{1}}{p}+\frac{\lambda_{2}}{q}\right)\frac{2p}{n(p-2)\lambda_{1}}\right)^{\frac{1}{2}}
\|\nabla\varphi(\tau_{10})\|_{2}\||x|\varphi(\tau_{10})\|_{2}\\
\geq&\frac{1}{4}\|\nabla \varphi(\tau_{10})\|^{2}_{2}+\frac{n}{2}\left(\frac{1}{2}-\left(\frac{\lambda_{1}}{p}+\frac{\lambda_{2}}{q}\right)
\frac{2p}{n(p-2)\lambda_{1}}\right)^{\frac{1}{2}}\|\varphi(\tau_{10})\|^{2}_{2}\\&+\left(\frac{1}{2}-\frac{\lambda_{2}}{q}\right)
\|\varphi(\tau_{10})\|^{2}_{2}.
\end{split}
\end{align}
Then, according to (\ref{10-2jdengyu2-2}) and Gagliardo-Nirenberg inequality (\ref{GN-budengshi}), we know
\begin{align*}
\||x|\varphi(\tau_{10})\|^{2}_{2}=&\frac{n(p-2)\lambda_{1} }{2p}\|\varphi(\tau_{10})\|^{p}_{p}\nonumber\\
\leq&\frac{n(p-2)\lambda_{1}C_{1} }{2p}\|\nabla\varphi(\tau_{10})\|^{\frac{n(p-2)}{2}}_{2}\|\varphi(\tau_{10})\|^{p-\frac{n(p-2)}{2}}_{2}\nonumber\\
=&\frac{n(p-2)\lambda_{1}C_{1} }{2p}\|\nabla\varphi(\tau_{10})\|^{\frac{n(p-2)}{2}}_{2}\|\varphi(\tau_{10})\|^{p-\frac{n(p-2)}{2}-4}_{2}\|\varphi(\tau_{10})\|^{4}_{2}\nonumber\\
\leq&\frac{2(p-2)\lambda_{1}C_{1} }{np}\|\nabla\varphi(\tau_{10})\|^{\frac{n(p-2)}{2}}_{2}\|\varphi(\tau_{10})\|^{p-\frac{n(p-2)}{2}-4}_{2}\|\nabla\varphi(\tau_{10})\|^{2}_{2}\||x|\varphi(\tau_{10})\|^{2}_{2},
\end{align*}
which means
\begin{align}
\|\nabla\varphi(\tau_{10})\|^{2}_{2}\geq\left(\frac{np}{2(p-2)\lambda_{1}C_{1}\|\varphi(\tau_{10})\|^{p-\frac{n(p-2)}{2}-4}_{2}}\right)
^{\frac{4}{n(p-2)+4}}.\label{10-Gnbudengshi-1-1}
\end{align}
By substituting (\ref{10-Gnbudengshi-1-1}) into (\ref{2.17-jdengyu10}), we have
\begin{align}\label{2.17-jdengyu--10}
\begin{split}
K(\varphi(\tau_{10}))
>&\frac{1}{4}\left(\frac{np}{2(p-2)\lambda_{1}C_{1}\|\varphi(\tau_{10})\|^{p-\frac{n(p-2)}{2}-4}_{2}}\right)
^{\frac{4}{n(p-2)+4}}\\
&+\frac{n}{2}\left(\frac{1}{2}-\left(\frac{\lambda_{1}}{p}+\frac{\lambda_{2}}{q}\right)\frac{2p}{n(p-2)\lambda_{1}}\right)^{\frac{1}{2}}\|\varphi(\tau_{10})\|^{2}_{2}
+\left(\frac{1}{2}-\frac{\lambda_{2}}{q}\right)\|\varphi(\tau_{10})\|^{2}_{2}.
\end{split}
\end{align}
Due to (\ref{2.2}) and (\ref{2.3}), we know (\ref{2.17-jdengyu--10}) turns to
\begin{align*}
K(\varphi_{0})
>&\frac{1}{4}\left(\frac{np}{2(p-2)\lambda_{1}C_{1}\|\varphi_{0}\|^{p-\frac{n(p-2)}{2}-4}_{2}}\right)
^{\frac{4}{n(p-2)+4}}\nonumber\\
&+\left(\frac{n}{2}\left(\frac{1}{2}-\left(\frac{\lambda_{1}}{p}+\frac{\lambda_{2}}{q}\right)\frac{2p}{n(p-2)\lambda_{1}}\right)^{\frac{1}{2}}+\frac{1}{2}-\frac{\lambda_{2}}{q}\right)\|\varphi_{0}\|^{2}_{2},
\end{align*}
which contradicts (\ref{K0detiaoj1}) for $j=10$. Thus, we know $S_{10}(\varphi(t))>0$ for $t\in(0,T_{\rm{max}})$.\\
{\bf{Case \textrm{III}: \bm{$j=11.$}}}

We prove $S_{11}(\varphi(t))>0$ for all $t\in(0,T_{\rm{max}})$. Arguing by contradiction, we suppose that there exists the first time $\tau_{11}\in(0,T_{\rm{max}})$ to make
$S_{11}(\varphi(\tau_{11}))=0$ and $S_{11}(\varphi(t))>0$ for $t\in(0,\tau_{11})$, and we know $\varphi(\tau_{11})\neq 0$ from (\ref{2.2}) and $\varphi_{0}\neq0$.
Next, we estimate $K(\varphi(\tau_{11}))$.
According to $S_{11}(\varphi(\tau_{11}))=0$, we have
\begin{align}
\|\varphi(\tau_{11})\|^{p}_{p}=&\frac{2p}{n(p-2)\lambda_{1}}\left(c\|\varphi(\tau_{11})\|^{2}_{2}+\||x|\varphi(\tau_{11})\|^{2}_{2}\right).\label{11-gujij4-p}
\end{align}
By substituting (\ref{11-gujij4-p}) into (\ref{2.16}), we have
\begin{align}\label{p11---zhunbeimaodun+1}
\begin{split}
K(\varphi(\tau_{11}))>&\frac{1}{2}\|\nabla \varphi(\tau_{11})\|^{2}_{2}+\left(\frac{1}{2}-\left(\frac{\lambda_{1}}{p}+\frac{\lambda_{2}}{q}\right)\frac{2p}{n(p-2)\lambda_{1}}\right)\||x|\varphi(\tau_{11})\|^{2}_{2}
\\
&+\left(\frac{1}{2}-\left(\frac{\lambda_{1}}{p}+\frac{\lambda_{2}}{q}\right)\frac{2p c}{n(p-2)\lambda_{1}}-\frac{\lambda_{2}}{q}\right)\|\varphi(\tau_{11})\|^{2}_{2}\\
=&\frac{1}{4}\|\nabla \varphi(\tau_{11})\|^{2}_{2}+\frac{1}{4}\|\nabla \varphi(\tau_{11})\|^{2}_{2}\\
&+\left(\frac{1}{2}-\left(\frac{\lambda_{1}}{p}+\frac{\lambda_{2}}{q}\right)\frac{2p}{n(p-2)\lambda_{1}}\right)
\||x|\varphi(\tau_{11})\|^{2}_{2}\\
&+\left(\frac{1}{2}-\left(\frac{\lambda_{1}}{p}+\frac{\lambda_{2}}{q}\right)\frac{2p c}{n(p-2)\lambda_{1}}-\frac{\lambda_{2}}{q}\right)\|\varphi(\tau_{11})\|^{2}_{2}.
\end{split}
\end{align}
In (\ref{p11---zhunbeimaodun+1}), the assumptions (\ref{2.9}) and (\ref{2.10}) ensure $$\frac{1}{2}-\left(\frac{\lambda_{1}}{p}+\frac{\lambda_{2}}{q}\right)\frac{2p}{n(p-2)\lambda_{1}}>0$$
and $$\frac{1}{2}-\frac{\lambda_{2}}{q}>0,$$ respectively. Then, $c>0$ is chosen such that $$\frac{1}{2}-\left(\frac{\lambda_{1}}{p}+\frac{\lambda_{2}}{q}\right)\frac{2p c}{n(p-2)\lambda_{1}}-\frac{\lambda_{2}}{q}>0,$$ i.e., (\ref{cdetiaojian}), holds.
According to Cauchy-Schwartz inequality and Lemma \ref{lem3.2}, we know (\ref{p11---zhunbeimaodun+1}) turns to
\begin{align}\label{p11+---zhunbeimaodun+1}
\begin{split}
K(\varphi(\tau_{11}))>
&\frac{1}{4}\|\nabla \varphi(\tau_{11})\|^{2}_{2}+\left(\frac{1}{2}-\left(\frac{\lambda_{1}}{p}+\frac{\lambda_{2}}{q}\right)\frac{2p}{n(p-2)\lambda_{1}}\right)^{\frac{1}{2}}
\|\nabla \varphi(\tau_{11})\|_{2}\||x|\varphi(\tau_{11})\|_{2}\\
&+\left(\frac{1}{2}-\left(\frac{\lambda_{1}}{p}+\frac{\lambda_{2}}{q}\right)\frac{2p c}{n(p-2)\lambda_{1}}-\frac{\lambda_{2}}{q}\right)\|\varphi(\tau_{11})\|^{2}_{2}\\
\geq&\frac{1}{4}\|\nabla \varphi(\tau_{11})\|^{2}_{2}+\frac{n}{2}\left(\frac{1}{2}-\left(\frac{\lambda_{1}}{p}+\frac{\lambda_{2}}{q}\right)\frac{2p}{n(p-2)\lambda_{1}}\right)^{\frac{1}{2}}
\|\varphi(\tau_{11})\|^{2}_{2}\\
&+\left(\frac{1}{2}-\left(\frac{\lambda_{1}}{p}+\frac{\lambda_{2}}{q}\right)\frac{2p c}{n(p-2)\lambda_{1}}-\frac{\lambda_{2}}{q}\right)\|\varphi(\tau_{11})\|^{2}_{2}.
\end{split}
\end{align}
Meanwhile, according to (\ref{11-gujij4-p}), we have
\begin{align}
\|\varphi(\tau_{11})\|^{p}_{p}>&\frac{2p c}{n(p-2)\lambda_{1}}\|\varphi(\tau_{11})\|^{2}_{2}.\label{p-gujij9-11}
\end{align}
By the similar process of proving (\ref{jdengyu11}), we can use (\ref{p-gujij9-11}) to give \begin{align}
\left(\frac{2p c}{n(p-2)\lambda_{1}C_{1}\|\varphi_{0}\|^{p-\frac{n(p-2)}{2}-2}_{2}}\right)^{\frac{4}{n(p-2)}}\leq\|\nabla\varphi(\tau_{11})\|^{2}_{2}.\label{jdengyu11---}
\end{align}
By substituting (\ref{jdengyu11---}) into (\ref{p11+---zhunbeimaodun+1}), we have
\begin{align}\label{j11}
\begin{split}
K(\varphi(\tau_{11}))
>&\frac{1}{4}\left(\frac{2p c}{n(p-2)\lambda_{1}C_{1}\|\varphi_{0}\|^{p-\frac{n(p-2)}{2}-2}_{2}}\right)^{\frac{4}{n(p-2)}}\\
&+\left(\frac{n}{2}\left(\frac{1}{2}-\left(\frac{\lambda_{1}}{p}+\frac{\lambda_{2}}{q}\right)\frac{2p}{n(p-2)\lambda_{1}}\right)^{\frac{1}{2}}+\frac{1}{2}-\left(\frac{\lambda_{1}}{p}+\frac{\lambda_{2}}{q}\right)\frac{2p c}{n(p-2)\lambda_{1}}
\right.\\&\left.-\frac{\lambda_{2}}{q}\right)\|\varphi(\tau_{11})\|^{2}_{2}.
\end{split}
\end{align}
According to (\ref{2.2}) and (\ref{2.3}), we know (\ref{j11}) turns to
\begin{align*}
K(\varphi_{0})
>&\frac{1}{4}\left(\frac{2p c}{n(p-2)\lambda_{1}C_{1}\|\varphi_{0}\|^{p-\frac{n(p-2)}{2}-2}_{2}}\right)^{\frac{4}{n(p-2)}}\nonumber\\
&+\left(\frac{n}{2}\left(\frac{1}{2}-\left(\frac{\lambda_{1}}{p}+\frac{\lambda_{2}}{q}\right)\frac{2p}{n(p-2)\lambda_{1}}\right)^{\frac{1}{2}}+\frac{1}{2}-\left(\frac{\lambda_{1}}{p}+\frac{\lambda_{2}}{q}\right)\frac{2p c}{n(p-2)\lambda_{1}}-\frac{\lambda_{2}}{q}\right)\|\varphi_{0}\|^{2}_{2},
\end{align*}
which contradicts (\ref{K0detiaoj1}) for $j=11$. Thus, we have $S_{11}(\varphi(t))>0$ for $t\in(0,T_{\rm{max}})$.\\
{\bf{Case \textrm{IV}: \bm{$j=12.$}}}

We prove $S_{12}(\varphi(t))>0$ for $t\in(0,T_{\rm{max}})$. Arguing by contradiction, we suppose that there exists the first time $\tau_{12}\in(0,T_{\rm{max}})$ to make $S_{12}(\varphi(\tau_{12}))=0$ and $S_{12}(\varphi(t))>0$ for $t\in(0,\tau_{12})$, and we know $\varphi(\tau_{12})\neq 0$ from (\ref{2.2}) and $\varphi_{0}\neq0$. Next, we estimate $K(\varphi(\tau_{12}))$.
According to $S_{12}(\varphi(\tau_{12}))=0$, we have
\begin{align}
\frac{n(p-2)\lambda_{1}}{2p}\|\varphi(\tau_{12})\|^{p}_{p}+\frac{n(q-2)\lambda_{2}}{2q}\|\varphi(\tau_{12})\|^{q}_{q}=\| \varphi(\tau_{12})\|^{2}_{2}.\label{jdengyu12}
\end{align}
Using (\ref{jdengyu12}) and retrieving (\ref{chuli1-p-q}) for $j=12$, we know
\begin{align}
&\frac{\lambda_{1}}{p}\|\varphi(\tau_{12})\|^{p}_{p}+\frac{\lambda_{2}}{q}\|\varphi(\tau_{12})\|^{q}_{q}
<\frac{2}{n(q-2)}\|\varphi(\tau_{12})\|^{2}_{2}.\label{pq12gujijgaoneng3+}
\end{align}
By substituting (\ref{pq12gujijgaoneng3+}) into $K(\varphi(\tau_{12}))$, we have
\begin{align}\label{pq12--+zhunbeimaodun-1-5}
\begin{split}
K(\varphi(\tau_{12}))>&\frac{1}{2}\|\nabla \varphi(\tau_{12})\|^{2}_{2}+\left(\frac{1}{2}-\frac{2}{n(q-2)}\right)\|\varphi(\tau_{12})\|^{2}_{2}+\frac{1}{2}\||x|\varphi(\tau_{12})\|^{2}_{2}\\
=&\frac{1}{4}\|\nabla \varphi(\tau_{12})\|^{2}_{2}+\frac{1}{4}\|\nabla \varphi(\tau_{12})\|^{2}_{2}+\frac{1}{2}\||x|\varphi(\tau_{12})\|^{2}_{2}+\left(\frac{1}{2}-\frac{2}{n(q-2)}\right)\|\varphi(\tau_{12})\|^{2}_{2}.
\end{split}
\end{align}
In (\ref{pq12--+zhunbeimaodun-1-5}), the assumption (\ref{1.2}) ensures $$\frac{1}{2}-\frac{2}{n(q-2)}>0.$$ By Cauchy-Schwartz inequality and Lemma \ref{lem3.2}, (\ref{pq12--+zhunbeimaodun-1-5}) becomes
\begin{align}\label{pq12-+-+zhunbeimaodun-1-5}
\begin{split}
K(\varphi(\tau_{12}))
>&\frac{1}{4}\|\nabla \varphi(\tau_{12})\|^{2}_{2}+\frac{1}{2^{\frac{1}{2}}}\|\nabla \varphi(\tau_{12})\|_{2}\||x|\varphi(\tau_{12})\|_{2}+\left(\frac{1}{2}-\frac{2}{n(q-2)}\right)\|\varphi(\tau_{12})\|^{2}_{2}\\
\geq&\frac{1}{4}\|\nabla \varphi(\tau_{12})\|^{2}_{2}+\frac{n}{2^{\frac{3}{2}}}\| \varphi(\tau_{12})\|^{2}_{2}+\left(\frac{1}{2}-\frac{2}{n(q-2)}\right)\|\varphi(\tau_{12})\|^{2}_{2}.
\end{split}
\end{align}
Meanwhile, according to $(\ref{jdengyu12})$, (\ref{GN-budengshi}) and (\ref{GN-budengshi1}), we have
\begin{align*}
\| \varphi(\tau_{12})\|^{2}_{2}=&\frac{n(p-2)\lambda_{1}}{2p}\|\varphi(\tau_{12})\|^{p}_{p}+\frac{n(q-2)\lambda_{2}}{2q}\|\varphi(\tau_{12})\|^{q}_{q}\nonumber\\
\leq&\frac{n(p-2)C_{1}\lambda_{1}}{2p}\|\nabla\varphi(\tau_{12})\|^{\frac{n(p-2)}{2}}_{2}\|\varphi(\tau_{12})\|^{p-\frac{n(p-2)}{2}}_{2}\nonumber\\
&+\frac{n(q-2)C_{3}\lambda_{2}}{2q}\|\nabla\varphi(\tau_{12})\|^{\frac{n(q-2)}{2}}_{2}\|\varphi(\tau_{12})\|^{q-\frac{n(q-2)}{2}}_{2},
\end{align*}
i.e.,\label{jdengyu5-5}
\begin{align}
\begin{split}
1
\leq&\frac{n(p-2)C_{1}\lambda_{1}}{2p}\|\nabla\varphi(\tau_{12})\|^{\frac{n(p-2)}{2}}_{2}\|\varphi(\tau_{12})\|^{p-\frac{n(p-2)}{2}-2}_{2}\\
&+\frac{n(q-2)C_{3}\lambda_{2}}{2q}\|\nabla\varphi(\tau_{12})\|^{\frac{n(q-2)}{2}}_{2}\|\varphi(\tau_{12})\|^{q-\frac{n(q-2)}{2}-2}_{2},
\end{split}
\end{align}
which means
\begin{align*}
\frac{1}{2}
\leq&\frac{n(p-2)C_{1}\lambda_{1}}{2p}\|\nabla\varphi(\tau_{12})\|^{\frac{n(p-2)}{2}}_{2}\|\varphi(\tau_{12})\|^{p-\frac{n(p-2)}{2}-2}_{2},
\end{align*}
i.e.,
\begin{align}
\left(\frac{p }{n(p-2)C_{1}\lambda_{1}\|\varphi(\tau_{12})\|^{p-\frac{n(p-2)}{2}-2}_{2}}\right)^{\frac{4}{n(p-2)}}
\leq&\|\nabla\varphi(\tau_{12})\|^{2}_{2},\label{jdengyu12-1}
\end{align}
or
\begin{align*}
\frac{1}{2}
\leq&\frac{n(q-2)C_{3}\lambda_{2}}{2q}\|\nabla\varphi(\tau_{12})\|^{\frac{n(q-2)}{2}}_{2}\|\varphi(\tau_{12})\|^{q-\frac{n(q-2)}{2}-2}_{2},
\end{align*}
i.e.,
\begin{align}
\left(\frac{q }{n(q-2)C_{3}\lambda_{2}\|\varphi(\tau_{12})\|^{q-\frac{n(q-2)}{2}-2}_{2}}\right)^{\frac{4}{n(q-2)}}
\leq&\|\nabla\varphi(\tau_{12})\|^{2}_{2}.\label{jdengyu12-2}
\end{align}
By (\ref{2.2}), we see that (\ref{jdengyu12-1}) and (\ref{jdengyu12-2}) become
\begin{align}
\left(\frac{p }{n(p-2)C_{1}\lambda_{1}\|\varphi_{0}\|^{p-\frac{n(p-2)}{2}-2}_{2}}\right)^{\frac{4}{n(p-2)}}
\leq&\|\nabla\varphi(\tau_{12})\|^{2}_{2}\label{jdengyu12-1-1}
\end{align}
and
\begin{align}
\left(\frac{q }{n(q-2)C_{3}\lambda_{2}\|\varphi_{0}\|^{q-\frac{n(q-2)}{2}-2}_{2}}\right)^{\frac{4}{n(q-2)}}
\leq&\|\nabla\varphi(\tau_{12})\|^{2}_{2},\label{jdengyu12-2-1}
\end{align}
respectively. From (\ref{jdengyu12-1-1}) and (\ref{jdengyu12-2-1}), we have
\begin{align}
\|\nabla\varphi(\tau_{12})\|^{2}_{2}\geq&\frac{A_{5}(\|\varphi_{0}\|_{2})+B_{5}(\|\varphi_{0}\|_{2})-\left|A_{5}(\|\varphi_{0}\|_{2})-B_{5}(\|\varphi_{0}\|_{2})\right|}{2},\label{jdengyu12-2-2}
\end{align}
where $A_{5}(\|\varphi_{0}\|_{2})$ and $B_{5}(\|\varphi_{0}\|_{2})$ are defined by (\ref{A5}) and (\ref{B5}), respectively.
According to (\ref{jdengyu12-2-2}), we know (\ref{pq12-+-+zhunbeimaodun-1-5}) turns to
\begin{align}
\begin{split}
K(\varphi(\tau_{12}))>&\frac{A_{5}(\|\varphi_{0}\|_{2})+B_{5}(\|\varphi_{0}\|_{2})-\left|A_{5}(\|\varphi_{0}\|_{2})-B_{5}(\|\varphi_{0}\|_{2})\right|}{8}\\
&+\left(\frac{n}{2^{\frac{3}{2}}}+\frac{1}{2}-\frac{2}{n(q-2)}\right)
\|\varphi(\tau_{12})\|^{2}_{2}.\label{pq12--+-+zhunbeimaodun-1-5}
\end{split}
\end{align}
Due to (\ref{2.2}) and (\ref{2.3}), (\ref{pq12--+-+zhunbeimaodun-1-5}) becomes
\begin{align*}
K(\varphi_{0})>&\frac{A_{5}(\|\varphi_{0}\|_{2})+B_{5}(\|\varphi_{0}\|_{2})-\left|A_{5}(\|\varphi_{0}\|_{2})-B_{5}(\|\varphi_{0}\|_{2})\right|}{8}+\left(\frac{n}{2^{\frac{3}{2}}}+\frac{1}{2}-\frac{2}{n(q-2)}\right)
\|\varphi_{0}\|^{2}_{2},
\end{align*}
which contradicts (\ref{K0detiaoj1}) for $j=12$. Thus, we have $S_{12}(\varphi(t))>0$ for $t\in(0,T_{\rm{max}})$.\\
{\bf{Case \textrm{V}: \bm{$j=13.$}}}

We prove $S_{13}(\varphi(t))>0$ for $t\in(0,T_{\rm{max}})$. Arguing by contradiction, we suppose that there exists the first time $\tau_{13}\in(0,T_{\rm{max}})$ to make $S_{13}(\varphi(\tau_{13}))=0$ and $S_{13}(\varphi(t))>0$ for $t\in(0,\tau_{13})$, and we know $\varphi(\tau_{13})\neq 0$ due to (\ref{2.2}) and $\varphi_{0}\neq0$. Next, we estimate $K(\varphi(\tau_{13}))$.
According to $S_{13}(\varphi(\tau_{13}))=0$, we have
\begin{align}
\frac{n(p-2)\lambda_{1}}{2p}\|\varphi(\tau_{13})\|^{p}_{p}+\frac{n(q-2)\lambda_{2}}{2q}
\|\varphi(\tau_{13})\|^{q}_{q}=\||x|\varphi(\tau_{13})\|^{2}_{2}.\label{pq13jdengyu6}
\end{align}
Using (\ref{pq13jdengyu6}) and retrieving (\ref{chuli1-p-q}) for $j=13$, we know
\begin{align}
&\frac{\lambda_{1}}{p}\|\varphi(\tau_{13})\|^{p}_{p}+\frac{\lambda_{2}}{q}\|\varphi(\tau_{13})\|^{q}_{q}
<\frac{2}{n(q-2)}\||x|\varphi(\tau_{13})\|^{2}_{2}.\label{gujijgaoneng13+}
\end{align}
By substituting (\ref{gujijgaoneng13+}) into $K(\varphi(\tau_{13}))$, we have
\begin{align}\label{pq13--+zhunbeimaodun-1-5}
\begin{split}
K(\varphi(\tau_{13}))>&\frac{1}{2}\|\nabla \varphi(\tau_{13})\|^{2}_{2}+\left(\frac{1}{2}-\frac{2}{n(q-2)}\right)\||x|\varphi(\tau_{13})\|^{2}_{2}+\frac{1}{2}\| \varphi(\tau_{13})\|^{2}_{2}\\
=&\frac{1}{4}\|\nabla \varphi(\tau_{13})\|^{2}_{2}+\frac{1}{4}\|\nabla \varphi(\tau_{13})\|^{2}_{2}+\left(\frac{1}{2}-\frac{2}{n(q-2)}\right)\||x|\varphi(\tau_{13})\|^{2}_{2}+\frac{1}{2}\| \varphi(\tau_{13})\|^{2}_{2}.
\end{split}
\end{align}
In (\ref{pq13--+zhunbeimaodun-1-5}), the assumption (\ref{1.2}) ensures $$\frac{1}{2}-\frac{2}{n(q-2)}>0.$$ By Cauchy-Schwartz inequality and Lemma \ref{lem3.2}, we know (\ref{pq13--+zhunbeimaodun-1-5}) becomes
\begin{align}\label{j13}
\begin{split}
K(\varphi(\tau_{13}))>&\frac{1}{4}\|\nabla \varphi(\tau_{13})\|^{2}_{2}+\left(\frac{1}{2}-\frac{2}{n(q-2)}\right)^{\frac{1}{2}}\|\nabla \varphi(\tau_{13})\|_{2}\||x|\varphi(\tau_{13})\|_{2}+\frac{1}{2}\| \varphi(\tau_{13})\|^{2}_{2}\\
\geq&\frac{1}{4}\|\nabla \varphi(\tau_{13})\|^{2}_{2}+\frac{n}{2}\left(\frac{1}{2}-\frac{2}{n(q-2)}\right)^{\frac{1}{2}}\|\varphi(\tau_{13})\|^{2}_{2}+\frac{1}{2}\| \varphi(\tau_{13})\|^{2}_{2}.
\end{split}
\end{align}
Meanwhile, according to (\ref{GN-budengshi}), (\ref{GN-budengshi1}) and Lemma \ref{lem3.2},
we know (\ref{pq13jdengyu6}) turns to
\begin{align*}
\||x|\varphi(\tau_{13})\|^{2}_{2}=&\frac{n(p-2)\lambda_{1}}{2p}\|\varphi(\tau_{13})\|^{p}_{p}+\frac{n(q-2)\lambda_{2}}{2q}
\|\varphi(\tau_{13})\|^{q}_{q}\nonumber\\
\leq&\frac{n(p-2)\lambda_{1}C_{1}}{2p}\|\nabla\varphi(\tau_{13})\|^{\frac{n(p-2)}{2}}_{2}\|\varphi(\tau_{13})\|^{p-\frac{n(p-2)}{2}}_{2}\nonumber\\
&+\frac{n(q-2)\lambda_{2}C_{3}}{2q}\|\nabla\varphi(\tau_{13})\|^{\frac{n(q-2)}{2}}_{2}\|\varphi(\tau_{13})\|^{q-\frac{n(q-2)}{2}}_{2}\nonumber\\
=&\frac{n(p-2)\lambda_{1}C_{1}}{2p}\|\nabla\varphi(\tau_{13})\|^{\frac{n(p-2)}{2}}_{2}\|\varphi(\tau_{13})\|^{p-\frac{n(p-2)}{2}-4}_{2}
\|\varphi(\tau_{13})\|^{4}_{2}\nonumber\\
&+\frac{n(q-2)\lambda_{2}C_{3}}{2q}\|\nabla\varphi(\tau_{13})\|^{\frac{n(q-2)}{2}}_{2}\|\varphi(\tau_{13})\|^{q-\frac{n(q-2)}{2}-4}_{2}
\|\varphi(\tau_{13})\|^{4}_{2}\nonumber\\
\leq&\frac{2(p-2)\lambda_{1}C_{1}}{np}\|\nabla\varphi(\tau_{13})\|^{\frac{n(p-2)}{2}}_{2}\|\varphi(\tau_{13})\|^{p-\frac{n(p-2)}{2}-4}_{2}
\|\nabla\varphi(\tau_{13})\|^{2}_{2}\||x|\varphi(\tau_{13})\|^{2}_{2}\nonumber\\
&+\frac{2(q-2)\lambda_{2}C_{3}}{nq}\|\nabla\varphi(\tau_{13})\|^{\frac{n(q-2)}{2}}_{2}\|\varphi(\tau_{13})\|^{q-\frac{n(q-2)}{2}-4}_{2}
\|\nabla\varphi(\tau_{13})\|^{2}_{2}\||x|\varphi(\tau_{13})\|^{2}_{2},
\end{align*}
i.e.,
\begin{align*}
1
\leq&\frac{2(p-2)\lambda_{1}C_{1}}{np}\|\nabla\varphi(\tau_{13})\|^{\frac{n(p-2)+4}{2}}_{2}\|\varphi(\tau_{13})\|^{p-\frac{n(p-2)}{2}-4}_{2}
\nonumber\\
&+\frac{2(q-2)\lambda_{2}C_{3}}{nq}\|\nabla\varphi(\tau_{13})\|^{\frac{n(q-2)+4}{2}}_{2}\|\varphi(\tau_{13})\|^{q-\frac{n(q-2)}{2}-4}_{2}
,
\end{align*}
which means
\begin{align}
\frac{1}{2}
\leq&\frac{2(p-2)\lambda_{1}C_{1}}{np}\|\nabla\varphi(\tau_{13})\|^{\frac{n(p-2)+4}{2}}_{2}\|\varphi(\tau_{13})\|^{p-\frac{n(p-2)}{2}-4}_{2}\label{13-1}
\end{align}
or
\begin{align}
\frac{1}{2}
\leq
&\frac{2(q-2)\lambda_{2}C_{3}}{nq}\|\nabla\varphi(\tau_{13})\|^{\frac{n(q-2)+4}{2}}_{2}\|\varphi(\tau_{13})\|^{q-\frac{n(q-2)}{2}-4}_{2}.\label{13-2}
\end{align}
Here, from (\ref{2.2}), we see that (\ref{13-1}) and (\ref{13-2}) give
\begin{align}
\|\nabla\varphi(\tau_{13})\|^{2}_{2}\geq\left(\frac{np}{4(p-2)\lambda_{1}C_{1}\|\varphi_{0}\|^{p-\frac{n(p-2)}{2}-4}_{2}}\right)^{\frac{4}{n(p-2)+4}}
\label{13-1-p}
\end{align}
and
\begin{align}
\|\nabla\varphi(\tau_{13})\|^{2}_{2}\geq\left(\frac{nq}{4(q-2)\lambda_{2}C_{3}\|\varphi_{0}\|^{q-\frac{n(q-2)}{2}-4}_{2}}\right)^{\frac{4}{n(q-2)+4}}
,\label{13-2-q}
\end{align}
respectively.
Then, according to (\ref{13-1-p}) and (\ref{13-2-q}), we have
\begin{align}
\|\nabla\varphi(\tau_{13})\|^{2}_{2}\geq\frac{A_{6}(\|\varphi_{0}\|_{2})+B_{6}(\|\varphi_{0}\|_{2})-\left|A_{6}(\|\varphi_{0}\|_{2})-B_{6}(\|\varphi_{0}\|_{2})\right|}{2},\label{jdengyu13pq-2-2}
\end{align}
where $A_{6}(\|\varphi_{0}\|_{2}$ and $B_{6}(\|\varphi_{0}\|_{2}$ are defined by (\ref{A6}) and (\ref{B6}), respectively.
By substituting (\ref{jdengyu13pq-2-2}) into (\ref{j13}), we have
\begin{align}\label{j13--}
\begin{split}
K(\varphi(\tau_{13}))>&\frac{A_{6}(\|\varphi_{0}\|_{2})+B_{6}(\|\varphi_{0}\|_{2})-\left|A_{6}(\|\varphi_{0}\|_{2})-B_{6}(\|\varphi_{0}\|_{2})\right|}{8}
\\
&+\frac{n}{2}\left(\frac{1}{2}-\frac{2}{n(q-2)}\right)^{\frac{1}{2}}\|\varphi(\tau_{13})\|^{2}_{2}+\frac{1}{2}\| \varphi(\tau_{13})\|^{2}_{2}.
\end{split}
\end{align}
Due to (\ref{2.2}) and (\ref{2.3}), we know (\ref{j13--}) turns to
\begin{align*}
K(\varphi_{0})>&\frac{A_{6}(\|\varphi_{0}\|_{2})+B_{6}(\|\varphi_{0}\|_{2})-\left|A_{6}(\|\varphi_{0}\|_{2})-B_{6}(\|\varphi_{0}\|_{2})\right|}{8}
\\&+\left(\frac{1}{2}+\frac{n}{2}\left(\frac{1}{2}-\frac{2}{n(q-2)}\right)^{\frac{1}{2}}\right)\|\varphi_{0}\|^{2}_{2},
\end{align*}
which contradicts (\ref{K0detiaoj1}) for $j=13$. Thus, we have $S_{13}(\varphi(t))>0$ for $t\in(0,T_{\rm{max}})$.\\
{\bf{Case \textrm{VI}: \bm{$j=14.$}}}

We prove $S_{14}(\varphi(t))>0$ for $t\in(0,T_{\rm{max}})$. Arguing by contradiction, we suppose that there exists the first time $\tau_{14}\in(0,T_{\rm{max}})$ to make $S_{14}(\varphi(\tau_{14}))=0$ and $S_{14}(\varphi(t))>0$ for $t\in(0,\tau_{14})$, and we know $\varphi(\tau_{14})\neq 0$ from (\ref{2.2}) and $\varphi_{0}\neq0$. Next, we estimate $K(\varphi(\tau_{14}))$.
According to $S_{14}(\varphi(\tau_{14}))=0$, we have
\begin{align}
\frac{n(p-2)\lambda_{1}}{2p}\|\varphi(\tau_{14})\|^{p}_{p}+\frac{n(q-2)\lambda_{2}}{2q}
\|\varphi(\tau_{14})\|^{q}_{q}=\|\varphi(\tau_{14})\|^{2}_{2}+\||x|\varphi(\tau_{14})\|^{2}_{2}.\label{jdengyu14}
\end{align}
Using (\ref{jdengyu14}) and retrieving (\ref{chuli1-p-q}) for $j=14$, we know
\begin{align}
&\frac{\lambda_{1}}{p}\|\varphi(\tau_{14})\|^{p}_{p}+\frac{\lambda_{2}}{q}\|\varphi(\tau_{14})\|^{q}_{q}<\frac{2}{n(q-2)}\left(\|\varphi(\tau_{14})\|^{2}_{2}+\||x|\varphi(\tau_{14})\|^{2}_{2}\right).\label{gujijgaoneng3-14+}
\end{align}
By substituting (\ref{gujijgaoneng3-14+}) into $K(\varphi(\tau_{14}))$, we have
\begin{align}\label{pq14zhunbeimaodun-1-5}
\begin{split}
K(\varphi(\tau_{14}))>&\frac{1}{2}\|\nabla \varphi(\tau_{14})\|^{2}_{2}+\left(\frac{1}{2}-\frac{2}{n(q-2)}\right)\||x|\varphi(\tau_{14})\|^{2}_{2}+\left(\frac{1}{2}-\frac{2}{n(q-2)}\right)\| \varphi(\tau_{14})\|^{2}_{2}\\
=&\frac{1}{4}\|\nabla \varphi(\tau_{14})\|^{2}_{2}+\frac{1}{4}\|\nabla \varphi(\tau_{14})\|^{2}_{2}+\left(\frac{1}{2}-\frac{2}{n(q-2)}\right)\||x|\varphi(\tau_{14})\|^{2}_{2}\\&+\left(\frac{1}{2}-\frac{2}{n(q-2)}\right)\|\varphi(\tau_{14})\|^{2}_{2}
.
\end{split}
\end{align}
In (\ref{pq14zhunbeimaodun-1-5}), the assumption (\ref{1.2}) ensures $$\frac{1}{2}-\frac{2}{n(q-2)}>0.$$ By Cauchy-Schwartz inequality and Lemma \ref{lem3.2}, we know (\ref{pq14zhunbeimaodun-1-5}) turns to
\begin{align}\label{pq--14zhunbeimaodun-1-5}
\begin{split}
K(\varphi(\tau_{14}))>&\frac{1}{4}\|\nabla \varphi(\tau_{14})\|^{2}_{2}+\left(\frac{1}{2}-\frac{2}{n(q-2)}\right)^{\frac{1}{2}}\|\nabla \varphi(\tau_{14})\|_{2}\||x|\varphi(\tau_{14})\|_{2}\\
&+\left(\frac{1}{2}-\frac{2}{n(q-2)}\right)\|\varphi(\tau_{14})\|^{2}_{2}\\
\geq&\frac{1}{4}\|\nabla \varphi(\tau_{14})\|^{2}_{2}+\frac{n}{2}\left(\frac{1}{2}-\frac{2}{n(q-2)}\right)^{\frac{1}{2}}\| \varphi(\tau_{14})\|^{2}_{2}+\left(\frac{1}{2}-\frac{2}{n(q-2)}\right)\|\varphi(\tau_{14})\|^{2}_{2}
.\end{split}
\end{align}
Meanwhile, (\ref{jdengyu14}) can also give
\begin{align}\label{14deshizi}
\frac{n(p-2)\lambda_{1}}{2p}\|\varphi(\tau_{14})\|^{p}_{p}+\frac{n(q-2)\lambda_{2}}{2q}\|\varphi(\tau_{14})\|^{q}_{q}>\| \varphi(\tau_{14})\|^{2}_{2}.
\end{align}By the similar process of obtaining (\ref{jdengyu12-2-2}), we can use (\ref{14deshizi}) to give
\begin{align}
\|\nabla\varphi(\tau_{14})\|^{2}_{2}\geq&\frac{A_{5}(\|\varphi_{0}\|_{2})+B_{5}(\|\varphi_{0}\|_{2})-\left|A_{5}(\|\varphi_{0}\|_{2})-B_{5}(\|\varphi_{0}\|_{2})\right|}{2},
\label{--14jdengyu12-2-2}
\end{align}
where $A_{5}(\|\varphi_{0}\|_{2})$ and $B_{5}(\|\varphi_{0}\|_{2})$ are defined by (\ref{A5}) and (\ref{B6}), respectively.
 Substituting (\ref{--14jdengyu12-2-2}) into (\ref{pq--14zhunbeimaodun-1-5}), we have
\begin{align}\label{pq-+-14zhunbeimaodun-1-5}
\begin{split}
K(\varphi(\tau_{14}))
>&\frac{A_{5}(\|\varphi_{0}\|_{2})+B_{5}(\|\varphi_{0}\|_{2})-
\left|A_{5}(\|\varphi_{0}\|_{2})-B_{5}(\|\varphi_{0}\|_{2})\right|}{8}
\\&+\frac{n}{2}\left(\frac{1}{2}-\frac{2}{n(q-2)}\right)^{\frac{1}{2}}\| \varphi(\tau_{14})\|^{2}_{2}+\left(\frac{1}{2}-\frac{2}{n(q-2)}\right)\|\varphi(\tau_{14})\|^{2}_{2}
.
\end{split}
\end{align}
According to (\ref{2.2}) and (\ref{2.3}), (\ref{pq-+-14zhunbeimaodun-1-5}) turns to
\begin{align*}
K(\varphi_{0})
>&\frac{A_{5}(\|\varphi_{0}\|_{2})+B_{5}(\|\varphi_{0}\|_{2})
-\left|A_{5}(\|\varphi_{0}\|_{2})-B_{5}(\|\varphi_{0}\|_{2})\right|}{8}\\&+\left(\frac{n}{2}\left(\frac{1}{2}-\frac{2}{n(q-2)}\right)^{\frac{1}{2}}+\frac{1}{2}-\frac{2}{n(q-2)}\right)\|\varphi_{0}\|^{2}_{2},
\end{align*}
which contradicts (\ref{K0detiaoj1}) for $j=14$. Thus, we have $S_{14}(\varphi(t))>0$ for $t\in(0,T_{\rm{max}})$.\\
\textbf{Step \textrm{II}:\ Global existence.}

For the initial data $\varphi_{0}\in \mathcal{W}_{R_{j}}$, $j=9,10,\cdots,14$, we know the corresponding solution $\varphi\in \mathcal{W}_{R_{j}}$, i.e., $S_{j}(\varphi)>0$, according to Step \textrm{I}. Then, by the process similar to the proof of Theorem \ref{th5.1}, we can use the condition $S_{j}(\varphi)>0$, $j=9,10,\cdots,14$, to obtain
(\ref{youjie1})
for $j=9,10,11$ and (\ref{youjie2})
for $j=12,13,14$.
Thus, we know $T_{\rm{max}}=+\infty$ according to Lemma \ref{lem2.1}.
\end{proof}

\begin{remark}\label{rem5.3}(Fail to prove the finite time blowup by the cross-constrained variational functionals $S_{j}(\varphi)$, $j=9,10,\cdots,14$, without the term $\|\nabla\varphi\|^{2}_{2}$, for high initial energy $K(\varphi_{0})\geq d_{_{\mathcal{N}_{j}}}$, $j=9,10,\cdots,14$)
When we select and combine the energy structural terms to construct the cross-constrained variational structures, for the mandatory terms $\|\nabla\varphi\|^{2}_{2}$ and $\|\varphi\|^{p}_{p}$, if we only select the term $\|\varphi\|^{p}_{p}$ and not select the term $\|\nabla\varphi\|^{2}_{2}$, i.e., to consider the case of $S_{j}(\varphi)$, $j=9,10,\cdots,14$, it is too difficult to prove the finite time blowup for the high initial energy case $K(\varphi_{0})\geq d_{_{\mathcal{N}_{j}}}$, $j=9,10,\cdots,14$. Especially, we notice that we can prove the invariance of the manifold $\mathcal{V}_{R_{j}}$, $j=9,10,\cdots,14$, by the similar idea of proving Lemma \ref{lem5.1}. However, we fail to prove the finite time blowup by using the manifold $\mathcal{V}_{R_{j}}$, $j=9,10,\cdots,14$. This is mainly because we need to use auxiliary functional
$J(t)=\||x|\varphi\|^{2}_{2}$ and to verify $J''(t)<0$ when we prove the finite time blowup. Then, due to the lack of the term $\|\nabla\varphi\|^{2}_{2}$ in the cross-constrained variational functionals $S_{j}(\varphi)$, $j=9,10,\cdots,14$, we shall meet the similar difficulty pointed out in Remark \ref{rem4.2}, which means that we fail to prove the finite time blowup for the high initial energy case $K(\varphi_{0})\geq d_{_{\mathcal{N}_{j}}}$, $j=9,10,\cdots,14$.
\end{remark}

\begin{remark}\label{rem5.4}(Fail to prove the global existence and finite time blowup by the cross-constrained variational functionals $S_{j}(\varphi)$, $j=15,16,\cdots,21$, without the term $\|\varphi\|^{p}_{p}$ for high initial energy $K(\varphi_{0})\geq d_{_{\mathcal{N}_{j}}}$, $j=15,16,\cdots,21$)
When we construct the cross-constrained variational structures by selecting and combining the energy structure terms, if we do not select the mandatory term $\|\varphi\|^{p}_{p}$, i.e., we consider the case of $S_{j}(\varphi)$, $j=15,16,\cdots,21$, for the high initial energy case $K(\varphi_{0})\geq d_{_{\mathcal{N}_{j}}}$, $j=15,16,\cdots,21$, we can neither prove the global existence nor the finite time blowup. The reason is that we fail to prove the invariance of the corresponding manifolds.
For high initial energy case $K(\varphi_{0})\geq d_{_{\mathcal{N}_{j}}}$, $j=15,16,\cdots,21$, in order to prove the invariance of the manifolds, by the similar analysis in Remark \ref{rem5.2}, we need to find additional conditions to retrieve the contradiction between the sign-change of $S_{j}(\varphi)$ and the assumption $K(\varphi_{0})\geq d_{{\mathcal{N}_{j}}}$, $j=15,16,\cdots,21$. Here, like the cases $j=1,2,\cdots,14$, in order to construct such additional conditions, we need to apply the technique of controlling the negative terms in ``energy'' functional $K(\varphi)$, i.e., $-\frac{\lambda_{1}}{p}\|\varphi\|^{p}_{p}$ and $-\frac{\lambda_{2}}{q}\|\varphi\|^{q}_{q}$, in $\mathcal{N}_{j}$, $j=1,2,\cdots,14$, as shown in Lemma \ref{lem5.1} and Theorem \ref{dinglis9--14}. However, for the cross-constrained variational functionals $S_{j}(\varphi)$, $j=15,16,\cdots,21$, due to the lack of the term $\|\varphi\|^{p}_{p}$, by the similar analysis in Remark \ref{rem4.3}, we fail to
control the above negative terms in $\mathcal{N}_{j}$, $j=15,16,\cdots,21$, which means that we fail to prove the invariance of the manifolds. This is the reason why we fail to prove the global existence and finite time blowup by the cross-constrained variational functionals $S_{j}(\varphi)$, $j=15,16,\cdots,21$, without the term $\|\varphi\|^{p}_{p}$ for high initial energy $K(\varphi_{0})\geq d_{_{\mathcal{N}_{j}}}$, $j=15,16,\cdots,21$.
\end{remark}

\subsection{The relationship between the eight classifications of the initial data given by the eight different cross-constrained variational functionals $S_{j}(\varphi)$, $j=1,2,\cdots,8$}\label{sec5.3}

In Section \ref{sec4}, we used the sign of eight cross-constrained variational functionals $S_{j}(\varphi)$, $j=1,2,\cdots,8$, having different structures to classify the initial data to derive eight different criteria to distinguish the global existence and the finite time blowup for the non-negative low initial energy case $0\leq K(\varphi_{0})<d_{\mathcal{N}_{j}}$, $j=1,2,\cdots,8$, as well as for the high initial energy case $K(\varphi_{0})\geq d_{\mathcal{N}_{j}}$, $j=1,2,\cdots,8$, in Section \ref{sec5}. Exactly speaking, we use the positive sign of $S_{j}(\varphi)$, i.e., $S_{j}(\varphi)>0$, to describe the initial data leading to the global existence and use the negative sign of $S_{j}(\varphi)$, i.e., $S_{j}(\varphi)<0$, to describe the initial data leading to the finite time blowup. Because these $S_{j}(\varphi)$, $j=1,2,\cdots,8$, have different structures based on their different combinations of the energy structural terms, these classification criteria are different from each other. Clearly speaking,  for different $S_{j}(\varphi)$, for example, for any $k\neq l$, $S_{k}(\varphi)>0$ dose not mean $S_{l}(\varphi)>0$. Next, we shall show an interesting fact that if we impose some conditions on the initial energy $K(\varphi_{0})$ to ``press'' it to a sufficient low level, these eight different cross-constrained variational functionals $S_{j}(\varphi)$, $j=1,2,\cdots,8$, play the same role.

 In Theorem \ref{th4.1} and Theorem \ref{th4.2}, for the non-negative low initial energy $0\leq K(\varphi_{0})<d_{\mathcal{N}_{j}}$, $j=1,2,\cdots,8$, we used the sign of the eight cross-constrained variational functionals $S_{j}(\varphi)$, $j=1,2,\cdots,8$, to divide the space of the initial data
 \begin{align*}
\mathcal{H}_{j}=\left\{\varphi\in H\big|0\leq K(\varphi)<d_{\mathcal{N}_{j}}\right\}
\end{align*}
into
\begin{align*}
\mathcal{W}_{j}=\left\{\varphi\in H\big|0\leq K(\varphi)<d_{{\mathcal{N}_{j}}},S_{j}(\varphi)>0\right\}\cup\big\{0\big\}
\end{align*}
and
\begin{align*}
\mathcal{V}_{j}=\left\{\varphi\in H\big|0\leq K(\varphi)<d_{{\mathcal{N}_{j}}},S_{j}(\varphi)<0\right\},
\end{align*}
which are applied to prove the global existence and the finite time blowup, respectively. And in Theorem \ref{th5.1} and Theorem \ref{th5.2}, for the high initial energy $K(\varphi_{0})\geq d_{\mathcal{N}_{j}}$, $j=1,2,\cdots,8$, we used the sign of $S_{j}(\varphi)$, $j=1,2,\cdots,8$, to divide the space of the initial data
\begin{align*}
\mathcal{H}_{R_{j}}=\left\{\varphi\in H\Big|d_{_{\mathcal{N}_{j}}}\leq K(\varphi)\leq R_{j}(\|\varphi\|_{2})\right\}
\end{align*}into
\begin{align*}
\mathcal{W}_{R_{j}}=\left\{\varphi\in H\Big|d_{_{\mathcal{N}_{j}}}\leq K(\varphi)\leq R_{j}(\|\varphi\|_{2}),S_{j}(\varphi)>0\right\}
\end{align*}
and
\begin{align*}
\mathcal{V}_{R_{j}}=\left\{\varphi\in H\Big|d_{_{\mathcal{N}_{j}}}\leq K(\varphi)\leq R_{j}(\|\varphi\|_{2}),S_{j}(\varphi)<0\right\},
\end{align*}
 which are applied to prove the global existence and the finite time blowup, respectively.
 Now we have to deal with the situation that we have many, for example, $s$ different criteria, $s\in\left\{1,2,\cdots,8\right\}$, which means we have $s$ different Nehari manifolds $\mathcal{N}_{j_{1}}, \mathcal{N}_{j_{2}}, \cdots, \mathcal{N}_{j_{s}}$ with corresponding $d_{_{\mathcal{N}_{j_{1}}}}, d_{_{\mathcal{N}_{j_{2}}}}, \cdots, d_{_{\mathcal{N}_{j_{s}}}}$ and $S_{j_{1}}(\varphi), S_{j_{2}}(\varphi), \cdots, S_{j_{s}}(\varphi)$. Here, for $s\in\{1,2,\cdots,8\}$, we use the minimum of these $d_{_{\mathcal{N}_{j_{1}}}}, d_{_{\mathcal{N}_{j_{2}}}}, \cdots,$ $d_{_{\mathcal{N}_{j_{s}}}}$ defined as \begin{align}\label{dmin}
d^{{\rm min}}_{{\mathcal{N}_{j_{1}}},\cdots,{\mathcal{N}_{j_{s}}}}:=\min\left\{d_{\mathcal{N}_{j_{1}}}
,d_{\mathcal{N}_{j_{2}}},\cdots,d_{\mathcal{N}_{j_{s}}}\right\}
\end{align} to ``press'' the energy level to the lowest, i.e.,
\begin{align}\label{guangyi-dineng-1}
\mathcal{H}_{j_{1},\cdots,j_{s}}:=\left\{\varphi\in H\Big|0\leq K(\varphi)<d^{{\rm min}}_{{\mathcal{N}_{j_{1}}},\cdots,{\mathcal{N}_{j_{s}}}}\right\}.
\end{align}
Obviously, (\ref{guangyi-dineng-1}) can be regarded as \begin{align*}
&\mathcal{H}_{j_{1},\cdots,j_{s}}=\bigcap_{k=1}^{s}\mathcal{H}_{j_{k}}.
\end{align*}
 Now we consider these $s$ different cross-constrained variational problems. We use $j_{k}\in\{j_{1},j_{2},\cdots,\\
 j_{s}\}$ to denote these $s$ cross-constrained variational functionals $S_{j_{k}}(\varphi)$, $j_{k}\in\{j_{1},j_{2},\cdots,j_{s}\}$, having different structures. According to the signs of $S_{j_{k}}(\varphi)$, we divide the space of the initial data $\mathcal{H}_{j_{1},\cdots,j_{s}}$ into
\begin{align*}
\mathcal{W}_{j_{k},d^{{\rm min}}_{{\mathcal{N}_{j_{1}}},\cdots,{\mathcal{N}_{j_{s}}}}}:=\left\{\varphi\in H\Big|0\leq K(\varphi)<d^{{\rm min}}_{{\mathcal{N}_{j_{1}}},\cdots,{\mathcal{N}_{j_{s}}}},S_{j_{k}}(\varphi)>0\right\}\cup \Big\{0\Big\}
\end{align*}
and
\begin{align*}
\mathcal{V}_{j_{k},d^{{\rm min}}_{{\mathcal{N}_{j_{1}}},\cdots,{\mathcal{N}_{j_{s}}}}}:=\left\{\varphi\in H\Big|0\leq K(\varphi)<d^{{\rm min}}_{{\mathcal{N}_{j_{1}}},\cdots,{\mathcal{N}_{j_{s}}}},S_{j_{k}}(\varphi)<0\right\},
\end{align*}that is,
\begin{align}
\mathcal{H}_{j_{1},\cdots,j_{s}}=\mathcal{W}_{j_{k},d^{{\rm min}}_{{\mathcal{N}_{j_{1}}},\cdots,{\mathcal{N}_{j_{s}}}}}\cup\mathcal{V}_{j_{k},d^{{\rm min}}_{{\mathcal{N}_{j_{1}}},\cdots,{\mathcal{N}_{j_{s}}}}}.
\end{align}
In following Theorem \ref{th5.3}, we claim that if we control the initial energy $K(\varphi_{0})$ to the lowest level, i.e.,
\begin{align}\label{Kxiaoyud0}
0\leq K(\varphi_{0})<d^{{\rm min}}_{{\mathcal{N}_{j_{1}}},\cdots,{\mathcal{N}_{j_{s}}}},
\end{align}the results of classifying the initial data by using the sign of the different cross-constrained variational functionals $S_{j_{k}}(\varphi)$, $j_{k}\in\{j_{1},j_{2},\cdots,j_{s}\}$ are same. In other words, these different $S_{j_{k}}(\varphi)$ play the same role in this case.

\begin{theorem}[The relationship between the classifications of the initial data given by the $s$ different cross-constrained variational functionals $S_{j_{k}}(\varphi)$, $j_{k}\in\{j_{1},j_{2},\cdots,j_{s}\}$, for the non-negative low initial energy]\label{th5.3}
Let $p$ and $q$ satisfy (\ref{1.2}), $\lambda_{1}>0$ and $\lambda_{2}>0$ satisfy (\ref{2.9}) and (\ref{2.10}). For $s\in\left\{1,2,\cdots,8\right\}$, $j_{k}\neq j_{m}$, $j_{k},j_{m}\in\{j_{1},j_{2},\cdots,j_{s}\}\subset\{1,2,\cdots,8\}$, one has \begin{align}\label{guangyi-dineng}
\mathcal{W}_{j_{k},d^{{\rm min}}_{{\mathcal{N}_{j_{1}}},\cdots,{\mathcal{N}_{j_{s}}}}}=\mathcal{W}_{j_{m},d^{{\rm min}}_{{\mathcal{N}_{j_{1}}},\cdots,{\mathcal{N}_{j_{s}}}}}
\end{align}
and
\begin{align}\label{guangyi-dineng2}
\mathcal{V}_{j_{k},d^{{\rm min}}_{{\mathcal{N}_{j_{1}}},\cdots,{\mathcal{N}_{j_{s}}}}}=\mathcal{V}_{j_{m},d^{{\rm min}}_{{\mathcal{N}_{j_{1}}},\cdots,{\mathcal{N}_{j_{s}}}}}.
\end{align}
 \end{theorem}
 \begin{proof}
 We only write down the details of proving the conclusion (\ref{guangyi-dineng}) since he proof of the conclusion (\ref{guangyi-dineng2}) is similar.
For $j_{k}\neq j_{m}$, $j_{k},j_{m}\in\{j_{1},j_{2},\cdots,j_{s}\}$, from (\ref{Kxiaoyud0}), we know
\begin{align}\label{Kxiaoyudjk}
0\leq K(\varphi_{0})<d_{{\mathcal{N}_{j_{k}}}}
\end{align}
and
\begin{align}\label{Kxiaoyudjm}
0\leq K(\varphi_{0})<d_{{\mathcal{N}_{j_{m}}}}.
\end{align}
Then, to prove the conclusion (\ref{guangyi-dineng}), we only need to claim that for the initial data $\varphi_{0}$ satisfying both (\ref{Kxiaoyudjk}) and (\ref{Kxiaoyudjm}), $S_{j_{k}}(\varphi_{0})>0$ means $S_{j_{m}}(\varphi_{0})>0$. Arguing by contradiction, we suppose that $S_{j_{m}}(\varphi_{0})\leq0$. If $S_{j_{m}}(\varphi_{0})=0$ holds, from the definition of $d_{{\mathcal{N}_{j_{m}}}}$, i.e., (\ref{2.5}), and $\varphi_{0}\neq0$ given by the condition $S_{j_{k}}(\varphi_{0})>0$, we deduce
\begin{align*}
K(\varphi_{0})\geq d_{{\mathcal{N}_{j_{m}}}},
\end{align*}which contradicts (\ref{Kxiaoyudjm}). And if $S_{j_{m}}(\varphi_{0})<0$ holds, by (\ref{Kxiaoyudjm}), we see that $\varphi_{0}\in\mathcal{V}_{j_{m}}$. Then, according to Theorem \ref{th4.2}, we know that the corresponding solution $\varphi$ to the problem (\ref{1.1}) blows up in finite time. However, the conditions $S_{j_{k}}(\varphi_{0})>0$ and (\ref{Kxiaoyudjk}) imply $\varphi_{0}\in\mathcal{W}_{j_{k}}$, which means that the solution $\varphi$ exists globally according to Theorem \ref{th4.1}. This leads to a contradiction. Thus, we have $S_{j_{m}}(\varphi_{0})>0$.
\end{proof}

The $s$ different cross-constrained variational functionals $S_{j_{k}}(\varphi)$, $j_{k}\in\{j_{1},j_{2},\cdots,j_{s}\}$, reflect the $s$ different combinations of the energy structural terms and thus have the different structures, which leads to the $s$ different Nehari manifolds $\mathcal{N}_{j_{k}}$, i.e., $S_{j_{k}}(\varphi)=0$, $j_{k}\in\{j_{1},j_{2},\cdots,j_{s}\}$. Based on common sense and intuition, we should have believed that these $s$ different cross-constrained variational functionals should give the $s$ different classifications of the initial data.
 However, the conclusions (\ref{guangyi-dineng}) and (\ref{guangyi-dineng2}) present a fact contrary to this intuition, that is, if we control the initial energy $K(\varphi_{0})$ to the lowest energy level, the classifications of the initial data given by these $s$ different cross-constrained variational functionals are same. Next, in order to clearly show the classifications of the initial data, i.e., the conclusions (\ref{guangyi-dineng}) and (\ref{guangyi-dineng2}), we take the example of $s=3$ and assume
\begin{align}\label{dizhengdejiashe}
d_{{\mathcal{N}_{j_{1}}}}<d_{{\mathcal{N}_{j_{2}}}}<d_{{\mathcal{N}_{j_{3}}}},
\end{align} and then use Figure 2 to make a comparison of the classifications and the sensory misclassifications of the initial data.
\\
\begin{figure}[!h]
  \centering \includegraphics[width=0.9\textwidth]{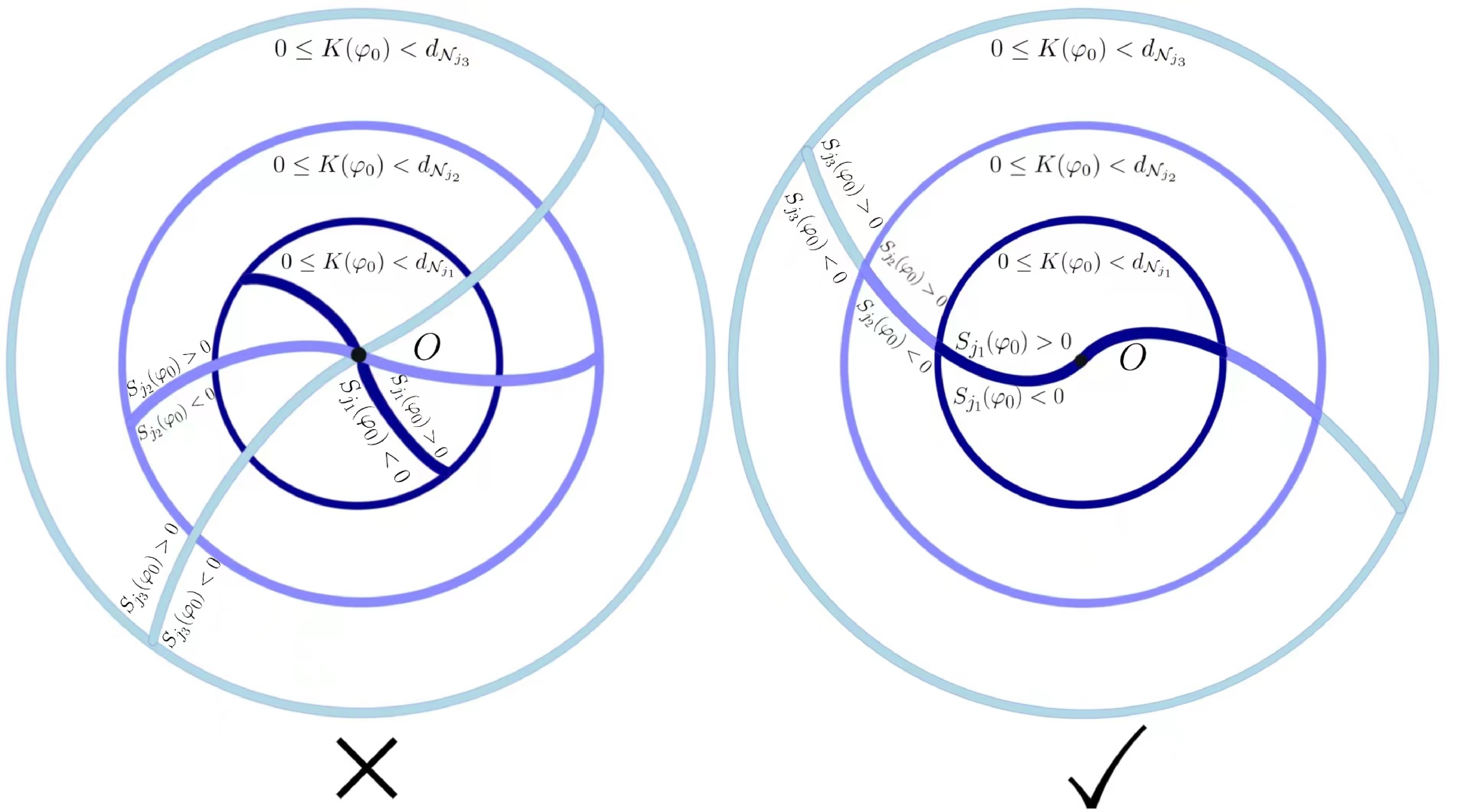}\\
  \caption{The sensory misclassifications of the initial data (left) v.s. the actual classifications of the initial data (right) by using $S_{j_{1}}(\varphi)-S_{j_{3}}(\varphi)$}
  \label{b}
\end{figure}

In the left and right schematic diagrams in Figure 2, we use the parts inside the three circles with different colors to represent the initial data with energy levels $0\leq K(\varphi_{0})<d_{{\mathcal{N}_{j_{k}}}}$, $k=1,2,3$, respectively. Here, using a sphere to represent $0\leq K(\varphi_{0})<d_{{\mathcal{N}_{j_{k}}}}$, $k=1,2,3$, makes it convenient for us to intuitively understand the relationship between $0\leq K(\varphi_{0})<d_{{\mathcal{N}_{j_{k}}}}$, $k=1,2,3$, and $S_{j_{k}}(\varphi_{0})>0$ or $S_{j_{k}}(\varphi_{0})<0$, $k=1,2,3$. However, from the definition of the ``energy'' functional $K(\varphi)$, i.e., (\ref{2.4}), we have no reason to believe that $K(\varphi)$ can represent the norm of a certain space, hence it is not rigorous mathematically to regard $K(\varphi)$ as the radius to represent $0\leq K(\varphi_{0})<d_{{\mathcal{N}_{j_{k}}}}$, $k=1,2,3$. Thus, we use three concentric circles with increasing radii as a schematic diagram to show the inclusion relationship between these three different energy levels.
In the left schematic diagram in Figure 2, we show the sensory misclassifications of the initial data with energy levels $0\leq K(\varphi_{0})<d_{{\mathcal{N}_{j_{k}}}}$, $k=1,2,3$, by using $S_{j_{k}}(\varphi)$, $k=1,2,3$. More precisely, according to our intuition, three different cross-constrained variational functionals $S_{j_{k}}(\varphi)$, $k=1,2,3$, should give three different classifications of the initial data. In order to present this intuition, we use three curves with three different colors that do not overlap to represent $S_{j_{k}}(\varphi_{0})=0$, $k=1,2,3$, respectively, which are used to divide the inner parts of the three circles with three colors, i.e., $0\leq K(\varphi_{0})<d_{{\mathcal{N}_{j_{k}}}}$, $k=1,2,3$. However, as shown in the right schematic diagram in Figure 2,
the conclusions (\ref{guangyi-dineng}) and (\ref{guangyi-dineng2}) tell us that the three curves representing $S_{j_{k}}(\varphi_{0})=0$, $k=1,2,3$, respectively, with the different colors overlap in the lowest energy level $0\leq K(\varphi_{0})<d_{{\mathcal{N}_{j_{1}}}}$. And the curves representing $S_{j_{2}}(\varphi_{0})=0$ and $S_{j_{3}}(\varphi_{0})=0$ overlap in the second lowest energy level $0\leq K(\varphi_{0})<d_{{\mathcal{N}_{j_{2}}}}$. Certainly, this observation is also valid for any case when we use $s$ ($s=2,3,\cdots,8$) different cross-constrained variational functionals $S_{j_{k}}(\varphi)$, $j_{k}\in\{j_{1},j_{2},\cdots,j_{s}\}$, to classify the initial data.

For the high initial energy case and $s\in\{2,3,\cdots,8\}$, we define the space of the initial data
\begin{align}\label{guangyi-gaoneng-1}
\mathcal{H}_{R_{j_{1},\cdots,j_{s}}}:=\left\{\varphi\in H\Big|d^{{\rm max}}_{{\mathcal{N}_{j_{1}}},\cdots,{\mathcal{N}_{j_{s}}}}\leq K(\varphi)\leq R^{{\rm min}}_{j_{1},\cdots,j_{s}}(\|\varphi\|_{2})\right\},
\end{align}where \begin{align}\label{dmax}
d^{{\rm max}}_{{\mathcal{N}_{j_{1}}},\cdots,{\mathcal{N}_{j_{s}}}}:=\max\left\{d_{\mathcal{N}_{j_{1}}},d_{\mathcal{N}_{j_{2}}},\cdots,d_{\mathcal{N}_{j_{s}}}\right\}
\end{align}and
\begin{align}\label{Rmin}
R^{{\rm min}}_{j_{1},\cdots,j_{s}}(\|\varphi\|_{2}):=\min\Big\{R_{j_{1}}(\|\varphi\|_{2}),R_{j_{2}}(\|\varphi\|_{2}),\cdots,
R_{j_{s}}(\|\varphi\|_{2})\Big\}.
\end{align}
Here, (\ref{guangyi-gaoneng-1}) can be regarded as \begin{align*}
&\mathcal{H}_{R_{j_{1},\cdots,j_{s}}}=\bigcap_{k=1}^{s}\mathcal{H}_{R_{j_{k}}}.
\end{align*}
For any cross-constrained variational functionals $S_{j_{k}}(\varphi)$, $j_{k}\in\{j_{1},j_{2},\cdots,j_{s}\}$,
we use the sign of $S_{j_{k}}(\varphi)$
to divide the space of the initial data $\mathcal{H}_{R_{j_{1},\cdots,j_{s}}}$ into
\begin{align*}
\mathcal{W}_{R_{j_{k}},d_{\mathcal{N}_{j_{1}},\cdots,\mathcal{N}_{j_{s}}}^{{\rm max}}}:=\left\{\varphi\in H\Big|d^{{\rm max}}_{{\mathcal{N}_{j_{1}}},\cdots,{\mathcal{N}_{j_{s}}}}\leq K(\varphi)\leq R^{{\rm min}}_{j_{1},\cdots,j_{s}}(\|\varphi\|_{2}),S_{j_{k}}(\varphi)>0\right\}
\end{align*}
and
\begin{align*}
\mathcal{V}_{R_{j_{k}},d^{{\rm max}}_{{\mathcal{N}_{j_{1}}},\cdots,{\mathcal{N}_{j_{s}}}}}:=\left\{\varphi\in H\Big|d^{{\rm max}}_{{\mathcal{N}_{j_{1}}},\cdots,{\mathcal{N}_{j_{s}}}}\leq K(\varphi)\leq R^{{\rm min}}_{j_{1},\cdots,j_{s}}(\|\varphi\|_{2}),S_{j_{k}}(\varphi)<0\right\},
\end{align*}respectively, that is,
\begin{align*}
\mathcal{H}_{R_{j_{1},\cdots,j_{s}}}=\mathcal{W}_{R_{j_{k}},d_{\mathcal{N}_{j_{1}},\cdots,\mathcal{N}_{j_{s}}}^{{\rm max}}}\cup\mathcal{V}_{R_{j_{k}},d_{\mathcal{N}_{j_{1}},\cdots,\mathcal{N}_{j_{s}}}^{{\rm max}}}.
\end{align*} Next, in the following Theorem \ref{th5.4}, we shall show that
if we apply a stricter condition
\begin{align}\label{gaonengdebijiao}
d^{{\rm max}}_{{\mathcal{N}_{j_{1}}},\cdots,{\mathcal{N}_{j_{s}}}}\leq K(\varphi_{0})\leq R^{{\rm min}}_{j_{1},j_{2},\cdots,j_{s}}(\|\varphi_{0}\|_{2})
\end{align} on the initial energy, the results of classifying the initial data by using the sign of the different cross-constrained variational functionals $S_{j_{k}}(\varphi)$, $j_{k}\in\{j_{1},j_{2},\cdots,j_{s}\}$, are also the same.
\begin{theorem}\label{th5.4}(Relationship between the classifications of the initial data given by $s$ different cross-constrained variational functionals $S_{j_{k}}(\varphi)$, $j_{k}\in\{j_{1},j_{2},\cdots,j_{s}\}$, for the high initial energy)
Let $p$ and $q$ satisfy (\ref{1.2}), $\lambda_{1}>0$ and $\lambda_{2}>0$ satisfy (\ref{2.9}) and (\ref{2.10}). For $s\in\left\{1,2,\cdots,8\right\}$ and any $j_{k}\neq j_{m}$, $j_{k},j_{m}\in\{j_{1},j_{2},\cdots,j_{s}\}\subset\{1,2,\cdots,8\}$, one has \begin{align}\label{guangyi-gaoneng}
\mathcal{W}_{R_{j_{k}},d^{{\rm max}}_{{\mathcal{N}_{j_{1}}},\cdots,{\mathcal{N}_{j_{s}}}}}=\mathcal{W}_{R_{j_{m}},d^{{\rm max}}_{{\mathcal{N}_{j_{1}}},\cdots,{\mathcal{N}_{j_{s}}}}}
\end{align}
and
\begin{align}\label{guangyi-gaoneng2}
\mathcal{V}_{R_{j_{k}},d^{{\rm max}}_{{\mathcal{N}_{j_{1}}},\cdots,{\mathcal{N}_{j_{s}}}}}=\mathcal{V}_{R_{j_{m}},d^{{\rm max}}_{{\mathcal{N}_{j_{1}}},\cdots,{\mathcal{N}_{j_{s}}}}}.
\end{align}
\end{theorem}
Here, since the proofs of the conclusions (\ref{guangyi-gaoneng}) and (\ref{guangyi-gaoneng2}) are similar to (\ref{guangyi-dineng}) and (\ref{guangyi-dineng2}), respectively, we omit the details of proof of Theorem \ref{th5.4}. And by the similar analysis in Remark \ref{Manifoldsfeikong}, we can find the example of the initial data belonging to $\mathcal{W}_{R_{j_{k}},d^{{\rm max}}_{{\mathcal{N}_{j_{1}}},\cdots,{\mathcal{N}_{j_{s}}}}}$ or $\mathcal{V}_{R_{j_{k}},d^{{\rm max}}_{{\mathcal{N}_{j_{1}}},\cdots,{\mathcal{N}_{j_{s}}}}}$, $j_{k}\in\{j_{1},j_{2},\cdots,j_{s}\}$, $s\in\{1,2,\cdots,8\}$.

From the above analysis, for any $s$ classifications of the initial data obtained by using the corresponding cross-constrained variational functionals $S_{j_{k}}(\varphi)$, $j_{k}\in\{j_{1},j_{2},\cdots,j_{s}\}$, $s\in\{1,2,\cdots,8\}$, in non-negative low initial energy case, as long as we decrease the initial energy from $0\leq K(\varphi_{0})<d_{\mathcal{N}_{j_{k}}}$, $j_{k}\in\{j_{1},j_{2},\cdots,j_{s}\}$, $s\in\{1,2,\cdots,8\}$, to (\ref{Kxiaoyud0}), we can see that these classifications of the initial data are same. On the surface, by imposing the stricter conditions on initial energy $K(\varphi_{0})$, we can make these classifications obtained by the different cross-constrained variational functionals $S_{j_{k}}(\varphi)$, $j_{k}\in\{j_{1},j_{2},\cdots,j_{s}\}$, $s\in\{1,2,\cdots,8\}$, indistinguishable. Through in-depth analysis, this fact actually reflects the relationship between the initial energy level and different combinations of the energy structural terms. More precisely, if we put the initial energy $K(\varphi_{0})$ at the different levels, the classifications of the initial data obtained by using the different cross-constrained variational functionals $S_{j}(\varphi)$, $j=1,2,\cdots,8$, will be different, and if we put the initial energy $K(\varphi_{0})$ at the lowest level, these classifications will be same. And for the high initial energy case, by the similar analysis, we can see the similar fact to the non-negative low energy case we analyzed above.

\section{High initial energy for semilinear Schr\"{o}dinger equation}\label{sec6}
In this section, we focus on the Cauchy problem of the standard semilinear Schr\"{o}dinger equation
\begin{flalign}\label{6.1}
(a) \ \ \ \ \ \ \ \ \ \ \ \ \ \ \ \ \ \ \ \ \ \ \ \ \ \ \ \
\begin{cases}
\displaystyle i \varphi_{t}+\Delta \varphi+|\varphi|^{p-2}\varphi=0, \ \ x\in \mathbb{R}^n, \ t>0,\\[2pt]
\displaystyle \varphi(x,0)=\varphi_0(x), \ \ \ \ \ \ \ \ \ \ \ \ \ x\in \mathbb{R}^n.
\end{cases}
\ \ \ \ \ \  \ \ \ \ \ \  \ \ \ \ \ \ \  \ \  \ \ \
\end{flalign}We aim to prove the global existence and the finite time blowup for the high initial energy case, like what we did in Section \ref{sec5}. On the surface, the model equation in (\ref{1.1}) studied in Section \ref{sec5} covers the model equation in (\ref{6.1}). It seems that either the problem (\ref{6.1}) need not be investigated in the high initial energy case, or we only need to directly apply the strategy proposed in Section \ref{sec5} to obtain the corresponding conclusions for the problem (\ref{6.1}) easily, however, from a mathematical perspective, we still cannot make such simple judgment. Although the five structural terms in $K(\varphi)$ of (\ref{1.1}), i.e., (\ref{2.4}), include all the three structural terms in $K_{(a)}(\varphi)$ of (\ref{6.1}), i.e., (\ref{Ka}), we need to consider a series of different combinations of the three ``energy'' structural terms for the problem (\ref{6.1}), while such analysis and discussion are different from those for the combinations of the five ``energy'' structural terms in Section \ref{sec5}.
Especially it is a classical and important model. By a quick review of the references \cite{r4-1-2-0,55-9,Annals-2005,r4-1-2-00,r14}, we see that high initial energy case for (\ref{6.1}) has remained an open problem and received constant focus for nearly half a century. In view of the characteristic of the strategy proposed in Section \ref{sec5}, i.e., to consider as many effective cross-constrained variational combinations as possible, for the model equation in (\ref{6.1}), we shall also find out as many effective cross-constrained variational combinations as possible to discuss the corresponding high energy case.

We first construct the cross-constrained variational structures for the model equation in (\ref{6.1}).
 More precisely, we select and combine the three structural terms in ``energy'' functional $K_{(a)}(\varphi)$ defined for the model equation in (\ref{6.1}) by adjusting their coefficients to present three cross-constrained variational combinations. Here, since we used the notation $(a)$ to denote the model equation in (\ref{6.1}) in the ``Introduction'' section, we also use the lower index $(a)$ to denote the cross-constrained variational functionals and the corresponding manifolds defined for the problem (\ref{6.1}). And since we used the notation $S_{1,(a)}(\varphi)$ to denote the classical form of the Nehari functional in the potential well theory in the ``Introduction'' section,
  we shall use the notations $S_{2,(a)}(\varphi)$-$S_{4,(a)}(\varphi)$ to denote three possible cross-constrained variational functionals for the Cauchy problem (\ref{6.1}).
The cross-constrained variational $S_{2,(a)}(\varphi)$ is defined by (\ref{S2a}), and we define
 \begin{align*}
S_{3,(a)}(\varphi):=\|\nabla \varphi\|^{2}_{2}+\tilde{c}\|\varphi\|^{2}_{2}-\frac{n(p-2)}{2p}\|\varphi\|^{p}_{p}, \ 0<\tilde{c}<1,
\end{align*}
and
\begin{align*}
S_{4,(a)}(\varphi):=\tilde{c}\|\varphi\|^{2}_{2}-\frac{n(p-2)}{2p}\|\varphi\|^{p}_{p}, \ 0<\tilde{c}<1.
\end{align*}
 Subsequently, we shall use the cross-constrained variational functionals $S_{j,(a)}(\varphi)$, $j=2,3,4$, to investigate the global existence and the finite time blowup for the high initial energy case.
 Hence, we shall use the lower index $j=2,3,4$ to denote the corresponding variational structures and manifolds. The cross-constrained variational structure
$d_{\mathcal{N}_{2,(a)}}$ has been defined in (\ref{d2a}). We only need to define
\begin{align}\label{dn3a}
d_{\mathcal{N}_{3,(a)}}:=\inf_{\varphi\in\mathcal{N}_{3,(a)}}K_{(a)}(\varphi)
\end{align}
and
\begin{align}\label{dn4a}
d_{\mathcal{N}_{4,(a)}}:=\inf_{\varphi\in\mathcal{N}_{4,(a)}}K_{(a)}(\varphi),
\end{align}
 where $K_{(a)}(\varphi)$ is defined by (\ref{Ka}),
\begin{align*}
\mathcal{N}_{3,(a)}:=\left\{\varphi\in H^{1}(\mathbb{R}^{n})\big|S_{3,(a)}(\varphi)=0\right\},
\end{align*}
and
\begin{align*}
\mathcal{N}_{4,(a)}:=\left\{\varphi\in H^{1}(\mathbb{R}^{n})\big|S_{4,(a)}(\varphi)=0\right\}.
\end{align*}
 Then, we define the space of the initial data
 \begin{align}\label{HR2a}
\mathcal{H}_{R_{2,(a)}}:=&\left\{\varphi\in H^{1}(\mathbb{R}^{n})\Big|d_{\mathcal{N}_{1,(a)}}\leq K_{(a)}(\varphi)<R_{2,(a)}(\|\varphi\|_{2})\right\},
\end{align}and
  \begin{align}
\mathcal{H}_{R_{j,(a)}}:=\left\{\varphi\in H^{1}(\mathbb{R}^{n})\Big|d_{\mathcal{N}_{j,(a)}}\leq K_{(a)}(\varphi)< R_{j,(a)}(\|\varphi\|_{2})\right\},\ j=3,4,
\end{align}
where
\begin{align}\label{R2a}R_{2,(a)}(\|\varphi\|_{2}):=&\left(\frac{1}{2}-\frac{2}{n(p-2)\gamma}\right)M_{(a)}(\|\varphi\|_{2})+\frac{1}{2}\|\varphi\|^{2}_{2},\end{align}
\begin{align*}
R_{3,(a)}(\|\varphi\|_{2}):=&\left(\frac{1}{2}-\frac{2}{n(p-2)}\right)\frac{\hat{M}_{(a)}(\|\varphi\|_{2})+\tilde{M}_{(a)}(\|\varphi\|_{2})+\left|\hat{M}_{(a)}(\|\varphi\|_{2})
-\tilde{M}_{(a)}(\|\varphi\|_{2})\right|}{2}
\\&+\left(\frac{1}{2}-\frac{2\tilde{c}}{n(p-2)}\right)\|\varphi\|^{2}_{2},\end{align*}
\begin{align*}
R_{4,(a)}(\|\varphi\|_{2}):=\left(\frac{1}{2}-\frac{2 \tilde{c}}{n(p-2)}\right)\|\varphi\|^{2}_{2}+\frac{1}{2}\left(\frac{2p\tilde{c}}{n(p-2) C_{1}\|\varphi\|^{p-\frac{n(p-2)}{2}-2}_{2}}\right)^{\frac{4}{n(p-2)}},
\end{align*}
\begin{align}\label{Maf}
M_{(a)}(\|\varphi\|_{2}):=\left(\frac{2p}{n(p-2)\gamma C_{1}\|\varphi\|^{p-\frac{n(p-2)}{2}}_{2}}\right)^{\frac{4}{n(p-2)-4}},
\end{align}
\begin{align}\label{Mhaf}
\hat{M}_{(a)}(\|\varphi\|_{2}):=\left(\frac{2p}{n(p-2) C_{1}\|\varphi\|^{p-\frac{n(p-2)}{2}}_{2}}\right)^{\frac{4}{n(p-2)-4}},
\end{align}and \begin{align}\label{Mtilea}
\tilde{M}_{(a)}(\|\varphi\|_{2}):=\left(\frac{4p \tilde{c}^{\frac{1}{2}}}{n(p-2)C_{1}\|\varphi\|^{p-\frac{n(p-2)}{2}-1}_{2}}\right)^{\frac{4}{n(p-2)-2}}.
\end{align}
Here, for the case $j=2$, in the definition of $\mathcal{H}_{R_{2,(a)}}$, i.e., (\ref{HR2a}), we consider the case $K_{(a)}(\varphi_{0})\geq d_{\mathcal{N}_{1,(a)}}$ instead of the case $K_{(a)}(\varphi_{0})\geq d_{\mathcal{N}_{2,(a)}}$, where $d_{\mathcal{N}_{1,(a)}}$ is defined by (\ref{d1a}), because in \cite{r14},
 the global well-posedness for low initial energy $K_{(a)}(\varphi_{0})<d_{\mathcal{N}_{1,(a)}}$ was investigated.
In addition, to deal with the high initial energy $K_{(a)}(\varphi_{0})\geq d_{\mathcal{N}_{1,(a)}}$, we choose coefficient $\gamma>0$ in the definition of $S_{2,(a)}(\varphi)$, i.e., (\ref{S2a}), such that
\begin{align}\label{gammadefanwei}
\frac{4}{n(p-2)}<\gamma<1.
\end{align}

Next, we use the cross-constrained variational functionals $S_{j,(a)}(\varphi)$, $j=2,3,4$, to divide the space of the initial data $\mathcal{H}_{R_{j,(a)}}$, $j=2,3,4$, into
\begin{align*}
\mathcal{W}_{R_{j,(a)}}:=\left\{\varphi\in H^{1}(\mathbb{R}^{n})\Big|d_{\mathcal{N}_{1,(a)}}\leq K_{(a)}(\varphi)<R_{j,(a)}(\|\varphi\|_{2}),S_{j,(a)}(\varphi)>0\right\},\ j=2,
\end{align*}
\begin{align*}
\mathcal{V}_{R_{j,(a)}}:=\left\{\varphi\in H^{1}(\mathbb{R}^{n})\Big|d_{\mathcal{N}_{1,(a)}}\leq K_{(a)}(\varphi)< R_{j,(a)}(\|\varphi\|_{2}),S_{j,(a)}(\varphi)<0\right\}, \ j=2,
\end{align*}
\begin{align*}
\mathcal{W}_{R_{j,(a)}}:=\left\{\varphi\in H^{1}(\mathbb{R}^{n})\Big|d_{\mathcal{N}_{j,(a)}}\leq K_{(a)}(\varphi)<R_{j,(a)}(\|\varphi\|_{2}),S_{j,(a)}(\varphi)>0\right\}, \ j=3,4,
\end{align*}
and
\begin{align*}
\mathcal{V}_{R_{j,(a)}}:=\left\{\varphi\in H^{1}(\mathbb{R}^{n})\Big|d_{\mathcal{N}_{j,(a)}}\leq K_{(a)}(\varphi)<R_{j,(a)}(\|\varphi\|_{2}),S_{j,(a)}(\varphi)<0\right\}, \ j=3,4.
\end{align*}
For the cases $j=2,3$, we shall also use the manifolds $\mathcal{W}_{R_{j,(a)}}$ and $\mathcal{V}_{R_{j,(a)}}$ to prove the global existence and the finite time blowup, respectively.
For the case $j=4$, we shall use the manifold $\mathcal{W}_{R_{j,(a)}}$ to prove the global existence for high initial energy $K_{(a)}(\varphi_{0})\geq d_{\mathcal{N}_{j,(a)}}$. And in Remark \ref{rem6.3}, we shall also point out that we fail to prove the finite time blowup by the manifold $\mathcal{V}_{R_{4,(a)}}$ due to the lack of the term $\|\nabla\varphi\|^{2}_{2}$ in the cross-constrained variational functional $S_{4,(a)}(\varphi)$.

Since we use the cross-constrained variational functionals $S_{j,(a)}(\varphi)$, $j=2,3,4$, to deal with the high initial energy case, we shall not consider the case $j=1$ and shall consider the cases $j=2,3,4$.

\subsection{Global existence $\&$ finite time blowup for high initial energy $K_{(a)}(\varphi_{0})\geq d_{{\mathcal{N}_{j,(a)}}}$, $j=2,3$}\label{sec6.1}

In order to use the manifolds $\mathcal{W}_{R_{j,(a)}}$ and $\mathcal{V}_{R_{j,(a)}}$, $j=2,3$, to prove the global existence and the finite time blowup, respectively, for high initial energy $K_{(a)}(\varphi_{0})\geq d_{{\mathcal{N}_{j,(a)}}}$, $j=2,3$, we first prove the invariance of these manifolds.

\begin{lemma}\label{lem6.1}(The invariance of the manifolds $\mathcal{W}_{R_{j,(a)}}$ and $\mathcal{V}_{R_{j,(a)}}$, $j=2,3$, for high initial energy $K_{(a)}(\varphi_{0})\geq d_{\mathcal{N}_{j,(a)}}$, $j=2,3$)
Let $p$ satisfy
\begin{equation}\label{jingdianptiao}
\begin{cases}
2+\frac{4}{n}<p<\frac{2n}{n-2},\ n\geq3,\\
2+\frac{4}{n}<p<+\infty, \ n=1,2.\\
\end{cases}
\end{equation}Let $\varphi$ be the solution to Cauchy problem (\ref{6.1}).
 For the initial data $\varphi_{0}\in\mathcal{W}_{R_{j,(a)}}$, one has $\varphi\in\mathcal{W}_{R_{j,(a)}}$, $j=2,3$.
And for the initial data $\varphi_{0}\in\mathcal{V}_{R_{j,(a)}}$, one has $\varphi\in\mathcal{V}_{R_{j,(a)}}$, $j=2,3$.

\end{lemma}

\begin{proof}
We only write down the details of proving the invariance of the manifold $\mathcal{W}_{R_{j,(a)}}$, $j=2,3$, since the proof of the invariance of the manifold $\mathcal{V}_{R_{j,(a)}}$, $j=2,3$, is similar.

For the initial data $\varphi_{0}\in \mathcal{W}_{R_{j,(a)}}$, $j=2,3$, we know that there exists the unique solution $\varphi\in C([0,T_{{\rm{max}},(a)}),H^{1}(\mathbb{R}^{n}))$ according to Theorem 4.3.1 in Chapter 4 in \cite{rz1}\ (see also Theorem 2.1 in \cite{r14}), which means the continuity of $S_{2,(a)}(\varphi(t))$ and $S_{3,(a)}(\varphi(t))$ with respect to time $t$. Here, we use the notation $T_{{\rm{max}},(a)}$ to denote the maximal existence time of the solution to Cauchy problem (\ref{6.1}) in order to distinguish it from $T_{\rm{max}}$ defined for the solution to Cauchy problem (\ref{1.1}). We prove $\varphi(t)\in \mathcal{W}_{R_{j,(a)}}$, $j=2,3$ for $t\in(0,T_{{\rm{max}},(a)})$.
From $\varphi_{0}\in\mathcal{W}_{R_{j,(a)}}$, $j=2,3$, we know that
 \begin{align}
d_{\mathcal{N}_{j,(a)}}\leq& K_{(a)}(\varphi_{0})<R_{j,(a)}\left(\|\varphi_{0}\|_{2}\right)\label{Ka-deguji}
\end{align}and $S_{j,(a)}(\varphi_{0})>0$ hold. According to $\|\varphi_{0}\|^{2}_{2}=\|\varphi(t)\|^{2}_{2}$ and $K_{(a)}(\varphi_{0})=K_{(a)}(\varphi(t))$ for $t\in(0,T_{{\rm{max}},(a)})$ shown by Theorem 4.3.1 in Chapter 4 in \cite{rz1}\ (see also Theorem 2.1 in \cite{r14}), we know
 \begin{align}
d_{\mathcal{N}_{j,(a)}}\leq& K_{(a)}(\varphi(t))<R_{j,(a)}(\|\varphi(t)\|_{2})\label{Kadeguji2-}
\end{align}
for $t\in(0,T_{{\rm{max}},(a)})$ and $j=2,3$. Next, we only need to prove $S_{j,(a)}(\varphi(t))>0$, $j=2,3$, for all $t\in(0,T_{{\rm{max}},(a)})$. The strategies of the proofs depend on the structures of the cross-constrained variational functionals $S_{2,(a)}(\varphi)$ and $S_{3,(a)}(\varphi)$. Thus, we treat these two cases one by one.\\
\textbf{Case \textrm{I}}:\ \bm{$j=2.$}

 We prove $S_{2,(a)}(\varphi(t))>0$ for all $t\in(0,T_{{\rm{max}},(a)})$.
 Arguing by contradiction, we suppose that there exists the first time $\tilde{\tau}_{2}\in(0,T_{{\rm{max}},(a)})$ to make $S_{2,(a)}(\varphi(\tilde{\tau}_{2}))=0$ and $S_{2,(a)}(\varphi(t))>0$ for $t\in(0,\tilde{\tau}_{2})$, and we know $\varphi(\tilde{\tau}_{2})\neq 0$ due to $\|\varphi_{0}\|^{2}_{2}=\|\varphi(t)\|^{2}_{2}$ for $t\in(0,T_{{\rm{max}},(a)})$. Next, we estimate $K_{(a)}(\varphi(\tilde{\tau}_{2}))$. In this process, we use the condition $S_{2,(a)}(\varphi(\tilde{\tau}_{2}))=0$,
i.e.,
\begin{align}
&\frac{1}{p}\|\varphi(\tilde{\tau}_{2})\|^{p}_{p}=\frac{2}{n(p-2)\gamma}\|\nabla\varphi(\tilde{\tau}_{2})\|^{2}_{2}.\label{S2aguji}
\end{align}
By substituting (\ref{S2aguji}) into $K_{(a)}(\tilde{\tau}_{2})$, we can deduce
\begin{align}
K_{(a)}(\varphi(\tilde{\tau}_{2}))=&\left(\frac{1}{2}-\frac{2}{n(p-2)\gamma}\right)\|\nabla\varphi(\tilde{\tau}_{2})\|^{2}_{2}
+\frac{1}{2}\|\varphi(\tilde{\tau}_{2})\|^{2}_{2}.\label{Kadeguji}
\end{align}Here, $\gamma>0$ is chosen to make (\ref{gammadefanwei}), i.e.,
$$\frac{1}{2}-\frac{2}{n(p-2)\gamma}>0,$$ hold.
Applying Gagliardo-Nirenberg inequality (\ref{GN-budengshi}) and (\ref{S2aguji}), we have
\begin{align}
\frac{2p}{n(p-2)\gamma C_{1}}\leq\|\nabla\varphi(\tilde{\tau}_{2})\|^{\frac{n(p-2)-4}{2}}_{2}\|\varphi(\tilde{\tau}_{2})\|^{p-\frac{n(p-2)}{2}}_{2}.\label{Gnbudengshi}
\end{align}
Due to the fact $\|\varphi(\tilde{\tau}_{2})\|^{2}_{2}=\|\varphi_{0}\|^{2}_{2}$, we know (\ref{Gnbudengshi}) turns to
\begin{align}
\frac{2p}{n(p-2)\gamma C_{1}}\leq\|\nabla\varphi(\tilde{\tau}_{2})\|^{\frac{n(p-2)-4}{2}}_{2}\|\varphi_{0}\|^{p-\frac{n(p-2)}{2}}_{2},\label{Gnbudengshi2}
\end{align}
i.e.,
\begin{align}
\|\nabla\varphi(\tilde{\tau}_{2})\|^{2}_{2}\geq M_{(a)}(\|\varphi_{0}\|_{2}),\label{a1Gnbudengshi2}
\end{align}where $M_{(a)}(\|\varphi_{0}\|_{2})$ is defined by (\ref{Maf}).
According to (\ref{a1Gnbudengshi2}), we know that (\ref{Kadeguji}) becomes
\begin{align}
K_{(a)}(\varphi(\tilde{\tau}_{2}))\geq&\left(\frac{1}{2}-\frac{2}{n(p-2)\gamma}\right)M_{(a)}(\|\varphi_{0}\|_{2})+\frac{1}{2}\|\varphi(\tilde{\tau}_{2})\|^{2}_{2}.\label{Kadeguji2--=--==}
\end{align}From the facts $K_{(a)}(\varphi(\tilde{\tau}_{2}))=K_{(a)}(\varphi_{0})$ and $\|\varphi(\tilde{\tau}_{2})\|^{2}_{2}=\|\varphi_{0}\|^{2}_{2}$, we see that (\ref{Kadeguji2--=--==}) gives
\begin{align}
K_{(a)}(\varphi_{0})\geq&\left(\frac{1}{2}-\frac{2}{n(p-2)\gamma}\right)M_{(a)}(\|\varphi_{0}\|_{2})+\frac{1}{2}\|\varphi_{0}\|^{2}_{2},\label{Kadeguji2}
\end{align}
which contradicts (\ref{Ka-deguji}) for $j=2$. Thus, we know $S_{2,(a)}(\varphi(t))>0$ for all $t\in(0,T_{{\rm{max}},(a)})$, which means $\varphi\in\mathcal{W}_{R_{2,(a)}}$.\\
\textbf{Case \textrm{II}}:\ \bm{$j=3.$}

 We prove $S_{3,(a)}(\varphi(t))>0$ for all $t\in(0,T_{{\rm{max}},(a)})$.
 Arguing by contradiction, we suppose that there exists the first time $\tilde{\tau}_{3}\in(0,T_{{\rm{max}},(a)})$ to make $S_{3,(a)}(\varphi(\tilde{\tau}_{3}))=0$ and $S_{3,(a)}(\varphi(t))>0$ for $t\in(0,\tilde{\tau}_{3})$, and we know $\varphi(\tilde{\tau}_{3})\neq 0$ due to $\|\varphi_{0}\|^{2}_{2}=\|\varphi(t)\|^{2}_{2}$ for $t\in(0,T_{{\rm{max}},(a)})$ shown by Theorem 4.3.1 in Chapter 4 in \cite{rz1}\ (see also Theorem 2.1 in \cite{r14}). Next, we estimate $K_{(a)}(\varphi(\tilde{\tau}_{3}))$. In this process, we use the condition $S_{3,(a)}(\varphi(\tilde{\tau}_{3}))=0$,
i.e.,
\begin{align}
&\frac{1}{p}\|\varphi(\tilde{\tau}_{3})\|^{p}_{p}=\frac{2}{n(p-2)}\left(\|\nabla\varphi(\tilde{\tau}_{3})\|^{2}_{2}
+\tilde{c}\|\varphi(\tilde{\tau}_{3})\|^{2}_{2}\right),\ 0<\tilde{c}<1.\label{S3aguji}
\end{align}
 By substituting (\ref{S3aguji}) into $K_{(a)}(\varphi(\tilde{\tau}_{3}))$, we can deduce
\begin{align}\label{3Kadeguji}
\begin{split}
K_{(a)}(\varphi(\tilde{\tau}_{3}))=&\left(\frac{1}{2}-\frac{2}{n(p-2)}\right)\|\nabla\varphi(\tilde{\tau}_{3})\|^{2}_{2}
+\left(\frac{1}{2}-\frac{2\tilde{c}}{n(p-2)}\right)\|\varphi(\tilde{\tau}_{3})\|^{2}_{2}.
\end{split}
\end{align}
Applying Cauchy-Schwartz inequality, we know (\ref{S3aguji}) turns to
\begin{align}
\frac{1}{p}\|\varphi(\tilde{\tau}_{3})\|^{p}_{p}\geq\frac{4 \tilde{c}^{\frac{1}{2}}}{n(p-2)}\|\nabla\varphi(\tilde{\tau}_{3})\|_{2}\|\varphi(\tilde{\tau}_{3})\|_{2}.\label{S--3aguji}
\end{align}
By (\ref{GN-budengshi}), we see that (\ref{S--3aguji}) becomes
\begin{align*}
\frac{4 p\tilde{c}^{\frac{1}{2}}}{n(p-2)C_{1}}\|\nabla\varphi(\tilde{\tau}_{3})\|_{2}\|\varphi(\tilde{\tau}_{3})\|_{2}\leq
\|\nabla\varphi(\tilde{\tau}_{3})\|^{\frac{n(p-2)}{2}}_{2}\|\varphi(\tilde{\tau}_{3})\|^{p-\frac{n(p-2)}{2}}_{2},
\end{align*}
i.e.,
\begin{align*}
\frac{4 p\tilde{c}^{\frac{1}{2}}}{n(p-2)C_{1}}\leq
\|\nabla\varphi(\tilde{\tau}_{3})\|^{\frac{n(p-2)}{2}-1}_{2}\|\varphi(\tilde{\tau}_{3})\|^{p-\frac{n(p-2)}{2}-1}_{2},
\end{align*}
which means
\begin{align}\label{guodudezhisi}
\frac{4 p\tilde{c}^{\frac{1}{2}}}{n(p-2)C_{1}}\leq
\|\nabla\varphi(\tilde{\tau}_{3})\|^{\frac{n(p-2)}{2}-1}_{2}\|\varphi_{0}\|^{p-\frac{n(p-2)}{2}-1}_{2}
\end{align}according to the fact $\|\varphi(\tilde{\tau}_{3})\|^{2}_{2}=\|\varphi_{0}\|^{2}_{2}$.
From (\ref{guodudezhisi}), we have
\begin{align}
\|\nabla\varphi(\tilde{\tau}_{3})\|^{2}_{2}\geq&\tilde{M}_{(a)}(\|\varphi_{0}\|_{2}),\label{Gnbudengshi2}
\end{align}
where $\tilde{M}_{(a)}(\|\varphi_{0}\|_{2})$ is defined by (\ref{Mtilea}).
Meanwhile, (\ref{S3aguji}) also gives
\begin{align}\label{xuyaoshiyong}
&\frac{n(p-2)}{2p}\|\varphi(\tilde{\tau}_{3})\|^{p}_{p}>\|\nabla\varphi(\tilde{\tau}_{3})\|^{2}_{2}.
\end{align}By the similar process of obtaining (\ref{a1Gnbudengshi2}), we can use (\ref{xuyaoshiyong}) to give
\begin{align}
\|\nabla\varphi(\tilde{\tau}_{3})\|^{2}_{2}\geq \hat{M}_{(a)}(\|\varphi_{0}\|_{2}),\label{3a1Gnbudengshi2}
\end{align}where
$\hat{M}_{(a)}(\|\varphi_{0}\|_{2})$ is defined by (\ref{Mhaf}).
By (\ref{Gnbudengshi2}) and (\ref{3a1Gnbudengshi2}), we have
\begin{align*}
\|\nabla\varphi(\tilde{\tau}_{3})\|^{2}_{2}\geq \frac{\hat{M}_{(a)}(\|\varphi_{0}\|_{2})+\tilde{M}_{(a)}(\|\varphi_{0}\|_{2})+\left|\hat{M}_{(a)}(\|\varphi_{0}\|_{2})-\tilde{M}_{(a)}(\|\varphi_{0}\|_{2})\right|}{2},
\end{align*}
which makes
(\ref{3Kadeguji}) become
\begin{align}\label{3Kadeguji--0}
\begin{split}
K_{(a)}(\varphi(\tilde{\tau}_{3}))
\geq&\left(\frac{1}{2}-\frac{2}{n(p-2)}\right)\frac{\hat{M}_{(a)}(\|\varphi_{0}\|_{2})+\tilde{M}_{(a)}(\|\varphi_{0}\|_{2})+\left|\hat{M}_{(a)}(\|\varphi_{0}\|_{2})-\tilde{M}_{(a)}(\|\varphi_{0}\|_{2})\right|}{2}\\
&+\left(\frac{1}{2}-\frac{2\tilde{c}}{n(p-2)}\right)\|\varphi(\tilde{\tau}_{3})\|^{2}_{2}.
\end{split}
\end{align}
By the facts $K_{(a)}(\varphi(\tilde{\tau}_{3}))=K_{(a)}(\varphi_{0})$ and $\|\varphi(\tilde{\tau}_{3})\|^{2}_{2}=\|\varphi_{0}\|^{2}_{2}$, we know that (\ref{3Kadeguji--0}) turns to
\begin{align}\label{3Kadeguji--}
\begin{split}
K_{(a)}(\varphi_{0})
\geq&\left(\frac{1}{2}-\frac{2}{n(p-2)}\right)\frac{\hat{M}_{(a)}(\|\varphi_{0}\|_{2})+\tilde{M}_{(a)}(\|\varphi_{0}\|_{2})+\left|\hat{M}_{(a)}(\|\varphi_{0}\|_{2})-\tilde{M}_{(a)}(\|\varphi_{0}\|_{2})\right|}{2}
\\&+
\left(\frac{1}{2}-\frac{2\tilde{c}}{n(p-2)}\right)\|\varphi_{0}\|^{2}_{2},
\end{split}
\end{align}
 which contradicts (\ref{Ka-deguji}) for $j=3$.
 Thus, we know $S_{3,(a)}(\varphi(t))>0$ for all $t\in(0,T_{{\rm{max}},(a)})$, which means $\varphi\in\mathcal{W}_{R_{3,(a)}}$.
\end{proof}

\begin{theorem}\label{th6.1}(Global existence for high initial energy $K_{(a)}(\varphi_{0})\geq d_{\mathcal{N}_{j,(a)}}$, $j=2,3$)
Let $p$ satisfy (\ref{jingdianptiao}).
Then,
the corresponding solution $\varphi$ to problem (\ref{6.1}) exists globally provided $\varphi_{0}\in \mathcal{W}_{R_{j,(a)}}$, $j=2,3$.
\begin{proof}
The proofs of the global existence depend on the structures of the cross-constrained variational functionals $S_{j,(a)}(\varphi)$, $j=2,3$. Thus, we treat the cases $j=2,3$ separately.\\
\textbf{Case \textrm{I}}:\ \bm{$j=2.$}

For the initial data $\varphi_{0}\in \mathcal{W}_{R_{2,(a)}}$, we know the corresponding solution $\varphi\in \mathcal{W}_{R_{2,(a)}}$, i.e., $S_{2,(a)}(\varphi)>0$, meaning
\begin{align}
\|\nabla\varphi\|^{2}_{2}>\frac{n(p-2)\gamma}{2p}\|\varphi\|^{p}_{p},\label{S2adayuling}
\end{align}
according to Lemma \ref{lem6.1}. Then, substituting (\ref{S2adayuling}) into $K_{(a)}(\varphi)$, we have
\begin{align*}
K_{(a)}(\varphi)\geq&\frac{1}{2}\|\varphi\|^{2}_{2}+\left(\frac{1}{2}-\frac{2}{n(p-2)\gamma}\right)\|\nabla\varphi\|^{2}_{2}\nonumber\\
\geq&\left(\frac{1}{2}-\frac{2}{n(p-2)\gamma}\right)\|\varphi\|^{2}_{H^{1}(\mathbb{R}^{n})},
\end{align*}
which means
\begin{align}\label{gujishi}
K_{(a)}(\varphi_{0})\geq&\left(\frac{1}{2}-\frac{2}{n(p-2)\gamma}\right)\|\varphi\|^{2}_{H^{1}(\mathbb{R}^{n})}
\end{align}
due to the fact $K_{(a)}(\varphi)=K_{(a)}(\varphi_{0})$ given by Theorem 4.3.1 in Chapter 4 in \cite{rz1}\ (see also Theorem 2.1 in \cite{r14}). Here, $\gamma>0$ is chosen to make (\ref{gammadefanwei}), i.e.,
$$\frac{1}{2}-\frac{2}{n(p-2)\gamma}>0,$$ hold.
Due to (\ref{gujishi}) and Theorem 2.1 in \cite{r14}, we know $T_{{\rm {max}},(a)}=+\infty$.\\
\textbf{Case \textrm{II}}:\ \bm{$j=3.$}

For the initial data $\varphi_{0}\in \mathcal{W}_{R_{3,(a)}}$, we know that the corresponding solution $\varphi\in \mathcal{W}_{R_{3,(a)}}$, i.e., $S_{3,(a)}(\varphi)>0$, takes the following inequality
\begin{align}
\|\nabla\varphi\|^{2}_{2}+\tilde{c}\|\varphi\|^{2}_{2}>\frac{n(p-2)}{2p}\|\varphi\|^{p}_{p},\ 0<\tilde{c}<1, \label{S3adayuling}
\end{align}
according to Lemma \ref{lem6.1}. Then, substituting (\ref{S3adayuling}) into $K_{(a)}(\varphi)$, we have
\begin{align*}
K_{(a)}(\varphi)\geq&\left(\frac{1}{2}-\frac{2}{n(p-2)}\right)\|\nabla\varphi\|^{2}_{2}+\left(\frac{1}{2}-\frac{2 \tilde{c}}{n(p-2)}\right)\|\varphi\|^{2}_{2}
\nonumber\\
\geq&\left(\frac{1}{2}-\frac{2}{n(p-2)}\right)\|\varphi\|^{2}_{H^{1}(\mathbb{R}^{n})},
\end{align*}
which means
\begin{align}\label{youjiexinggujiS3a}
K_{(a)}(\varphi_{0})\geq&\left(\frac{1}{2}-\frac{2}{n(p-2)}\right)\|\varphi\|^{2}_{H^{1}(\mathbb{R}^{n})}
\end{align}
due to the fact $K_{(a)}(\varphi)=K_{(a)}(\varphi_{0})$.
From (\ref{youjiexinggujiS3a}) and Theorem 2.1 in \cite{r14}, we know $T_{{\rm {max}},(a)}=+\infty$.
\end{proof}
\end{theorem}

\begin{theorem}\label{th6.2}(Finite time blowup for high initial energy $K_{(a)}(\varphi_{0})\geq d_{\mathcal{N}_{j,(a)}}$, $j=2,3$)
Let $p$ satisfy (\ref{jingdianptiao}).
Then,
the corresponding solution $\varphi$ to problem (\ref{6.1}) blows up in finite time provided $\varphi_{0}\in \mathcal{V}_{R_{j,(a)}}$, $j=2,3$, and $\||x|\varphi_{0}\|_{2}<+\infty$.
\end{theorem}

\begin{proof}
The proofs of the finite time blowup depend on the structures of the cross-constrained variational functionals $S_{j,(a)}(\varphi)$, $j=2,3$. Thus, we treat the cases $j=2,3$ separately.\\
\textbf{Case \textrm{I}}:\ \bm{$j=2.$}

For the initial data $\varphi_{0}\in \mathcal{V}_{R_{2,(a)}}$, we prove $T_{{\rm max},(a)}<+\infty$. Arguing by contradiction, we suppose $T_{{\rm max},(a)}=+\infty$. According to Theorem 2.2 in \cite{r14}, we know
\begin{align}\label{Jadaoshu}
\begin{split}
J''_{(a)}(t)=&8\|\nabla\varphi(t)\|^{2}_{2}-\frac{8n(p-2)}{2p}\|\varphi(t)\|^{p}_{p}\\
=&\frac{8}{\gamma}\|\nabla\varphi(t)\|^{2}_{2}-\frac{8n(p-2)}{2p}\|\varphi(t)\|^{p}_{p}-
\frac{8-8\gamma}{\gamma}\|\nabla\varphi(t)\|^{2}_{2}\\
=&\frac{8}{\gamma}S_{2,(a)}(\varphi(t))-
\frac{8-8\gamma}{\gamma}\|\nabla\varphi(t)\|^{2}_{2}
\end{split}
\end{align}
for $t\in(0,+\infty)$. Due to $\varphi(t)\in \mathcal{V}_{R_{1,(a)}}$, i.e., $S_{2,(a)}(\varphi(t))<0$, for $t\in(0,+\infty)$, we know that (\ref{Jadaoshu}) turns to
\begin{align}
J''_{(a)}(t)<-\frac{8-8\gamma}{\gamma}\|\nabla\varphi(t)\|^{2}_{2}\label{Jadaoshu2}
\end{align}
for $t\in(0,+\infty)$. Next, we estimate $\|\nabla\varphi(t)\|^{2}_{2}$ in (\ref{Jadaoshu2}). By the similar process of obtaining (\ref{a1Gnbudengshi2}), we can apply $S_{2,(a)}(\varphi(t))<0$, i.e.,
\begin{align}
\frac{2p}{n(p-2)\gamma}\|\nabla \varphi(t)\|^{2}_{2}<\|\varphi(t)\|^{p}_{p},
\end{align}
for $t\in(0,+\infty)$ to obtain
\begin{align}
\|\nabla\varphi(t)\|^{2}_{2}\geq M_{(a)}(\|\varphi_{0}\|_{2})\label{Gnbudengshi2-2}
\end{align}
for $t\in(0,+\infty)$, where
$M_{(a)}(\|\varphi_{0}\|_{2})$ is defined by (\ref{Maf}).
By substituting (\ref{Gnbudengshi2-2}) into (\ref{Jadaoshu2}), we have
\begin{align}
J''_{(a)}(t)<-\frac{8-8\gamma}{\gamma}M_{(a)}(\|\varphi_{0}\|_{2})<0.\label{Jadaoshu4}
\end{align}
The rest proof is similar to Theorem \ref{th3.1}.\\
\textbf{Case \textrm{II}}:\ \bm{$j=3.$}

For the initial data $\varphi_{0}\in \mathcal{V}_{R_{3,(a)}}$, we prove $T_{{\rm {max}},(a)}<+\infty$. Arguing by contradiction, we suppose $T_{{\rm {max}},(a)}=+\infty$. According to Theorem 2.2 in \cite{r14}, we know
\begin{align}\label{Jadaoshujdengyu3}
\begin{split}
J''_{(a)}(t)=&8\|\nabla\varphi(t)\|^{2}_{2}-\frac{8n(p-2)}{2p}\|\varphi(t)\|^{p}_{p}\\
=&8\|\nabla\varphi(t)\|^{2}_{2}+8\tilde{c}\|\varphi(t)\|^{2}_{2}-\frac{8n(p-2)}{2p}\|\varphi(t)\|^{p}_{p}
-8\tilde{c}\|\varphi(t)\|^{2}_{2}\\
=&8S_{3,(a)}(\varphi(t))-8\tilde{c}\|\varphi(t)\|^{2}_{2}
\end{split}
\end{align}
for $t\in(0,+\infty)$. Due to $\varphi(t)\in \mathcal{V}_{R_{3,(a)}}$, i.e., $S_{3,(a)}(\varphi(t))<0$, for $t\in(0,+\infty)$ obtained by Lemma \ref{lem6.1}, and the fact $\|\varphi\|^{2}_{2}=\|\varphi_{0}\|^{2}_{2}$ shown by Theorem 4.3.1 in Chapter 4 in \cite{rz1}\ (see also Theorem 2.1 in \cite{r14}), we know that (\ref{Jadaoshujdengyu3}) turns to
\begin{align}
J''_{(a)}(t)\leq-8\tilde{c}\|\varphi(t)\|^{2}_{2}=-8\tilde{c}\|\varphi_{0}\|^{2}_{2}.\label{Jadaoshu3p}
\end{align}
The rest proof is similar to Theorem \ref{th3.1}.
\end{proof}

\begin{remark}\label{bubianjiheKa}(Examples of the initial data belonging to the manifolds $\mathcal{W}_{R_{j,(a)}}$, $\mathcal{V}_{R_{j,(a)}}$, $j=2,3$, and $\mathcal{W}_{R_{j,(a)}}$, $j=4$, for high initial energy $K_{(a)}(\varphi_{0})\geq d_{{\mathcal{N}_{j,(a)}}}$, $j=2,3,4$)
As shown in Theorem \ref{th6.1} and Theorem \ref{th6.2},
for high initial energy $K_{(a)}(\varphi_{0})\geq d_{{\mathcal{N}_{j,(a)}}}$, $j=2,3$, the global existence and the finite time blowup of the solution are established in the invariant manifolds $\mathcal{W}_{R_{j,(a)}}$ and $\mathcal{V}_{R_{j,(a)}}$, $j=2,3$, respectively. And we shall use the manifold $\mathcal{W}_{R_{4,(a)}}$ to prove the global existence for the high initial energy $K_{(a)}(\varphi_{0})\geq d_{{\mathcal{N}_{4,(a)}}}$ in Subsection \ref{sec6.2}. Here, since the manifold $\mathcal{V}_{R_{4,(a)}}$ cannot to be used to prove the finite time blowup for the high initial energy $K_{(a)}(\varphi_{0})\geq d_{{\mathcal{N}_{4,(a)}}}$, which will be explained in Remark \ref{rem6.3}. Hence, it is necessary to explain that these invariant manifolds are non-empty, which means there exist some initial data in the invariant manifolds $\mathcal{W}_{R_{j,(a)}}$, $\mathcal{V}_{R_{j,(a)}}$, $j=2,3$, and $\mathcal{W}_{R_{j,(a)}}$, $j=4$, similar to Remark \ref{Manifoldsfeikong}. Next, we give some examples of such initial data.\\
(1)\ Example of initial data belonging to manifold $\mathcal{W}_{R_{2,(a)}}$.

 Since the manifolds $\mathcal{W}_{R_{2,(a)}}$ and $\mathcal{V}_{R_{2,(a)}}$ are obtained by dividing the space of the initial data $\mathcal{H}_{R_{2,(a)}}$, in order to find a function belonging to $\mathcal{W}_{R_{2,(a)}}$ or $\mathcal{V}_{R_{2,(a)}}$, we first find a function satisfying the condition $$d_{\mathcal{N}_{1,(a)}}\leq K_{(a)}(\varphi)<R_{2,(a)}(\|\varphi\|_{2})$$ in the definition of $\mathcal{H}_{R_{2,(a)}}$, i.e., (\ref{HR2a}). Here, we need to ensure $d_{\mathcal{N}_{1,(a)}}<R_{2,(a)}(\|\varphi\|_{2})$ holds, which means
\begin{align}\label{dn1xiaoyuA1jiaA2}
d_{\mathcal{N}_{1,(a)}}<\left(\frac{1}{2}-\frac{2}{n(p-2)\gamma}\right)M_{(a)}(\|\varphi\|_{2})+\frac{1}{2}\|\varphi\|^{2}_{2}
\end{align} according to (\ref{R2a}).
 To make (\ref{dn1xiaoyuA1jiaA2}) hold, we only need to make
\begin{align}\label{dn1xiaoyuA1}
d_{\mathcal{N}_{1,(a)}}<\left(\frac{1}{2}-\frac{2}{n(p-2)\gamma}\right)M_{(a)}(\|\varphi\|_{2})
\end{align}
or
\begin{align}\label{dn1xiaoyuA2}
d_{\mathcal{N}_{1,(a)}}<\frac{1}{2}\|\varphi\|^{2}_{2}
\end{align}
hold, where
$M_{(a)}(\|\varphi\|_{2})$ is defined by (\ref{Maf}).
Subsequently, for the case $j=2$, we have two strategies to show the examples of the initial data belonging to $\mathcal{W}_{R_{2,(a)}}$ or $\mathcal{V}_{R_{2,(a)}}$, which are to show such examples by using (\ref{dn1xiaoyuA1}) and (\ref{dn1xiaoyuA2}), respectively.

Find a function $\eta_{2}\in H^{1}(\mathbb{R}^{n})$ satisfying $\eta_{2}\neq0$,
\begin{align}\label{Ma-eta}
d_{_{\mathcal{N}_{1,(a)}}}< \left(\frac{1}{2}-\frac{2}{n(p-2)\gamma}\right)M_{(a)}(\|\eta_{2}\|_{2}),
\end{align} i.e., satisfying $R_{2,(a)}(\|\eta_{2}\|_{2})>d_{_{\mathcal{N}_{1},(a)}}$,
\begin{align}\label{Ma-eta1}
\left(\frac{1}{2}-\frac{2}{n(p-2)\gamma}\right)M_{(a)}(\|\eta_{2}\|_{2})>\frac{1}{2}\|\eta_{2}\|^{2}_{2},
\end{align}and $\||x|\eta_{2}\|_{2}<+\infty$. Next, for the case $j=2$, we shall use the function $\eta_{2}$ to show the examples of the initial data belonging to the manifolds $\mathcal{W}_{R_{2,(a)}}$ and $\mathcal{V}_{R_{2,(a)}}$. Hence we use the lower index $2$ to denote this function $\eta_{2}$.
We take a constant $\rho>0$ and consider a function $\rho^{\frac{n}{2}}\eta_{2}(\rho x)$.
By the similar analysis in Remark \ref{Manifoldsfeikong}, we have
\begin{align}
\left\|\nabla \left(\rho^{\frac{n}{2}}\eta_{2}(\rho x)\right)\right\|^{2}_{2}=&\rho^{2}\|\nabla\eta_{2}\|^{2}_{2},\label{daihuan---}
\end{align}
\begin{align}
\|\rho^{\frac{n}{2}}\eta_{2}(\rho x)\|^{2}_{2}
=&\|\eta_{2}\|^{2}_{2},\label{daihuan2---}
\end{align}
and
\begin{align}
\|\rho^{\frac{n}{2}}\eta_{2}(\rho x)\|^{p}_{p}
=&\rho^{\frac{n(p-2)}{2}}\|\eta_{2}\|^{p}_{p}.\label{daihuan3---}
\end{align}
Thus, we know
\begin{align}\label{daihuan3--0===-0-}
\begin{split}
S_{2,(a)}(\rho^{\frac{n}{2}}\eta_{2}(\rho x))=&\left\|\nabla \left(\rho^{\frac{n}{2}}\eta_{2}(\rho x)\right)\right\|^{2}_{2}-\frac{n(p-2)\gamma}{2p}\|\rho^{\frac{n}{2}}\eta_{2}(\rho x)\|^{p}_{p}\\
=&\left(\|\nabla\eta_{2}\|^{2}_{2}\right)\rho^{2}-\left(\frac{n(p-2)\gamma}{2p}\|\eta_{2}\|^{p}_{p}\right)\rho^{\frac{n(p-2)}{2}}\\
=&\rho^{2}\left(\|\nabla\eta_{2}\|^{2}_{2}-\left(\frac{n(p-2)\gamma}{2p}\|\eta_{2}\|^{p}_{p}\right)\rho^{\frac{n(p-2)}{2}-2}\right).
\end{split}
\end{align}
There exists a unique $\rho^{*}>0$ such that $S_{2,(a)}(\left(\rho^{*}\right)^{\frac{n}{2}}\eta_{2}(\rho^{*} x))=0$, $S_{2,(a)}(\rho^{\frac{n}{2}}\eta_{2}(\rho x))>0$ for $0<\rho<\rho^{*}$ and $S_{2,(a)}(\rho^{\frac{n}{2}}\eta_{2}(\rho x))<0$ for $\rho>\rho^{*}$.
Next, we have
\begin{align}\label{daihuan3-}
\begin{split}
E_{(a)}(\rho^{\frac{n}{2}}\eta_{2}(\rho x)):=&\frac{1}{2}\left\|\nabla \left(\rho^{\frac{n}{2}}\eta_{2}(\rho x)\right)\right\|^{2}_{2}-\frac{1}{p}\|\rho^{\frac{n}{2}}\eta_{2}(\rho x)\|^{p}_{p}\\
=&\left(\frac{1}{2}\|\nabla\eta_{2}\|^{2}_{2}\right)\rho^{2}-\left(\frac{1}{p}\|\eta_{2}\|^{p}_{p}\right)\rho^{\frac{n(p-2)}{2}}\\
=&\rho^{2}\left(\frac{1}{2}\|\nabla\eta_{2}\|^{2}_{2}-\left(\frac{1}{p}\|\eta_{2}\|^{p}_{p}\right)\rho^{\frac{n(p-2)}{2}-2}\right).
\end{split}
\end{align}
From $S_{2,(a)}(\left(\rho^{*}\right)^{\frac{n}{2}}\eta_{2}(\rho^{*} x))=0$, by the similar process of dealing with $j=2$ in the proof of Lemma \ref{lem6.1}, we know
\begin{align}\label{maodundian}
E_{(a)}(\left(\rho^{*}\right)^{\frac{n}{2}}\eta_{2}(\rho^{*} x))\geq R_{2,(a)}(\|\left(\rho^{*}\right)^{\frac{n}{2}}\eta_{2}(\rho^{*} x)\|_{2})-\frac{1}{2}\|\left(\rho^{*}\right)^{\frac{n}{2}}\eta_{2}(\rho^{*} x)\|^{2}_{2}.
\end{align}
From (\ref{daihuan2---}), we know (\ref{maodundian}) becomes
\begin{align}\label{maodundian===}
E_{(a)}(\left(\rho^{*}\right)^{\frac{n}{2}}\eta_{2}(\rho^{*} x))\geq R_{2,(a)}(\|\eta_{2}\|_{2})-\frac{1}{2}\|\eta_{2}\|^{2}_{2}.
\end{align}
 According to (\ref{daihuan3-}), for $\rho=0$, we know $E_{(a)}(\rho^{\frac{n}{2}}\eta_{2}(\rho x))=0$. Combining this fact, (\ref{maodundian===}) and the continuity of $E_{(a)}(\rho^{\frac{n}{2}}\eta_{2}(\rho x))$ with respect to $\rho$, we know that there exists a $\bar{\rho}\in(0,\rho^{*})$ such that \begin{align}\label{duizhao}
 d_{\mathcal{N}_{1,(a)}}\leq E_{(a)}(\left(\bar{\rho}\right)^{\frac{n}{2}}\eta_{2}(\bar{\rho} x))<R_{2,(a)}(\|\eta_{2}\|_{2})-\frac{1}{2}\|\eta_{2}\|^{2}_{2},
 \end{align}which means
  \begin{align}\label{compare}
  d_{\mathcal{N}_{1,(a)}}\leq K_{(a)}(\left(\bar{\rho}\right)^{\frac{n}{2}}\eta_{2}(\bar{\rho} x))<R_{2,(a)}(\|\left(\bar{\rho}\right)^{\frac{n}{2}}\eta_{2}(\bar{\rho} x)\|_{2})
  \end{align}
 due to (\ref{daihuan2---}).
   And noticing (\ref{daihuan3--0===-0-}), $S_{2,(a)}(\left(\rho^{*}\right)^{\frac{n}{2}}\eta_{2}(\rho^{*} x))=0$, we have $S_{2,(a)}(\left(\bar{\rho}\right)^{\frac{n}{2}}\eta_{2}(\bar{\rho} x))>0$ due to $\bar{\rho}<\rho^{*}$. Thus, we know $\left(\bar{\rho}\right)^{\frac{n}{2}}\eta_{2}(\bar{\rho} x)\in\mathcal{W}_{R_{2,(a)}}$.\\

 \noindent{(2)\ Example of the initial data belonging to the manifold $\mathcal{V}_{R_{2,(a)}}$.}

  We use the function $\eta_{2}$ satisfying (\ref{Ma-eta}), (\ref{Ma-eta1}) and $\||x|\eta_{2}\|_{2}<+\infty$ to show the example of the initial data belonging to the manifold $\mathcal{V}_{R_{2,(a)}}$. According to (\ref{daihuan3-}), there exists a $\tilde{\rho}>\rho^{*}$ large enough such that $E_{(a)}(\left(\tilde{\rho}\right)^{\frac{n}{2}}\eta_{2}(\tilde{\rho} x))<0$. Combining this fact, (\ref{maodundian===}) and the continuity of $E_{(a)}(\rho^{\frac{n}{2}}\eta_{2}(\rho x))$ with respect to $\rho$, we know that there exists a $\hat{\rho}\in(\rho^{*},\tilde{\rho})$ such that  $$d_{\mathcal{N}_{1,(a)}}\leq E_{(a)}(\left(\hat{\rho}\right)^{\frac{n}{2}}\eta_{2}(\hat{\rho} x))<R_{2,(a)}(\|\eta_{2}\|_{2})-\frac{1}{2}\|\eta_{2}\|^{2}_{2},$$which means
  $$d_{\mathcal{N}_{1,(a)}}\leq K_{(a)}(\left(\hat{\rho}\right)^{\frac{n}{2}}\eta_{2}(\hat{\rho} x))<R_{2,(a)}(\|\left(\hat{\rho}\right)^{\frac{n}{2}}\eta_{2}(\hat{\rho} x)\|_{2})$$
 due to (\ref{daihuan2---}). And noticing (\ref{daihuan3--0===-0-}), $S_{2,(a)}(\left(\rho^{*}\right)^{\frac{n}{2}}\eta_{2}(\rho^{*} x))=0$ , we know $S_{2,(a)}(\left(\hat{\rho}\right)^{\frac{n}{2}}\eta_{2}(\hat{\rho} x))<0$ from $\hat{\rho}>\rho^{*}$. Thus, we see $\left(\hat{\rho}\right)^{\frac{n}{2}}\eta_{2}(\hat{\rho} x)\in\mathcal{V}_{R_{2,(a)}}$.\\

\noindent{(3)\ Another version of example of the initial data belonging to the manifold $\mathcal{W}_{R_{2,(a)}}$.}

We shall use another
strategy, i.e., to apply (\ref{dn1xiaoyuA2}), to show a different example of the initial data belonging to $\mathcal{W}_{R_{2,(a)}}$ or $\mathcal{V}_{R_{2,(a)}}$.
Find a function $\tilde{\eta}_{2}\in H^{1}(\mathbb{R}^{n})$ satisfying $\tilde{\eta}_{2}\neq0$,  \begin{align}\label{tildeeta-dayu2d}\|\tilde{\eta}_{2}\|_{2}\geq 2d_{\mathcal{N}_{1,(a)}},\end{align}
and $\||x|\tilde{\eta}_{2}\|_{2}<+\infty$.
Then, we take the constant $\varrho>0$ and consider the function $\varrho^{\frac{n}{2}}\tilde{\eta}_{2}(\varrho x)$.
By the similar analysis in Remark \ref{Manifoldsfeikong}, we have
\begin{align}
\left\|\nabla \left(\varrho^{\frac{n}{2}}\tilde{\eta}_{2}(\varrho x)\right)\right\|^{2}_{2}=&\varrho^{2}\|\nabla\tilde{\eta}_{2}\|^{2}_{2},\label{daihuan0}
\end{align}
\begin{align}
\|\varrho^{\frac{n}{2}}\tilde{\eta}_{2}(\varrho x)\|^{2}_{2}
=&\|\tilde{\eta}_{2}\|^{2}_{2},\label{daihuan2-0}
\end{align}
and
\begin{align}
\|\varrho^{\frac{n}{2}}\tilde{\eta}_{2}(\varrho x)\|^{p}_{p}
=&\varrho^{\frac{n(p-2)}{2}}\|\tilde{\eta}_{2}\|^{p}_{p}.\label{daihuan3-0}
\end{align}
Thus, we know
\begin{align}\label{daihuan3--2}
\begin{split}
S_{2,(a)}(\varrho^{\frac{n}{2}}\tilde{\eta}_{2}(\varrho x))=&\left\|\nabla \left(\varrho^{\frac{n}{2}}\tilde{\eta}_{2}(\varrho x)\right)\right\|^{2}_{2}-\frac{n(p-2)\gamma}{2p}\|\varrho^{\frac{n}{2}}\tilde{\eta}_{2}(\varrho x)\|^{p}_{p}\\
=&\left(\|\nabla\tilde{\eta}_{2}\|^{2}_{2}\right)\varrho^{2}-\left(\frac{n(p-2)\gamma}{2p}\|\tilde{\eta}_{2}\|^{p}_{p}\right)\varrho^{\frac{n(p-2)}{2}}\\
=&\varrho^{2}\left(\|\nabla\tilde{\eta}_{2}\|^{2}_{2}-\left(\frac{n(p-2)\gamma}{2p}\|\tilde{\eta}_{2}\|^{p}_{p}\right)\varrho^{\frac{n(p-2)}{2}-2}\right).
\end{split}
\end{align}
There exists a unique $\varrho_{0}>0$ such that $S_{2,(a)}(\varrho_{0}^{\frac{n}{2}}\tilde{\eta}_{2}(\varrho_{0} x))=0$, $S_{2,(a)}(\varrho^{\frac{n}{2}}\tilde{\eta}_{2}(\varrho x))>0$ for $0<\varrho<\varrho_{0}$ and $S_{2,(a)}(\varrho^{\frac{n}{2}}\tilde{\eta}_{2}(\varrho x))<0$ for $\varrho>\varrho_{0}$.
Next, we have
\begin{align}\label{daihuan3--}
\begin{split}
E_{(a)}(\varrho^{\frac{n}{2}}\tilde{\eta}_{2}(\varrho x)):=&\frac{1}{2}\left\|\nabla\left(\varrho^{\frac{n}{2}}\tilde{\eta}_{2}(\varrho x)\right)\right\|^{2}_{2}-\frac{1}{p}\|\varrho^{\frac{n}{2}}\tilde{\eta}_{2}(\varrho x)\|^{p}_{p}\\
=&\left(\frac{1}{2}\|\nabla\tilde{\eta}_{2}\|^{2}_{2}\right)\varrho^{2}-\left(\frac{1}{p}\|\tilde{\eta}_{2}\|^{p}_{p}\right)\varrho^{\frac{n(p-2)}{2}}\\
=&\varrho^{2}\left(\frac{1}{2}\|\nabla\tilde{\eta}_{2}\|^{2}_{2}-\left(\frac{1}{p}\|\tilde{\eta}_{2}\|^{p}_{p}\right)\varrho^{\frac{n(p-2)}{2}-2}\right).
\end{split}
\end{align}
From $S_{2,(a)}(\varrho_{0}^{\frac{n}{2}}\tilde{\eta}_{2}(\varrho_{0} x))=0$, by the similar process of dealing with $j=2$ in the proof of Lemma \ref{lem6.1}, we know
\begin{align}\label{var0}
E_{(a)}(\varrho_{0}^{\frac{n}{2}}\tilde{\eta}_{2}(\varrho_{0} x))\geq \left(\frac{1}{2}-\frac{2}{n(p-2)\gamma}\right)M_{(a)}(\|\varrho_{0}^{\frac{n}{2}}\tilde{\eta}_{2}(\varrho_{0} x)\|_{2}),
\end{align}where $M_{(a)}(\|\varrho_{0}^{\frac{n}{2}}\tilde{\eta}_{2}(\varrho_{0} x)\|_{2})$ is defined by (\ref{Maf}).
From (\ref{daihuan2-0}), we see that (\ref{var0}) becomes
\begin{align}\label{var0--}
E_{(a)}(\varrho_{0}^{\frac{n}{2}}\tilde{\eta}_{2}(\varrho_{0} x))\geq \left(\frac{1}{2}-\frac{2}{n(p-2)\gamma}\right)M_{(a)}(\|\tilde{\eta}_{2}\|_{2}),
\end{align}where $M_{(a)}(\|\tilde{\eta}_{2}\|_{2})$ is defined by (\ref{Maf}).
 According to (\ref{daihuan3--}), for $\varrho=0$, we know $E_{(a)}(\varrho^{\frac{n}{2}}\tilde{\eta}_{2}(\varrho x))=0$. Combining this fact, (\ref{var0--}) and the continuity of $E_{(a)}(\varrho^{\frac{n}{2}}\tilde{\eta}_{2}(\varrho x))$ with respect to $\varrho$, we know that there exists a $\varrho^{*}\in(0,\varrho_{0})$ such that
\begin{align*}0\leq E_{(a)}(\left(\varrho^{*}\right)^{\frac{n}{2}}\tilde{\eta}_{2}(\varrho^{*} x))<\left(\frac{1}{2}-\frac{2}{n(p-2)\gamma}\right)M_{(a)}(\|\tilde{\eta}_{2}\|_{2}),
\end{align*} which means \begin{align*}0\leq E_{(a)}(\left(\varrho^{*}\right)^{\frac{n}{2}}\tilde{\eta}_{2}(\varrho^{*} x))<\left(\frac{1}{2}-\frac{2}{n(p-2)\gamma}\right)M_{(a)}(\|\left(\varrho^{*}\right)^{\frac{n}{2}}\tilde{\eta}_{2}(\varrho^{*} x)\|_{2}),
\end{align*} according to (\ref{daihuan2-0}). Meanwhile, from (\ref{daihuan2-0}) and (\ref{tildeeta-dayu2d}), we have \begin{align*}\|\left(\varrho^{*}\right)^{\frac{n}{2}}\tilde{\eta}_{2}(\varrho^{*} x)\|_{2}\geq 2d_{\mathcal{N}_{1,(a)}}.\end{align*}And noticing (\ref{daihuan3--2}), $S_{2,(a)}(\varrho_{0}^{\frac{n}{2}}\tilde{\eta}_{2}(\varrho_{0} x))=0$ and $\varrho^{*}<\varrho_{0}$, we have $S_{2,(a)}(\left(\varrho^{*}\right)^{\frac{n}{2}}\tilde{\eta}_{2}(\varrho^{*} x))>0$. Thus, we know $\left(\varrho^{*}\right)^{\frac{n}{2}}\tilde{\eta}_{2}(\varrho^{*} x)\in\mathcal{W}_{R_{2,(a)}}$.\\

\noindent{(4)\ Another version of example of the initial data belonging to the manifold $\mathcal{V}_{R_{2,(a)}}$.}

We use the function $\tilde{\eta}_{2}$ satisfying (\ref{tildeeta-dayu2d})
and $\||x|\tilde{\eta}_{2}\|_{2}<+\infty$ to show the example of the initial data belonging to the manifold $\mathcal{V}_{R_{2,(a)}}$.
According to (\ref{daihuan3--}), there exists a $\tilde{\varrho}>\varrho_{0}$ large enough such that $E_{(a)}(\left(\tilde{\varrho}\right)^{\frac{n}{2}}\tilde{\eta}_{2}(\tilde{\varrho} x))<0$. Combining this fact, (\ref{var0--}) and the continuity of $E_{(a)}(\varrho^{\frac{n}{2}}\tilde{\eta}_{2}(\varrho x))$ with respect to $\varrho$, we know that there exists a $\hat{\varrho}\in(\varrho_{0},\tilde{\varrho})$ such that \begin{align*}0\leq E_{(a)}(\left(\hat{\varrho}\right)^{\frac{n}{2}}\tilde{\eta}_{2}(\hat{\varrho} x))<\left(\frac{1}{2}-\frac{2}{n(p-2)\gamma}\right)M_{(a)}(\|\tilde{\eta}_{2}\|_{2}),
\end{align*} which means \begin{align*}0\leq E_{(a)}(\left(\hat{\varrho}\right)^{\frac{n}{2}}\tilde{\eta}_{2}(\hat{\varrho} x))<\left(\frac{1}{2}-\frac{2}{n(p-2)\gamma}\right)M_{(a)}(\|\left(\hat{\varrho}\right)^{\frac{n}{2}}\tilde{\eta}_{2}(\hat{\varrho} x)\|_{2})
\end{align*} according to (\ref{daihuan2-0}). Meanwhile, from (\ref{daihuan2-0}) and (\ref{tildeeta-dayu2d}), we have \begin{align*}\|\left(\hat{\varrho}\right)^{\frac{n}{2}}\tilde{\eta}_{2}(\hat{\varrho} x)\|_{2}\geq 2d_{\mathcal{N}_{1,(a)}}.\end{align*}And noticing (\ref{daihuan3--2}), $S_{2,(a)}(\varrho_{0}^{\frac{n}{2}}\tilde{\eta}_{2}(\varrho_{0} x))=0$ and $\hat{\varrho}>\varrho_{0}$, we have $S_{2,(a)}(\left(\hat{\varrho}\right)^{\frac{n}{2}}\tilde{\eta}_{2}(\hat{\varrho} x))<0$. Thus, we have $\left(\hat{\varrho}\right)^{\frac{n}{2}}\tilde{\eta}_{2}(\hat{\varrho}x)\in\mathcal{V}_{R_{2,(a)}}$.\\

 \noindent{(5)\ Examples of the initial data belonging to the manifolds $\mathcal{W}_{R_{3,(a)}}$, $\mathcal{V}_{R_{3,(a)}}$, and $\mathcal{W}_{R_{4,(a)}}$.}

Here, we need not to show the example of the initial data belonging to $\mathcal{V}_{R_{4,(a)}}$ since it is unapplicable for proving the finite time blowup in the high initial energy case and we do not consider this manifold.

We first focus on the case $j=3$.
 Find a function $\eta_{3}\in H^{1}(\mathbb{R}^{n})$ satisfying $\eta_{3}\neq0$,
\begin{align}
\left(\frac{1}{2}-\frac{2}{n(p-2)}\right)\hat{M}_{(a)}(\|\eta_{3}\|_{2})>\frac{2\tilde{c}}{n(p-2)}\|\eta_{3}\|^{2}_{2}+d_{_{\mathcal{N}_{3,(a)}}}, 0<\tilde{c}<1,\label{hanshu----0-1}
 \end{align}and $\||x|\eta_{3}\|_{2}<+\infty$, where
 $\hat{M}_{(a)}(\|\eta_{3}\|_{2})$ is defined by (\ref{Mhaf}).
Because of
\begin{align*}
\hat{M}_{(a)}(\|\eta_{3}\|_{2})\leq\frac{\hat{M}_{(a)}(\|\eta_{3}\|_{2})+\tilde{M}_{(a)}(\|\eta_{3}\|_{2})+\left|\hat{M}_{(a)}(\|\eta_{3}\|_{2})-\tilde{M}_{(a)}(\|\eta_{3}\|_{2})\right|}{2},
\end{align*}
where $\tilde{M}_{(a)}(\|\eta_{3}\|_{2})$ is defined by (\ref{Mtilea}),
we know (\ref{hanshu----0-1}) gives
\begin{align*}
d_{_{\mathcal{N}_{3,(a)}}}<\left(\frac{1}{2}-\frac{2}{n(p-2)}\right)\frac{\hat{M}_{(a)}(\|\eta_{3}\|_{2})+\tilde{M}_{(a)}(\|\eta_{3}\|_{2})
+\left|\hat{M}_{(a)}(\|\eta_{3}\|_{2})-\tilde{M}_{(a)}(\|\eta_{3}\|_{2})\right|}{2}.
\end{align*}
By the similar arguments for the manifolds $\mathcal{W}_{R_{2,(a)}}$ and $\mathcal{V}_{R_{2,(a)}}$, we can use the above function $\eta_{3}$ to give the examples of the initial data belonging to the manifolds $\mathcal{W}_{R_{3,(a)}}$ and $\mathcal{V}_{R_{3,(a)}}$.

Next, we deal with the case $j=4$.
Find a function $\eta_{4}\in H^{1}(\mathbb{R}^{n})$ satisfying $\eta_{4}\neq0$,
\begin{align}
\frac{1}{2}\left(\frac{2p\tilde{c}}{n(p-2) C_{1}\|\eta_{4}\|^{p-\frac{n(p-2)}{2}-2}_{2}}\right)^{\frac{4}{n(p-2)}}>\frac{2\tilde{c}}{n(p-2)}\|\eta_{4}\|^{2}_{2}+d_{_{\mathcal{N}_{4,(a)}}}, 0<\tilde{c}<1,\label{hanshu----0}
 \end{align}
and $\||x|\eta_{4}\|_{2}<+\infty$.
By the similar arguments for the manifold $\mathcal{W}_{R_{2,(a)}}$, we can use the above function $\eta_{4}$ to give the example of the initial data belonging to the manifold $\mathcal{W}_{R_{4,(a)}}$.
\end{remark}
\subsection{An imperfect criterion of global existence for high initial energy $K_{(a)}(\varphi_{0})\geq d_{{\mathcal{N}_{4},(a)}}$}\label{sec6.2}

 In this subsection, we investigate the cross-constrained variational combination $S_{4,(a)}(\varphi)$ without the mandatory term $\|\nabla\varphi\|^{2}_{2}$ to give a global existence criterion for $K_{(a)}(\varphi_{0})\geq d_{{\mathcal{N}_{4},(a)}}$. More precisely, for the initial data $\varphi_{0}\in\mathcal{W}_{R_{4,(a)}}$, we prove the corresponding solution $\varphi$ exists globally.
And we shall also point out that we fail to prove the finite time blowup by $S_{4,(a)}(\varphi)$ due to the lack of the term $\|\nabla\varphi\|^{2}_{2}$.

\begin{theorem}\label{thm6.-}(Global existence for high initial energy $K_{(a)}(\varphi_{0})\geq d_{\mathcal{N}_{4,(a)}}$)
Let $p$ satisfy (\ref{jingdianptiao}).
Then,
the corresponding solution $\varphi$ to problem (\ref{6.1}) exists globally provided $\varphi_{0}\in \mathcal{W}_{R_{4,(a)}}$.
\end{theorem}

\begin{proof}
In order to prove the global existence, we need to prove the invariance of the manifold $\mathcal{W}_{R_{4,(a)}}$. Thus, we divide the proof into two steps, i.e., Sept \textrm{I} (the manifold $\mathcal{W}_{R_{4,(a)}}$ is invariant) and Sept \textrm{II} (Global existence).\\
\textbf{Step \textrm{I}:\ The manifold \bm{$\mathcal{W}_{R_{4,(a)}}$} is invariant.}

According to Theorem 4.3.1 in Chapter 4 in \cite{rz1}\ (see also Theorem 2.1 in \cite{r14}), for the initial data $\varphi_{0}\in \mathcal{W}_{R_{4,(a)}}$, we know that there exists the unique solution $\varphi\in C([0,T_{{\rm{max}},(a)}),H^{1}(\mathbb{R}^{n}))$, which means the continuity of $S_{4,(a)}(\varphi(t))$ with respect to time $t$. We prove $\varphi(t)\in \mathcal{W}_{R_{4,(a)}}$ for $t\in(0,T_{{\rm{max}},(a)})$.
For $\varphi_{0}\in\mathcal{W}_{R_{4,(a)}}$, we know that
 \begin{align}
d_{\mathcal{N}_{4,(a)}}\leq& K_{(a)}(\varphi_{0})<R_{4,(a)}(\|\varphi_{0}\|_{2})\label{Ka-deguji4}
\end{align}and $S_{4,(a)}(\varphi_{0})>0$ hold. According to $\|\varphi_{0}\|^{2}_{2}=\|\varphi(t)\|^{2}_{2}$ and $K_{(a)}(\varphi_{0})=K_{(a)}(\varphi(t))$ for $t\in(0,T_{{\rm{max}},(a)})$ shown by Theorem 4.3.1 in Chapter 4 in \cite{rz1}\ (see also Theorem 2.1 in \cite{r14}), we know
 \begin{align}
d_{\mathcal{N}_{4,(a)}}\leq& K_{(a)}(\varphi(t))<R_{4,(a)}(\|\varphi(t)\|_{2})\label{Kadeguji2-}
\end{align}
for $t\in(0,T_{{\rm{max}},(a)})$. Next, we prove $S_{4,(a)}(\varphi(t))>0$ for all $t\in(0,T_{{\rm{max}},(a)})$. Arguing by contradiction, we suppose that there exists the first time $\tilde{\tau}_{4}\in(0,T_{{\rm{max}},(a)})$ to make $S_{4,(a)}(\varphi(\tilde{\tau}_{4}))=0$, i.e.,
\begin{align}
\tilde{c}\|\varphi(\tilde{\tau}_{4})\|^{2}_{2}=\frac{n(p-2)}{2p}\|\varphi(\tilde{\tau}_{4})\|^{p}_{p},
\ 0<\tilde{c}<1,\label{Ka4}
\end{align} and $S_{4,(a)}(\varphi(t))>0$ for $t\in(0,\tilde{\tau}_{4})$,
 and we know $\varphi(\tilde{\tau}_{4})\neq 0$ due to $\|\varphi_{0}\|^{2}_{2}=\|\varphi(\tilde{\tau}_{4})\|^{2}_{2}$ and $\varphi_{0}\neq0$.
Next, we estimate $K(\varphi(\tilde{\tau}_{4}))$.
By substituting (\ref{Ka4}) into $K_{(a)}(\varphi(\tilde{\tau}_{4}))$, we have
\begin{align}
K_{(a)}(\varphi(\tilde{\tau}_{4}))=&\frac{1}{2}\|\nabla\varphi(\tilde{\tau}_{4})\|^{2}_{2}
+\left(\frac{1}{2}-\frac{2\tilde{c}}{n(p-2)}\right)\|\varphi(\tilde{\tau}_{4})\|^{2}_{2}.\label{4Kadeguji}
\end{align}Here, from $0<\tilde{c}<1$ and the assumption (\ref{1.2}), we have $$\frac{1}{2}-\frac{2\tilde{c}}{n(p-2)}>0$$ in (\ref{4Kadeguji}).
Applying (\ref{GN-budengshi}) and (\ref{Ka4}), we have
\begin{align*}
&\frac{2p\tilde{c}}{n(p-2)}\|\varphi(\tilde{\tau}_{4})\|^{2}_{2}=\|\varphi(\tilde{\tau}_{4})\|^{p}_{p}\leq C_{1}\|\nabla\varphi(\tilde{\tau}_{4})\|^{\frac{n(p-2)}{2}}_{2}\|\varphi(\tilde{\tau}_{4})\|^{p-\frac{n(p-2)}{2}}_{2},
\end{align*}
i.e.,
\begin{align}
&\frac{2p\tilde{c}}{n(p-2)}\leq C_{1}\|\nabla\varphi(\tilde{\tau}_{4})\|^{\frac{n(p-2)}{2}}_{2}\|\varphi(\tilde{\tau}_{4})\|^{p-\frac{n(p-2)}{2}-2}_{2}.\label{Kap-3+Gnbudengshi}
\end{align}
Due to the fact $\|\varphi(\tilde{\tau}_{4})\|^{2}_{2}=\|\varphi_{0}\|^{2}_{2}$, we know (\ref{Kap-3+Gnbudengshi}) turns to
\begin{align*}
\frac{2p\tilde{c}}{n(p-2) C_{1}}\leq\|\nabla\varphi(\tilde{\tau}_{4})\|^{\frac{n(p-2)}{2}}_{2}\|\varphi_{0}\|^{p-\frac{n(p-2)}{2}-2}_{2},
\end{align*}
i.e.,
\begin{align}
\|\nabla\varphi(\tilde{\tau}_{4})\|^{2}_{2}\geq\left(\frac{2p\tilde{c}}{n(p-2) C_{1}\|\varphi_{0}\|^{p-\frac{n(p-2)}{2}-2}_{2}}\right)^{\frac{4}{n(p-2)}}.\label{Ka4Gnbudengshi2}
\end{align}
According to (\ref{Ka4Gnbudengshi2}) and the facts $K_{(a)}(\varphi(\tilde{\tau}_{4}))=K_{(a)}(\varphi_{0})$ and $\|\varphi(\tilde{\tau}_{4})\|^{2}_{2}=\|\varphi_{0}\|^{2}_{2}$, we know that (\ref{4Kadeguji}) turns to
\begin{align*}
K_{(a)}(\varphi_{0})\geq&\frac{1}{2}\left(\frac{2p\tilde{c}}{n(p-2) C_{1}\|\varphi_{0}\|^{p-\frac{n(p-2)}{2}-2}_{2}}\right)^{\frac{4}{n(p-2)}}+\left(\frac{1}{2}-\frac{2\tilde{c}}{n(p-2)}\right)\|\varphi_{0}\|^{2}_{2},
\end{align*}
which contradicts (\ref{Ka-deguji4}).
 Thus, we know $S_{4,(a)}(\varphi(t))>0$ for all $t\in(0,T_{{\rm{max}},(a)})$, which means $\varphi\in\mathcal{W}_{R_{4,(a)}}$.\\
\textbf{Step \textrm{II}:\ Global existence.}

For the initial data $\varphi_{0}\in \mathcal{W}_{R_{4,(a)}}$, we know the corresponding solution $\varphi\in \mathcal{W}_{R_{4,(a)}}$, i.e., $S_{4,(a)}(\varphi)>0$, meaning
\begin{align}
\tilde{c}\|\varphi\|^{2}_{2}>\frac{n(p-2)}{2p}\|\varphi\|^{p}_{p},\ 0<\tilde{c}<1,\label{S4adayuling}
\end{align}
according to Sept \textrm{I}. Then, by substituting (\ref{S4adayuling}) into $K_{(a)}(\varphi)$, we have
\begin{align*}
K_{(a)}(\varphi)\geq&\frac{1}{2}\|\nabla\varphi\|^{2}_{2}+\left(\frac{1}{2}-\frac{2 \tilde{c}}{n(p-2)}\right)\|\varphi\|^{2}_{2}\nonumber\\
\geq&\left(\frac{1}{2}-\frac{2 \tilde{c}}{n(p-2)}\right)\|\varphi\|^{2}_{H^{1}(\mathbb{R}^{n})},
\end{align*}
which means
\begin{align}\label{gujishiS4a}
K_{(a)}(\varphi_{0})
\geq&\left(\frac{1}{2}-\frac{2 \tilde{c}}{n(p-2)}\right)\|\varphi\|^{2}_{H^{1}(\mathbb{R}^{n})}
\end{align}
due to the fact $K_{(a)}(\varphi)=K_{(a)}(\varphi_{0})$. Here, from $0<\tilde{c}<1$ and the assumption (\ref{1.2}), we have $$\frac{1}{2}-\frac{2\tilde{c}}{n(p-2)}>0.$$
From (\ref{gujishiS4a}) and Theorem 2.1 in \cite{r14}, we know $T_{{\rm {max}},(a)}=+\infty$.
\end{proof}

Next, we explain the reason we fail to prove the finite time blowup by the cross-constrained variational functional $S_{4,(a)}(\varphi)$ for the case $K_{(a)}(\varphi_{0})\geq d_{\mathcal{N}_{4,(a)}}$.

\begin{remark}\label{rem6.3}(Fail to prove the finite time blowup by the cross-constrained variational functional $S_{4,(a)}(\varphi)$ without the term $\|\nabla\varphi\|^{2}_{2}$)
When we select and combine the energy structural terms to construct the cross-constrained variational structures for (\ref{6.1}), if we do not select the term $\|\nabla\varphi\|^{2}_{2}$, i.e., to consider the case of $S_{4,(a)}(\varphi)$, we fail to prove the finite time blowup for the high initial energy case $K_{(a)}(\varphi_{0})\geq d_{_{\mathcal{N}_{4,(a)}}}$, similar to the analysis in Remark \ref{rem5.3}. Although we can prove the invariance of the manifold $\mathcal{V}_{R_{4,(a)}}$ by the similar process of proving Lemma \ref{lem6.1}, we fail to prove the finite time blowup by using the manifold $\mathcal{V}_{R_{4,(a)}}$. The reason is that we need to apply the auxiliary functional
$J_{(a)}(t)=\||x|\varphi\|^{2}_{2}$ and verify $J''_{(a)}(t)<0$ in order to prove the finite time blowup. Then, for the cross-constrained variational functional $S_{4,(a)}(\varphi)$, because of the lack of the term $\|\nabla\varphi\|^{2}_{2}$, we shall meet the similar difficulty revealed in Remark \ref{rem4.2}. Thus we fail to prove the finite time blowup for the high initial energy case $K(\varphi_{0})\geq d_{_{\mathcal{N}_{4,(a)}}}$.
 \end{remark}

\section*{Acknowledgement}
Runzhang Xu was supported by the Hainan Provincial Natural Science Foundation of China (126MS0251), the National Natural Science Foundation of China (12271122), and the Fundamental Research Funds for the Central Universities. In the process of this work (October 2019-October 2026), Runzhang Xu visited the Institute of Mathematical Sciences, The Chinese University of Hong Kong, on several occasions (June 9-21, 2025; August 25-27, 2025; December 21-24, 2025; January 15-27, 2026; February 20-March 1, 2026; April 21-May 3, 2026; and May 20-31, 2026), during which he benefited greatly from valuable discussions with Professor Zhouping Xin.

\end{document}